\documentclass[11pt]{article}
\usepackage[utf8x]{inputenc}
\usepackage[margin=1.5in]{geometry}
\usepackage{microtype}
\usepackage[T1]{fontenc}
\usepackage{ae,aecompl}
\usepackage{times}
\usepackage{tikz-cd}
\usepackage{mathrsfs} 
\usepackage{verbatim,amsmath,amsthm,amsfonts,amssymb,latexsym,graphicx,mathtools,extpfeil,color,mathabx}
\usepackage{tikz}
\usepackage{epstopdf,pinlabel}
\usepackage[all]{xy}
\usepackage{graphicx}
\usepackage{caption}
\usepackage{tabularx}
\usepackage{subcaption}
\usepackage{slashed}
\DeclareMathAlphabet{\mathpzc}{OT1}{pzc}{m}{it}
\usepackage[colorlinks,pagebackref,hypertexnames=false]{hyperref} \usepackage[alphabetic,backrefs,msc-links]{amsrefs}

\usepackage{aliascnt}
\numberwithin{equation}{section}

\newcommand{\bh}{{\bf h}}

\newcommand{\bC}{{\bf C}}

\newcommand{\bF}{{\bf F}}

\newcommand{\bP}{{\bf P}}
\newcommand{\bQ}{{\bf Q}}
\newcommand{\bR}{{\bf R}}

\newcommand{\bZ}{{\bf Z}}

\newcommand{\cB}{\mathcal{B}}
\newcommand{\cC}{\mathcal{C}}

\newcommand{\cF}{\mathcal{F}}

\newcommand{\cH}{\mathcal{H}}

\newcommand{\cM}{\mathcal{M}}

\newcommand{\cS}{\mathcal{S}}

\newcommand{\sB}{\mathscr{B}}
\newcommand{\sC}{\mathscr{C}}
\newcommand{\sD}{\mathscr{D}}

\newcommand{\sF}{\mathscr{F}}
\newcommand{\sG}{\mathscr{G}}

\newcommand{\sM}{\mathscr{M}}

\newcommand{\sP}{\mathscr{P}}

\newcommand{\sR}{\mathscr{R}}
\newcommand{\sS}{\mathscr{S}}
\newcommand{\sT}{\mathscr{T}}

\newcommand{\sX}{\mathscr{X}}

\newcommand{\fa}{{\mathfrak a}}

\newcommand{\fd}{{\mathfrak d}}

\newcommand{\fg}{{\mathfrak g}}

\newcommand{\fj}{{\mathfrak j}}

\newcommand{\fm}{{\mathfrak m}}

\newcommand{\fp}{{\mathfrak p}}

\newcommand{\fu}{{\mathfrak u}}

\newcommand{\fC}{{\mathfrak C}}

\newcommand{\fE}{{\mathfrak E}}

\newcommand{\fI}{{\mathfrak I}}

\newcommand{\fJ}{{\mathfrak J}}

\newcommand{\Z}{\bZ}
\newcommand{\Q}{\bQ}
\newcommand{\R}{\bR}
\newcommand{\C}{\bC}

\newcommand{\su}{\mathfrak{su}}

\DeclareMathOperator{\coker}{coker}

\DeclareMathOperator{\ind}{ind}

\DeclareMathOperator{\tr}{tr}

\renewcommand{\det}{\operatorname{det}}

\renewcommand{\P}{\bP}
\renewcommand{\epsilon}{\varepsilon}

\def\({\mathopen{}\left(}
\def\){\right)\mathclose{}}
\def\<{\mathopen{}\left<}
\def\>{\right>\mathclose{}}

\newcommand{\iso}{\cong}

\usepackage{multicol, color}

\definecolor{gold}{rgb}{0.85,.66,0}
\definecolor{cherry}{rgb}{0.9,.1,.2}
\definecolor{burgundy}{rgb}{0.8,.2,.2}
\definecolor{orangered}{rgb}{0.85,.3,0}
\definecolor{orange}{rgb}{0.85,.4,0}
\definecolor{olive}{rgb}{.45,.4,0}
\definecolor{lime}{rgb}{.6,.9,0}
\definecolor{green}{rgb}{.2,.7,0}
\definecolor{grey}{rgb}{.4,.4,.2}
\definecolor{brown}{rgb}{.4,.3,.1}

\def\makeautorefname#1#2{\AtBeginDocument{\expandafter\def\csname#1autorefname\endcsname{#2}}}

\newcommand{\mynewtheorem}[2]{
  \newaliascnt{#1}{equation}          
  \newtheorem{#1}[#1]{#2}
  \aliascntresetthe{#1}
  \makeautorefname{#1}{#2}
}
\mynewtheorem{theorem}{Theorem}
\mynewtheorem{prop}{Proposition}
\mynewtheorem{cor}{Corollary}
\mynewtheorem{construction}{Construction}
\mynewtheorem{lemma}{Lemma}
\mynewtheorem{conjecture}{Conjecture}

\numberwithin{substep}{step}
\makeautorefname{step}{Step}
\makeautorefname{substep}{Step}

\numberwithin{subcase}{case}
\makeautorefname{case}{Case}
\makeautorefname{subcase}{case}

\theoremstyle{remark}
\mynewtheorem{remark}{Remark}

\theoremstyle{definition}
\mynewtheorem{definition}{Definition}
\mynewtheorem{example}{Example}
\mynewtheorem{exercise}{Exercise}
\mynewtheorem{convention}{Convention}
\newtheorem*{convention*}{Convention}
\newtheorem*{conventions*}{Conventions}
\mynewtheorem{question}{Question}
        
\makeautorefname{chapter}{Chapter}
\makeautorefname{section}{Section}
\makeautorefname{subsection}{Section}
\makeautorefname{subsubsection}{Section}

\newcommand\hrI{{\widehat { I}}}
\newcommand\hrC{{\widehat { C}}}
\newcommand\fhrC{{\widehat { \mathfrak C}}}
\newcommand\crI{{\widecheck { I}}}
\newcommand\crC{{\widecheck { C}}}
\newcommand\fcrC{{\widecheck {\mathfrak C}}}
\newcommand\brI{{\overline { I}}}
\newcommand\brC{{\overline {C}}}
\newcommand\fbrC{{\overline {\mathfrak C}}}
\newcommand\diamd{\diamond}
\newcommand\hinv{h}

\title{Equivariant aspects of singular instanton Floer homology}
\author{Aliakbar Daemi\thanks{The work of AD was supported by NSF Grant DMS-1812033.} \hspace{1cm} Christopher Scaduto}
\date{}

\newcommand{\Addresses}{{
  \bigskip
  \footnotesize
  Aliakbar Daemi, \textsc{Department of Mathematics, Washington University in St. Louis, One Brookings drive, Room 207A,
  St. Louis, MO 63130}\par\nopagebreak
  \textit{E-mail address}: \texttt{adaemi@wustl.edu}
  \vspace{.2cm}

Christopher Scaduto, \textsc{Department of Mathematics, University of Miami, 1365 Memorial Dr 515, Coral Gables, FL 33124}\par\nopagebreak
  \textit{E-mail address}: \texttt{cscaduto@miami.edu}
}}

\begin{document}
\maketitle

{\abstract{We associate several invariants to a knot in an integer homology 3-sphere using $SU(2)$ singular instanton gauge theory. There is a space of framed singular connections for such a knot, equipped with a circle action and an equivariant Chern--Simons functional, and our constructions are morally derived from the associated equivariant Morse chain complexes. In particular, we construct a triad of groups analogous to the knot Floer homology package in Heegaard Floer homology, several Fr\o yshov-type invariants which are concordance invariants, and more. The behavior of our constructions under connected sums are determined. We recover most of Kronheimer and Mrowka's singular instanton homology constructions from our invariants. Finally, the ADHM description of the moduli space of instantons on the 4-sphere can be used to give a concrete characterization of the moduli spaces involved in the invariants of spherical knots, and we demonstrate this point in several examples.}}

\newpage

\hypersetup{linkcolor=black}
\tableofcontents

\newpage

%!TEX root = main.tex
\section{Introduction}

Instanton Floer homology \cite{floer:inst1} and Heegaard Floer homology \cite{os-1} provide two powerful invariants of 3-manifolds, each of which have knot-theoretic variations: singular instanton Floer homology \cite{KM:YAFT} and knot Floer homology \cite{os-knot, rasmussen-thesis}. These knot invariants share many formal properties: they are both functorial with respect to surface cobordisms, they each have skein exact triangles, and it is even conjectured that some versions of the theories agree with one another \cite[Conjecture 7.25]{km-sutures}. Despite their similarities, each of the two theories have some advantage over the other.

On the one hand, singular instanton Floer homology is more directly related to the fundamental group of the knot complement. For example, this Floer homology can be used to show that the knot group of any non-trivial knot admits a non-abelian representation into the Lie group $SU(2)$ \cite{km-propertyp, km-sutures}. On the other hand, knot Floer homology currently has a richer algebraic structure which can be used to obtain invariants of closed 3-manifolds obtained by surgery on a knot \cite{os-intsurgeries, os-rationalsurgeries}. Moreover, knot Floer homology is more computable, and in fact has combinatorial descriptions \cite{mos-knot-comb, os-borderedknot}.

A natural question is whether there is a refinement of singular instanton Floer homology that helps bridge the gap between the two theories. An important step in this direction was recently taken by Kronheimer and Mrowka \cite{km-barnatan}. The main goal of the present paper is to propose a different approach to this question. Like \cite{KM:YAFT, KM:unknot}, we construct invariants of knots in integer homology spheres using singular instantons. However, in contrast to those constructions, we do not avoid reducibles, and instead exploit them to derive equivariant homological invariants. As we explain below, the relevant symmetry group in this setting is $S^1$.

The knot invariants in this paper recover various versions of singular instanton Floer homology in the literature, including all of the ones constructed in \cite{KM:unknot, km-rasmussen, km-barnatan}. Moreover, some of the structures of our invariants do not seem to have any obvious analogues in the context of Heegaard Floer invariants. For instance, a {\it filtration by the Chern-Simons functional} and a Floer homology group {\it categorifying the knot signature} can be derived from the main construction of the present work.

\subsection*{Motivation}
The basic idea behind the main construction of the present paper is to construct a configuration space of singular connections with an $S^1$-action. Let $K$ be a knot in an integer homology sphere $Y$ and fix a basepoint on $K$. Consider the space of connections on the trivial $SU(2)$-bundle $E$ over $Y$ which are singular along $K$ and such that the holonomy along any meridian of $K$ is asymptotic to a conjugate of
\begin{equation}\label{i}
  \left[
  \begin{array}{cc}
   	i&0\\
	0&-i\\
   \end{array}
   \right]\in SU(2)
\end{equation}
as the size of meridian goes to zero. (See Section \ref{sec:defns} for a more precise review of the definition of such singular connections.) A {\it framed singular connection} is a singular connection with a trivialization of $E$ at the basepoint of $K$ such that the holonomy of the connection along a meridian of $K$ at the basepoint is asymptotic to \eqref{i} (rather than just conjugate to it). The space of automorphisms of $E$ acts freely on the space of framed singular connections and we denote the quotient by $\widetilde \sB(Y,K)$. There is an $S^1$-action on $\widetilde \sB(Y,K)$ given by changing the framing at the basepoint. 

An important feature of this $S^1$-action is that the stabilizers of elements in $\widetilde \sB(Y,K)$ are not all the same. The element $-1\in S^1$ acts trivially on $\widetilde \sB(Y,K)$. Thus the action factors through $S^1\cong S^1/\{\pm1\}$ which acts freely on a singular framed connection in $\widetilde \sB(Y,K)$ unless the underlying singular connection is {\it $S^1$-reducible}, namely, it respects a decomposition of $E$ into a sum of two (necessarily dual) complex line bundles. Although framed connections do not appear in the sequel, our constructions are motivated by the above $S^1$-action and the interactions between framed singular connections with different stabilizers. An important source of inspiration for the authors was a similar story for non-singular connections which is developed in \cite{donaldson-book, froyshov, miller}.

\subsection*{$\cS$-complexes associated to knots}

The fundamental object that we associate to a knot $K\subset Y$ in an integer homology 3-sphere is a chain complex $(\widetilde C_*(Y,K),\widetilde d)$ which is a module over the graded ring $\Z[\chi]/(\chi^2)$, where $\chi$ has degree $1$. The ring $\Z[\chi]/(\chi^2)$ should be thought of as the homology ring of $S^1$ where the ring structure is induced by the multiplication map. In particular, one expects a similar structure arising from the singular chain complex of a topological space with an $S^1$-action. In our setup, singular homology is replaced with Floer homology. In fact, the chain complex $\widetilde C_*(Y,K)$ we associate to a knot $K$ decomposes as:
\begin{equation}\label{cS-com}
  \widetilde C_*(Y,K)=C_*(Y,K)\oplus C_{*}(Y,K)[1]\oplus \Z.
\end{equation}
Here $C_*(Y,K)$ is $\Z/4$-graded, $C_{*}(Y,K)[1]$ is the same complex as $C_*(Y,K)$ with the grading shifted up by $1$, and $\Z$ is in grading $0$. The action of $\chi$ on $\widetilde C_*(Y,K)$ maps the first factor by the identity to the second factor, and maps the remaining two factors to zero. We call a chain complex over the graded ring $\Z[\chi]/(\chi^2)$ of the form \eqref{cS-com} an {\emph{$\cS$-complex}}. Although the complex $\widetilde C_*(Y,K)$ depends on some auxiliary choices (e.g. a Riemannian metric), the chain homotopy type of $\widetilde C_*(Y,K)$ in the category of $\cS$-complexes is an invariant of $(Y,K)$. (See Section \ref{sec:tilde} for more details.) In particular, the homology
\[
	\widetilde I_\ast(Y,K):=H(\widetilde C_*(Y,K),\widetilde d)
\]
is an invariant of the pair $(Y,K)$. We will see below that this homology group is naturally isomorphic to Kronheimer and Mrowka's $I^\natural (Y,K)$ from \cite{KM:unknot}.

By applying various algebraic constructions to the $\cS$-complex $\widetilde C_*(Y,K)$ we can recover various knot invariants and also construct new ones. One of the invariants we recover is a counterpart of Floer's instanton homology for integer homology spheres, and may be compared to a version of Collin and Steer's orbifold instanton homology from \cite{collin-steer}. The differential $\widetilde d$ gives rise to a differential $d$ on $C_*(Y,K)$, and we write $I_\ast(Y,K)$ for the homology of the complex $(C_*(Y,K),d)$. The Euler characteristic of $I_\ast(Y,K)$ was essentially computed by Herald \cite{herald}, generalizing the work of Lin \cite{lin}. In summary, we have the following:

\begin{theorem}\label{singular-Floer-Euler}
	Let $K\subset Y$ be a knot embedded in an integer homology 3-sphere $Y$. The $\Z/4$-graded abelian group $I_\ast(Y,K)$ is an invariant of the equivalence class of the knot $(Y,K)$. Its Euler characteristic satisfies
	\[
		\chi\left(I_\ast(Y,K)\right)= 4\lambda(Y) + \frac{1}{2}\sigma(K)
	\]
	where $\lambda(Y)$ is the Casson invariant of $Y$ and $\sigma(K)$ is the signature of the knot $K\subset Y$.
\end{theorem}

Another chain complex that can be constructed from $(\widetilde C_*(Y,K),\widetilde d)$ is given by:
\[
  \widehat C_*(Y,K):=\widetilde C_*(Y,K)\otimes_{\Z}\Z[x]\hspace{1cm} \widehat d:= -\widetilde d+x\cdot \chi
\]
The homology of this complex can be regarded, morally, as the $S^1$-equivariant homology of $\widetilde \sB(Y,K)$. This equivariant complex inherits a $\Z/4$-grading from the tensor product grading of $\widetilde C_*(Y,K)$ and $\Z[x]$, where the latter has $x^i$ in grading $-2i$. The homology of $(\widehat C_*(Y,K),\widehat d)$ gives a counterpart of $HFK^-$ in the context of singular instanton Floer homology.
\begin{theorem}
	The homology of the complex $(\widehat C_*(Y,K),\widehat d)$, denoted by $\hrI_*(Y,K)$, is a topological 
	invariant of the pair $(Y,K)$ as a $ \Z/4$-graded $\Z[x]$-module. Moreover, one can construct $\Z/4$-graded $\Z[x]$-modules 
	$\crI_*(Y,K)$ and $\brI_\ast(Y,K)$ from $(\widehat C_*(Y,K),\widehat d)$ which are invariants of the pair
	$(Y,K)$. These modules fit into two exact triangles:

\noindent\begin{minipage}{.5\textwidth}
\begin{equation}\label{top-equiv-triangle-intro}
	\begin{tikzcd}[column sep=1ex, row sep=5ex, fill=none, /tikz/baseline=-10pt]
\crI_*(Y,K) \arrow{rr} & & \hrI_*(Y,K) \arrow{dl} \\
& \brI_\ast(Y,K) \arrow{ul} & \\
\end{tikzcd}
\end{equation}
\end{minipage}
\begin{minipage}{.5\textwidth}
\begin{equation}\label{top-equiv-triangle-intro-2}
	\begin{tikzcd}[column sep=1ex, fill=none, row sep=5ex, fill=none, /tikz/baseline=-10pt]
\hrI_*(Y,K) \arrow{rr} & & \hrI_*(Y,K) \arrow{dl} \\
& \widetilde I_\ast(Y,K) \arrow{ul} & \\
\end{tikzcd}
\end{equation}
\end{minipage}
	
	\noindent The top arrow in \eqref{top-equiv-triangle-intro-2} is induced by multiplication by $x$. 
	Furthermore, $\brI_*(Y,K)$ is isomorphic to $\Z[\![x^{-1},x]$ as a $\Z[x]$-module.
\end{theorem}
\noindent
The invariants $\crI_*(Y,K)$, $\brI_\ast(Y,K)$ and $\widetilde I_\ast(Y,K)$ are analogues of the Heegaard Floer knot homology groups $HFK^+(Y,K)$, $HFK^\infty(Y,K)$ and $\widehat {HFK}(Y,K)$. The exact triangles in \eqref{top-equiv-triangle-intro} and \eqref{top-equiv-triangle-intro-2} are also counterparts of similar exact triangles for the knot Floer homology groups in Heegaard Floer theory.

\begin{remark}
Recently, Li introduced $KHI^-$ in \cite{Li:KHIm} as another approach to define the instanton counterpart of $HFK^-$ using sutured manifolds. We expect that $KHI^-$ for a knot $K$ in an integer homology $3$-sphere $Y$ can be recovered from $\widetilde{C}_\ast(Y,K)$ using an algebraic construction similar to what appears in Subsection \ref{subsec:rasmussen}. $\diamd$
\end{remark}

There are even further refinements of $(\widetilde C_*(Y,K),\widetilde d)$ which can be constructed following ideas contained in \cite{KM:YAFT, AD:CS-Th}. The refinements come from equivariant local coefficient systems on the framed configuration space $\widetilde \sB(Y,K)$ that can be used to define twisted versions of the complex $(\widetilde C_*(Y,K),\widetilde d)$. The universal local coefficient system $\Delta$ that we consider is defined over the two-variable Laurent polynomial ring
\[
	\mathscr{R}:=\Z[U^{\pm1},T^{\pm 1}],
\]
and it gives rise to an $\cS$-complex $(\widetilde C_*(Y,K;\Delta),\widetilde d)$. Roughly, the variable $T$ is related to the holonomy of flat connections around the knot and the ``monopole charge'' of instantons, while the variable $U$ is related to the Chern--Simons functional on flat connections and the topological energy, or action, of instantons. All of the invariants in this paper may be derived from $(\widetilde C_*(Y,K;\Delta),\widetilde d)$, assuming one keeps track of all of its relevant structures. (See Section \ref{sec:loccoeffs} for more details.)

If $\mathscr{S}$ is an $\mathscr{R}$-algebra, we can change our local coefficient system by a base change and define an $\cS$-complex $(\widetilde C_*(Y,K;\Delta_\mathscr{S}),\widetilde d)$ over the ring $\mathscr{S}$, and its chain homotopy type as an $\cS$-complex over $\sS$ is again an invariant of the knot. We then obtain, for example, an $\mathscr{S}$-module $I(Y,K;\Delta_\mathscr{S})$ and an $\mathscr{S}[x]$-module $\hrI(Y,K;\Delta_\mathscr{S})$ which are also knot invariants. Evaluation of $T$ and $U$ at $1$ defines an $\mathscr{R}$-algebra structure on $\Z$, and the associated $\cS$-complex recovers the untwisted complex $(\widetilde C_*(Y,K),\widetilde d)$. Another case of interest is the base change given by $\mathscr{T}=\Z[T^{\pm 1}]$, which is an $\mathscr{R}$-algebra by evaluation of $U$ at $1$, and this gives the $\cS$-complex $(\widetilde C_*(Y,K;\Delta_\mathscr{T}),\widetilde d)$.

\subsection*{A connected sum theorem}

Given two pairs $(Y,K)$ and $(Y',K')$ of knots in integer homology spheres, we may form another such pair $(Y\# Y',K\# K')$ by taking the connected sum of 3-manifolds and knots. It is natural to ask if the $\cS$-complex associated to $(Y\# Y',K\# K')$ can be related to those of $(Y,K)$ and $(Y',K')$. The following theorem answers this question affirmatively, and should be compared with the connected sum theorem for instanton Floer homology of integer homology spheres \cite{fukaya}. In fact, our proof is inspired by the treatment of Fukaya's connected sum theorem in \cite[Section 7.4]{donaldson-book}.

\begin{theorem}\label{thm:connectedsumlocceoff} There is a chain homotopy equivalence of $\Z/4$-graded $\cS$-complexes
\begin{equation*}
	\widetilde C(Y\# Y',K\# K') \simeq  \widetilde C(Y,K)\otimes_\Z \widetilde C(Y',K').
\end{equation*}
More generally, in the setting of local coefficients, we have a chain homotopy equivalence of $\Z/4$-graded $\cS$-complexes over $\mathscr{R}=\Z[U^{\pm 1}, T^{\pm 1}]$:
\begin{equation*}
	\widetilde C(Y\# Y',K\# K';\Delta) \simeq  \widetilde C(Y,K;\Delta)\otimes_\mathscr{R} \widetilde C(Y',K';\Delta).
\end{equation*}
\end{theorem}
We remark that the tensor product of two $\cS$-complexes is naturally an $\cS$-complex and refer the reader to Section \ref{sec:equivtheories} for more details. The above theorem allows us to recover the invariants of $I(Y\# Y',K\# K';\Delta_\mathscr{S})$ and  $\hrI(Y\# Y',K\# K';\Delta_\mathscr{S})$ from $ \widetilde C(Y,K;\Delta)$ and $\widetilde C(Y',K';\Delta)$. In particular, if the ring $\mathscr{S}[x]$ is a PID, then there is K\"unneth formula relating $\hrI(Y\# Y',K\# K';\Delta_\mathscr{S})$ to $\hrI(Y,K;\Delta_\mathscr{S})$ and $\hrI(Y',K';\Delta_\mathscr{S})$.

\subsection*{Recovering invariants of Kronheimer and Mrowka}

Kronheimer and Mrowka have defined several versions of singular instanton Floer homology groups. There are the {\it reduced} invariants $I^\natural(Y,K)$,  first defined as abelian groups in \cite{KM:unknot}, and later defined using local coefficients as modules over the ring $\bF[T_1^{\pm1}, T_2^{\pm1}, T_3^{\pm1}]$ where $\bF$ is the field of two elements \cite{km-barnatan}. There are also the {\it unreduced} invariants $I^\#(Y,K)$, first defined as abelian groups in \cite{KM:unknot}, then defined using local coefficients as modules over the ring $\Q[T^{\pm 1}]$ in \cite{km-rasmussen}, and finally as modules over the ring $\bF[T_0^{\pm1},T_1^{\pm1}, T_2^{\pm1}, T_3^{\pm1}]$ in \cite{km-barnatan}. 

The definition of each of these invariants follows a similar pattern. To avoid working with reducible singular connections, one firstly picks $(Y_0,K_0)$ such that there is no reducible singular connection associated to this pair. This assumption requires working in a set up that allows $K_0$ to be a link or more generally a web \cite{km-tait}, equipped with a bundle of structure group $SO(3)$, instead of $SU(2)$. Then the invariant of the pair $(Y,K)$ is defined by applying Floer theoretical methods to the configuration space of singular connections on the pair $(Y\#Y_0,K\#K_0)$. A variation of our connected sum theorem allows us to prove the following theorem. (For more details, see Section \ref{sec:kmgroups}.)

\begin{theorem} \label{relation-KM}
	All the different versions of the invariants $I^\natural(Y,K)$ and $I^\#(Y,K)$ can be recovered from the homotopy type of the chain complex $(\widehat C_*(Y,K;\Delta),\widehat d)$ over $\mathscr{R}[x]$. For instance, $I^\natural(Y,K)$, defined as in  \cite{KM:unknot},
	is isomorphic to $H(\widetilde{C}(Y,K),\widetilde{d})$:
	\begin{equation*}
		 I^\natural(Y,K) \cong \widetilde I (Y,K).
	\end{equation*}
	Furthermore, $I^\natural(Y,K)$, with local coefficients defined as in \cite{km-barnatan}, is isomorphic to 
	\begin{equation}
		\hrI(Y,K;\Delta_\mathscr{T})\otimes_{\mathscr{T}[x]} \bF[T_1^{\pm1}, T_2^{\pm1}, T_3^{\pm1}]\label{eq:inatequivintro}
	\end{equation}
	where the $\mathscr{T}[x]$-module structure on $\bF[T_1^{\pm1}, T_2^{\pm1}, T_3^{\pm1}]$ is given by mapping $T\in \sT=\Z[T^{\pm 1}]$ to $T_1$ and $x$ to the element
	\[
		P:= T_1T_2T_3 + T_1^{-1}T_2^{-1} T_3 + T_1^{-1}T_2T_3^{-1} + T_1T_2^{-1}T_3^{-1}.
	\]
	Similarly, $I^\#(Y,K)$, with local coefficients defined as in \cite{km-barnatan}, is isomorphic to 
	\begin{equation}
		 \hrI(Y,K;\Delta_\mathscr{T})\otimes_{\mathscr{T}[x]} \bF[T_{0}^{\pm 1}, T_1^{\pm1}, T_2^{\pm1}, T_3^{\pm1}] ^{\oplus 2}\label{eq:isharpequivintro}
	\end{equation}
	 where the $\mathscr{T}[x]$-module structure on $\bF[T_0^{\pm 1}, T_1^{\pm1}, T_2^{\pm1}, T_3^{\pm1}]$ sends $T\mapsto T_0$, $x\mapsto P$. 
\end{theorem}

The isomorphisms between the local coefficients versions of $I^\natural(Y,K)$ and $I^\#(Y,K)$ and the modules \eqref{eq:inatequivintro} and \eqref{eq:isharpequivintro} are given more precisely in Corollaries \ref{equiv-natural} and \ref{equiv-sharp}.

Although it is not clear from its definition, Theorem \ref{relation-KM} suggests that the most recent version of $I^\natural(Y,K)$ from \cite{km-barnatan} can be regarded as an $S^1$-equivariant theory, and similar results hold for the other versions of singular instanton Floer homology in an appropriate sense. The knot homology  $I^\natural(Y,K)$ is defined in \cite{km-barnatan} only for characteristic $2$ rings because of a feature of instanton Floer homology for webs. On the other hand, the above theorem suggests that this restriction is not essential. The above theorem also asserts that $I^\natural(Y,K)$ is given by applying a base change to a module defined over the subring $\bF[T_1^{\pm1}, P]$ of $\bF[T_1^{\pm1}, T_2^{\pm1}, T_3^{\pm1}]$. Thus $I^\natural(Y,K)$ is essentially a module over this smaller ring, and the $\bF[T_1^{\pm1}, T_2^{\pm1}, T_3^{\pm1}]$-module structure is obtained by applying a formal algebraic construction. Furthermore, while $I^\natural(Y,K)$ as defined in \cite{km-barnatan} only has a $\Z/2$-grading, our invariant \eqref{eq:inatequivintro} comes equipped with a $\Z/4$-grading.

\subsection*{Spherical knots and ADHM construction}

For a spherical knot $K$, the moduli spaces of singular instantons involved in the definition of the chain complex $(\widetilde C_*(Y,K;\Delta),\widetilde d)$ can be characterized in terms of the moduli spaces of (non-singular) instantons on $S^4$. In particular, it is reasonable to expect that the ADHM description of instantons on $S^4$ can be used to directly compute the $\cS$-complex $(\widetilde C_*(Y,K;\Delta),\widetilde d)$. To manifest this idea, let $K_{p,q}$ be the $(p,q)$ two-bridge knot whose branched double cover is the lens space $L(p,q)$. Using the results of \cite{austin, furuta-invariant} we can compute part of the $\cS$-complex $(\widetilde C_*(K_{p,q};\Delta),\widetilde d)$. In particular, a specialization of our instanton homology for $K_{p,q}$ recovers a version of instanton homology for the lens space $L(p,q)$ defined by Sasahira in \cite{sasahira-lens} (see also \cite{furuta-invariant}), which takes the form of a $\Z/4$-graded $\bF$-vector space $I_\ast(L(p,q))$. For the following, let $\bF_4:=\bF[x]/(x^2+x+1)$ be the field with four elements. (See Subsection \ref{subsec:sasahira} for more details.)

\begin{theorem}\label{thm:sasahiracomparison}
	There is an isomorphism of $\Z/4$-graded vector spaces over $\bF_4$
	\[
	I_\ast(K_{p,q};\Delta_{\bF_4}) \cong I_\ast(L(p,-q))\otimes \bF_4
	\]
	where the local system $\Delta_{\bF_4}$ is obtained from $\Delta_{\sT}\otimes \bF$ via the base change sending $T$ to $x$.
\end{theorem}

\subsection*{Concordance invariants}

We say a knot $K$ in an integer homology sphere $Y$ is {\emph{homology concordant}} to a knot $K'$ in another integer homology sphere $Y'$ if there is an integer homology cobordism $W$ from $Y$ to $Y'$ and a properly and smoothly embedded cylinder $S$ in $W$ such that $\partial S=-K\cup K'$. In particular, a classical concordance for knots in $S^3$ produces a homology concordance. The collection of knots modulo this relation defines an abelian group $\mathcal C_\Z$, where addition is given by taking the connected sum of the knots within the connected sum of the ambient homology spheres. The $\cS$-complex $\widetilde C(Y,K;\Delta)$ can be used to define various algebraic objects invariant under homology concordance.

The simplest version of our concordance invariants is an integer-valued homomorphism from the homology concordance group, and its definition is inspired by Fr\o yshov's homomorphism $h$ from the homology cobordism group to the integers \cite{froyshov}, a predecessor to the Heegaard Floer $d$-invariant of Ozsv\'{a}th and Szab\'{o} \cite{os-dinv} and Fr\o yshov's monopole $h$-invariant \cite{froyshov-monopole}. In fact, we obtain a homology concordance homomorphism for each $(\widetilde C_*(Y,K;\Delta_\mathscr{S}),\widetilde d)$ that depends on the choice of an $\mathscr{R}$-algebra $\mathscr{S}$ and is denoted by $\hinv_\mathscr{S}(Y,K)\in \Z$. Its basic properties are summarized as follows.

\begin{theorem}\label{thm:hinvproperties}
	Let $\sS$ be an integral domain $\sR$-algebra. The invariant $\hinv_{\sS}$ satisfies:
	\begin{enumerate}
		\item[{\emph{(i)}}] $\hinv_{\sS}(Y\#Y',K\#K') = \hinv_{\sS}(Y,K) + \hinv_{\sS}(Y',K')$.
		\item[{\emph{(ii)}}] Suppose $(W,S):(Y,K)\to (Y',K')$ is a cobordism of pairs such that $H_1(W;\Z)=0$, the homology class of $S$ is divisible by $4$, 
		and the double cover of $W$ branched over $S$ is negative definite. Then we have:
		\[
			\hinv_{\sS}(Y,K) \leqslant \hinv_{\sS}(Y',K').
		\]
\end{enumerate}
In particular, $\hinv_{\sS}$ induces a homomorphism from the homology concordance group to the integers, which in turn induces a homomorphism from the smooth concordance group of knots in the 3-sphere to the integers.
\end{theorem}

The cobordism $(W,S)$ appearing in (ii) is an example of what we call a {\emph{negative definite pair}} in the sequel. When $K$ is a knot in the 3-sphere, we simply write $\hinv_{\sS}(K)$ for the invariant $\hinv_{\sS}(S^3,K)$, and similarly for the other invariants we define. The two choices of $\sS$ that we focus on are $\Z$ and $\sT=\Z[T^{\pm 1}]$. For the former choice we simply write $h$:
\[
	h(Y,K) := h_\Z(Y,K)
\]
The two invariants $h$ and $h_\sT$ take on different values for simple knots in the 3-sphere. Some of our computations from Section \ref{sec:computations} are summarized as follows.

\begin{theorem}\label{thm:hinvcomps}
We have the following computations for the invariants $\hinv$ and $\hinv_{\sT}$:
	\begin{enumerate}
		\item[\rm{(i)}] For any two-bridge knot we have $\hinv=0$.
		\item[\rm{(ii)}] For the positive (right-handed) trefoil we have $\hinv_{\sT}=1$.
		\item[\rm{(iii)}] For the positive $(3,4)$ and $(3,5)$ torus knots we have $\hinv=1$.
		\item[\rm{(iv)}] For the following families of torus knots, we have $h=0$:
		        \begin{align*}
 	       	(p,2pk+2), & \qquad k\geqslant 1,\;\; p\equiv \phantom{\pm}1 \pmod 2\\
 	       	(p,2pk\pm (2-p)), & \qquad k \geqslant 1,\;\; p \equiv \pm 1 \pmod 4
        \end{align*}
	\end{enumerate}
\end{theorem}
Although Theorem \ref{thm:hinvcomps} computes $\hinv_{\sT}$ only for one knot, we expect that $\hinv_{\sT}(K)$ for a general knot can be evaluated in terms of classical invariants of $K$. We will address this claim in a forthcoming work.

\begin{remark}
	Recently, a version of the invariant $\hinv_{\sS}(K)$ and a 1-parameter family variation of it is used in \cite{SFO} to study a Furuta-Ohta type invariant for tori embedded in a 4-manifold with the integral homology of $S^1\times S^3$. $\diamd$
\end{remark}

A refinement of $\hinv_\mathscr{S}(Y,K)$ has the form of a nested sequence of ideals of $\mathscr{S}$,
\[
  \dots \subseteq J_{i+1}^\mathscr{S}(Y,K) \subseteq J_{i}^\mathscr{S}(Y,K)\subseteq J_{i-1}^\mathscr{S}(Y,K)\subseteq \dots \subseteq\mathscr{S}.
\]
This sequence depends only on the homology concordance of $(Y,K)$, and recovers the invariant $\hinv_{\sS}(Y,K)$. Its basic properties are summarized as follows.

\begin{theorem}\label{thm:jideals}
	The nested sequence of ideals $\{J_i^{\sS}(Y,K)\}_{i\in \Z}$ in $\sS$ satisfy: 
	\begin{enumerate}
		\item[{\emph{(i)}}]  $J_i^{\sS}(Y,K)\cdot J_{j}^{\sS}(Y',K') \subset J_{i+j}^\sS(Y\# Y', K\# K')$
		\item[{\emph{(ii)}}] If $(W,S):(Y,K)\to (Y',K')$ is a negative definite pair, $J^{\sS}_i(Y,K)\subset J^{\sS}_i(Y',K')$.
		\item[{\emph{(iii)}}] $\hinv_{\sS}(Y,K) = \max \left\{ i\in \Z : J^{\sS}_i(Y,K)\neq 0\right\}$.
\end{enumerate}
\end{theorem}

All of the constructions discussed thus far are derived from the chain homotopy type of the $\cS$-complex $\widetilde C(Y,K;\Delta_\mathscr{S})$. However,  there is more structure to exploit on this complex, coming from a filtration induced by the Chern--Simons functional. (The terminology that we use for $\cS$-complexes with this extra structure is an {\emph{enriched $\cS$-complex}}. We refer the reader to Subsection \ref{enriched-cS} for a more precise definition.) 

The Chern--Simons filtration can also be used to define homology concordance invariants. To illustrate this, we associate $\Gamma^{R}_{(Y,K)}:\Z\to \R^{\geq 0}\cup \infty$ to a pair $(Y,K)$ by adapting the construction of  \cite{AD:CS-Th} to our setup. Here $R$ is any integral domain which is an algebra over the ring $\Z[T^{\pm 1}]$. The function $\Gamma_{(Y,K)}^R$ depends only on the homology concordance class of $(Y,K)$. Some other properties are mentioned in the following theorem. For a slightly stronger version see Theorem \ref{Gamma-Y-K}.

\begin{theorem}
	Let $(Y,K)$ be a knot in an integer homology 3-sphere.
	\begin{itemize}
		\item[\emph{(i)}] The function $\Gamma_{(Y,K)}^R$ is an invariant of the homology concordance class of $(Y,K)$.
		\item[\emph{(ii)}] For each $i\in\Z$, we have $\Gamma_{(Y,K)}^R(i)<\infty$ if and only if $i\leqslant \hinv_R(Y,K)$.
		\item[\emph{(iii)}] For each $i\in \Z$, if $\Gamma_{(Y,K)}^R(i)\not\in\{ 0, \infty\}$, then it is congruent {\emph{(mod $\Z$)}} to the value of the Chern-Simons functional at an irreducible singular flat $SU(2)$ connection on $(Y,K)$.
	\end{itemize}
\end{theorem}

A {\it traceless} $SU(2)$-representation for a pair $(Y,K)$ is a representation of $\pi_1(Y\backslash K)$ into $SU(2)$ such that a (and hence any) meridian of $K$ is mapped to an element of $SU(2)$ with vanishing trace. For instance, the unknot has a unique conjugacy class of such representations which, of course, has an abelian image. Similarly, for a given homology concordance $(W,S):(Y,K)\to (Y',K')$, a traceless representation is a homomorphism of $\pi_1(W\backslash S)$ into $SU(2)$ such that a meridian of $S$ is mapped to a traceless element of $SU(2)$. In particular, any traceless representation of the pair $(Y,K)$ (resp. $(W,S)$) induces an $SO(3)$-representation of the orbifold fundamental group of the $\Z/2$-orbifold structure on $Y$ (resp. $W$) with singular locus $K$ (resp. $S$).

The following is a corollary of the invariance of $\Gamma_{(Y,K)}^\mathscr{S}$ under homology concordances. (Compare to the case for integer homology 3-spheres in \cite[Theorem 3]{AD:CS-Th}.)

\begin{cor}
	Let $(W,S):(Y,K)\to (Y',K')$ be a homology concordance with $\Gamma_{(Y,K)}^R\neq \Gamma_{(S^3,U)}^R$. Then there exists a traceless representation of $(W,S)$ that extends non-abelian traceless representations of $(Y,K)$ and $(Y',K')$.
	In particular, the images of $\pi_1(Y\backslash K)$ and $\pi_1(Y'\backslash K')$ in $\pi_1(W\backslash S)$ are non-abelian.
\end{cor}

Note that the condition $\Gamma_{(Y,K)}^R\neq \Gamma_{(S^3,U)}^R$ is satisfied if $\hinv_R(Y,K) \neq 0$, examples for which can be found in Theorem \ref{thm:hinvcomps} (and more examples may be generated by additivity).

\subsection*{Further discussion}

Functoriality of the $\cS$-complex $\widetilde C(Y,K;\Delta)$ with respect to homology concordances plays the key role in proving the desired properties of the above concordance invariants. In fact, if $(W,S):(Y,K)\to (Y',K')$ is a negative definite pair, then there is an induced morphism $\widetilde C(W,S;\Delta):\widetilde C(Y,K;\Delta) \to \widetilde C(Y',K';\Delta)$ in the category of $\cS$-complexes, which preserves the Chern-Simons filtration, in the sense of enriched $\cS$-complexes. This notion of functoriality implies that the chain complexes and homology groups constructed from $\widetilde C(Y,K;\Delta)$, such as $\hrC(Y,K;\Delta)$, $\hrI(Y,K;\Delta)$ and $I(Y,K;\Delta)$, are functorial with respect to such negative definite pairs.

The main reason that we develop the functoriality for this limited family of cobordisms is to avoid working with moduli spaces of singular instantons that have reducible elements that are not cut out transversely. To achieve a regular moduli space, one cannot simply perturb these connections, due to the well-known phenomenon that equivariant transversality does not hold generically. However, there is enough evidence to believe that at least the equivariant theory $\hrC(Y,K;\Delta)$ (and hence the homology theory $\hrI(Y,K;\Delta)$) is functorial with respect to more general cobordisms. We plan to return to this issue elsewhere.

In addition to extending the theory to include more general cobordisms, the authors also expect that an Alexander grading may be constructed on the homology groups studied here, perhaps adapting the ideas used in \cite{KM:alexander}. 

In \cite{km-concordance,km-rasmussen}, Kronheimer and Mrowka introduce various concordance invariants out of the singular instanton homology groups $I^\#(S^3,K)$ and $I^\natural(S^3,K)$. In fact, they show that their invariants can be used to obtain lower bounds for the slice genus, unoriented slice genus, and unknotting number. Due to our limited functoriality, at this point we cannot examine our concordance invariants in this generality here. We hope that our conjectured functoriality for $\hrC(Y,K;\Delta)$ allows us to achieve this goal. In light of Theorem \ref{relation-KM}, we believe that this extended functoriality would be useful to answer the following:

\begin{question}\label{question:kmconc}
	Is there any relationship between the concordance invariants in \cite{km-concordance,km-rasmussen} and the ideals $\{J_{i}^\mathscr{S}(Y,K)\}_{i\in \Z}$ appearing in Theorem \ref{thm:jideals}?
\end{question}

In Subsection \ref{subsection:remarks-conc} we propose an approach to construct yet another family of concordance invariants which we argue should recover Kronheimer and Mrowka's invariants in \cite{km-concordance}. Moreover, if $K$ is a knot in $S^3$ satisfying the following slice genus identity:
\begin{equation}\label{slice-genus-cond}
  g_4(K)=-\sigma(K)/2,
\end{equation}
such as the right-handed trefoil, then the functoriality developed in this paper allows us to carry out the proposed construction. In particular, we show that the concordance invariants obtained from the unreduced theory $I^\#(S^3,K)$ and the reduced theory $I^\natural(S^3,K)$ in \cite{km-concordance} are essentially equal to each other, a relation which is not obvious from the constructions of \cite{km-concordance}. For the knots satisfying \eqref{slice-genus-cond}, we also give a partial answer to Question \ref{question:kmconc} by providing some relations between the concordance invariants of \cite{km-concordance} and the ideals $\{J_{i}^\mathscr{S}(Y,K)\}_{i\in \Z}$.\\

{\it Organization.} The necessary background on the gauge theory of singular connections, which was developed by Kronheimer and Mrowka, is reviewed in Section \ref{sec:defns}. In particular, we devote Subsection \ref{subsec:red} to analyzing reducible singular ASD connections, which play an important role in our construction. The definition of negative definite pair arises naturally from this analysis. The geometrical setup of Section \ref{sec:defns} allows us to define the $\cS$-complex $(\widetilde C(Y,K),\widetilde d)$ in Section \ref{sec:tilde}. Some technical constructions involving holonomy maps of singular connections used in Section 3 are explained in Appendix A at the end of the paper.

We make a digression in Section \ref{sec:equivtheories} to develop the homological algebra of $\cS$-complexes. In Subsections \ref{equiv-model-1} and \ref{small-model}, we give two models for the chain complexes underlying the equivariant homology groups $\hrI$, $\crI$ and $\brI$.  We also define tensor products (needed for Theorem \ref{thm:connectedsumlocceoff}) and duals of $\cS$-complexes in Section \ref{sec:equivtheories}. We use these operations to define a {\it local equivalence group} following the construction of \cite{stoffregen}. The algebraic framework for the ideals $J_{i}^\mathscr{S}(Y,K)$ is defined in Subsection \ref{subsec:ideals}. Next, in Section \ref{sec:eq-I}, the algebraic constructions of Section \ref{sec:equivtheories} are used to define equivariant Floer homology groups $\hrI_*(Y,K)$, $\crI_*(Y,K)$ and  $\brI_\ast(Y,K)$ and the concordance invariant $\hinv(Y,K)$.

Theorem \ref{thm:connectedsumlocceoff} on invariants of connected sums is proved in Section \ref{sec:consum}. In Section \ref{sec:loccoeffs} we explain how one can obtain additional algebraic structures on $\widetilde C(Y,K)$ using local coefficient systems. Here the general concordance invariants $h_\sS(Y,K)$, $\{ J_i^\sS(Y,K)\}$ and $\Gamma^R_{(Y,K)}$ are defined. Theorem \ref{relation-KM} is discussed in detail in Section \ref{sec:kmgroups}. In the final section of the paper, we focus on computations, where proofs of Theorems \ref{thm:sasahiracomparison} and \ref{thm:hinvcomps} are given.\\

{\it Acknowledgements.} The authors would like to thank the Simons Center for Geometry and Physics where this project started, as well as the organizers of the 2019 PCMI Research Program ``Quantum Field Theory and Manifold Invariants'' where some of the work was carried out. Discussions with Mike Miller Eismeier were instrumental in improving the treatment of the foundations of the established constructions. In particular, the appendix at the end of this paper developed after conversations with him. The authors would also like to thank Peter Kronheimer, Tom Mrowka and Nikolai Saveliev for helpful discussions. The authors thank Hayato Imori, Kouki Sato and the anonymous referee for pointing out several mistakes and providing comments on an earlier version of the paper.

\newpage

%!TEX root = main.tex
\section{Background on singular $SU(2)$ gauge theory}\label{sec:defns}

In this section we survey the relevant aspects of singular $SU(2)$ gauge theory. The objects we begin with are $SU(2)$ connections on a homology 3-sphere which are singular along a knot, with limiting holonomies of order 4 around small meridional loops. Most of the definitions and results are due to Kronheimer and Mrowka \cite{KM:YAFT}. The main difference in our setup is the presence of a distinguished flat reducible $\theta$. In particular, we modify the holonomy perturbation scheme of \cite{KM:YAFT} so as to not disturb $\theta$, which is isolated and non-degenerate in the moduli space of singular flat connections. 

Next, we consider ASD connections on cobordisms of homology spheres which are singular along an embedded cobordism of knots. We start with the product case, and then move to the arbitrary case. To any such connection, we can associate an elliptic operator, called the ASD operator. We study the index of such operators for reducible singular connections on a cobordism, which motivates the definition of {\it negative definite pairs}. We also use the ASD operator to define an absolute $\Z/4$-grading for irreducible critical points using $\theta$, analogous to Floer's grading in the non-singular setting.

We review a formula due to Herald \cite{herald} that expresses the signed count of singular flat $SU(2)$ connections in terms of the Casson invariant of the homology 3-sphere and the signature of the knot. Finally, we review the data needed to fix orientations on moduli spaces of singular ASD connections.

\subsection{Singular $SU(2)$ connections}\label{sec:connections}

Let $Y$ be an integer homology 3-sphere, and $K\subset Y$ a smoothly embedded knot. Fix a rank 2 Hermitian vector bundle $E$ over $Y$ with structure group $SU(2)$, with a reduction $E|_K=L\oplus L^\ast$ over the knot for some Hermitian line bundle $L$. Note that $E$ and $L$ are necessarily trivializable bundles. The pair $(Y,K)$ determines a smooth 3-dimensional $\Z/2$-orbifold $\check{Y}$, with underlying topological space $Y$ and singular locus $K$.

Choose a regular neighborhood of $K\subset Y$ diffeomorphic to $S^1\times D^2$, in which $K$ is identified with $S^1\times\{0\}$. Let $(r,\theta)\in D^2$ be polar coordinates normal to $K$. Define
\[
	\lambda_0 = b(r)\frac{1}{4}id\theta
\]
where $b(r)$ is a bump function equal to $1$ for $r<1/2$ and zero for $r>1$. Then $\lambda_0$ is a 1-form on $Y\setminus K$ with values in $i\R=\fu(1)$. 
Using trivializations of $E$ and $L$ that respect the splitting $E|_K=L\oplus L^\ast$, we view $B_0:=\lambda_0\oplus\lambda_0^\ast$ as a connection on $E|_{Y\setminus K}$. The holonomy of this connection is of order 4 around small meridional loops of $K$.

The adjoint bundle of $E$, written $\fg_E$, is the subbundle of $\text{End}(E)$ consisting of skew-Hermitian endomophisms, and has structure group $SO(3)=\text{Aut}(\su(2))$. The singular connection $B_0$ induces a connection on $\fg_E$ denoted $B_0^\text{ad}$. It has holonomy of order 2 around small meridional loops of $K$, and so it extends to an orbifold connection $\check{B}_0^\text{ad}$ on an orbifold bundle $\check{\fg}_E$ over $\check{Y}$ whose underlying topological bundle is $\fg_E$.

Fix $k\geqslant 3$, and choose a Riemannian metric $g_o$ on $Y$ with cone angle $\pi$ along $K$, which induces a Riemannian metric on the orbifold $\check{Y}$. The space of $SU(2)$ connections on $Y$ with singularities of order 4 along $K$ is defined as follows:
\begin{equation*}
	\sC(Y,K) = B_0 + \check{L}_{k,B^\text{ad}_0}^2(\check{Y}; \check{\Lambda}^\ast \otimes \check{\fg}_E) \label{eq:connections}
\end{equation*}
The function space on the right-hand side consists of sections $b$ of $\check{\Lambda}^\ast \otimes \check{\fg}_E$ such that $\nabla^i b$ are $L^2$ for $0\leqslant i \leqslant k$, where the orbifold connection $\nabla$ is defined using the Levi-Civita derivative induced by $g_o$ and the covariant derivative induced by the adjoint of $B_0$. We have written $\check{\Lambda}^\ast$ for the orbifold bundle of exterior forms on $\check{Y}$.

The gauge transformation group $\sG(Y,K)$ consists of the orbifold automorphisms $g$ of the bundle $E$ such that $\nabla_{B
_0} g\in \check{L}_{k,B^\text{ad}_0}^{2}$. Write $\sB(Y,K)=\sC(Y,K)/\sG(Y,K)$ for the quotient configuration space. The homotopy type of $\sG(Y,K)$ is the same as that of the space of continuous automorphisms of $E$ that preserve each factor of $E|_K=L\oplus L^\ast$, and this latter group may be identified with the space of continuous maps $g:Y\to SU(2)$ such that $g(K)\subset U(1)$. We have an isomorphism
\begin{equation}
	d:\pi_0(\sG(Y,K)) \to \Z\oplus \Z, \quad d(g) = (k,l)\label{eq:gaugedegree}
\end{equation}
With the above homotopy identifications understood, the number $k$ is the degree of the map $g:Y\to SU(2)$, and $l$ is the degree of the restriction $g|_K:K\to U(1)$.

\subsection{The Chern-Simons functional and flat connections}
There is defined a Chern-Simons functional $\text{CS}:\sC(Y,K)\to \R$, uniquely characterized up to a constant as the functional whose formal $L^2$ gradient is given by
\begin{equation*}
	(\text{grad}\;\text{CS})_{B}=\frac{1}{4\pi^2}\ast F_{B}\label{eq:csgrad}
\end{equation*}
for each $B\in\sC(Y,K)$, where $F_{B}$ is the curvature of $B$. For a gauge transformation $g\in \sG(Y,K)$ with homotopy invariants $d(g)=(k,l)$ as in \eqref{eq:gaugedegree}, we have
\[
	\text{CS}(B) - \text{CS}(g(B)) = 2k + l.
\]
We thus obtain a circle-valued functional $\text{CS}:\sB(Y,K)\to \R/\Z$, defined up to the addition of a constant, denoted by the same name. The critical points of $\text{CS}$ are flat connections on $E|_{Y\setminus K}$ with prescribed holonomy around meridians of $K$. We denote by $\fC\subset\mathscr{B}(Y,K)$ the set of gauge equivalence classes of flat connections.

By choosing a basepoint in $Y\setminus K$ and taking holonomy around based loops in $Y\setminus K$, we obtain a homeomorphism between $\fC$ and the traceless $SU(2)$ character variety,
\begin{equation}
	\mathscr{X}(Y,K) := \left\{ \rho: \pi_1(Y\setminus K)\to SU(2): \tr \rho(\mu) =0 \right\}/SU(2).\label{eq:charvar}
\end{equation}
Here $\mu$ is any meridional loop around $K$, and the action of $SU(2)$ is by conjugation. This correspondence does not depend on the chosen basepoint. 

There is a distinguished class $\theta\in\sB(Y,K)$, the {\emph{(flat) reducible}}, characterized as the orbit of flat connections in $\sC(Y,K)$ corresponding to the unique conjugacy class of representations in \eqref{eq:charvar} that factor through $H_1(Y\setminus K;\Z)\cong \Z\cdot\mu$. We call a class of flat connections $[B]\in\fC$ {\emph{non-degenerate}} if the Hessian of $\text{CS}$ at $B$ is non-degenerate. The following result is implied by \cite[Lemma 3.13]{KM:YAFT}. See also Proposition \ref{reg-dbl-cover}.

\begin{prop}\label{prop:nondeg}
	The reducible $\theta\in\fC$ is isolated and non-degenerate.
\end{prop}

\noindent Flat connections in the class $\theta$ have $\sG(Y,K)$-stabilizer isomorphic to $U(1)$. Indeed, gauge stabilizers arise as centralizers of holonomy groups, and the holonomy group of a connection in the class $\theta$ is conjugate to the subgroup $\{\pm 1, \pm i \}\subset SU(2)$, with centralizer $U(1)$.

We note that $\theta$ is not the only gauge equivalence class of reducible connections: any connection in $\sC(Y,K)$ compatible with a reduction of $E|_{Y\setminus K}$ into a sum of line bundles also has stabilizer $U(1)$. However, among such reducibles, the connections in the orbit $\theta$ are the only ones that are flat. As the other reducibles are not relevant to the sequel, we feel justified in calling $\theta$ {\emph{the}} reducible, with ``flat'' being implicit.

We see now that $\fC$ may be written as the disjoint union $\{\theta\}\sqcup \fC^\text{irr}$ where $\fC^\text{irr}$ consists of flat {\emph{irreducible}} connection classes, each with $\sG(Y,K)$-stabilizer $\{\pm 1\}$. Finally, we may fix the ambiguity in the definition of $\text{CS}:\sB(Y,K)\to \R/\Z$ by declaring that $\text{CS}(\theta)=0$.

\subsection{The flip symmetry}\label{subsec:flip}

There is an involution $\iota$ on the configuration space $\sB(Y,K)$, defined as follows. Consider a flat $\Z/2$ bundle-with-connection $\xi$ over $Y\setminus K$ with holonomy $-1$ around meridians of $K$, corresponding to a generator of $H^1(Y\setminus K;\Z/2)$. Then for $[B]\in\sB(Y,K)$ we have
\begin{equation}
	\iota [B]  = [B\otimes \xi] \label{eq:flip}
\end{equation}
The involution $\iota$ is the ``flip symmetry'' considered, for example, in \cite[Section 2(iv)]{km-embedded-i}. (The flip symmetry there is in fact in the 4-dimensional setting, but is defined similarly.) Although the involution $\iota$ will not play an essential role in most of the sequel, it inevitably appears in the structure of our examples in Section \ref{sec:computations}. In a forthcoming work, we give a more systematic study of the interaction of $\iota$ with the $S^1$-equivariant theories introduced throughout this paper.

The flip symmetry $\iota$ restricts to an involution on the critical set $\fC$. In terms of the character variety $\sX(Y,K)$, the action of $\iota$ is induced by the assignment which sends a representation $\rho:\pi_1(Y\setminus K)\to SU(2)$ to the representation $\chi_\mu \cdot \rho$, where $\chi_\mu$ is the unique non-trivial representation $\chi_\mu:\pi_1(Y\setminus K)\to \{\pm 1\}$, again corresponding to a generator of $H^1(Y\setminus K;\Z/2)$. (In particular, note $\chi_\mu$ itself does not define a class in $\sX(Y,K)$.) From this it is clear that $\iota (\theta)=\theta$. More generally, the following elementary lemma is observed in \cite{PS}, where this involution on the character variety is studied:

\begin{lemma}\label{lemma:binarydihedral}
	An element of the critical set $\fC$ is fixed by the flip symmetry $\iota$ if and only if its corresponding representation class in $\sX(Y,K)$ has image in $SU(2)$ conjugate to a binary dihedral subgroup.
\end{lemma}

\subsection{Perturbing the critical set}\label{sec:perturb}

In general, the critical set $\fC^\text{irr}$ is degenerate. To fix this, we add a small perturbation to the Chern-Simons functional. Kronheimer and Mrowka use holonomy perturbations in \cite[Section 3]{KM:YAFT} modelled after those used in the non-singular setting, see \cite{taubes, donaldson:orientations, floer:inst1}. Although not essential, we would like to have a class of perturbations that leave the reducible alone, just as in Floer's instanton homology for integer homology 3-spheres.

We first describe the pertubations used in \cite[Section 3]{KM:YAFT}. Let $q:S^1\times D^2\to Y\setminus K$ be a smooth immersion. Let $s$ and $z$ be the coordinates of $S^1=\R/\Z$ and $D^2$, respectively. Consider the bundle $G_E\to Y$ whose sections are gauge transformations in $\sG(Y,K)$, and for each $B\in\sC(Y,K)$ and $z\in D^2$ let $\text{Hol}_{q(-,z)}(B)\in (G_E)_{q(0,z)}$ be the holonomy of $B$ around the corresponding loop based at $q(0,z)$. As $z$ varies we obtain a section $\text{Hol}_q(B)$ of the bundle $q^\ast (G_E)$ over the disk $D^2$.

Suppose we have a tuple of such immersions, $\mathbf{q}=(q_1,\ldots,q_r)$, with the property that they all agree on $[-\eta,\eta]\times D^2$ for some $\eta > 0$. The bundles $q^\ast_j(G_E)$ are canonically isomorphic over this neighborhood, and for each $B\in\sC(Y,K)$, the holonomy maps define a section $\text{Hol}_\mathbf{q}(B):D^2\to q_1^\ast(G_E^r)$. Choose a smooth function $h:SU(2)^r\to \R$ invariant under the diagonal adjoint action on the factors. Then $h$ also defines a function on $q_1^\ast(G^r_E)$. Choose a non-negative 2-form $\mu$ supported on the interior of $D^2$ with integral 1. Define
\[
	f_\mathbf{q}(B) = \int_{D^2} h(\text{Hol}_\mathbf{q}(B))\mu
\]
Kronheimer and Mrowka call such functions {\emph{cylinder functions}}. The space of perturbations they consider is a Banach space completion of sums of cylinder functions where $\mathbf{q}$ and $h$ run over a fixed dense set. This Banach space is called $\sP$.

When adding a cylinder function to the Chern-Simons functional, the reducible $\theta$ may be perturbed.  To avoid this, consider the point in $SU(2)^r$ obtained by choosing a representative connection for $\theta$ and taking its holonomy around the loops $q_1,\ldots,q_r$. The orbit of this point under the conjugation action of $SU(2)$ defines a subset $O_\theta\subset SU(2)^r$ independent of the choice of the representative for $\theta$. Note that if $h:SU(2)^r\to \R$ is constant on a neighborhood of $O_\theta$ then any associated cylinder function $f_\mathbf{q}$ which is small leaves the reducible $\theta$ unperturbed, isolated and non-degenerate. We may form a Banach space $\sP'\subset \sP$ of such perturbations.  We write $\fC_\pi$ for the critical set of the Chern-Simons functional perturbed by $\pi \in \sP$.

\begin{prop}\label{prop:perturb1}
	There is a residual subset of $\sP'$ such that for all sufficiently small $\pi$ in this subset, the set of irreducible critical points $\fC_\pi^\text{\emph{irr}}$ of the perturbed Chern-Simons functional is finite and non-degenerate, and $\fC_\pi=\{\theta\}\sqcup \fC_\pi^\text{\emph{irr}}$, where $\theta$ remains non-degenerate.
\end{prop}
\begin{proof}[Sketch of the proof]
	This is analogue of \cite[Proposition 3.10]{KM:YAFT} and the the proof is similar. The essential point is that for any compact 
	finite dimensional submanifold $M$ of the space of irreducibles in $\sB(Y,K)$ the restrictions of the perturbation functions in 
	$\sP'$ form a dense subset of $C^\infty(M)$.
\end{proof}

\begin{remark}
In \cite[Section 5.5]{donaldson-book}, Donaldson uses a different class of holonomy perturbations to deform the ordinary, non-singular flat equation. As mentioned in \cite[Section 3]{KM:YAFT}, this approach may also be adapted to the singular setting. The above perturbations are modified as follows: each immersion $q_j$ from above is assumed to be an embedding, but we no longer require that the $q_j$'s agree on $[-\eta,\eta]\times D^2$; and we now require that $h:SU(2)^r\to \R$ is invariant under the adjoint action on each factor separately. Following the discussion in \cite[Section 5.5]{donaldson-book}, we may proceed just as in the non-singular case, ensuring that the reducible remains unmoved and non-degenerate. $\diamd$
\end{remark}

\begin{remark}
Although not needed in the sequel, we may actually perturb the Chern-Simons functional, keeping the reducible isolated and non-degenerate, and achieving non-degeneracy at the remaining elements of the critical set, by a perturbation which is invariant with respect to the flip symmetry $\iota$. If $\mathbf{q}=(q_1,\ldots,q_r)$ is as above, the involution $\iota$ either fixes the holonomy of a singular connection along a loop $q_i$ or changes it by a sign, depending on whether the homology class of $q_i$ is an even or odd multiple of the meridian of $K$. This induces an action of $\Z/2$ on $SU(2)^r$ and we consider functions $h:SU(2)^r\to \R$ which are additionally invariant with respect to this $\Z/2$-action. The induced function on $\sB(Y,K)$ is invariant with respect to the action of $\iota$. We may proceed as above to define a space $\sP''$ and an analogue of Proposition \ref{prop:perturb1} holds for this more constrained space of perturbations. Indeed, in this new set up we must show that for $M$ a compact $\iota$-invariant submanifold of irreducibles in $\sB(Y,K)$, the restrictions of functions in $\sP''$ is dense in the space of $\iota$-invariant smooth functions on $M$; this can be done as in \cite{Was:G-Morse}. $\diamd$
\end{remark}

\subsection{Gradient trajectories and gradings}\label{subsec:gr}

Solutions to the formal $L^2$ gradient flow of the Chern-Simons functional satisfy the anti-self-duality (ASD) equations on the cylinder $Z=\R\times Y$. To describe the latter, we consider connections $A=B + C dt$ on $Z$, where $B$ is a $t$-dependent singular $SU(2)$ connection on $(Y,K)$, and $C$ is a $t$-dependent section in $\check{L}^2_k(\check{Y};\check{\fg}_E)$. Then the 4-dimensional ASD equations on $\R\times Y$, perturbed by a holonomy perturbation $\pi$, are
\begin{equation}
	F_A^+ + \widehat{V}_\pi(A) =0.\label{eq:asd}
\end{equation}
Here $\widehat{V}_\pi(A)$ is the projection of $dt\wedge V_\pi(A)$ to the self-dual bundle-valued 2-forms, where $V_\pi$ is the pull-back of the gradient of the perturbation $\pi$ of the Chern-Simons functional. Solutions $A$ to \eqref{eq:asd} are called {\emph{(singular) instantons}} on the cylinder.

Let $\pi\in\sP'$ be a perturbation such that $\fC_\pi^\text{irr}$ is finite and non-degenerate. Consider irreducible classes $\alpha_i=[B_i]\in \fC_\pi^\text{irr}$ for $i=1,2$. Let $A_0$ be a connection on $\R\times Y$ as written above, which agrees with pullbacks of $B_1$ and $B_2$ for large negative and large positive $t\in \R$, respectively. The connection $A$ determines a path $\gamma:\R\to \sB(Y,K)$, constant outside of a compact set. From this we have a relative homotopy class $z=[\gamma]\in \pi_1(\sB(Y,K);\beta_1,\beta_2)$. We then have a space of connections
\[
	\sC_{\gamma}(Z,S;B_1,B_2) = \left\{ A: A-A_0 \in \check{L}_{k,A^\text{ad}_0}^2(\check{Z}; \check{\fg}_{F}\otimes\check{\Lambda}^1)\right\}
\]
where $S=\R\times K$, and $\check{Z}$ is the $\Z/2$-orbifold with the underlying space $Z$ and singular locus $S$. The corresponding gauge transformation group $\sG_\gamma(Y,K;B_1,B_2)$ consists of orbifold automorphisms $g$ of $E$ with $\nabla_{A_0} g\in \check{L}^2_{k,A_0}$. We then have the quotient space $\sB_z(Y,K;\alpha_1,\alpha_2)=\sC_{\gamma}(Z,S;B_1,B_2)/\sG_{\gamma}(Z,S;B_1,B_2)$.

The associated moduli space of ASD connections on the cylinder is defined as
\[
	M_z(\alpha_1,\alpha_2) = \left\{[A] \in \sB_z(Y,K;\alpha_1,\alpha_2): F_A^+  + \widehat{V}_\pi(A) =0  \right\}
\]
We write $M(\alpha_1,\alpha_2)$ for the disjoint union of the $M_z(\alpha_1,\alpha_2)$ as $z$ ranges over all relative homotopy classes from $\alpha_1$ to $\alpha_2$. There is an $\R$-action on $M(\alpha_1,\alpha_2)$ induced by translation in the $\R$-factor of the cylinder $\R\times Y$. This action is free on non-constant trajectories; we write $\breve{M}(\alpha_1,\alpha_2)$ for the subset of the quotient $M(\alpha_1,\alpha_2)/\R$ which excludes the constant trajectories. We have a relative grading
\[
	\text{gr}_z(\alpha_1,\alpha_2) = \text{ind}(\sD_{A}) = \text{v.dim}M_z(\alpha_1,\alpha_2) \in \Z
\]
Here $A$ is any connection in $\sC_\gamma(Z,S;B_1,B_2)$, for example $A=A_0$; and the elliptic operator $\sD_A = -d_A^\ast\oplus (d_A^+  + D\widehat{V}_\pi ) $ is the linearized (perturbed) ASD operator with gauge fixing:
\begin{equation}
	\sD_A : \check{L}^2_{k,A^\text{ad}}(\check{Z}; \check{\fg}_F\otimes \check{\Lambda}^1) \to \check{L}^2_{k-1,A^\text{ad}}(\check{Z}; \check{\fg}_F\otimes (\check{\Lambda}^0\oplus \check{\Lambda}^+)) \label{eq:linasdopcyl}
\end{equation}
We have written $\text{v.dim}M_z(\alpha_1,\alpha_2)$ for the virtual dimension of the moduli space; when $d_A^+  + D\widehat{V}_\pi $ is surjective, we say that $[A]\in M_z(\alpha_1,\alpha_2)$ is a {\emph{regular}} solution, and when this is true for all $[A]\in M_z(\alpha_1,\alpha_2)$, we say that the moduli space is regular. When $M_z(\alpha_1,\alpha_2)$ is regular, it is a smooth manifold of dimension $\text{gr}_z(\alpha_1,\alpha_2)$. We write $M(\alpha_1,\alpha_2)_d$ for the disjoint union of moduli spaces $M_z(\alpha_1,\alpha_2)$ with $\text{gr}_z(\alpha_1,\alpha_2)=d$, and $\breve{M}(\alpha_1,\alpha_2)_{d-1} =M(\alpha_1,\alpha_2)_d/\R$. In general, our conventions will be compatible with the rule that a subscript $d\in \Z$ in the notation for a moduli space is equal to its virtual dimension.

Now we slightly diverge from \cite{KM:YAFT} and consider moduli spaces with reducible flat limits. This is done exactly as in \cite{floer:inst1}. When one or both of $\alpha_i$ are reducible, then in the definition of $M_z(\alpha_1,\alpha_2)$ we consider classes $[A]$ such that $A-A_0$ is in
\begin{equation}
	\phi\check{L}^2_{k,A^{\text{ad}}_0}(\check{Z};\check{\fg}_F\otimes \check{\Lambda}^1),\label{eq:expweightspace}
\end{equation}
a weighted Sobolev space. The weight $\phi:Z\to \R$ is a smooth function equal to $e^{-\epsilon |t|}$ for some sufficiently small $\epsilon >0$ and $|t| \gg 0$. In particular, our sections decay exponentially along the ends of the cylinder. With this modification, we may define $\sD_A$ and $\text{gr}_z(\alpha_1,\alpha_2)=\text{ind}(\sD_A)$ when one or both of $\alpha_i$ are reducible.

The following is adapted from \cite[Proposition 3.8]{KM:YAFT}, and differs by our inclusion of the reducible $\theta$ and our restrictions on perturbations from the previous subsection. The essential point is that the only reducible element of $M_z(\alpha_1,\alpha_2)$ is the constant solution associated to $\theta$, which we already know is regular by Proposition \ref{reg-dbl-cover}. Now similar arguments as in \cite[Chapter 5]{donaldson-book} can be used to verify the following proposition.

\begin{prop}\label{prop:perturb2}
	Suppose $\pi_0\in\sP'$ is a perturbation such that the critical points of $\fC_{\pi_0}=\{\theta\}\cup\fC_{\pi_0}^\text{\emph{irr}}$ are non-degenerate. Then there exists $\pi\in \sP'$ such that
	\begin{enumerate}
		\item[{\emph{(i)}}] $f_\pi=f_{\pi_0}$ in a neighborhood of the critical points of $\text{\emph{CS}}+f_{\pi_0}$;
		\item[{\emph{(ii)}}] the critical sets for the two perturbations are the same, $\fC_{\pi}=\fC_{\pi_0}$;
		\item[{\emph{(iii)}}] all moduli spaces $M_z(\alpha_1,\alpha_2)$ for the perturbation $\pi$ are regular.
	\end{enumerate}
\end{prop}

\begin{remark} The involution $\iota$ may be defined on the singular connection classes we consider here on $\R\times (Y,K)$, just as in \eqref{eq:flip}, using the pullback of $\xi$. Following the discussion at the end of Subsection \ref{sec:perturb}, we may in fact choose a perturbation which is invariant under the involution $\iota$ and such that the conclusions of Proposition \ref{prop:perturb2} hold. The key point is that before perturbing, there are no non-constant gradient flow lines invariant under $\iota$.  $\diamd$ 
\end{remark}

From now on we assume that the perturbation $\pi$ in the definition of the moduli spaces $M_z(\alpha_1,\alpha_2)$ is chosen such that the claims in Propositions \ref{prop:perturb1} and \ref{prop:perturb2} hold. Some other important properties of the moduli spaces are summarized as follows. The first is essentially Proposition 3.22 of \cite{KM:YAFT}.

\begin{prop}\label{prop:mod1} Let $\alpha_1,\alpha_2\in\fC_\pi$. If $M_z(\alpha_1,\alpha_2)$ is of dimension less than 4, then the space of unparametrized broken trajectories $\breve{M}^+_z(\alpha_1,\alpha_2)$ is compact.
\end{prop}

\noindent Recall that an element of $\breve{M}^+_z(\alpha_1,\alpha_2)$, {\it an unparametrized broken trajectory}, is by definition a collection $[A_i]\in \breve{M}_{z_i}(\beta_{i},\beta_{i+1})$ for $i=1,\ldots ,l-1$ with $\beta_1=\alpha_1$ and $\beta_l=\alpha_2$, and such that the concatenation of the homotopy classes $z_i$ is equal to $z$. We use the standard approach to topologize $\breve{M}^+_z(\alpha_1,\alpha_2)$.

The second result follows from Corollary 3.25 of \cite{KM:YAFT}, and is special to our hypothesis that our model singular connection has order 4 holonomy around meridians.

\begin{prop}\label{prop:mod2} Given $d\geqslant 0$, there are only finitely many $\alpha_1,\alpha_2\in\fC_\pi$ and $z$ such that $M_z(\alpha_1,\alpha_2)$ is non-empty and $\text{\emph{gr}}_z(\alpha_1,\alpha_2)=d$.
\end{prop}
\noindent In particular, $\breve{M}(\alpha_1,\alpha_2)_0$ is a finite set of points.

We now discuss some aspects of these gradings. Let $g\in\sG(Y,K)$ have homotopy invariants $d(g)=(k,l)$ as in \eqref{eq:gaugedegree}. Choose $\alpha=[B]\in \sB(Y,K)$, and let $z\in \pi_1(\sB(Y,K);\alpha)$ be the homotopy class induced by a path from $B$ to $g(B)$. Then
\begin{equation}
	\text{gr}_z(\alpha,\alpha) = 8k + 4l,\label{eq:grcharge}
\end{equation}
see \cite[Lemma 3.14]{KM:YAFT}. From standard linear gluing theory, as in \cite[Chapter 3]{donaldson-book}, when $\alpha_2$ is non-degenerate and irreducible we have
\begin{equation}
	\text{gr}_{z_{12}}(\alpha_1,\alpha_2) + \text{gr}_{z_{23}}(\alpha_2,\alpha_3) = \text{gr}_{z_{13}}(\alpha_1,\alpha_3)\label{eq:irradd}
\end{equation}
where $z_{13}$ is the concatenation of $z_{12}$ and $z_{23}$. Using \eqref{eq:grcharge} and \eqref{eq:irradd} we conclude that $\text{gr}_z(\alpha_1,\alpha_2)$ modulo 4 does not depend on the homotopy class $z$, and we set
\[
	\text{gr}(\alpha_1,\alpha_2) := \text{gr}_z(\alpha_1,\alpha_2) \mod 4.
\]
This defines a relative $\Z/4$-grading on the irreducible ciritical set $\fC^\text{irr}_\pi$. We lift this to an absolute $\Z/4$-grading using the reducible, analogous to Floer \cite{floer:inst1}: for $\alpha\in \fC^\text{irr}_\pi$ set
\begin{equation}
	\text{gr}(\alpha) := \text{gr}_z(\alpha,\theta) \mod 4\label{eq:absgr}
\end{equation}
for any choice of homtopy class $z$. Now \eqref{eq:irradd} does not hold when $\alpha_2=\theta$; as the dimension of the gauge stabilizer of $\theta$ is $1=\dim U(1)$, by \cite[Section 3.3.1]{donaldson-book} we instead have
\begin{equation}
	\text{gr}_{z_{12}}(\alpha_1,\theta) + 1 +  \text{gr}_{z_{23}}(\theta,\alpha_3) = \text{gr}_{z_{13}}(\alpha_1,\alpha_3).\label{eq:irrred}
\end{equation}
In particular, if we write $\text{gr}(\alpha)=\text{gr}_Y(\alpha)\in\Z/4$ to emaphasize the underlying 3-manifold $Y$, we obtain the orientation-reversing property
\begin{equation}
	\text{gr}_{-Y}(\alpha) \equiv 3-\text{gr}_Y(\alpha) \mod 4. \label{eq:orrevgr}
\end{equation}

From \eqref{eq:irrred} we deduce the following, which is analogous to part of the compactness principle in the non-singular setting, see Section 5.1 of \cite{donaldson-book}.

\begin{prop}\label{prop:mod3} Let $\alpha_1,\alpha_2\in\fC^\text{\emph{irr}}_\pi$. If $M_z(\alpha_1,\alpha_2)$ is of dimension less than 3, then $\breve{M}^+_z(\alpha_1,\alpha_2)$ has no broken trajectories that factor through $\theta$.
\end{prop}

\noindent Indeed, suppose a broken trajectory $([A_1],\ldots,[A_{l-1}])$ factors through the reducible $N\geqslant 1$ times. Note that $l\geqslant 3$. Then from our discussion thus far we have
\[
	\text{gr}_z(\alpha_1,\alpha_2)=\sum_{i=1}^{l-1} \text{gr}_{z_i}(\beta_i,\beta_{i+1}) + N \geqslant l-1+N \geqslant  3,
\]
where we use $ \text{gr}_{z_i}(\beta_i,\beta_{i+1})=\ind(\sD_{A_i})\geq 1$ because $A_i$ is a non-constant singular instanton.
Thus we must have $\text{gr}_z(\alpha_1,\alpha_2)=\dim M_z(\alpha_1,\alpha_2)\geqslant 3$ for such a factoring to occur. Note that in the non-singular setting, the dimension of the moduli space must be less than 5 to avoid breaking at the reducible. This is because the dimension of the stabilizer of the reducible in that setting is $3=\dim SO(3)$ instead of $1$.

\subsection{Moduli spaces for cobordisms}\label{subsec:cobmoduli}

We next discuss moduli spaces of instantons on cobordisms. Suppose we have a cobordism of pairs $(W,S):(Y,K)\to (Y',K')$ between two homology 3-spheres $Y$ and $Y'$ with embedded knots $K$ and $K'$, respectively. More precisely, $W$ is an oriented 4-manifold with boundary $Y'\sqcup -Y$, and $S\subset W$ is an embedded surface intersecting the boundary transversely with $\partial S = S\cap \partial W=K\sqcup K'$. Although it is possible to consider unoriented surfaces as in \cite{KM:unknot}, in this paper we will only be concerned with the case in which $S$ is connected and oriented. For a pair of composable cobordisms $(W_1,S_1)$ and $(W_2,S_2)$  we write $(W_2,S_2)\circ (W_1,S_1) = (W_2\circ W_1,S_2\circ S_1)$ for the composite cobordism.

Given a cobordism $(W,S):(Y,K)\to (Y',K')$, equip $W$ with an orbifold metric that has a cone angle $\pi$ along $S$, and which is a product near the boundary. Let $W^+$ (resp. $S^+$) be obtained from $W$ (resp. $S$) by attaching cylindrical ends to the boundary, and extend the metric data in a translation-invariant fashion. Given classes $\alpha\in \sB(Y,K)$ and $\alpha'\in \sB(Y',K')$, choose an $SU(2)$ connection $A$ on $W^+$ singular along $S^+$ such that the restrictions of $A$ to the two ends are in the gauge equivalence classes of $\alpha$ and $\alpha'$. The homotopy class of $A$ mod gauge rel $\alpha,\alpha'$ will be denoted by $z$. Similar to the definition of $\sB_z(Y,K;\alpha_1,\alpha_2)$ in the cylindrical case, we may form $\sB_z(W,S;\alpha,\alpha')$, the gauge equivalence classes of singular connections on $W^+$ whose representatives differ from $A$ by elements of regularity $\check{L}_k^2$. Just as in the cylindrical case, when either of $\alpha$ or $\alpha'$ is reducible, we use an appropriately weighted Sobolev norm for the end(s).

\begin{remark}
	For a general discussion of the possibilities for the model connection and the associated ``singular bundle data'' see \cite[Section 2]{KM:unknot}. However, the construction of \cite[Section 2]{km-embedded-i} suffices for our purposes. In particular, we restrict our attention to the case of structure group $SU(2)$.  $\diamd$ 
\end{remark}

We may then form the moduli space of instantons $M_z(W,S;\alpha,\alpha') \subset \sB_z(W,S;\alpha,\alpha')$. The perturbed instanton equation defining the moduli space $M_z(W,S;\alpha,\alpha')$ is of the following form along the incoming end $(-\infty , 1] \times Y \subset W^+$:
\[
	F_A^+ + \psi(t) \widehat{V}_\pi(A) + \psi_0(t)\widehat{V}_{\pi_0}(A) =0.
\]
Here $\pi$ and $\pi_0$ are perturbations on $\R\times Y$ as in \eqref{eq:asd}, and $\psi(t)=1$ for $t<0$ and $0$ at $t=1$, while $\psi_0(t)$ is supported on $(0,1)$. We always choose $\pi\in \sP'$ such that $\fC_{\pi} \subset \sB(Y,K)$ is as in Propositions \ref{prop:perturb1} and  \ref{prop:perturb2}. Similar remarks hold for the other end. For generic choices of $\pi_0$ and its analogue at the end of $Y'$ the irreducible part of the moduli space $M_z(W,S;\alpha,\alpha')$ is cut out transversally, and is a smooth manifold of dimension $d$, where $d=\text{ind}(\sD_A)=:\text{gr}_z(W,S;\alpha,\alpha')$. (See \cite{KM:YAFT} and \cite[Section 24]{km:monopole} for more details.) Here $\sD_A$ is the linearized ASD operator on $(W^+,S^+)$, analogous to \eqref{eq:linasdopcyl}, defined using Sobolev spaces with exponential decay at the ends with reducible limits, as in \eqref{eq:expweightspace}. Write
\[
	M(W,S;\alpha,\alpha')_d = \bigcup_{\text{gr}_z(W,S;\alpha,\alpha')=d} M_z(W,S;\alpha,\alpha').
\]
Note that $\text{gr}_z(I\times (Y,K);\alpha_1,\alpha_2)=\text{gr}_z(\alpha_1,\alpha_2)$. Furthermore, the linear gluing formulae for $\text{gr}_z(\alpha_1,\alpha_2)$ in the cylindrical case extend in this more general context. In particular, for cobordisms $(W,S):(Y,K)\to (Y',K')$ and $(W',S'):(Y',K')\to (Y'',K'')$ with composite cobordism $(W'',S'')=(W',S')\circ (W,S)$, we have
\begin{align*}
	\text{gr}_{z}(W,S;\alpha,\alpha') + \dim \text{Stab}(\alpha') +  \text{gr}_{z'}(W',S';\alpha',\alpha'') = \text{gr}_{z''}(W'',S'';\alpha,\alpha'')
\end{align*}
where $\text{Stab}(\alpha')\subset \sG(Y',K')$ is isomorphic to $\{\pm 1\}$ if $\alpha'$ is irreducible, and $U(1)$ if it is reducible. Just as in the cylindrical case, the mod 4 congruence class of $\text{gr}_z(W,S;\alpha,\alpha')$ is independent of $z$, and for this we write $\text{gr}(W,S;\alpha,\alpha')\in \Z/4$.

An unparametrized broken trajectory for $M_z(W,S;\alpha,\alpha')$ is a triple consisting of an instanton in $M_{z_2}(W,S;\beta,\beta')$ and unparametrized broken trajectories in $\breve{M}_{z_1}^+(\alpha,\beta)$ and $\breve{M}_{z_3}^+(\beta',\alpha')$, where $\beta$ and $\beta'$ are critical points for $Y$ and $Y'$, and $z=z_3\circ z_2\circ z_1$. The space of such broken trajectories is denoted $M^+_z(W,S;\alpha,\alpha')$. 

\begin{remark}
When $z$ is dropped from either  $\sB_z(W,S;\alpha,\alpha')$ or $M_z(W,S;\alpha,\alpha')$, it should be understood that we are considering the union over all homotopy classes $z$. $\diamd$
\end{remark}

Suppose $[A]\in \sB_z(W,S;\alpha,\alpha')$ is a singular connection for the pair $(W,S)$. Then the {\emph{action}}, or {\emph{topological energy}}, of $[A]$ is defined to be the Chern-Weil integral
\[
	\kappa(A) := \frac{1}{8\pi^2} \int_{W^+\setminus S^+} \text{Tr}(F_A\wedge F_A).
\]
Instantons $[A]$ are characterized as having energy equal to $8\pi^2 \kappa(A)$, and in particular $\kappa(A)\geqslant 0$, with equality if and only if $A$ is flat. Furthermore, 
\begin{equation}
	2\kappa(A)=\text{CS}(\alpha)-\text{CS}(\alpha')-\frac{1}{8}S\cdot S \pmod \Z.\label{eq:csmodz}
\end{equation}
In this paper we will focus on the case in which the homology class of $S$ is divisible by $4$, in which case we can ignore the term $\frac{1}{8}S\cdot S$ in this formula. Next, we define the {\emph{monopole number}} of $[A]$, denoted $\nu(A)$, by the following integral:
\[
	\nu(A) := \frac{i}{\pi} \int_{S^+}  \Omega.
\]
The connection $F_A$ extends to the singular locus $S^+$, and $\Omega$ in the above formula is a 2-form with values in the orientation bundle of $S^+$ such that the restriction of $F_A$ to the singular locus has the following form:
\[
  F_A\vert_{S^+}=\left[
  \begin{array}{cc}
 	\Omega&0\\
	0&-\Omega
  \end{array}
  \right]
\] 
The numbers $\kappa(A)$ and $\nu(A)$ are invariants of the homotopy class $z$, and determine it. Moreover, the dimension $d$ of a moduli space $M(W,S;\alpha,\alpha')_d$ is determined by $\kappa(A)$, and the homotopy classes $z$ of the components of $M(W,S;\alpha,\alpha')_d$ are distinguished by their monopole numbers $\nu(A)$. 

The flip symmetry of Subsection \ref{subsec:flip} extends to the case in which the homology class of $S$ is a multiple of $2$ within $H_2(W;\Z)$. In this case, there exists a flat $\Z/2$ bundle-with-connection $\xi$ over $W^+\setminus S^+$ with holonomy $-1$ around small circles linking $S$; then $\iota [A]:= [A\otimes \xi]$ as before. We have the relations
\begin{equation}\label{iota-effect}
	\kappa(\iota A) = \kappa(A), \qquad \nu(\iota  A) = -\nu (A),
\end{equation}
see \cite[Lemma 2.12]{km-embedded-i}. In particular, note that when $(W^+,S^+)$ is a cylinder, the monopole number is negated under $\iota$.

\subsection{Reducible connections and negative definite pairs}\label{subsec:red}

The goal of this subsection is to study the reducible solutions of the ASD equation on a cobordism of pairs $(W,S):(Y,K)\to (Y',K')$ between knots in integer homology 3-spheres. Any such connection is necessarily asymptotic to the reducibles associated to $(Y,K)$ and $(Y',K')$. The following lemma gives a formula for the index of the ASD operator associated to a connection that is asymptotic to reducibles:

\begin{lemma}\label{index-triv-limit}
	Suppose the connection $A$ represents an element of $\sB_z(W,S;\theta,\theta')$. Then:
	\begin{equation}\label{ind-form-red-ends}
		\ind(\sD_A)=8\kappa(A)-\frac{3}{2}(\sigma(W)+\chi(W))+\chi(S)+\frac{1}{2}S\cdot S+\sigma(K)-\sigma(K')-1
	\end{equation}
	where $\sigma(W)$ and $\sigma(K)$ are respectively the signature of the 4-manifold $W$ and the signature of the knot $K$, and for a 
	topological space $X$, $\chi(X)$ denotes the Euler characteristic of $X$.
\end{lemma}
\begin{proof}
	We first compute the index for a slightly simpler case. Suppose that $(X,\Sigma)$ has only one outgoing end $(Y,K)$, the homology class of $\Sigma$ is trivial, and 
	$\widetilde X$ denotes a branched double cover of $X$ branched along $\Sigma$ 
	with covering involution $\tau:\widetilde X\to \widetilde X$. There is a flat singular connection associated to the pair 
	$(X,\Sigma)$ such that after lifting up to $\widetilde X$ and taking the induced $SO(3)$ adjoint connection, it can be extended to the trivial connection over $\widetilde X$. As an alternative description, we may consider the involution on the trivial bundle 
	$\underline{\R}^{3}$ over $\widetilde X$ which lifts the involution $\tau$ and is given by:
	\begin{equation}
		((v_1,v_2,v_3),x)\in \R^3\times \widetilde X \mapsto ((v_1,-v_2,-v_3),\tau(x)).\label{eq:locmodelbranched}
	\end{equation}
	The quotient by this involution sends the trivial connection to an orbifold $SO(3)$ connection on $X$ which lifts to our desired $SU(2)$ reducible singular connection $A_0$. 
	
	We define the ASD operator $\sD_{A_0}$ in the same way as before 
	using weighted Sobolev spaces with exponential decay at the end. From the description of $A_0$, it is clear that $\ker(\sD_{A_0})$ is isomorphic as a vector space to the subspace of
	\[
		H^1(\widetilde X; \underline{\R}) = H^1(\widetilde X) \otimes \R^3
	\]
	which is invariant under the involution induced by \eqref{eq:locmodelbranched}, which may be identified with $\tau^\ast\otimes \text{diag}(1,-1,-1)$ where $\tau^\ast$ acts on $H^1(\widetilde X)$ and the diagonal matrix acts on $\R^3$. We obtain that $\ker(\sD_{A_0})$ is isomorphic to $H^1_{+}(\widetilde X)\oplus H^1_{-}(\widetilde X)^{\oplus 2}$ where $H^i_{\pm}(\widetilde X)$ denotes the $(\pm 1)$-eigenspace of the action of $\tau$ on the cohomology group $H^i(\widetilde X)$. A similar argument applies to the cokernel. In summary, we obtain:
	\begin{equation}
	  \ker(\sD_{A_0})\iso H^1_{+}(\widetilde X)\oplus H^1_{-}(\widetilde X)^{\oplus 2}\label{ker}
	\end{equation}
	\begin{equation}
	  \coker(\sD_{A_0})\iso H^+_{+}(\widetilde X)\oplus  H^+_{-}(\widetilde X)^{\oplus 2}\oplus H^0(\widetilde X)\label{coker}
	\end{equation}
	A straightforward calculation shows that
	\begin{equation}\label{ind-sim}
	  \ind(\sD_{A_0})=-(\sigma(\widetilde X)+\chi(\widetilde X))+\frac{1}{2}(\sigma(X)+\chi(X))-\frac{1}{2}
	\end{equation}
	We further specialize to the case that $(X,\Sigma)$ is obtained by firstly pushing a Seifert surface for $K\subset \{1\}\times Y$ into $[0,1]\times Y$ and then capping 
	the incoming end of $[0,1]\times Y$ with a 4-manifold $X$ with boundary $Y$. 
	The signature of the 4-manifold obtained as the branched double cover of the Seifert surface $\Sigma$ pushed into 
	$[0,1]\times Y$ is equal to the signature of $K$. Therefore, in this case we have: 
	\[
	  \sigma(\widetilde X)=2\sigma(X)+\sigma(K)\hspace{1cm}\chi(\widetilde X)=2\chi(X)-\chi(\Sigma)
	\]
	and the formula in \eqref{ind-sim} simplifies to:
	\[
	  -\frac{3}{2}(\sigma(X)+\chi(X))-\sigma(K)+\chi(\Sigma)-\frac{1}{2}.
	\]
	Applying a similar construction as above to the pair $(Y',K')$ and then changing the orientation of the underlying 4-manifold produces a pair $(X',\Sigma')$
	with boundary $(-Y',K')$ and a reducible singular flat connection $A_0'$. A similar argument as above shows:
	\[
	  \ind(\sD_{A_0'})=-\frac{3}{2}(\sigma(X')+\chi(X'))+\sigma(K')+\chi(\Sigma')-\frac{1}{2}.
	\]

	Gluing $(X,\Sigma)$, $(W,S)$ and $(X',S')$ produces a closed pair $(\overline W,\overline S)$. We may also glue $A_0$, $A$ and $A_0'$ 
	to obtain a singular connection $\overline A$ on $(\overline W,\overline S)$ with the same topological energy as $A$. 
	Additivity of the ASD indices implies that:
	\[
	  \ind(\sD_{\overline A})=\ind(\sD_{A_0})+\ind(\sD_{A})+\ind(\sD_{A_0'})+2
	\]
	where the appearance of the term $2$ on the left hand side is due to reducibility of the connections $\theta$ and $\theta'$. Now we can obtain \eqref{ind-form-red-ends}
	using the index formula in the closed case \cite{km-embedded-i} and our calculation of the indices of $A_0$ and $A_0'$.
\end{proof}

The same elementary observation which was used in \eqref{ker}--\eqref{coker} implies that:
\begin{prop}\label{reg-dbl-cover} 
	Suppose a cobordism of pairs $(W,S)$ has a double branched cover $\pi:\widetilde W\to W$. 
	Let $A$ be a singular connection on $(W,S)$ for some singular bundle data. If the non-singular connection $\pi^\ast{A}^\text{\emph{ad}}$ 
	is a regular ASD connection on $\widetilde W$, then $A$ is regular. In particular, if $\pi^\ast{A}^\text{\emph{ad}}$ is trivial 
	and $b^+(\widetilde W)=0$, then $A$ is regular.
 \end{prop}

To simplify our discussion about reducible singular instantons, we henceforth assume $H_1(W;\Z)=0$, $b^+(W)=0$, and $S$ is an orientable surface of genus $g$ whose homology class is divisible by $4$. By a slight abuse of notation, we write $S$ for both the homology class of $S$ and its Poincar\'e dual. The space of reducible elements of the moduli spaces $M_z(W,S;\theta,\theta')$ (with the trivial perturbation term) is in correspondence with the set of isomorphism classes of $U(1)$-bundles on $W$. For any line bundle $L$ on $W$, there is a $U(1)$ reducible singular ASD connection $A_L:=\eta\oplus \eta^*$ such that $\eta$ is a singular ASD connection on $L$, defined over $W^+\backslash S^+$. The connection $\eta$ has the property that its holonomy along a meridian of $S$ is asymptotic to $i$ (rather than $-i$) as the size of the meridian goes to zero. The topological energy and the monopole number of $A_L$ are given as follows:
\[
  \kappa(A_L)=-(c_1(L)+\frac{1}{4}S)\cdot (c_1(L)+\frac{1}{4}S)\hspace{1cm} \nu(A_L)=2c_1(L)\cdot S + \frac{1}{2} S\cdot S
\]
In particular, the topological energy is strictly positive unless $c_1(L)+\frac{1}{4}S$ is a torsion cohomology class. By requiring $H_1(W;\Z)=0$, we guarantee that there is a unique reducible instanton with vanishing topological energy and monopole number.

The index formula of Lemma \ref{index-triv-limit}, under the current assumptions, simplifies to:
\begin{align*}
  \ind(\sD_{A_L})&=8\kappa(A_L)-2g+\frac{1}{2}S\cdot S+\sigma(K)-\sigma(K')-1\nonumber\\
  &=8\kappa(A_L)+2(b_1(\widetilde W)-b^+(\widetilde W))-1\label{ind-red}
\end{align*}
where $\widetilde W$ denotes the double cover of $W$ branched along $S$, and the second identity can be derived from the following standard identities:
\[
  \sigma(\widetilde W)=2\sigma(W)-\sigma(K)+\sigma(K')-\frac{1}{2}S\cdot S,\hspace{1cm} \chi(\widetilde W)=2\chi(W)-\chi(S).
\]
As another observation about the topology of $\widetilde W$, note that $b_1(\widetilde W)=0$. This is shown, for example, in \cite{rohlin} for the case that the pair $(W,S)$ is a closed pair and the homology class of $S$ is non-trivial and a similar argument can be used to verify the same identity in our case. In fact, we can reduce our case to the closed case by gluing the pairs $(X,\Sigma)$ and $(X',\Sigma')$ as in the proof of Lemma \ref{index-triv-limit}. We may also assume that the homology class of $S$ is non-trivial by taking the connected sum with the pair $(\overline{\C\P}^2,B)$ where $B$ is a surface representing a non-trivial homology class.

 \begin{definition}\label{def:negdef}
 	A cobordism of pairs $(W,S):(Y,K)\to (Y',K')$ between knots in integer homology 3-spheres is a {\emph{negative definite pair}} if $H_1(W;\Z)=0$, $b^+(W)=0$, 
	the homology class of $S$ is divisible by $4$, and $b^+(\widetilde W)=0$. The latter condition about the branched double cover can be replaced with the following 
	identity:
	\[
	  \sigma(K')=\sigma(K)+\frac{1}{2}S\cdot S+\chi(S).\quad \diamd
	\]
 \end{definition}

For a negative definite pair $(W,S)$, $\ind(\sD_{A_L})$ is equal to $8\kappa(A_L)-1$. In particular, the flat reducible $A_0$ has index $-1$, and it is regular and has 1-dimensional stabilizer. All remaining reducibles have higher indices. In fact, we may assume that all the other reducibles are also regular \cite[Subsection 7.3]{dcx}. However, we do not need this fact in the sequel. As the moduli space $M(W,S;\theta,\theta')_0$ defined with trivial perturbation contains a unique regular reducible with vanishing $\kappa$ and $\nu$, the same is true for a small enough perturbation.

\begin{example}
	For any pair $(Y,K)$ of a knot in an integer homology sphere, the product $([0,1]\times Y,[0,1]\times K)$ is a negative definite pair. We fix two perturbations of 
	the Chern-Simons functional for $(Y,K)$ and a perturbation of the ASD connection on the cobordism associated to the product cobordism. 
	The above discussion shows that if the perturbation of the ASD equation is small enough, then $M(W,S;\theta,\theta')_0$  contains a unique 
	regular reducible with vanishing topological energy and monopole number.  Here the ASD equation is defined 
	with respect to an orbifold metric, which is not necessarily a product metric. $\diamd$
\end{example}

\begin{example}
	Any homology concordance $(W,S):(Y,K) \to (Y',K')$, as defined in the introduction, is a negative definite pair. $\diamd$
\end{example}

For a negative definite pair $(W,S)$, we may use the discussions of the previous and present sections to ensure the regularity of the moduli spaces $M(W,S;\alpha,\alpha')_z$ of expected dimension at most $3$. The analogues of Propositions \ref{prop:mod1} and \ref{prop:mod2} carry over as stated to the setting of non-cylindrical cobordisms. However, we note that the analogue of  Proposition \ref{prop:mod3} requires $M_z(W,S;\alpha,\alpha')$ to be regular and of dimension less than 2, instead of 3.

We close this subsection with some remarks on compactness and gluing theory for singular instantons. Let $(W,S)$ be a cobordism with auxiliary data as in Subsection \ref{subsec:cobmoduli}. For simplicity of notation to follow we assume $(W,S):\emptyset \to (Y,K)$, i.e. it has one boundary component $(Y,K)$. We assume that the critical set $\fC_\pi$ is non-degenerate, and all moduli spaces $M_z(W,S;\alpha)$ for $\alpha\in \fC_\pi$ are unobstructed and smooth.

Suppose $\alpha$ is irreducible, and let $\beta_1,\ldots,\beta_l \in \fC_\pi^\text{irr}$ with $\beta_l=\alpha$. For any homotopy classes $z_1,\ldots,z_l$ with concatenation equal to $z$, there is a gluing map of the form
\[
	M_{z_1}(W,S;\beta_1)	\times \R_{>0} \times \breve{M}_{z_2}(\beta_1,\beta_2)\times \R_{>0}\times \cdots \times \R_{>0} \times \breve{M}_{z_l}(\beta_{l-1},\beta_l) \longrightarrow M_z(W,S;\alpha)
\]
which is an embedding. In many cases we consider, the moduli space $M(W,S;\alpha)$ is compact away from the image of this gluing map. There are only two other sources of non-compactness that might occur. The first is from bubbling, which can occur only when $\text{ind}(\sD_A)\geqslant 4$ for instantons $[A]\in M_z(W,S;\alpha)$. We will always be in a situation where bubbling can be avoided.

The other source of non-compactness comes from reducible connections. This case is important for us in the sequel. Suppose above that some $\beta_j$ for $1<j<l$ is reducible, i.e. $\beta=\theta$. Then the relevant gluing map in this situation takes the form
\[
	\cdots \times   \breve{M}_{z_{j}}(\beta_{j-1},\theta)\times S^1\times \R_{>0}\times \breve{M}_{z_{j+1}}(\theta,\beta_{j+1}) \times \cdots  \longrightarrow M_z(W,S;\alpha)
\]
where the rest of the domain is as before. The factor $S^1$ of ``gluing parameters'' may be identified with the stabilizer of $\theta$. In general one can glue along multiple reducibles, but in the sequel we will only see the above case.

Finally, consider the case in which $\beta_1=\theta$ is reducible, and $M_{z_1}(W,S;\theta)=\{[A_1]\}$ contains a unique {\emph{regular}} reducible $[A_1]$ and no irreducibles. There is a gluing map
\[
	\{[A_1]\}	\times \R_{>0} \times \breve{M}_{z_2}(\theta,\beta_2)\times \R_{>0}\times \cdots \times \R_{>0} \times \breve{M}_{z_l}(\beta_{l-1},\beta_l) \longrightarrow M_z(W,S;\alpha)
\]
where as before $\beta_j$ for $j>1$ are irreducible. Here there is no gluing parameter as in the previous case; from another viewpoint, the gluing parameter cancels with the stabilizer of $A_1$. Compare the discussion in \cite[p.325]{DK}.

The situations we encounter below are all minor variations of the above. We will allow $(W,S)$ to have multiple boundary components, for example. Further, for $M_z(W,S;\theta)$, a moduli space of irreducible singular instantons, the same constructions are available.

The compactness principle underlying our arguments in the sequel is as follows: if bubbling can be ruled out (as will always be the case), a sequence of instantons in $M_z(W,S;\alpha)$ which does not have a convergent subsequence does have a subsequence that eventually lies in the image of one of the types of gluing maps described above.

The moduli spaces $M^+_z(W,S;\alpha)$ which contain broken trajectories are topologized as follows. Let $\mathfrak{a}=([A],[B_1],\ldots,[B_{l-1}])\in M^+_z(W,S;\alpha)$ be a broken trajectory of the first type considered above, where $[A]$ is an instanton on $(W,S)$ and $[B_i]$ on $\R\times (Y,K)$, with all limits irreducible. Then any neighborhood $N$ of $\mathfrak{a}\in M^+_z(W,S;\alpha)$ contains the image under the gluing map of $U\times (T,\infty)\times U_1\times \cdots \times (T,\infty)\times U_{l-1}$ for some $T\gg 0$ depending on $N$, and neighborhoods $U$ and $U_i$ of $[A]$ and $[B_i]$ in their respective moduli spaces, which also depend on $N$. The other cases are similar. Note that the gluing parameter factors involved in breakings at reducibles are forgotten in the completion of $M^+_z(W,S;\alpha)$. Furthermore, if bubbling does not occur, then $M^+_z(W,S;\alpha)$ is compact.

The above summary relies on a substantial amount of technical work which is by now standard or treated elsewhere. Gluing theory in instanton theory began with Taubes \cite{taubes-glue}, and Floer developed the case of $\R\times Y$ with $Y$ an integer homology 3-sphere in his original construction of instanton homology \cite{floer:inst1}. In the singular setting, the machinary developed in \cite{KM:YAFT}, which also references \cite{km:monopole}, handles the cases in which reducibles can be avoided, and also describes the conditions under which bubbling occurs. The results are similar to the case of non-singular instanton homology as treated in \cite[4.4,5.1]{donaldson-book}. The cases involving breaking at reducibles are also analogous to the non-singular case as in \cite[4.4.1]{donaldson-book} and \cite[Theorem 5]{froyshov}, except that instances of $SO(3)$ (the stabilizer of the reducible) are replaced by $S^1$. All of the above fits into the general framework of {\emph{unobstructed}} gluing theory.

\subsection{Counting critical points}

In the non-singular $SU(2)$ gauge theory setting, Taubes showed in \cite{taubes} that the signed count of (perturbed) irreducible flat connections on an integer homology 3-sphere is equal to $2\lambda(Y)$, twice the Casson invariant. Here ``signed count'' means that a critical point $\alpha$ is counted with sign $(-1)^{\text{gr}(\alpha)}$ where $\text{gr}(\alpha)$ is defined analogously to \eqref{eq:absgr}. In the singular setting, we have the following, which is essentially a special case of a result due to Herald \cite{herald}, which followed the work of Lin \cite{lin}.

\begin{theorem}[cf. Theorem 0.1 of \cite{herald}]\label{thm:herald}
	Let $Y$ be an integer homology 3-sphere and $K\subset Y$ a knot. Suppose $\pi$ is a small perturbation such that $\fC_\pi$ is non-degenerate. Then
	\begin{equation}\label{eq:herald}
		\sum_{\alpha\in \fC^\text{\emph{irr}}_\pi} (-1)^{{\emph{\text{gr}}}(\alpha)} = 4\lambda(Y) + \frac{1}{2}\sigma(K)
	\end{equation}
	where $\lambda(Y)$ is the Casson invariant of $Y$ and $\sigma(K)$ is the signature of the knot $K\subset Y$. 
\end{theorem}

As our setup is different from Herald's, we explain how the work in \cite{herald} implies Theorem \ref{thm:herald}. Suppose, as in the statement, that $\pi$ is a small perturbation such that $\fC_\pi$ is non-degenerate. In particular, $\fC_\pi^\text{irr}$ is a finite set. Let $\alpha_1,\alpha_2\in \fC_\pi^\text{irr}$ and write $\alpha_i=[B_i]$. The orientation, or sign, associated to $\alpha_i$ is determined by the parity of $\text{gr}(\alpha_i)$. In particular, the signs for the two critical points $\alpha_1$ and $\alpha_2$ agree if and only if
\[
	\text{gr}(\alpha_1) - \text{gr}(\alpha_2) \equiv  \text{gr}(\alpha_1,\alpha_2) \equiv  \text{ind}(\sD_A) \mod 4
\]
is even. Here $A$ is a connection on the cylinder $\R\times Y$ with limits $B_1$ and $B_2$, as in Subsection \ref{subsec:gr}. In fact, we may arrange that the operator $\sD_A$ is of the form $\frac{\partial}{\partial t} + D_{B(t)}$, where $B(t)$ is a path of connections $\R\to \sC(Y,K)$ equal to $B_1$ and $B_2$ for $t\ll 0$ and $t \gg 0$, respectively, and $D_{B}$ is the extended Hessian of $\text{CS}+f_\pi$ at $B$. It is well-known, in this situation, that $\text{ind}(\sD_A)$ is equal to the {\emph{spectral flow}} of the path of operators $D_{B(t)}$.

Now, consider a closed tubular neighborhood $N\subset Y$ of the knot $K$, which as an orbifold is isomorphic to the pair $(S^1\times D^2,S^1)$. Extend the reduction of the bundle $E|_K=L\oplus L^\ast$ over $N$. We may assume that $\pi$ is compactly supported. In particular, we may assume that $N$ is chosen small enough so that $\pi$ is supported on $Y\setminus N$. The boundary of $N$ is a 2-torus. Write $\sM_{\partial N}$ for the moduli space of flat $U(1)$ connections on $L|_{\partial N}$, which is naturally identified with the dual torus of $\partial N$. The torus $\sM_{\partial N}$ is a 2-fold branched cover over the moduli space of flat $SU(2)$ connections on $E|_{\partial N}$.

Let $\cS\subset \cM_{\partial N}$ be the embedded circle consisting of flat connections which are trace free at the meridian of $K$. Let $Y^\circ$ be the closure of $Y\setminus N$. Denote by
\[
	\cM^\text{irr}_{Y^\circ}\subset \cM_{\partial N}
\]
the image of the moduli space of $\pi$-perturbed flat irreducible connections on $E|_{Y^\circ}$ which preserve the $U(1)$ bundle $L\subset E|_{\partial N}$. 
After perhaps changing our small perturbation $\pi$, we can assume that $\cM^\text{irr}_{Y^\circ}$ is immersed in $\cM_{\partial N}$ and also that $\cM^\text{irr}_{Y^\circ}$ intersects $\cS$ transversely, away from self-intersection points of $\cM^\text{irr}_{Y^\circ}$. We orient these manifolds as in \cite{herald}.

\begin{prop}\label{prop:localsigns}
	Let $\alpha_i=[B_i]\in \fC^\text{irr}_\pi$ for $i=1,2$. As points in $\cM^\text{\emph{irr}}_{Y^\circ} \cap \cS \subset \cM_{\partial N}$, the local orientations for $\alpha_1$ and $\alpha_2$ agree if and only if the spectral flow of $D_{B(t)}$ is even.
\end{prop}

A proof of Proposition \ref{prop:localsigns} follows by essentially repeating the proof of Proposition 7.2 in \cite{herald}, which itself is a modification of Proposition 5.2 from \cite{taubes}. In the proof of Proposition 7.2 in \cite{herald}, the spectral flow of $D_{B(t)}$ is related to data on $Y^\circ$ and $N$ by analyzing a Mayer--Vietoris sequence of Fredholm bundles. While Proposition 7.2 in \cite{herald} treats the case of flat connections with trivial holonomy around meridians of $K$, so that $N$ is instead, as an orbifold, simply $S^1\times D^2$, the argument easily adapts to our situation.

The main result of \cite{herald}, Theorem 0.1, with $\alpha=\pi/2$, tells us that the intersection number $\cM^\text{irr}_{Y^\circ} \cdot \cS$ is equal, up to a sign $\pm$ which depends on our conventions, to the quantity $\pm (4\lambda(Y) + \frac{1}{2}\sigma(K))$. Together with Proposition \ref{prop:localsigns}, this implies equation \eqref{eq:herald} of Theorem \ref{thm:herald} up to the ambiguity $\pm$. An adaptation of \cite[Lemma 7.5]{herald} shows that the sign $\pm$ is universal, i.e. independent of $(Y,K)$. Finally, we determine that the sign is in fact $+$ by computing a non-trivial example, see e.g. Subsection \ref{sec:35}.

\subsection{Orienting moduli spaces}\label{subsec:ors}

In this subsection, we fix our conventions for the orientation of ASD moduli spaces based on \cite{KM:YAFT}. A similar discussion about the orientation of moduli spaces in the non-singular case appears in \cite[Section 5]{donaldson-book}. For any cobordism of pairs $(W,S):(Y,K)\to (Y',K')$, and a path $z$ along $(W,S)$ between the critical points $\alpha$, $\alpha'$ for $Y$, $Y'$, the moduli space $M_z(W,S;\alpha,\alpha')$ is orientable. In fact, the index of the family of ASD operators $\sD_A$ associated to connections $[A]\in \sB_z(W,S;\alpha,\alpha')$ determines a trivial line bundle $l_z(W,S;\alpha,\alpha')$ on $\sB_z(W,S;\alpha,\alpha')$ and the restriction of this bundle to $M_z(W,S;\alpha,\alpha')$ is the orientation bundle of the moduli space $M_z(W,S;\alpha,\alpha')$.

 In the case that either $\alpha$ or $\alpha'$ is reducible, we may form a variation of the bundle $l_z(W,S;\alpha,\alpha')$. For example, in the case that $\alpha=\theta$, we may change the definition of the weighted Sobolev spaces of the domain and codomain of the ASD operator by allowing exponential growth for an exponent $\epsilon$ rather than the exponential decay condition that we used earlier. If $\epsilon$ is non-zero and small enough, then this new ASD operator is still Fredholm and its index is independent of $\epsilon$. Thus when $\alpha'$ is irreducible, we have two choices of determinant line bundles, both trivial, denoted by $l_z(W,S;\theta_-,\alpha')$ and $l_z(W,S;\theta_+,\alpha')$, depending on whether we require exponential decay or exponential growth. We use a similar notation in the case that $\alpha'$ is reducible.

We denote the set of orientations of the bundle $l_z(W,S;\alpha,\alpha')$ by $\Lambda_z[W,S;\alpha,\alpha']$, which is a $\Z/2$-torsor. For a composite cobordism, there is a natural isomorphism:
\[
  \Phi:\Lambda_{z_1}[W_1,S_1;\alpha,\alpha']\otimes_{\Z/2\Z}\Lambda_{z_2}[W_2,S_2;\alpha',\alpha'']\to \Lambda_{z_2\circ z_1}[W_2\circ W_1,S_2\circ S_1;\alpha,\alpha''].
\]
In the case that $\alpha'$ is reducible, we require that one of the appearances of $\alpha'$ in the domain of $\Phi$ is $\theta_+$, and the other $\theta_-$. The isomorphism $\Phi$ is associative when we compose three cobordisms. Moreover, there is also a natural isomorphism between $\Lambda_z[W,S;\alpha,\alpha']$ and $\Lambda_{z'}[W,S;\alpha,\alpha']$ for any two paths $z$, $z'$ from $\alpha$ to $\alpha'$ and this isomorphism is compatible with $\Phi$. This allows us to drop $z$ from our notation for $\Lambda_z[W,S;\alpha,\alpha']$. We also drop $(W,S)$ from our notation whenever the choice of $(W,S)$ is clear from the context.

For a cobordism of pairs $(W,S)$, an element in $\Lambda[W,S;\theta_-,\theta'_-]$ is identified with an orientation of the determinant line $l_z(W,S;\theta_-,\theta'_-)$ over any connection in the homotopy class of paths $z$. Assuming $S$ is oriented and $A_0$ represents a reducible element of $\sB_z(W,S;\theta,\theta')$, this line may be identified as follows:
\begin{equation}
	l_z(W,S;\theta_-,\theta'_-)|_{[A_0]} \cong \wedge^\text{top} (H^1(W) \oplus H^+(W)^\ast\oplus H^0(W)^\ast).\label{eq:reddetlinemin}
\end{equation}
This holds because the connection $A_0^{\rm ad}$ decomposes into a trivial connection on a trivial real line bundle and an $S^1$-connection on a complex line bundle $L_0$. The index of the operator $\sD_{A_0}$ decomposes accordingly. The contribution from $L_0$ can be oriented canonically as it is a complex vector space, and the orientation of the contribution from the trivial line bundle can be identified with the right hand side of \eqref{eq:reddetlinemin}. Similarly, we have an isomorphism
\begin{equation}\label{ori-bdle}
	l_z(W,S;\theta_+,\theta'_-)|_{[A_0]} \cong \wedge^\text{top} (H^1(W) \oplus H^+(W)^\ast).
\end{equation}
A {\it homology orientation} for $(W,S)$ is defined to be an element of $\Lambda[W,S;\theta_+,\theta'_-]$, which by \eqref{ori-bdle} amounts to an orientation of the vector space
\[
  H^1(W)\oplus H^+(W).
\]
In particular, if $(W,S)$ is a negative definite pair, we may take $A_0$ to be the unique flat reducible on $(W,S)$, and we have a canonical element of $\Lambda[W,S;\theta_+,\theta'_-]$.

Changing the orientation of $S$ changes the orientation of the trivial real line bundle and dualizes the complex line bundle $L_0$. In particular, the identifications in \eqref{eq:reddetlinemin} and \eqref{ori-bdle} change sign respectively according to the parities of $d+d'-1$ and $d+d'$, where: 
\[
	d=b^1(W)-b^+(W), \hspace{1cm} d'=\frac{\ind(\sD_{A_0})-d+1}{2}
\]
In the case that $(W,S)$ is a negative definite pair and $A_0$ is the unique flat reducible, then the identification in \eqref{eq:reddetlinemin} changes by a sign whereas the identification in \eqref{ori-bdle} is preserved. In particular, the canonical homology orientation is independent of the orientation of $S$.

Given $(Y,K)$ and $\alpha\in \fC_\pi$, let $\Lambda[\alpha]:=\Lambda[I\times Y, I\times K; \alpha,\theta_-]$ if $\alpha$ is irreducible, and $\Lambda[\alpha]:=\Lambda[I\times Y, I\times K; \theta_+,\theta_-]$ if $\alpha=\theta$. From the discussion in the previous paragraph, because the cobordism $(I\times Y, I\times K)$ has $b^1=b^+=0$, it has a canonical homology orientation, and there is thus a canonical element of $\Lambda[\theta]$. We use this canonical choice whenever we need an element from this set.

Given a homology orientation $o_W$ for $(W,S)$ and $o_\alpha\in \Lambda[\alpha]$ and $o_\beta\in \Lambda[\beta]$, we can fix $o_{(W,S;\alpha,\alpha')}\in \Lambda[W,S;\alpha,\alpha']$, and hence an orientation of $M_z(W,S;\alpha,\alpha')$, by requiring:
\begin{equation*}
	\Phi(o_\alpha\otimes o_{W})=\Phi(o_{(W,S;\alpha,\alpha')}\otimes o_{\alpha'}).
\end{equation*}
As a special case, we may apply this rule to orient a cylinder moduli space $M_z(\alpha_1,\alpha_2)$ from the data of $o_{\alpha_1}\in \Lambda[\alpha_1]$ and $o_{\alpha_2}\in \Lambda[\alpha_2]$. Let $\tau_s$ be the translation on $\R\times Y$ defined by $\tau_s(t,y)=(t-s,y)$. Then $\tau_s$ acts on $M_z(\alpha_1,\alpha_2)$ by pull-back, and the identification
\begin{equation}
	M_z(\alpha_1,\alpha_2) = \R \times \breve{M}_z(\alpha_1,\alpha_2)\label{eq:modrtransl}
\end{equation}
is such that the action of $\tau_s$ is by addition by $s$ on the $\R$-factor. Then we may orient $\breve{M}_z(\alpha_1,\alpha_2)$ using an orientation of $M_z(\alpha_1,\alpha_2)$, and requiring that the identification \eqref{eq:modrtransl} is orientation-preserving, with the ordering of factors as written.

\newpage

%!TEX root = main.tex

\section{Instanton Floer homology groups for knots}\label{sec:tilde}

In this section we introduce instanton homology groups for based knots in integer homology 3-spheres. Although we first introduce an analogue of Floer's instanton homology for homology 3-spheres, our main object of interest is a ``framed'' instanton homology, or rather its chain-level manifestation, analogous to Donaldson's theory for homology 3-spheres in \cite[Section 7.3.3]{donaldson-book}. The framed theory incorporates the reducible critical point and certain maps defined via holonomy.

\subsection{An analogue of Floer's instanton homology for knots}\label{sing-Floer}

Let $(Y,K)$ be an integer homology 3-sphere with an embedded knot, as in Section \ref{sec:defns}. We now also choose an orientation of the knot $K$, which will be required for some of the later constructions. Equip $Y$ with a Riemannian metric $g_o$ with cone angle $\pi$ along $K$. Choose a holonomy perturbation as in Propositions \ref{prop:perturb1} and \ref{prop:perturb2}, so that the critical set $\fC_\pi=\{\theta\}\cup \fC_\pi^\text{irr}$ is a finite set of non-degenerate points, and the moduli spaces $M(\alpha_1,\alpha_2)$ are all regular. We define $C=C(Y,K)$ to be the free abelian group generated by the {\emph{irreducible}} critical set:
\begin{equation}
	C(Y,K) = \bigoplus_{\alpha\in\fC_\pi^\text{irr}} \Z\cdot \alpha. \label{eq:complexwithoutors}
\end{equation}
The mod 4 grading $\text{gr}(\alpha)$ of \eqref{eq:absgr} gives $C$ the structure of a $\Z/4$-graded abelian group. As usual, the grading will be indicated as a subscript, $C_\ast=C_\ast(Y,K)$, but is often omitted.

The differential $d$ on the group $C_\ast$ is defined as follows:
\begin{equation}\label{eq:irrdiff}
	d(\alpha_1) = \sum_{\substack{\alpha_2\in\fC_\pi^\text{irr}\\ \text{gr}(\alpha_1,\alpha_2)\equiv 1}} \# \breve{M}(\alpha_1,\alpha_2)_0\cdot \alpha_2.
\end{equation}
In the sequel we depict the map $d$ by an undecorated cylinder {\smash{\raisebox{-.1\height}{\includegraphics[scale=0.5]{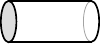}}}} which should be thought of as $I \times K$. By Propositions \ref{prop:mod1} and \ref{prop:mod2}, the moduli space $\breve{M}(\alpha_1,\alpha_2)_0$ is a finite set of points. The coefficient $\# \breve{M}(\alpha_1,\alpha_2)_0$ is a signed count of these points. More precisely, we may more invariantly define the chain group to be
\begin{equation}
	C(Y,K) = \bigoplus_{\alpha\in\fC_\pi^\text{irr}} \Z\Lambda[\alpha]\label{eq:complexwithors}
\end{equation}
where $\Z\Lambda[\alpha]$ is the rank 1 free abelian group with generators the elements of $\Lambda[\alpha]$ and the relation that the sum of the two elements in $\Lambda[\alpha]$ is equal to zero. Recall that $\Lambda[\alpha]$ is defined in Subsection \ref{subsec:ors}. A choice of an element in each $\Lambda[\alpha]$ identifies \eqref{eq:complexwithors} with \eqref{eq:complexwithoutors}. Then, given $o_{\alpha_1}\in \Lambda[\alpha_1]$ and $o_{\alpha_2}\in \Lambda[\alpha_2]$, the moduli space $\breve{M}(\alpha_1,\alpha_2)_0$ is oriented using the rules given in Subsection \ref{subsec:ors}, and $\# \breve{M}(\alpha_1,\alpha_2)_0$ is the count of this oriented $0$-manifold.

Now $(C(Y,K),d)$ is a chain complex, as follows from the usual argument by virtue of Propositions \ref{prop:mod1}, \ref{prop:mod2} and \ref{prop:mod3}. Specifically, the boundary of a compactified 1-dimensional moduli space $\breve{M}^+(\alpha_1,\alpha_2)_1$ consists of the components $\breve{M}^+(\alpha_1,\beta)_0 \times \breve{M}^+(\beta,\alpha_2)_0$. The relation $d^2=0$ is depicted as {\smash{\raisebox{-.1\height}{\includegraphics[scale=0.5]{graphics/dmap}}}}{\smash{\raisebox{-.1\height}{\includegraphics[scale=0.5]{graphics/dmap}}}} $=0$. Although not reflected in the notation, $C(Y,K)$ and $d$ depend on the choices of metric $g_o$ and perturbation $\pi$. Define
\[
	I_\ast(Y,K) = H_\ast(C(Y,K),d).
\]
Then $I(Y,K)$ is a $\Z/4$-graded abelian group. We call $I(Y,K)$ the {\emph{irreducible}} instanton homology of $(Y,K)$, as it only takes into account the irreducible critical points of the Chern-Simons functional. It is a singular, or orbifold, analogue of Floer's $\Z/8$-graded instanton homology $I(Y)$ for homology 3-spheres from \cite{floer:inst1}.

Now suppose we have a cobordism of pairs $(W,S):(Y,K)\to (Y',K')$ with metric and perturbations compatible with ones chosen for the boundaries. Suppose further that our cobordism $(W,S)$ is a negative definite pair in the sense of Definition \ref{def:negdef}. Then we have a map $\lambda=\lambda_{(W,S)}:C(Y,K)\to C(Y',K')$ defined by
\begin{equation*}\label{chain-map-sing-flo}
	\lambda(\alpha) = \sum_{\substack{\alpha'\in \fC_{\pi'}^\text{irr} \\ \text{gr}(W,S;\alpha,\alpha')\equiv 0}} \#M(W,S;\alpha,\alpha')_0\cdot \alpha'
\end{equation*}
where $\fC_{\pi}$ and $\fC_{\pi'}$ are the corresponding critical sets, and $\alpha\in \fC_{\pi}^\text{irr}$. More precisely, using the complexes \eqref{eq:complexwithors}, we orient $M(W,S;\alpha,\alpha')_0$ using $o_\alpha\in \Lambda[\alpha]$ and $o_{\alpha'}\in \Lambda[\alpha']$ as described in Subsection \ref{subsec:ors}, and $\#M(W,S;\alpha,\alpha')_0$ is defined using this orientation. Note that because $b^1(W)=b^+(W)=0$, we have a canonical homology orientation of $(W,S)$. 

Although the map $\lambda$ in general depends on the metric and perturbations, we have omitted these dependencies from the notation. We depict the map $\lambda$ by an undecorated picture of $S$, given for example by {\smash{\raisebox{-.2\height}{\includegraphics[scale=0.25]{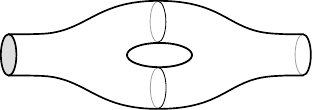}}}}. Write $d'$ for the differential of $C(Y',K')$. Then we have
\[
	d' \circ \lambda - \lambda \circ d = 0
\]
with the two terms representing factorizations {\smash{\raisebox{-.2\height}{\includegraphics[scale=0.25]{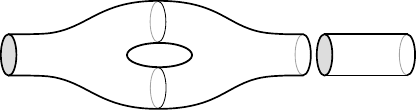}}}} and {\smash{\raisebox{-.2\height}{\includegraphics[scale=0.25]{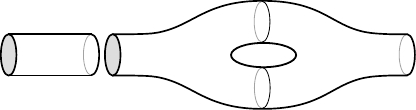}}}} corresponding to the boundary points of $M^+(W,S;\alpha,\alpha')_1$ from $M(W,S;\alpha,\beta')_0\times \breve{M}(\beta',\alpha')_0$ and $\breve{M}(\alpha,\beta)_0\times M(W,S;\beta,\alpha')_0$ respectively. 

We mention three standard properties of these cobordism maps. First, the composition of two cobordism maps is chain chomotopic to the map associated to the composite cobordism. That is, if we write $(W,S)=(W_2\circ W_1, S_2\circ S_1)$, then there exists $\phi$ such that
\begin{equation}
	\lambda_{(W_2,S_2)}\circ \lambda_{(W_1,S_1)}  - \lambda_{(W, S)} = d \circ \phi - \phi \circ d \label{eq:compositemap}
\end{equation}
where on the composite cobordism $W$ one takes the composite metric and perturbation data. The second property is similar: there is a chain homotopy between two cobordism maps $\lambda_{(W,S)}$ that are defined using different perturbation and metric data on the interior of $W$. Finally, the third property says that if $(W,S):(Y,K)\to (Y,K)$ is diffeomorphic to a cylinder with equal auxiliary data at the ends, then $\lambda_{(W,S)}$ is chain homotopic to the identity map on $C(Y,K)$. The chain homotopy in \eqref{eq:compositemap} is defined by counting isolated points in the moduli space of $G$-instantons, where $G$ is a 1-parameter family of metrics with perturbations interpolating between the composite auxiliary data on $W^+$ and the result of stretching $W^+$ along the 3-manifold at which $W_1$ and $W_2$ are glued. The other two properties are proven similarly. See e.g. \cite[Section 5.3]{donaldson-book}.

\begin{remark}
	In the verification of equation \eqref{eq:compositemap} it is important that there is only one reducible and that it remains unobstructed and isolated in the moduli space of $G$-instantons, where $G$ is the 1-parameter family of auxiliary data. This follows from our assumption that the cobordisms are negative definite pairs and the discussion in Subsection \ref{subsec:red}. This is in contrast to the analogous situation in Seiberg--Witten theory for 3-manifolds, where in the same situation one might encounter degenerate reducibles. $\diamd$
\end{remark}

We write $I(W,S):I(Y,K)\to I(Y',K')$ for the map induced by $\lambda_{(W,S)}$ on homology. Following \cite[Section 5.3]{donaldson-book}, the above properties imply that $I(Y,K)$ is an invariant of $(Y,K)$, i.e., its isomorphism class does not depend on the metric and perturbation chosen to define $C(Y,K)$. Along with Theorem \ref{thm:herald}, we obtain the following result, stated as Theorem \ref{singular-Floer-Euler} in the introduction.

\begin{theorem}\label{thm:irrhomthy}
	Let $Y$ be an integer homology 3-sphere and $K\subset Y$ a knot. Then the $\Z/4$-graded abelian group $I_\ast(Y,K)$ is an invariant of the equivalence class of the knot $(Y,K)$. The euler characteristic of the irreducible instanton homology $I_\ast(Y,K)$ is
	\[
		\chi \left(I_\ast(Y,K)\right)  = 4 \lambda(Y) + \frac{1}{2}\sigma(K).
	\]
	where $\lambda(Y)$ is the Casson invariant and $\sigma(K)$ is the knot signature.
\end{theorem}

\subsection{The operators $\delta_1$ and $\delta_2$}\label{sec:delta1and2}

The chain complex $(C,d)$ is defined only using irreducible critical points. To begin incorporating the reducible flat connection, we define two chain maps $\delta_1:C_1\to \Z$ and $\delta_2:\Z\to C_{-2}$, analagous to maps defined in the non-singular setting using the trivial connection, see \cite{froyshov} and \cite[Ch. 7]{donaldson-book}. We define 
\[
	\delta_1(\alpha) = \# \breve{M}(\alpha,\theta)_0, \qquad \delta_2(1) = \sum_{\substack{\alpha\in \fC^\text{irr}_{\pi} \\ \text{gr}(\alpha)\equiv 2}} \#\breve{M}(\theta,\alpha)_0.
\]
More precisely, the signs of these maps are defined, using the complexes \eqref{eq:complexwithors}, as follows. The map $\delta_1$ is straightforward: a choice $o_\alpha\in \Lambda[\alpha]$ in the chain complex $C(Y,K)$ determines an orientation of $M(\alpha,\theta)_0$ and hence of $\breve{M}(\alpha,\theta)_0$ as described in Subsection \ref{subsec:ors}, and $\# \breve{M}(\alpha,\theta)_0$ is defined by using this orientation. For $\delta_2$, we use the following rule. Given $o_\alpha\in \Lambda[\alpha]$, the moduli space $M(\theta,\alpha)_0$ obtains an orientation $o'$ by requiring that
\begin{equation}
	\Phi(o' \otimes o_\alpha) \in \Lambda[I\times Y, I\times K; \theta_-, \theta_-]\label{eq:orientdelta2}
\end{equation}
is the {\emph{negative}} of the preferred element in this set. (Our particular choice of convention is not important, but gives the signs that we use in our relations below.) Then $\breve{M}(\theta,\alpha)_0$ is oriented from $M(\theta,\alpha)_0$ as in Subsection \ref{subsec:ors}, from which $\#\breve{M}(\theta,\alpha)_0$ is defined.

\begin{remark}\label{rmk:delta2or}
	Recall that elements in the set appearing in \eqref{eq:orientdelta2} may be identified with orientations of \eqref{eq:reddetlinemin} upon setting $(W,S)=(I\times Y, I \times K)$. From the discussion there, a preferred orientation depends on the orientation of $S=I \times K$. This is the first point at which we use our chosen orientation of $K$. $\diamd$
\end{remark}

Just as in the non-singular case, see \cite[Section 7.1]{donaldson-book}, we have the chain relations
\begin{equation}
	\delta_1\circ d=0, \qquad d\circ \delta_2=0,\label{eq:deltamaps}
\end{equation}
which follow by counting the boundary points of 1-manifolds of the form $\breve{M}^+(\alpha_1,\alpha_2)_1$ where one of $\alpha_1$ or $\alpha_2$ is the reducible $\theta$. We depict $\delta_1$ and $\delta_2$ as {\smash{\raisebox{-.1\height}{\includegraphics[scale=0.5]{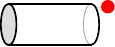}}}} and{\smash{\raisebox{-.1\height}{\includegraphics[scale=0.5]{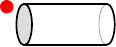}}}} respectively, placing a dot at the end of the cylinder that has a reducible flat limit. Then the relations in \eqref{eq:deltamaps} are {\smash{\raisebox{-.1\height}{\includegraphics[scale=0.5]{graphics/dmap}}}}{\smash{\raisebox{-.1\height}{\includegraphics[scale=0.5]{graphics/delta1map}}}} $=0$ and {\smash{\raisebox{-.1\height}{\includegraphics[scale=0.5]{graphics/delta2map}}}}{\smash{\raisebox{-.1\height}{\includegraphics[scale=0.5]{graphics/dmap}}}} $=0$, respectively.

Now suppose $(W,S):(Y,K)\to (Y',K')$ is a cobordism of pairs. Then we define maps $\Delta_1=\Delta_{1,(W,S)}:C(Y,K)\to \Z$ and $\Delta_2=\Delta_{2,(W,S)}:\Z\to C(Y',K')$ as follows:
\[
	\Delta_1(\alpha) = \# M(W,S;\alpha,\theta')_0, \qquad \Delta_2(1) = \sum_{\substack{\alpha'\in \fC^\text{irr}_{\pi'} \\ \text{gr}(W,S;\theta,\alpha')\equiv 0}} \#M(W,S;\theta,\alpha')_0.
\]
The signed counts are determined as follows. Identify the chain complexes as generated by orientations as in \eqref{eq:complexwithors}. First, for the map $\Delta_1$, an element $o_\alpha\in \Lambda[\alpha]$ determines an orientation $o'$ of the moduli space $M(W,S;\alpha,\theta')_0$ by the requirement
\[
	\Phi(o_\alpha\otimes o_W) = o'
\]
where $o_W$ is the canonical homology orientation of $(W,S)$. Next, for $\Delta_2$, we define an orientation $o'$ of $M(W,S;\theta,\alpha')_0$ given $o_{\alpha'}\in \Lambda[\alpha']$ by requiring that 
\[
	\Phi(o' \otimes o_{\alpha'}) \in \Lambda[W, S; \theta_-, \theta'_-]
\]
is the canonical element. Note that this rule for $\Delta_2$ depends on the orientation of $S$, just as when we defined the orientation rule for $\delta_2$.

The maps $\Delta_1$ and $\Delta_2$ are depicted by placing dots at the appropriate ends of a picture for $S$, e.g. {\smash{\raisebox{-.2\height}{\includegraphics[scale=0.25]{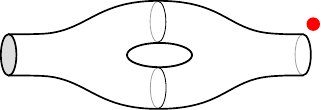}}}} and {\smash{\raisebox{-.2\height}{\includegraphics[scale=0.25]{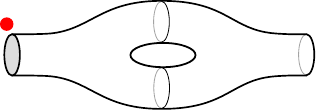}}}}.  The following is an analogue of \cite[Lemma 1]{froyshov}.

\begin{prop}\label{prop:cobdelta} Suppose $(W,S):(Y,K)\to (Y',K')$ is a negative definite pair. Then
\begin{itemize}
	\item[{\emph{(i)}}] $ \Delta_1 \circ d  + \delta_1 - \delta'_1\circ \lambda = 0$,
	\item[{\emph{(ii)}}] $d' \circ \Delta_2 -\delta'_2 + \lambda\circ \delta_2  = 0$.
\end{itemize}
\end{prop}

\begin{proof}
Consider (i). This relation can be verified by counting the ends of the 1-dimensional moduli space $M(W,S;\alpha,\theta')_1$. Studying the ends of such moduli spaces relies on the on the compactness and gluing theory of the moduli spaces of singular instantons, which were reviewed in Subsection \ref{subsec:red}. There are three types of ends in this moduli space. The first two types are cylinders on components of $\breve{M}(\alpha,\beta)_0\times M(W,S;\beta,\theta')_0$ and $M(W,S;\alpha,\beta')_0\times \breve{M}(\beta',\alpha')_0$, corresponding to instantons approaching trajectories that are broken along irreducible critical points. Counting these contributions gives $\Delta_1 \circ d(\alpha)- \delta_1'\circ \lambda(\alpha)$. 

\begin{figure}[t]
\centering
\includegraphics[scale=0.75]{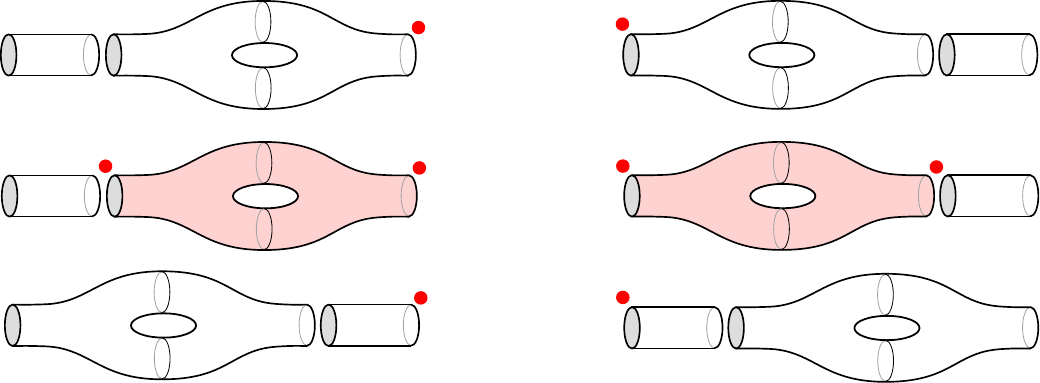}
\caption{The relations (i) (left) and (ii) (right) of Proposition \ref{prop:cobdelta} .}
\label{fig:deltarels}
\end{figure}

The third type of end in $M(W,S;\alpha,\theta')_1$ consists of singular instantons which factorize into an instanton $[A]\in \breve{M}(\alpha,\theta)_0$ grafted to a reducible instanton on $(W,S)$. The condition that $(W,S)$ is a negative definite pair implies that there is a unique such reducible connection class, and by our discussion in Subsection \ref{subsec:red} it is unobstructed, so that the standard gluing theory applies. This third type of end thus contributes the term $\delta_1(\alpha)$.

 The proof of (ii) is similar. 
\end{proof}

\begin{remark}
	The verification of the signs in the above relations is straightforward given our conventions for orienting moduli spaces. The argument is similar, for example, to the proof of \cite[Proposition 20.5.2]{km:monopole}. The same remark holds for the relations that appear below. $\diamd$
\end{remark}

For a depiction of the relations in Proposition \ref{prop:cobdelta}, see Figure \ref{fig:deltarels}. A shaded picture such as {\smash{\raisebox{-.2\height}{\includegraphics[scale=0.25]{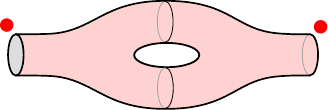}}}} is to be understood as representing a reducible singular instanton, and as a map sends the reducible to the reducible, each represented by a dot, as before.

\subsection{Holonomy operators and $v$-maps}\label{loop-mu}

In this section we describe maps that are obtained by taking the holonomies of instantons along an embedded curve $\gamma\subset S$, see e.g. \cite[Section 2.2]{KM:unknot}. We treat the following cases: (i) $\gamma$ is a closed loop and $(W,S)$ is a negative definite pair; (ii) $(W,S)$ and $\gamma$ are cylinders, i.e. $(W,S)=(I\times Y)$ and $\gamma=I\times\{y\}$, which yields the $v$-map; and (iii) $\gamma$ is a properly embedded interval intersecting both ends of a negative definite pair $(W,S)$.

\subsubsection{The case of closed loops}\label{subsec:closedloops}

Consider a negative definite pair $(W,S):(Y,K)\to (Y',K')$ and an associated configuration space $\sB(W,S;\alpha,\alpha')$ where $\alpha$ and $\alpha'$ are irreducible critical points for $(Y,K)$ and $(Y',K')$, respectively. Let $\gamma$ be a closed loop lying on the interior of the surface $S$. Suppose $\nu_S$ denotes the $S^1$-bundle associated to the normal bundle of $S$ and fix a trivialization of this bundle over $\gamma$. Given $[A]\in \sB(W,S;\alpha,\alpha')$, the adjoint connection $A^\text{ad}$ can be used to define an $S^1$-connection over $\gamma$, as described in the following paragraph.

The boundary of an $\varepsilon$ tubular neighborhood of $S$ in $W$ is naturally isomorphic to the $S^1$-bundle $\nu_S$ over $S$, and by pulling back we obtain a connection $A_\epsilon^\text{ad}$ on $\nu_S$. The limit of these connections as $\epsilon$ goes to zero defines a connection $A^\text{ad}_0$ on $\nu_S$ such that the curvature of $A^\text{ad}_0$ has the fiber of $\nu_S$ in its kernel. Fixing an orientation for the fiber of $\nu_S$ (or equivalently an orientation for $S$) determines an $S^1$-reduction of this bundle over $\nu_S$. In particular, the holonomy of this connection along a lift of $\gamma$ to $\nu_S$, given by the trivialization of $\nu_S$ over $\gamma$, defines a map which depends only on the gauge equivalence class of $A$:
\begin{equation}\label{hol-map-4d}
	 h_{\alpha\alpha'}^\gamma:\sB(W,S;\alpha,\alpha') \longrightarrow S^1.
\end{equation}
Note that to make sense of the holonomy, we must choose a basepoint, an orientation of the loop $\gamma$, an orientation of the fiber of $\nu_S$ and a lift of $\gamma$ to $\nu_S$. As $S^1$ is abelian, the map $h_{\alpha\alpha'}^\gamma$ is independent of the choice of basepoint defining the holonomy around $\gamma$. If we change the orientation of $\gamma$, then $h^\gamma_{\alpha\alpha'}$ is post-composed with the conjugation $S^1\to S^1$. Changing the orientation of $S$ has the same effect. Finally, changing the chosen lift of $\gamma$ would multiply $h_{\alpha\alpha'}^\gamma$ by $\pm 1$.

\begin{remark}
	In the sequel, we slightly abuse the description of this holonomy map by saying that we take the holonomy of 
	$A^\text{ad}$ along $\gamma$, and do not refer to the limiting process. $\diamd$
\end{remark}

Using this holonomy map we define $\mu=\mu_{(W,S,\gamma)}:C(Y,K)\to C(Y',K')$ by:
\[
	\mu(\beta_1) = \sum_{\substack{\alpha'\in\fC_{\pi'}^\text{irr} \\ \text{gr}(W,S;\alpha,\alpha')\equiv 1}} \text{deg}\left( h_{\alpha\alpha'}^\gamma|_{M(W,S;\alpha,\alpha')_1} \right)\cdot \alpha'
\]
This deserves some explanation, as the moduli space $M(W,S;\alpha,\alpha')_1$ is in general not compact. The boundary components of the compactified moduli space $M^+(W,S;\alpha,\alpha')_1$ come from two types of factorizations:
\begin{gather*}
	\breve{M}(\alpha,\beta)_0 \times M(W,S;\beta,\alpha')_0,\\
	 M(W,S;\alpha,\beta')_0\times \breve{M}(\beta',\alpha')_0.
\end{gather*}
As each of the maps $h_{\beta\alpha'}^\gamma$ and $h_{\alpha\beta}^\gamma$ are defined on $0$-dimensional moduli spaces and transverse to some generic $h\in S^1$, the preimage $(h^\gamma_{\alpha\alpha'})^{-1}(h)$ restricted to $M(W,S;\alpha,\alpha')_1$ is supported on the interior of $M^+(W,S;\beta_1,\beta_2)_1$. We may then define $\text{deg}( h_{\alpha\alpha'}^\gamma|_{M(W,S;\alpha,\alpha')_1})$ to be the oriented count of this preimage. We depict the map $\mu$ by a picture of the surface $S$ containing the closed loop $\gamma$, such as {\smash{\raisebox{-.2\height}{\includegraphics[scale=0.25]{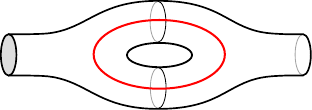}}}}.

Let us be more precise about our sign conventions for $\mu$. As usual, when determining signs, our complexes are given by \eqref{eq:complexwithors}. The moduli space $M(W,S;\alpha,\alpha')_1$ is then oriented from $o_{\alpha}\in\Lambda[\alpha]$ and $o_{\alpha'}\in \Lambda[\alpha']$ as described in Subsection \ref{subsec:ors}. In this paper, we use the convention that if $f:M\to N$ is a smooth map of oriented manifolds with regular value $y\in N$, then $f^{-1}(y)$ is oriented by the normal-directions-first convention. This orients the submanifold $(h^\gamma_{\alpha\alpha'})^{-1}(h)\subset M(W,S;\alpha,\alpha')_1$.

Now consider a holonomy map $h_{\alpha\alpha'}^\gamma$ restricted to a 2-dimensional space $M(W,S;\alpha,\alpha')_2$. The codimension-1 faces of $M^+(W,S;\alpha,\alpha')_2$ are
\begin{gather*}
	\breve{M}^+(\alpha,\beta)_{i-1}  \times M^+(W,S;\beta,\alpha')_{2-i},\\
	 M^+(W,S;\alpha,\beta')_{2-i} \times \breve{M}^+(\beta',\alpha')_{i-1}
\end{gather*} 
where $i\in\{1,2\}$. Consider the 1-manifold $(h^\gamma_{\alpha\alpha'})^{-1}(h)\subset M^+(W,S;\alpha,\alpha')_2$. Again, because $\gamma$ is supported on the interior of $S$, this 1-manifold is supported away from the codimension-1 faces with $i=2$. Counting the contributions from $i=1$ gives the relation
\begin{equation}
	d \circ \mu -  \mu \circ d = 0, \label{eq:closedloop}
\end{equation}
showing that $\mu$ is a chain map. See Figure \ref{fig:vmaprels1}.

\begin{figure}[t]
\centering
\includegraphics[scale=0.85]{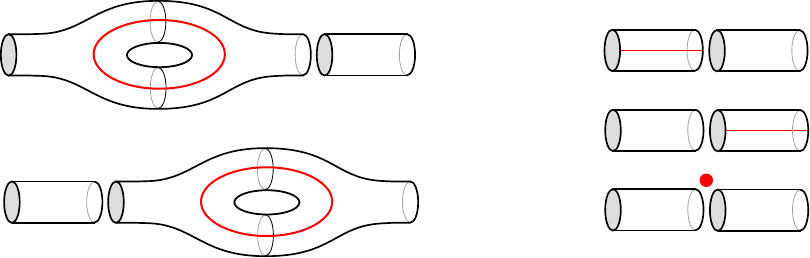}
\caption{On the left is depicted relation \eqref{eq:closedloop}; on the right, the relation of Proposition \ref{prop:vmap}.}
\label{fig:vmaprels1}
\end{figure}

\subsubsection{The case of a cylinder: the $v$-map}\label{sec:vmap}
We move on to the cases in which $\gamma$ is not closed. We first consider the case of the cylinder $\R\times (Y,K)$. The output of the construction is a degree $-2$ (mod $4$) endomorphism on $C_\ast = C_\ast(Y,K)$. Choose a basepoint $y\in K$ and a lift of each $\alpha\in\fC^\text{irr}_\pi$ to $\sC(Y,K)$ once and forever. Following a similar construction as above, we obtain a translation-invariant map
\[
	h_{\alpha_1\alpha_2}:\sB(Y,K;\alpha_1,\alpha_2) \longrightarrow S^1
\]
for each pair of irreducible critical points $\alpha_i\in\fC^\text{irr}_\pi$ on $(Y,K)$. To be more detailed, we define $\sB(Y,K;\alpha_1,\alpha_2)$ using the chosen lifts of $\alpha_1$ and $\alpha_2$, and for a given $[A]\in \sB(Y,K;\alpha_1,\alpha_2)$, the holonomy of $A^\text{ad}$ along $\R\times\{ y\}$ determines $h_{\alpha_1\alpha_2}([A])$. (A lift of $\alpha_i$'s to the space of framed connections $\widetilde \sB(Y,K)$ suffices for the definition of $h_{\alpha_1\alpha_2}$.) The induced map $h_{\alpha_1\alpha_2}$ on $\breve{M}(\alpha_1,\alpha_2)_{d}$ extends to the moduli space of unparametrized broken trajectories which break along irreducible critical points, and on the codimension-1 faces $\breve{M}^+(\alpha_1,\beta)_{i-1}\times \breve{M}^+(\beta,\alpha_2)_{d-i}$ with $\beta$ irreducible, factors accordingly as $h_{\alpha_1\beta}\cdot h_{\beta\alpha_2}$.

For $h_{\alpha_1\alpha_2}$ to be well-defined we also require that meridians have preferred directions; this is true because the knot $K$ is oriented. Changing the orientation of $K$ alters $h_{\alpha_1\alpha_2}$ by post-composition with the conjugation map $S^1\to S^1$.

To define a well-behaved map on $C_\ast$ in this situation, in general we must modify the above holonomy maps, similar to what is done in \cite{donaldson-book}. In particular, we define maps
\begin{equation*}
	H_{\alpha_1\alpha_2}:\breve{M}(\alpha_1,\alpha_2)_{d} \longrightarrow S^1 \label{eq:holonomymaps1}
\end{equation*}
by modifying the maps $h_{\alpha_1\alpha_2}$ near broken trajectories. Since we need these modified holonomy maps only for moduli spaces of up to dimension $2$, we may assume that $d\leq 2$. The maps $H_{\alpha_1\alpha_2}$ extend to unparametrized broken trajectories which break along irreducible critical points, and satisfy the following properties.
\begin{itemize}
\item[(H1)] $H_{\alpha_1\alpha_2}=1$ if the dimension $d$ of the unparametrized moduli space $\breve{M}(\alpha_1,\alpha_2)_{d}$ on which $H_{\alpha_1\alpha_2}$ is defined is equal to zero.

\item[(H2)]  For a given sequence of instantons $[B_i]$ in $\breve{M}(\alpha_1,\alpha_2)_d$ converging to a broken instanton $([B],[B'])\in \breve{M}(\alpha_1,\beta)_{i-1}\times \breve{M}(\beta,\alpha_2)_{d-i}$ where $\beta$ is an irreducible critical point, the holonomies $H_{\alpha_1\alpha_2}(B_i)$ converge to the product $H_{\alpha_1\beta}(B)\cdot H_{\beta\alpha_2}(B')$.

\item[(H3)]  If $\text{gr}(\alpha_1)\equiv 1$ and $\text{gr}(\alpha_2)\equiv 2$, there is an end of $\breve{M}(\alpha_1,\alpha_2)_{2}$, corresponding to trajectories broken along the reducible $\theta$, which by standard gluing theory can be identified with $\breve{M}(\alpha_1,\theta)_{0}\times  S^1\times \R_{>0}\times \breve{M}(\theta,\alpha_2)_{0}$. We require that the restriction of $H_{\alpha_1\alpha_2}$ to $\breve{M}(\alpha_1,\theta)_{0}\times \{T\} \times S^1\times \breve{M}(\theta,\alpha_2)_{0}$ for some (and hence any) $T\in\R_{>0}$ is a degree $1$ map on each circle component.
\end{itemize}
The unmodified holonomy maps $h_{\alpha_1\alpha_2}$ satisfy the properties (H2) and (H3) but do not necessarily satisfy (H1). The modified holonomy maps are constructed in a way that is inspired by what is done in \cite[Section 7.3.2]{donaldson-book}. See Appendix \ref{app:hol} for details on the contruction of the modified holonomy maps $H_{\alpha_1\alpha_2}$ which satisfy (H1)--(H3).

We may now define an operator $v: C_\ast \to C_{\ast-2}$ as follows:
\begin{equation}
	v(\alpha_1) = \sum_{\substack{\alpha_2\in\fC_\pi^\text{irr} \\ \text{gr}(\alpha_1,\alpha_2)\equiv 2}} \text{deg}\left( H_{\alpha_1\alpha_2}|_{\breve{M}(\alpha_1,\alpha_2)_1} \right)\cdot \alpha_2 \label{eq:vmapdef}
\end{equation}
The degree may be computed by taking the preimage of a generic $h\in S^1\setminus \{1\}$. By property (H1), such a preimage is supported away from the ends of $\breve{M}(\alpha_1,\alpha_2)_1$, and generically is a finite set of oriented points. The expression $\text{deg}( H_{\alpha_1\alpha_2}|_{\breve{M}(\alpha_1,\alpha_2)_1} )$ is defined to be the signed count of these points.

 The following proposition is a singular instanton analogue of \cite[Proposition 7.8]{donaldson-book} and \cite[Theorem 4]{froyshov}, and its proof is analogous. The main difference is that $SO(3)$, which in the non-singular setting plays both the role of the stabilizer of the trivial connection and the codomain of the (adjoint) holonomy maps, is replaced in the singular setting by $S^1$.

\begin{prop}\label{prop:vmap}
	$d \circ v - v \circ d  - \delta_2 \circ \delta_1 = 0$. 
\end{prop}

\begin{proof}
Consider the 1-dimensional moduli space 
\[
	M:=\{[A]\in \breve{M}(\alpha_1,\alpha_2)_2: H_{\alpha_1\alpha_2}([A])=h\}
\]
for some generic $h\in S^1\setminus\{1\}$. Studying the ends of $M$ will lead to the desired relation. As the dimension of $M(\alpha_1,\alpha_2)_3$ is 3, there is no bubbling, and $\breve{M}^+(\alpha_1,\alpha_2)_2$ is compact. Thus an end of $M$ contains a sequence of instantons that approaches an unparametrized broken trajectory $\fa=([A_1],\ldots,[A_{l-1}])$ where $l\geqslant 3$, $[A_i]\in \breve{M}(\beta_i,\beta_{i+1})_{d_i}$ and $\beta_1=\alpha_1$, $\beta_{l}=\alpha_2$. By index additivity and dimension considerations, $l\leqslant 4$.

 First suppose that each $\beta_i$ is irreducible. Then $H_{\alpha_1\alpha_2}$ factors as $\prod_{i=1}^{l-1} H_{\beta_i\beta_{i+1}}$ near $\fa$, as follows inductively from property (H2) above. If $l=4$, then each $[A_i]$ is of index 1, and since $H_{\beta_i\beta_{i+1}}=1$ for such instantons by property (H1), this case does not occur.  From now on we may assume $l=3$, so that $\fa=([A_1],[A_2])$. Write $\beta=\beta_2$.

 Suppose $[A_1]$ is of index 1 and $[A_2]$ of index 2, i.e. $\fa\in \breve{M}(\alpha_1,\beta)_0\times \breve{M}(\beta,\alpha_2)_1$. For these types of breakings, consider a corresponding gluing map
 \[
 	  \psi_\beta:\breve{M}(\alpha_1,\beta)_0  \times \R_{>0}\times \breve{M}(\beta,\alpha_2)_1\to \breve{M}(\alpha_1,\alpha_2)_2
 \]
which is a diffeomorphism onto its image. Let $N_{0,1}^\beta\subset M$ be the image of this map, with $\R_{>0}$ restricted to some $(T,\infty)$, intersected with $M$. Here and in what follows, $T>0$ is some large generic number to be specified later. Consider the map
\[
		f:\breve{M}(\alpha_1,\beta)_0\times \breve{M}(\beta,\alpha_2)_1 \to S^1
\]
given by $f([A_1],[A_2])=H_{\alpha_1\alpha_2}(\psi_\beta([A_1],T,[A_2]))$. With $T$ generically chosen, $h$ is a regular value of $f$. Then $f^{-1}(h)$ consists of the boundary points of $M\setminus N_{0,1}^\beta$ that are adjacent to $N^\beta_{0,1}$. Taking $T\to \infty$ and using (H2), $f$ is homotopic to $([A_1],[A_2])\mapsto H_{\alpha_1\beta}(A_1)H_{\beta\alpha_2}(A_2)$. For $[A_1]$ fixed, this latter map has the same degree as $H_{\beta\alpha_2}$, which is $\pm \langle v(\beta),\alpha_2\rangle$. There are $\pm \langle d(\alpha_1),\beta\rangle$ many choices for $[A_1]\in \breve{M}(\alpha_1,\beta)_0$. Ranging over all possibilities for $\beta$, and using our orientation conventions, we find that the total number of boundary points of this type is given by $-\langle vd(\alpha_1),\alpha_2\rangle$.

For breakings where instead $[A_1]$ is of index 2 and $[A_2]$ of index 1, we define a corresponding subset $N_{1,0}^\beta\subset M$ using the image of a gluing map and intersecting with $M$. A similar argument to the above shows that the number of boundary points in this situation, again ranging over all relevant irreducible $\beta$, is given by $\langle d v (\alpha_1), \alpha_2\rangle$.

There is also the case in which there is a sequence of instantons in $M$ approaching an unparametrized broken trajectory $\fa=([A_1],[A_2])$ that factors through the reducible $\theta$. Such a sequence eventually lies in the image of a gluing map
 \[
 	\psi_\theta :  \breve{M}(\alpha_1,\theta)_0 \times S^1 \times \R_{>0}\times \breve{M}(\theta,\alpha_2)_0\to \breve{M}(\alpha_1,\alpha_2)_2
 \]
Let $N^\theta$ be the image of this map, with $\R_{>0}$ restricted to the interval $(T,\infty)$, intersected with the moduli space $M$. Consider the boundary points of $M\setminus N^\theta$ that are adjacent to $N^\theta$. These are in $(f')^{-1}(h)$ where
\[
		f':\breve{M}(\alpha_1,\theta)_0\times S^1\times \breve{M}(\theta,\alpha_2)_0 \to S^1
\]
is defined by $f'([A_1],[A_2]) = H_{\alpha_1\alpha_2}(\psi_\theta([A_1],g,T,[A_2]))$. Using (H3), the map $f'$ is homotopic to a map which for each point in $\breve{M}(\alpha_1,\theta)_0\times \breve{M}(\theta,\alpha_2)_0$ restricts to a degree $1$ map $S^1\to S^1$. The number of boundary points, using our orientation conventions, is equal to $-\langle \delta_2 \delta_1(\alpha_1),\alpha_2\rangle$.

Now consider $M'$, the complement in $M$ of the open sets $N_{0,1}^\beta$, $N_{1,0}^\beta$, $N^\theta$ defined above, where $\beta$ ranges over irreducible critical points with appropriate index in each case. We choose $T$ large enough so that these open sets are disjoint, and generic so that $M'$ is a 1-manifold with boundary. By the discussion in Subsection \ref{subsec:red}, $M'$ is compact. The boundary points have been counted above, and yield the desired relation. This completes the proof. The three types of factorizations are depicted in Figure \ref{fig:vmaprels1}.
\end{proof}

\begin{remark}\label{rmk:truncations}
	Ideally, the compactified moduli space would have the structure of a smooth manifold with corners near broken trajectories, induced by the natural compactifications of the domains of the gluing maps. In the above proof, we would then define a smooth manifold with boundary $M^+\subset \breve{M}^+(\alpha_1,\alpha_2)_2$ and simply count its boundary points. (Technically, in this strategy, we would also modify our compactification to include gluing parameters for reducibles.) However, proving that the compactified moduli space indeed has such structure is a technical issue which also arises in other Floer theories, and we would like to avoid it. This is why, in the above proof, we consider a truncation of the moduli space to obtain a compact manifold with boundary. A similar approach is employed in \cite{SS}. In the sequel, we will often ignore this subtlety, with the understanding that one should always really truncate the moduli spaces as done in the proof above.	$\diamd$
\end{remark}

\begin{remark}\label{rmk:deg}
	By properties (H1) and (H2), the map $H_{\alpha_1\alpha_2}$ is defined on the compactified 1-manifold $\breve{M}^+(\alpha_1,\alpha_2)_1$ and sends the boundary components $\breve{M}(\alpha_1,\beta)_0\times\breve{M}(\beta,\alpha_2)_0$ to the identity $1\in S^1$. Thus $H_{\alpha_1\alpha_2}$ descends to a map $M\to S^1$, where $M$ is the disjoint union of circles obtained by identifying the boundary points of $\breve{M}^+(\alpha_1,\alpha_2)_1$ in pairs. Then the degree of $H_{\alpha_1\alpha_2}$ appearing in \eqref{eq:vmapdef} is nothing more than the degree of $M\to S^1$.  In particular, this shows that the map $v:C_\ast \to C_{\ast-2}$ is independent of the choice of $h\in S^1$. Changing the choice of modified holonomy maps will alter $v$ by a chain homotopy. The proof is a standard continuation map argument. See also Theorem \ref{thm:framedcat} below.
	$\diamd$
\end{remark}

\begin{remark}\label{rmk:signsofvmaprel}
Our construction of $v$ depends on the orientation of $K$. If the orientation is reversed, then $v$ changes sign. Similar remarks hold for $\delta_2$, following Remark \ref{rmk:delta2or}. However, the maps $d$ and $\delta_1$ do not depend on the orientation of $K$. It follows that the relation of Proposition \ref{prop:vmap} is invariant under orientation-reversal of $K$. $\diamd$
\end{remark}

\subsubsection{Curves with boundary in cobordisms}\label{cut-down}

Now consider any negative definite pair $(W,S):(Y,K)\to (Y',K')$. We assume that $\gamma$ is an embedded interval in $(W,S)$ with its boundary intersecting both $(Y,K)$ and $(Y',K')$. Denote by $p\in K$ and $p'\in K'$ the boundary points of $\gamma$ in $S$. Recall that $W^+$ is obtained from $W$ by attaching cylindrical ends $(-\infty, 0]\times Y$ and $[0,\infty)\times Y'$, and $S^+$ is obtained from $S$ similarly, by attaching $(-\infty, 0]\times K$ and $[0,\infty)\times K'$. We extend $\gamma$ to a non-compact curve $\gamma^+\subset S^+$ by attaching $(-\infty,0]\times \{p\}$ and $[0,\infty )\times \{p'\}$. We obtain a map
\[
	h_{\alpha\alpha'}^\gamma:\sB(W,S;\alpha,\alpha') \longrightarrow S^1
\]
by taking as before the holonomy of the adjoint connection $A^\text{ad}$ along $\gamma^+$ compatible with the $S^1$ reduction of the bundle along $S^+$. Modified holonomy maps $H^\gamma_{\alpha\alpha'}$ of Appendix \ref{app:hol} give
\[
	\mu(\alpha) = \sum_{\substack{\alpha'\in\fC_{\pi'}^\text{irr} \\ \text{gr}(W,S;\alpha,\alpha')\equiv 1}} \text{deg}\left( H^\gamma_{\alpha\alpha'}|_{M(W,S;\alpha,\alpha')_1} \right)\cdot \alpha'.
\]
The relation in the following proposition is depicted in Figure \ref{fig:vmaprels2} in the case that the curve $\gamma\subset S$ looks like {\smash{\raisebox{-.2\height}{\includegraphics[scale=0.25]{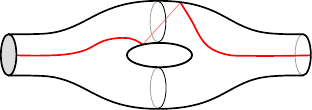}}}}. It is an analogue of the relation in the non-singular setting of \cite[Theorem 6]{froyshov}, and the proof is similar to that of Proposition \ref{prop:vmap}.

\begin{figure}[t]
\centering
\includegraphics[scale=0.75]{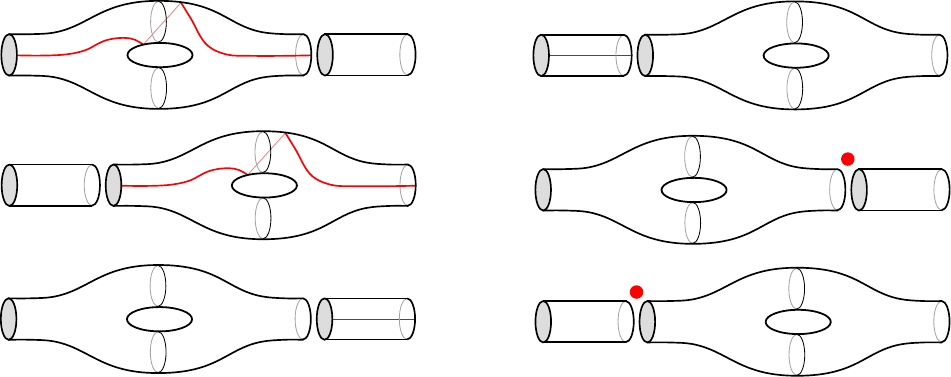}
\caption{The relation of Proposition \ref{prop:vmaprels2}.}
\label{fig:vmaprels2}
\end{figure}

\begin{prop}\label{prop:vmaprels2}
	 $d' \circ \mu + \mu \circ d + \Delta_2 \circ \delta_1 - \delta_2' \circ \Delta_1 - v' \circ \lambda +\lambda\circ v = 0$.
\end{prop}

Similar to Remark \ref{rmk:signsofvmaprel}, in this situation, the maps $d,d',\Delta_1,\lambda$ do not depend on the orientations of $K,K',S$, while $v,v',\mu,\delta_2,\delta_2',\Delta_2$ do, and change sign under orientation-reversal. Thus the relation of Proposition \ref{prop:vmaprels2} is invariant under orientation-reversal.

\subsection{A framed instanton homology for knots}\label{sec:framed}

We now assemble the above data to define a framed instanton chain group $(\widetilde{C}(Y,K), \widetilde{d})$. 
The main apparatus is the following.

\begin{definition}\label{S-comp-first}
	A chain complex $(\widetilde C_*,\widetilde d)$ is an {\emph{$\cS$-complex}} if there are a finitely generated free
	chain complex $(C_*,d)$ and graded maps 
	$v:C_*\to C_{*-2}$, $\delta_1:C_1\to \Z$ and $\delta_2:\Z \to C_{-2}$ such that $\widetilde C_\ast=C_*\oplus C_{*-1}\oplus \Z$, and such that the differential is given by
				\begin{equation}\label{eq:d-tilde}
					\widetilde d=\left[
					\begin{array}{ccc}
						d&0&0\\
						v&-d&\delta_2\\
						\delta_1&0&0\\
					\end{array}
					\right].
				\end{equation}
The copy of $\Z$ in the decomposition of $\widetilde C_\ast$ is supported in grading $0$. $\diamd$
\end{definition}
Because $(\widetilde C_\ast, \widetilde d)$ is a chain complex, $\widetilde d\circ \widetilde d=0$, which is equivalent to $d\circ d=0$ and
\begin{align}
	\delta_1\circ d &=0 \label{eq:d-tilderel1}\\
	d\circ \delta_2& =0\label{eq:d-tilderel2}\\
	d \circ v-v \circ d - \delta_2\circ \delta_1 & = 0\label{eq:d-tilderel3}
\end{align}

\begin{remark}\label{cS-coefficient}
	 More generally, if $R$ is any commutative ring, an $\cS$-complex over $R$ is defined to be a finitely generated free chain complex over $R$, with the same structure as in the definition above, replacing each instance of $\Z$ with $R$. If no ring is specified, the reader can safely assume that we are working over the integers. $\diamd$
\end{remark}

\begin{remark}\label{cS-gradings}
	In the above definition of an $\cS$-complex, the chain complex comes with a $\Z$-grading which decreases the differential by 1. However, in the sequel we will consider $\cS$-complexes graded by $\Z/2N$ for some positive integer $N$, in which case all grading subscripts in the above definition should be taken modulo $2N$. For technical reasons, an $\cS$-complex with no grading must be defined over a ring of characteristic two.
	$\diamd$
\end{remark}

\begin{definition}\label{def:morphism}
A {\emph{morphism}} $\widetilde{\lambda}: (\widetilde C_\ast, \widetilde d)\to (\widetilde C'_\ast, \widetilde d')$ of $\cS$-complexes is a degree zero chain map that may be written in the form
\begin{equation}\label{eq:tilde-map}
	\widetilde \lambda=\left[
					\begin{array}{ccc}
						\lambda &0&0\\
						\mu &\lambda &\Delta_2\\
						\Delta_1&0&1\\
					\end{array}
					\right]
\end{equation}
with respect to decompositions $\widetilde C_\ast=C_*\oplus C_{*-1}\oplus \Z$ and $\widetilde C'_\ast=C'_*\oplus C'_{*-1}\oplus \Z$. $\diamd$
\end{definition}

The condition of $\widetilde \lambda$ being a chain map is equivalent to the following relations:
\begin{align*}
	\lambda \circ d - d' \circ \lambda & = 0\\
	  \Delta_1 \circ d  + \delta_1 - \delta_1'\circ \lambda & = 0\\
	   d' \circ \Delta_2 -\delta_2' + \lambda\circ \delta_2 & = 0 \\
	    \mu \circ d +  \lambda \circ v + \Delta_2 \circ \delta_1 - v' \circ \lambda + d' \circ \mu - \delta_2'\circ \Delta_1 &= 0
\end{align*}

\begin{definition}\label{def:shomotopy}
An $\cS$-{\emph{chain homotopy}} $ (\widetilde C_\ast, \widetilde d)\to (\widetilde C'_\ast, \widetilde d')$ of $\cS$-complexes is a chain homotopy of complexes that when written with respect to decompositions $\widetilde C_\ast=C_*\oplus C_{*-1}\oplus \Z$ and $\widetilde C'_\ast=C'_*\oplus C'_{*-1}\oplus \Z$ takes the following form:
				\begin{equation*}\label{eq:tilde-chainhomotopy}
					\left[
					\begin{array}{ccc}
						K &0&0\\
						L&-K& M_2\\
						M_1&0&0\\
					\end{array}
					\right]. 
				\end{equation*}
A {\emph{chain homotopy equivalence}} of $\cS$-complexes is a pair of morphisms $f:\widetilde C_\ast\to \widetilde C'_\ast$ and $g: \widetilde C'_\ast\to\widetilde C_\ast$ of $\cS$-complexes with the property that $f\circ g $ and $g\circ f$ are $\cS$-chain homotopy equivalent to identity morphisms. $\diamd$
\end{definition}

Let $Y$ be an integer homology 3-sphere containing a knot $K$, and $(C_\ast,d) = (C_\ast(Y,K),d)$ its irreducible instanton chain complex for some choice of metric and perturbation. Choose a basepoint $p\in K$. We define $\widetilde C_\ast (Y, K) =  C_\ast (Y, K) \oplus  C_{\ast-1} (Y, K)\oplus \Z$, and $\widetilde d$ by \eqref{eq:d-tilde} using the maps from the previous subsections; note that the choices of basepoint and modified holonomy maps are required to define $v$. Then \eqref{eq:deltamaps} verifies \eqref{eq:d-tilderel1}--\eqref{eq:d-tilderel2}, and Proposition \ref{prop:vmap} gives \eqref{eq:d-tilderel3}. Thus $(\widetilde C_\ast (Y, K),\tilde d)$ is a $\Z/4$-graded $\cS$-complex. Its homology,
\[
	\widetilde I_\ast(Y,K) = H_\ast ( \widetilde C(Y,K), \widetilde d),
\]
is a $\Z/4$-graded abelian group, which we call the {\emph{framed}} instanton homology of $(Y,K)$.

\begin{remark}
	Our framed instanton chain complex should be distinguished from what is called framed instanton homology for knots in \cite{KM:YAFT}. However, the relationships established in Section \ref{sec:kmgroups} justify the overlapping use of terminology. $\diamd$
\end{remark}

\begin{remark}\label{cS-com-non-sing}
	In the non-singular set up, Donaldson introduces a similar object for any integer homology 
	sphere $Y$ called an $(\cF,\sigma)$-complex \cite{donaldson-book}. 
	This object, denoted by $(\widetilde C_*(Y),\widetilde d)$, is essentially an 
	$\cS$-complex. The main differences are that the complex $(\widetilde C_*(Y),\widetilde d)$ is defined over $\Q$,
	and $\widetilde C_*(Y)=C_*(Y)\oplus C_{*-3}(Y)\oplus \Q$, where $C_*(Y)$ is a $\Z/8$-graded complex.
	The complex $\widetilde C_*(Y)$ is defined out of a theory which is $SO(3)$-equivariant, rather than 
	$S^1$-equivariant. This is the primary reason for the appearance of a degree shift by $3$ (dimension of $SO(3)$) 
	instead of $1$ (dimension of $S^1$). $\diamd$
\end{remark}

Now suppose $(W,S):(Y,K)\to (Y',K')$ is a negative definite pair. We also assume that a properly embedded arc $\gamma$ in $S$ is fixed such that it forms a cobordism from the basepoint of $K$ to the basepoint of $K'$. We slightly abuse terminology and refer to the data of $(W,S)$ with $\gamma$ a negative definite pair, and omit $\gamma$ from our notation. Upon choosing metric, perturbation data, framings of critical points and modified holonomy maps, we define a map
\[
  \widetilde \lambda= \widetilde \lambda_{(W,S)}:\widetilde C(Y,K) \longrightarrow \widetilde C(Y',K')
\]
by the expression \eqref{eq:tilde-map}, using the maps previously defined for $(W,S)$ and $\gamma$. By virtue of Propositions \ref{prop:cobdelta}  and \ref{prop:vmaprels2}, $\widetilde \lambda$ is a morphism of $\cS$-complexes.

\begin{definition}
	Define $\cH$ to be the category whose objects are pairs $(Y,K)$ where $Y$ is an integer homology 3-sphere and $K\subset Y$ is an oriented based knot, and whose morphisms are negative definite cobordisms of pairs $(W,S):(Y,K)\to (Y',K')$ in the sense of Definition \ref{def:negdef}, equipped with an embedded arc on $S$ connecting the basepoints of $K$ and $K'$. Composition of morphisms in this category is defined by composing cobordisms. $\diamd$
\end{definition}

For any category with a notion of chain homotopy, we define the associated {\emph{homotopy category}} to be the category with the same objects, but whose morphisms are chain homotopy equivalences of morphisms. We have the following analogue of \cite[Theorem 7.11]{donaldson-book}.

\begin{theorem}\label{thm:framedcat}
	The assignments $(Y,K)\mapsto (\widetilde C(Y,K),\widetilde d)$ and $(W,S)\mapsto \widetilde \lambda_{(W,S)}$ induce a functor of categories from $\mathcal{H}$ to the homotopy category of $\Z/4$-graded $\cS$-complexes.
\end{theorem}

\begin{proof}
	Let  $(W,S):(Y,K)\to (Y',K')$ be a negative definite pair with an embedded arc $\gamma$ on $S$ connecting the basepoints. 
	Fix the required auxiliary choices (Riemannian metric, perturbation and modified holonomy maps) and let 
	$\widetilde C(Y,K)$ and $\widetilde C(Y',K')$ be the $\cS$-complexes associated to $(Y,K)$ and $(Y',K')$. As explained above, we obtain a  morphism
	\[
	  \widetilde \lambda_{(W,S)}:\widetilde C(Y,K) \longrightarrow \widetilde C(Y',K')
	\]
	after fixing a metric, perturbation data and modified holonomy maps for $(W,S)$ compatible with the auxiliary data of $(Y,K)$ and $(Y',K')$.
	For any two sets of auxiliary choices for $(W,S)$, a standard argument  shows that the resulting morphism differs from 
	$\widetilde\lambda$ by an $\cS$-chain homotopy of $\cS$-complexes. To be a bit more specific, one first connects the two collections of auxiliary choices for $(W,S)$
	using a 1-parameter family of orbifold metrics on $W$, perturbation data and modified holonomy maps. Then the associated family moduli spaces can be used to define the 
	required $\cS$-chain homotopy. In fact, one can use a similar proof to show that homotoping $\gamma$ changes $\widetilde \lambda_{(W,S)}$ by an 
	$\cS$-chain homotopy. We may also check that the morphism associated to a composition of negative definite pairs is $\cS$-chain homotopic to the composition of the morphisms associated to the cobordism pairs in the same way as in 
	\eqref{eq:compositemap}.
	
	Using another standard argument involving continuation maps we see that the $\cS$-chain homotopy type of the $\cS$-complex $\widetilde C(Y,K)$ of $(Y,K)$
	does not depend on the orbifold Riemannian metric on $Y$, perturbation of the Chern--Simons functional and modified holonomy maps. This uses connectedness of the spaces of these auxiliary choices (see Proposition 
	\ref{cons-cont-almost-hom} for the relevant result for modified holonomy maps). Technically, the assignment $(Y,K)\mapsto (\widetilde C(Y,K),\widetilde d)$ only determines an object in the homotopy category up to canonical 
	isomorphism. However, this is remedied as follows. 
	
	Let $\mathsf{C}$ be a category of modules closed under taking arbitrary products and submodules. A {\emph{transitive system}} in $\mathsf{C}$ is the data $(C,\phi, I)$ where $C=\{ C^i \}_{i\in I}$ is a set of objects in $\mathsf{C}$ 
	and $\phi=\{\phi_i^j:C^i\to C^j\}_{i,j\in I}$ is a set of isomorphisms such that $\phi_i^i=\text{id}$ and $\phi_k^j \circ \phi_i^k =\phi_i^j$. A {\emph{morphism}} of transitive systems $(C,\phi,I)\to (D,\psi,J)$ is a collection of morphisms 
	$\{\lambda_i^j:C^i\to D^j\}_{i\in I, j\in J}$ such that $\psi_k^l \lambda_{i}^k=\lambda_j^l\phi_i^j$ for $i,j\in I$ and $k,l\in J$. Transitive systems form a category $\mathsf{C}_\text{Trans}$. 
	There is a functor $\mathsf{C}_\text{Trans} \to\mathsf{C}$ that sends $(C,\phi,I)$ to the submodule of $\prod_{i\in I} C^i$ consisting of $\{c_i\}_{i\in I}$ with $c_i\in C^i$ and $c_j = \phi_i^j (c_i)$ for all $i,j\in I$.
	
	Thus to a based knot $(Y,K)$ we actually assign a transitive system of $\Z/4$-graded $\cS$-complexes in the homotopy category, indexed by admissible  metric, perturbation and modified holonomy maps data, 
	and to this transistive system we may then assign a $\Z/4$-graded $\cS$-complex in the homotopy category as described in the previous paragraph. Similar remarks hold for the morphisms, and this is precisely how the functor 
	in Theorem \ref{thm:framedcat} is defined. The situation is essentially the same as in any other construction of Floer homology, see e.g. \cite[p. 453]{km:monopole}.
\end{proof}

\begin{remark}
	The isomorphism class of $\widetilde I(Y,K)$ does not depend on the basepoint on $K$. Indeed, the identity cobordism of $(Y,K)$ with an arbitrary path from one choice of the basepoint on $K$ to another choice induces 
	an isomorphism between the Floer homology groups $\widetilde I(Y,K)$ for different choices of basepoints. $\diamd$
\end{remark}

\newpage

%!TEX root = main.tex

\section{The algebra of $\cS$-complexes}\label{sec:equivtheories}

The goal of this section is to further develop the algebraic aspects of $\cS$-complexes. In particular, we associate various equivariant homology theories to an $\cS$-complex and discuss how they give rise to Fr\o yshov-type invariants. We also study the behavior of $\cS$-complexes with respect to taking duals and tensor products. Although we work over $\Z$ throughout this section, all of the constructions carry over for any commutative ring $R$.

We also introduce the set $\Theta_R^\cS$ of local equivalence classes of $\cS$-complexes over any commutative ring $R$, following \cite{stoffregen}. This set has the structure of a partially ordered abelian group. When $R$ is an integral domain, the Fr\o yshov invariant $h$ may be viewed as a homomorphism $h:\Theta_R^\cS \to \Z$ of partially ordered abelian groups. 

The constructions here are applied to the setting of singular instanton Floer theory in the next section. However, the material in this section is entirely algebraic, and much of it fits into the framework of $S^1$-equivariant algebraic topology. In particular, the equivariant chain complexes we consider are particular models of the borel, co-borel, and Tate homology theories; see e.g. the discussions in \cite{manolescu:pin2, stoffregen, miller}. Although much of the material is standard in some circles, we include it here for completeness.

\subsection{An equivalent formulation of $\cS$-complexes}

In this subsection, we first give another definition of $\cS$-complexes:
\begin{definition}\label{cs-somplex-2}
	An $\cS$-complex is a finitely generated free abelian graded group $\widetilde C_*$ together with homomorphisms $\widetilde d:\widetilde C_*\to \widetilde C_{*-1}$ and $\chi: \widetilde C_*\to \widetilde C_{*+1}$ which respectively have degree $-1$ and $1$, and which satisfy the following properties:
	\begin{itemize}
		\item[(i)] $\widetilde d \circ \widetilde d=0$, $\chi \circ \chi=0$ and $\chi \circ \widetilde d+\widetilde d \circ \chi=0$.
		\item[(ii)]There is a subgroup $\Z$ of $\widetilde C_0$ such that ${\rm ker}(\chi)$ is equal to ${\rm image}(\chi)\oplus \Z$. 
	\end{itemize}
\end{definition}

\begin{remark}
	Remarks \ref{cS-coefficient} and \ref{cS-gradings} about coefficient rings and gradings for $\cS$-complexes still apply here.  The algebraic results in this section will hold for $\Z/2N$-graded $\cS$-complexes over any commutative ring, and $\cS$-complexes with arbitrary grading $\Z/N$ (in particular no grading) over any commutative ring of characteristic two. However, for concreteness we will typically work with $\Z$-graded $\cS$-complexes over $\Z$. $\diamd$
\end{remark}

This definition of $\cS$-complexes essentially agrees with the definition from the previous section. Let $C_*$ be ${\rm image}(\chi)$ after shifting down the degree by $1$ and $d:=-\widetilde d\vert_{C_*}$. The identities in (i) imply that $d$ is an endomorphism of $C_*$ and defines a differential on $C_*$. The assumption (ii) implies that $\widetilde C_*$ fits into a short exact sequence with degree preserving maps:
\begin{equation}\label{exact-seq}
  0\longrightarrow{}C_{*-1}\oplus \Z\xhookrightarrow{\hspace{.4cm}}\widetilde C_*\xrightarrow{\;\;\chi\;\;}C_{*}\longrightarrow{}0
\end{equation}
By splitting this exact sequence, we have an identification of $\widetilde C_*$ with $C_*\oplus C_{*-1}\oplus \Z$, with respect to which $\chi$ has the following form:
\[
  \chi(\alpha,\beta,r)=(0,\alpha,0).
\]
Since the map $\chi$ anti-commutes with $\widetilde d$, and $\widetilde d$ has degree $-1$, it is easy to see that $\widetilde d$ has the form given in \eqref{eq:d-tilde}. 

The notions of morphism and homotopy of $\cS$-complexes can be also reformulated using the new definition of $\cS$-complexes. Suppose $(\widetilde C_*,\widetilde d,\chi)$ and $(\widetilde C_*',\widetilde d',\chi')$ are  $\cS$-complexes and $\widetilde \lambda:\widetilde C_*\to \widetilde C_*'$ is a degree $0$ homomorphism of abelian groups such that $\widetilde \lambda\circ \widetilde d=\widetilde d'\circ \widetilde \lambda$ and $\widetilde \lambda\circ \chi=\chi'\circ \widetilde \lambda$. Once we split $\widetilde C_*$ and $\widetilde C_*'$ as $C_*\oplus C_{*-1}\oplus \Z$ and $C_*'\oplus C_{*-1}'\oplus \Z$ following the above strategy, the map $\widetilde \lambda$ has the form
\begin{equation*}
	\widetilde \lambda=\left[
					\begin{array}{ccc}
						\lambda &0&0\\
						\mu &\lambda &\Delta_2\\
						\Delta_1&0&c\\
					\end{array}
					\right]
\end{equation*}
for a constant $c\in \Z$. Thus $\widetilde \lambda$ is a morphism of $\cS$-complexes if we require that $c=1$. If $\widetilde \lambda':\widetilde C_*\to \widetilde C_*'$ is another morphism of $\cS$-complexes, an $\cS$-chain homotopy of  $\widetilde \lambda$ and $\widetilde \lambda'$ is a degree $1$ map $\widetilde h$ such that $\widetilde d' \circ \widetilde h+\widetilde h \circ \widetilde d=\widetilde \lambda-\widetilde \lambda'$ and $\chi' \circ h+h \circ \chi=0$.

\begin{remark}
	Suppose we have two different 
	splittings of $\widetilde C_*$. Then the identity map 
	of $\widetilde C_*$ induces a map between the two splittings which have the form in \eqref{eq:tilde-map}
	and is a morphism in the sense of Definition \ref{def:morphism}. Moreover, this morphism defines 
	an $\cS$-chain homotopy equivalence. This shows that the $\cS$-chain homotopy equivalences of $\cS$-complexes 
	with respect to Definitions \ref{S-comp-first} and \ref{cs-somplex-2} coincide with one another. 
	This justifies our switching between Definitions \ref{S-comp-first} and \ref{cs-somplex-2}. $\diamd$
\end{remark}

\begin{remark}
	Let $X$ be a finite CW complex on which $S^1$ acts cellularly, and freely away from a unique fixed point, a 0-cell $e^0$. The CW chain complex of $S^1$ is a differential graded algebra isomorphic to $\Z[\chi]/(\chi^2)$ with trivial differential. This dg-algebra acts on $(\widetilde C_\ast, \widetilde d)$, the CW chain complex of $X$, making it into a dg-module over $\Z[\chi]/(\chi^2)$. Then $(\widetilde C_{\ast}, \widetilde d, \chi)$ is an $\cS$-complex, with $\Z\subset \widetilde C$ generated by the fixed point $e^0$. $\diamd$
\end{remark}

In what follows, we freely switch between the two equivalent formulations of $\cS$-complexes. If we wish to use Definition \ref{cs-somplex-2}, we also assume that a splitting of the sequence \eqref{exact-seq} has been chosen. This allows us to obtain a splitting of $\widetilde C_*$ as $C_*\oplus C_{*-1}\oplus \Z$ and hence we can talk about the maps $v:C_*\to C_{*-2}$, $\delta_1:C_1\to \Z$, $\delta_2:\Z \to C_{-2}$. We denote a typical element of $\widetilde C_*$ by $\zeta$. Typical elements of the summand $C_*$ of $\widetilde C_*$ are denoted by $\alpha$, $\beta$ and the corresponding elements in the summand $C_{*-1}$ are denoted by $\underline \alpha$, $\underline \beta$.

\subsection{Equivariant homology theories associated to $\cS$-complexes}\label{equiv-model-1}

Suppose $(\widetilde C_*,\widetilde d, \chi)$ is as above.
In what follows, we will write $\Z[x]$ for the ring of polynomials with integer coefficients. Let also $\Z[\![x^{-1},x]$ be the ring of Laurent power series in the variable $x^{-1}$. That is to say, any element of $\Z[\![x^{-1},x]$ has finitely many terms with positive powers of $x$ and possibly infinitely many terms with negative powers of $x$. We shall regard $\Z[\![x^{-1},x]$ as a module over the ring $\Z[x]$.

We associate chain complexes $(\hrC_*,\widehat d)$, $(\crC_*,\widecheck d)$ to $(\widetilde C_*,\widetilde d,\chi)$ defined as follows:
\begin{align*}
   \hrC_* & :=\Z[x]\otimes \widetilde C_*\hspace{1cm}  & \widehat d(x^i \cdot \zeta) & =-x^i \cdot \widetilde d\zeta+x^{i+1}\cdot \chi(\zeta)\\
   \crC_*  & :=(\Z[\![x^{-1},x]/\Z[x])\otimes \widetilde C_*\hspace{1cm}  & \widecheck d(x^{i} \cdot \zeta) & = \phantom{-}x^{i} \cdot
  \widetilde d\zeta-x^{i+1}\cdot \chi(\zeta)
\end{align*}
We may define a chain map $j:\crC_{*}\to\hrC_{*-1}$ as follows:
\[
  j(x^k\xi)=\left\{
  \begin{array}{ll}
  	-\chi(\xi)&k=-1\\
	0&k<-1
  \end{array}
  \right.
\]
We can then define the mapping cone of $j$, which takes the following form:
\[
  \brC_*:=\widetilde C_*\otimes \Z[\![x^{-1},x]\hspace{1cm}\overline d(x^{i} \cdot \zeta)=-x^{i} \cdot \widetilde d\zeta+x^{i+1}\cdot \chi(\zeta)
\]
These complexes inherit $\Z$-gradings as follows: we declare that $x$ has grading $-2$, and use the natural tensor product gradings. In particular, if $\zeta\in \widetilde C_i$, then $x^j\cdot \zeta$ has grading $-2j+i$. With this convention, the differentials on the three complexes defined above have degree $-1$. By definition, we have a triangle of chain maps between chain complexes which induces an exact triangle at the level of homology:
\begin{equation}\label{equiv-triangle}
	\xymatrix{
	\crC_* \ar[rr]^{j}& &
	 \hrC_* \ar[dl]^{i}\\
	& \brC_* \ar[ul]^{p} &}
\end{equation}
Here $i$ is the inclusion map and $p$ is given by the composition of the projection map and the sign map $\epsilon$ associated to the graded vector space $\crC_*$:

\begin{definition}\label{sign-hom}
	For a $\Z$-graded vector space $V_*$, the sign homomorphism $\epsilon:V_*\to V_*$
	is defined by $\epsilon(a)=(-1)^{k}a$ where $a$ is a homogeneous element of $V_*$ with grading $k$. $\diamd$
\end{definition}

While $i$ and $p$ in \eqref{equiv-triangle} preserve gradings, $j$ has degree $-1$. It is clear from the definitions that all chain complexes here are defined over the graded ring $\Z[x]$ and all chain maps are $\Z[x]$-module homomorphisms, up to homotopy. We call \eqref{equiv-triangle} the {\it large equivariant triangle} associated to the $\cS$-complex $(\widetilde C_*,\widetilde d)$.

There is another exact triangle associated to an $\cS$-complex $\widetilde C_*$. Multiplication by $x$ defines an injective chain map from $(\hrC_*,\widehat d)$ to itself and the quotient complex is isomorphic to $(\widetilde C_*,\widetilde d)$. In particular, we have a triangle of the following form which induces an exact triangle at the level of homology:
\begin{equation}\label{equiv-triangle-3}
	\xymatrix{
	 \widetilde C_* \ar[rr]^{z}& &
	 \hrC_* \ar[dl]^{x}\\
	& \hrC_* \ar[ul]^{y} &}
\end{equation}
Here $x$ denotes the map given by multiplication by $x\in \Z[x]$, $y$ is the composition of the projection to the constant term and the sign map, and $z$ is given by $\chi$. In particular, $x$, $y$ and $z$ have respective degrees $-2$, $0$ and $1$.

\begin{prop}\label{functoriality}
	For any morphism $\widetilde \lambda:\widetilde C_* \to \widetilde C_*'$ of $\cS$-complexes, there are maps
	$\widecheck \lambda:\crC_*\to \crC_*'$, $\widehat \lambda:\hrC_*\to \hrC_*'$ and $\widebar \lambda:\brC_*\to \brC_*'$
	which satisfy the following properties:
	\begin{itemize}
		\item[(i)] $\widecheck \lambda$, $\widehat \lambda$ and $\widebar \lambda$ are chain maps over $\Z[x]$, and preserve gradings.
		\item[(ii)] If $\widetilde K$ is an $\cS$-chain homotopy of morphisms $\widetilde \lambda,\,\widetilde \lambda':\widetilde C_* \to \widetilde C_*'$,
		there are $\Z[x]$-module homomorphisms $\widecheck K:\crC_*\to \crC_*$, $\widehat K:\hrC_*\to \hrC_*'$, $\widebar K:\brC_*\to \brC_*'$ such that:
		\begin{align*}
		  \widecheck \lambda'-\widecheck \lambda & =\widecheck K\circ \widecheck d+\widecheck d\circ \widecheck K\\
		  \widehat \lambda'-\widehat \lambda & =\widehat K\circ \widehat d+\widehat d\circ \widehat K\\
		  \widebar \lambda'-\widebar \lambda & =\widebar K\circ \widebar d+\widebar d\circ \widebar K
		\end{align*}
		\item[(iii)] If  $\widetilde \lambda:\widetilde C_* \to \widetilde C_*'$ and $\widetilde \lambda':\widetilde C_*' \to \widetilde C_*''$ are 
		morphisms and $\widetilde {\lambda'\circ \lambda}:\widetilde C_* \to \widetilde C_*''$ is the composed morphism,
		then the following identities hold:
		\[
		  \widecheck{\lambda'\circ \lambda}=\widecheck \lambda'\circ \widecheck \lambda,\hspace{1cm}
		  \widehat{\lambda'\circ \lambda}=\widehat \lambda'\circ \widehat \lambda,\hspace{1cm}
		  \widebar{\lambda'\circ \lambda}=\widebar \lambda'\circ \widebar \lambda.
		\]
		If $\widetilde {\rm id}:\widetilde C_* \to \widetilde C_*$ is the identity morphism, then $\widehat {\rm id}$, $\widecheck {\rm id}$ and
		$\widebar {\rm id}$ are also identity maps.
	\end{itemize}
\end{prop}
\begin{proof}
	Given a morphism $\widetilde \lambda:\widetilde C_* \to \widetilde C_*'$ of $\cS$-complexes and an element $x^i\cdot \zeta \in \hrC_*$, we define $\widehat \lambda(x^i\cdot \zeta)=x^i\cdot  \widetilde \lambda(\zeta)\in \hrC_*'$. The maps $\widecheck \lambda$ and $\widebar \lambda$ are defined similarly. It is straightforward to check that these maps satisfy the required properties.
\end{proof}

\subsection{Small equivariant complexes and the $h$-invariant} \label{small-model}

Suppose $(\widetilde C_*,\widetilde d,\chi)$ is an $\cS$-complex as above. We introduce two other chain complexes over $\Z[x]$, denoted by $(\fhrC_*,\widehat\fd)$ and $(\fcrC_*,\widecheck \fd)$. Essentially the same complexes are defined in \cite{donaldson-book}. It is also shown in \cite{AD:CS-Th} that homology of these chain complexes and $\Z[\![x^{-1},x]$ form an exact triangle. In this subsection, we review these constructions and show that this information is equivalent to the triangle \eqref{equiv-triangle} up to homotopy. 

The chain complexes $(\fhrC_*,\widehat\fd)$ and $(\fcrC_*,\widecheck \fd)$ are given as follows:
\begin{align*}
  \fhrC_* & :=C_{*-1}\oplus \Z[x]  & \widehat \fd(\alpha,\sum_{i=0}^Na_ix^i) & =(d\alpha-\sum_{i=0}^{N}v^i\delta_2(a_i),0)\\
  \fcrC_* & :=C_*\oplus (\Z[\![x^{-1},x]/\Z[x])  & \widecheck \fd(\alpha,\sum_{i=-\infty}^{-1}a_ix^i)  & =(d\alpha,
  \sum_{i=-\infty}^{-1}\delta_1v^{-i-1}(\alpha)x^i)
\end{align*}
The $\Z$-grading on $\fhrC_*$ is given by the shifted grading on $C_*$ and the grading on $\Z[x]$ where $x^i$ has grading $-2i$. The $\Z$-grading on $\fcrC_*$ is defined similarly, except that we do not shift the grading on $C_*$. The $\Z[x]$-module structures on these complexes are defined as follows:
\begin{align*}
	 x\cdot (\alpha,\sum_{i=0}^Na_ix^i) &=(v\alpha,\delta_1(\alpha)+\sum_{i=0}^{N}a_ix^{i+1})&\forall(\alpha,\sum_{i=0}^Na_ix^i)\in \fhrC_*\\
	 x\cdot (\alpha,\sum_{i=-\infty}^{-1}a_ix^i) &=(v\alpha+\delta_2(a_{-1}),\sum_{i=-\infty}^{-2}a_ix^{i+1})&\forall(\alpha,\sum_{i=-\infty}^{-1}a_ix^i)\in  \fcrC_*
\end{align*}
We also define $\fbrC_*=\Z[\![x^{-1},x]$ with the trivial differential and the obvious module structure, again such that $x^i$ has grading $-2i$.

Next, we define chain maps:
\begin{equation}\label{equiv-triangle-2}
	\dots \xrightarrow{\hspace{5pt}\fp\hspace{5pt}} \fcrC_*\xrightarrow{\hspace{5pt}\fj\hspace{5pt}}\fhrC_*\xrightarrow{\hspace{5pt}\mathfrak i\hspace{5pt}}\fbrC_*\xrightarrow{\hspace{5pt}\mathfrak \fp\hspace{5pt}}  
	\fcrC_* \xrightarrow{\hspace{5pt}\fj\hspace{5pt}}\dots
\end{equation}
by the following formulas:
\begin{align*}
	  \mathfrak i(\alpha,\sum_{i=0}^{N}a_ix^i) &=\sum_{i=-\infty}^{-1}\delta_1v^{-i-1}(\alpha)x^{i}+\sum_{i=0}^{N}a_ix^i\\
	 \fj(\alpha,\sum_{i=-\infty}^{-1}a_ix^{i}) &=(-\alpha,0), \\
	 \fp(\sum_{i=-\infty}^{N}a_ix^i) &=(\sum_{i=0}^{N}v^{i}\delta_2(a_i),\sum_{i=-\infty}^{-1}a_ix^{i}),
\end{align*}
We call \eqref{equiv-triangle-2} the {\it small equivariant triangle} associated to the $\cS$-complex $(\widetilde C_*,\widetilde d)$. It is shown in \cite[Subsection 2.3]{AD:CS-Th} that the maps $\mathfrak i$ and $\fp$ are $\Z[x]$-module homomorphisms and $\fj$ commutes with the action of $x$ up to chain homotopy. Moreover, the maps at the level of homology groups induced by $\mathfrak i$, $\fj$ and $\fp$ form an exact triangle:
\begin{equation*}\label{hom-triangle-top}
	\xymatrix{
	H(\fcrC_*,\widecheck d) \ar[rr]^{\fj_*}& &
	H (\fhrC_* ,\widehat d)\ar[dl]^{\mathfrak i_*}\\
	& \fbrC_* \ar[ul]^{\fp_*} &
	}
\end{equation*}
We now show that the small and large equivariant chain complexes, and their associated triangles, are chain homotopy equivalent over $\Z[x]$. To this end, we define linear maps $\widehat \Phi:\hrC_* \to \fhrC_*$, $\widecheck \Phi:\crC_* \to \fcrC_*$ and $\widebar \Phi:\brC_* \to \fbrC_*$ as follows:
\begin{align*}
	\widehat \Phi(\sum_{i=0}^N\alpha_ix^i,\sum_{i=0}^N\beta_ix^i,\sum_{i=0}^Na_ix^i) &:=(\sum_{i=0}^Nv^i(\beta_i),\sum_{i=0}^Na_ix^i+\sum_{i=1}^N\sum_{j=0}^{i-1}\delta_1v^j(\beta_i)x^{i-j-1})\\
	 \widecheck \Phi(\sum_{i=-\infty}^{-1}\alpha_ix^i,\sum_{i=-\infty}^{-1}\beta_ix^i,\sum_{i=-\infty}^{-1}a_ix^i) &:=(\alpha_{-1},\sum_{i=-\infty}^{-1}a_ix^i+\sum_{i=-\infty}^{-1}\sum_{j=0}^{\infty}\delta_1v^j(\beta_i)x^{i-j-1})\\
	 \widebar \Phi(\sum_{i=-\infty}^{N}\alpha_ix^i,\sum_{i=-\infty}^{N}\beta_ix^i,\sum_{i=-\infty}^{N}a_ix^i) &:=\sum_{i=-\infty}^Na_ix^i+\sum_{i=-\infty}^{N}\sum_{j=0}^{\infty}\delta_1v^j(\beta_i)x^{i-j-1}
\end{align*}
We also define linear maps $\widehat \Psi:\fhrC_* \to \hrC_*$, $\widecheck \Psi:\fcrC_* \to \crC_*$ and $\widebar \Psi:\fbrC_* \to \brC_*$:
\begin{align*}
	 \widehat \Psi(\alpha,\sum_{i=0}^Na_ix^i) &:=(\sum_{i=1}^N\sum_{j=0}^{i-1}v^j\delta_2(a_i)x^{i-j-1},\alpha,\sum_{i=0}^Na_ix^i)\\
	\widecheck \Psi(\alpha,\sum_{i=-\infty}^{-1}a_ix^i) &:=(\sum_{i=-\infty}^{-1}v^{-i-1}(\alpha)x^i+\sum_{i=-\infty}^{-1}\sum_{j=0}^{\infty}v^j\delta_2(a_i)x^{i-j-1},0,\sum_{i=-\infty}^{-1}a_ix^i)\\
	\widebar \Psi(\sum_{i=-\infty}^{N}a_ix^i) &:=(\sum_{i=-\infty}^{N}\sum_{j=0}^{\infty}v^j\delta_2(a_i)x^{i-j-1},0,\sum_{i=-\infty}^{N}a_ix^i).
\end{align*}
The following lemma summarizes the properties of these maps:
\begin{lemma}\label{equiv-2-models}
	The above maps are chain maps, which commute with the action of $x$ 
	and the maps in the triangles \eqref{equiv-triangle} and \eqref{equiv-triangle-2} up to chain homotopy. We have:
	\[\widehat \Phi\circ \widehat \Psi={\rm id},\hspace{1cm}\widecheck \Phi\circ \widecheck \Psi={\rm id},\hspace{1cm}\widebar \Phi\circ \widebar \Psi={\rm id}.\]
	Moreover, the compositions $\widehat \Psi\circ \widehat \Phi$, $\widecheck \Psi\circ \widecheck \Phi$ and $\widebar \Psi\circ \widebar \Phi$ are chain homotopic to the identity.
\end{lemma}
\begin{proof}
	It is straightforward to check that the maps  $\widehat \Phi$, $\widecheck \Phi$, $\widebar \Phi$, $\widehat \Psi$, $\widecheck \Psi$ and $\widebar \Psi$
	are chain maps and they satisfy the following identities:
	\begin{align*}
	  \widehat \Phi \circ x &=x \circ \widehat \Phi, & \widecheck \Psi \circ x &=x \circ \widecheck \Psi,&
	  \widebar \Phi \circ x &=x \circ \widebar \Phi, & \widebar \Psi\circ x &=x \circ \widebar \Psi,\\
	  \widehat \Phi \circ j &=\fj \circ \widecheck \Phi, & \widehat \Psi \circ \fj &=j \circ \widecheck \Psi,&
	  \widebar \Phi \circ i &=\mathfrak i \circ \widehat \Phi, & \widecheck \Psi \circ \fp &=p \circ \widebar \Psi.
	\end{align*}
	Moreover, if we define the following four additional maps:
	\begin{align*}
	  \widehat K_{x}(\alpha,\sum_{i=0}^{N}a_ix^i) &:=(\alpha,0,0),\\
	  K_{\mathfrak i}(\alpha,\sum_{i=0}^{N}a_ix^i) &:=(-\sum_{i=0}^\infty v^i(\alpha)x^{-i-1},0,0),\\
	  \widecheck K_{x}(\sum_{i=-\infty}^{-1}\alpha_ix^i,\sum_{i=-\infty}^{-1}\beta_ix^i,\sum_{i=-\infty}^{-1}a_ix^i) &:=(\beta_{-1},0),\\
	  K'_{\fp}(\sum_{i=-\infty}^{N}\alpha_ix^i,\sum_{i=-\infty}^{N}\beta_ix^i,\sum_{i=-\infty}^{N}a_ix^i) &:=(-\sum_{i=0}^N v^i(\beta_i),0),\quad K_{\fp} = K'_{\fp}\circ \epsilon
	\end{align*}
	then we have the following relations, which are straightforward to verify:
	\begin{align*}
	x \circ \widehat \Psi- \widehat \Psi \circ x&=\widehat d \widehat K_{x}+\widehat K_{x}\widehat \fd &
	x \circ \widecheck \Phi-\widecheck \Phi \circ x&= \widecheck \fd\widecheck K_{x}+\widecheck K_{x}\widecheck d,\\
	\widebar \Psi \circ \mathfrak i-i \circ \widehat \Psi&=\widebar dK_{\mathfrak i}+K_{\mathfrak i}\widehat \fd &
	\widecheck \Phi \circ p-\fp \circ \widebar \Phi&=\widecheck \fd K_{\fp}+K_{\fp} \widebar d.
	\end{align*}
	In order to verify the last part of the lemma, define the following maps:
	\begin{align*}
		\widehat K(\sum_{i=0}^N\alpha_ix^i,\sum_{i=0}^N\beta_ix^i,\sum_{i=0}^Na_ix^i)&:=(-\sum_{i=0}^N\sum_{j=0}^{i-1} v^j(\beta_i)x^{i-j-1},0,0)\\
		\widecheck K(\sum_{i=-\infty}^{-1}\alpha_ix^i,\sum_{i=-\infty}^{-1}\beta_ix^i,\sum_{i=-\infty}^{-1}a_ix^i)&:=(\sum_{i=-\infty}^{-1}\sum_{j=0}^{\infty} v^j(\beta_i)x^{i-j-1},0,0)\\
		\widebar K(\sum_{i=-\infty}^{N}\alpha_ix^i,\sum_{i=-\infty}^{N}\beta_ix^i,\sum_{i=-\infty}^{N}a_ix^i)&:=(-\sum_{i=-\infty}^{N}\sum_{j=0}^{\infty} v^j(\beta_i)x^{i-j-1},0,0)
	\end{align*}
	Again, it is straightforward to verify the following relations:
	\begin{align*}
	  \widehat \Psi\circ \widehat \Phi-{\rm id} &=\widehat d\widehat K+\widehat K\widehat d,\\
	  \widecheck \Psi\circ \widecheck \Phi-{\rm id} &=\widecheck d\widecheck K+\widecheck K\widehat d,\\
	\widebar \Psi\circ \widebar \Phi-{\rm id} &=\widebar d\widebar K+\widebar K\widebar d. \qedhere
	\end{align*}
\end{proof}

Let $\widetilde C'$ be another $\cS$-complex and $\widetilde \lambda:\widetilde C\to \widetilde C'$ be a morphism. We write $\widehat \Phi'$, $\widecheck \Phi'$, $\widebar \Phi'$, $\widehat \Psi'$, $\widecheck \Psi'$ and $\widebar \Psi'$ for the chain homotopy equivalences associated to $\widetilde C'$. These chain homotopies can be employed to define $\widehat {\fm}_{\widetilde \lambda}:\fhrC_*\to \fhrC_*'$, $\widecheck {\fm}_{\widetilde \lambda}:\fcrC_*\to \fcrC_*'$ and $\widebar {\fm}_{\widetilde \lambda}:\fbrC_*\to \fbrC_*'$ as follows:
\begin{equation*}
	\widehat {\fm}_{\widetilde \lambda}:=\widehat \Phi' \circ \widehat \lambda \circ \widehat \Psi\hspace{1cm}
	\widecheck {\fm}_{\widetilde \lambda}:=\widecheck \Phi' \circ \widecheck {\lambda}\circ \widecheck \Psi\hspace{1cm}
	\widebar {\fm}_{\widetilde \lambda}:=\widebar \Phi' \circ \widebar{\lambda}\circ \widebar \Psi
\end{equation*}
These homomorphisms agree with the definitions given in \cite[Subsection 2.3]{AD:CS-Th}. For example, we have the following explicit formula:
\begin{align*}
  \widebar {\fm}_{\widetilde \lambda}(\sum_{i=-\infty}^{N}a_ix^i)=&(\sum_{i=-\infty}^{N}a_ix^i)\left(1+\sum_{i=0}^{\infty}\Delta_1v^{i}\delta_2(1)x^{-i-1}
  +\sum_{i=0}^{\infty}\delta_1'(v')^{i}\Delta_2(1)x^{-i-1}\right.\\
  &\left. +\sum_{k=0}^{\infty}\sum_{j=0}^{\infty}\delta_1'(v')^{j}\mu v^{k}\delta_2(1)x^{-k-j-2}\right)
\end{align*}
In particular, the map $\widebar {\fm}_{\widetilde \lambda}$ is an isomorphism of the $\Z[x]$-module $\Z[\![x^{-1},x]$. Therefore, we have the following consequence of this observation and Lemma \ref{equiv-2-models}:
\begin{cor}\label{localization}
	The $\Z[x]$-module $H(\brC_*,\widebar d)$ is naturally isomorphic to $\Z[\![x^{-1},x]$. Any morphism of $\cS$-complexes $\widetilde \lambda : \widetilde C_\ast\to \widetilde C_\ast'$ induces an isomorphism $\smash{\widebar {\fm}_{\widetilde \lambda}}:\brC_*\to \brC_*'$. Moreover, there are $b_i\in \Z$ such that after the identifications of 
	$H(\brC_*,\widebar d)$ and $H(\brC_*',\widebar d')$ with $\Z[\![x^{-1},x]$, the map $\widebar {\fm}_{\widetilde \lambda}$ is multiplication by $1+\sum_{i=-\infty}^{-1}b_ix^i$.
\end{cor}

Consider the $\Z[x]$-submodule $\fI\subset \Z[\![x^{-1},x]$ given by the image of $\mathfrak i_*: H (\fhrC_* ,\widehat d)\to \Z[\![x^{-1},x]$, or equivalently the kernel of the map $\fp_*: \Z[\![x^{-1},x]\to H (\fcrC_* ,\widecheck d)$. Corollary \ref{localization} implies that if two $\cS$-complexes are related to each other by a morphism of $\cS$-complexes, then the associated $\Z[x]$-modules $\fI$ are related by multiplication by an element of the form $1+\sum_{i=-\infty}^{-1}b_ix^i$. This observation suggests the following definition:
\begin{definition}\label{h-def}
	For an $\cS$-complex $\widetilde C$ as above we define its $h$-invariant as follows:
	\[h(\widetilde C):=-\inf_{Q(x)\in \fI}\{{\rm Deg}(Q(x))\}\]
	where for $Q(x)\in \Z[\![x^{-1},x]$, ${\rm Deg}(Q(x))$ denotes the degree of $Q(x)$. 
\end{definition}

The following result is an immediate consequence of Corollary \ref{localization}:
\begin{cor}\label{s-mor-h}
	If there is a morphism of $\cS$-complexes from $\widetilde C$ to $\widetilde C'$, then $h(\widetilde C) \leq h(\widetilde C') $. 
\end{cor}

The following proposition gives an alternative definition for $h(\widetilde C)$, and it can be easily verified from the definitions. It is also closely related to Proposition 4 in \cite{froyshov}.

\begin{prop}\label{h-g-reinterpret}
	The constant $h(\widetilde C_*)$ is positive if and only if there is $\alpha\in C_*$ such that $d(\alpha)=0$ and $\delta_1(\alpha)\neq 0$.
	If $h(\widetilde C_*)=k$ for a positive integer $k$, then $k$ is the largest integer such that there exists $\alpha\in C_*$
	satisfying the following properties:
	\begin{equation}\label{h-pos}
	  d\alpha=0,\hspace{1cm}\delta_1v^{k-1}(\alpha)\neq 0,\hspace{1cm}\delta_1v^{i}(\alpha)=0\hspace{.5cm} \text{ for }i\leq k-2.
	\end{equation}
	If $h(\widetilde C_*)=k$ for a non-positive integer $k$, then $k$ is the largest integer such that there are 
	elements $a_0,\ldots , a_{-k}\in\Z$ and $\alpha\in C_*$ such that
	\begin{equation}\label{h-neg}
	  d\alpha=\sum_{i=0}^{-k}v^i\delta_2(a_i),\hspace{1cm} a_{-k}\neq 0.
	\end{equation}
\end{prop}

\subsection{Dual $\cS$-complexes}\label{sec:dual}

Let $(V_\ast,d)$ be an arbitrary $\Z$-graded chain complex. Our convention is that the dual complex $\text{Hom}(V_\ast,\Z)$ has differential $f\mapsto -\epsilon(f)\circ d$, with the $\Z$-grading that declares $f:V_i\to \Z$ to be in grading $-i$.  Recall that the sign map $\epsilon$ is given in Definition \ref{sign-hom}.

Now let $(\widetilde C_*,\widetilde d,\chi)$ be an $\cS$-complex. Let $\widetilde C_*^\dagger$ be the dual chain complex ${\rm Hom}(\widetilde C_*,\Z)$ with differential $\widetilde d^\dagger$. The endomorphism $\chi^\dagger$ of $\widetilde C_*^\dagger$ is defined similarly, so that we have:
\begin{align*}
	\widetilde d^\dagger(f)&:=-\epsilon(f)\circ \widetilde d,\\
	\chi^\dagger(f)&:=-\epsilon(f)\circ \chi.
\end{align*}
These clearly satisfy property (i) of Definition \ref{cs-somplex-2}. The spaces $\ker(\chi^\dagger)$ and ${\rm image}(\chi^\dagger)$ are given by the subspaces of $\widetilde C_*^\dagger$ which vanish respectively on ${\rm image}(\chi)$ and $\ker(\chi)$. Thus property (ii) of Definition \ref{cs-somplex-2} is also satisfied, and $(\widetilde C_*^\dagger,\widetilde d^\dagger,\chi^\dagger)$ is an $\cS$-complex.

Let $C_*^\dagger$ be given by ${\rm image}(\chi^\dagger)$ after shifting gradings down by $1$, and $d^\dagger$ be the differential on $C_*^\dagger$ given by the restriction of $-\widetilde d^\dagger$. The space $C_*^\dagger$ can be identified with ${\rm Hom}( C_{*},\Z)[1]$ and the differential $d^\dagger$ of $f:C_i\to \Z$, which has degree $-i-1$, is equal to $(-1)^{i}f\circ d$. If we split $\widetilde C_*^\dagger$ as the sum $C_*^\dagger\oplus C_{*-1}^\dagger\oplus \Z$ using a corresponding splitting of $\widetilde C_*$, then the $\cS$-complex structure on $\widetilde C_*^\dagger$ determines maps $v^\dagger:C_*^\dagger\to C_{*-2}^\dagger$, $\delta^\dagger_1:C_1^\dagger \to \Z$ and $\delta^\dagger_2:\Z\to C_0^\dagger$ which are given by the following formulas:
\begin{equation*}
	v^\dagger(f):=f\circ v, \hspace{1cm}
	\delta^\dagger_1(f):=-f\circ \delta_2(1),\hspace{1cm}
	\delta_2^\dagger(m):=m\delta_1.
\end{equation*}

In what follows, for an $\cS$-complex $(\widetilde C, \widetilde d,\chi)$ write $\fcrC^\ast = \text{Hom}(\fcrC_\ast ,\Z)$, $\fhrC^\ast = \text{Hom}(\fhrC_\ast ,\Z)$, and $\fbrC^\ast = \text{Hom}(\fbrC_\ast ,\Z)$ for the duals of the small equivariant chain complexes associated to $\widetilde C_\ast$, and call these the equivariant {\emph{cohomology}} chain complexes. We write 
\[
	(\fcrC^\dagger_\ast, \widecheck{\fd}^\dagger), \qquad (\fhrC_\ast^\dagger, \widehat{\fd}^\dagger ), \qquad (\fbrC^\dagger_\ast, \overline{\fd}^\dagger)
\]
for the small equivariant homology chain complexes associated to the dual $\cS$-complex $(\widetilde C^\dagger,\widetilde d^\dagger, \chi^\dagger)$, and $\mathfrak{i}^\dagger$, $\mathfrak{j}^\dagger$, $\mathfrak{p}^\dagger$ for the maps in the associated small equivariant triangle. We write $\mathfrak{i}^\vee$, $\mathfrak{j}^\vee$, $\mathfrak{p}^\vee$ for the duals of the maps  $\mathfrak{i}$, $\mathfrak{j}$, $\mathfrak{p}$ in the small equivariant triangle for $(\widetilde C,\widetilde d, \chi)$. The relationship between these is summarized as follows.

\begin{lemma}\label{dual-small-equiv-tr}
	We have a commutative diagram of small equivariant exact triangles, where the vertical maps are isomorphisms of graded chain complexes over $\Z[x]$:
\begin{equation}\label{eq:dualeqtri}
	\xymatrix@C+=1.5cm{
	\cdots & \fhrC^\dagger_*  \ar[d]  \ar[l]_{\mathfrak i^\dagger}   & \fcrC^\dagger_* \ar[d]  \ar[l]_{\mathfrak j^\dagger}&  \fbrC^\dagger_*   \ar[l]_{\mathfrak p^\dagger}  \ar[d] &  \ar[l]_{ \mathfrak i^\dagger}  \cdots\\
	\cdots   & \fcrC^* \ar[l]_{ \mathfrak p^\vee }& \fhrC^*   \ar[l]_{\mathfrak j^\vee} &  \fbrC^*  \ar[l]_{\mathfrak i^\vee}   & \cdots \ar[l]_{\mathfrak p^\vee}
	}
\end{equation}
\end{lemma}

\begin{proof}
	The vertical maps, written from left to right as $F_1$, $F_2$ and $F_3$, are defined by:
	\begin{align*} 
		\widehat F(f,\sum_{j=0}^Nb_jx^j)(\alpha,\sum_{i=-\infty}^{-1}a_ix^i) &=f(\alpha)+\sum_{j=0}^Nb_ja_{-j-1},\\
		\widecheck F(f,\sum_{j=-\infty}^{-1}b_jx^j)(\alpha,\sum_{i=0}^{N}a_ix^i) &=f(\alpha)+\sum_{j=-N-1}^{-1}b_ja_{-j-1},\\
		\overline F(\sum_{j=-\infty}^{N}b_jx^j)(\sum_{i=-\infty }^{N'}a_ix^i) &=\sum_{j=-N'-1}^{N}b_ja_{-j-1}.
	\end{align*}
	It is straightforward to check that these are isomorphisms of chain complexes, and fit into the diagram \eqref{eq:dualeqtri} above, matching the two small equivariant triangles.
\end{proof}

\begin{remark}
	An analogous statement holds for the large equivariant triangles. $\diamd$
\end{remark}

\begin{prop} \label{h-dual}
	$h(\widetilde C^\dagger_*)=-h(\widetilde C_*)$.
\end{prop}
\begin{proof}
	First assume that $h(\widetilde C_*)=k>0$. By Proposition \ref{h-g-reinterpret} there is $\alpha\in C_*$
	such that $d\alpha=0$ and the smallest integer $i$ such that $\delta_1v^{i}(\alpha)$ is non-zero is equal to $k-1$.
	This implies that there is no $f: C_*\to \Z$ and $a_0,\ldots,a_{k-1}\in \Z$ such that $a_{k-1}\neq 0$ and
	\[f d+\sum_{i=0}^{k-1}a_i\delta_1v^i=0,\]
	because we can apply this linear form to $\alpha$. This implies that $h(\widetilde C^\dagger_*)\leq -k$. 
	
	Next, we consider the set $\{\delta_1,\delta_1 v,\dots, \delta_1 v^{k-1}\}$. By restriction, elements of this set define linearly independent maps on $\ker(d)$. 
	Our assumption implies that we obtain linearly dependent elements by adding $\delta_1 v^{k}$. Therefore, there are $a_0$, $\dots$, $a_{k}$ such that $a_{k}\neq 0$ and:
	\[\sum_{i=0}^{k}a_i\delta_1v^i\vert_{\ker(d)}=0\]
	This implies that there is $f:C_*\to \Z$ such that after multiplying all constants $a_i$ by a non-zero integer,
	we have:
	\[f d+\sum_{i=0}^{k}a_i\delta_1v^i=0\]
	Consequently, we have $h(\widetilde C^\dagger_*)\geq -k$ which completes the proof in the case that  
	$h(\widetilde C_*)$ is positive. The case that  $h(\widetilde C_*)$ is negative can be verified similarly.
\end{proof}

\begin{remark}
	An alternative proof of Proposition \ref{h-dual} may be obtained using the perspective of Definition \ref{h-def} and directly applying Lemma \ref{dual-small-equiv-tr}. $\diamd$
\end{remark}

\subsection{Tensor products of $\cS$-complexes}\label{sec:tensor}

Here we show that the tensor product of two $\cS$-complexes is naturally isomorphic to an $\cS$-complex. Let $(\widetilde C_\ast, \widetilde d,\chi)$ and $(\widetilde C'_\ast, \widetilde d',\chi')$ be two $\cS$-complexes. We also fix splittings $\widetilde C_\ast=C_*\oplus C_{*-1}\oplus \Z$ and $\widetilde C'_\ast=C'_*\oplus C'_{*-1}\oplus \Z$ and let $d,v,\delta_1,\delta_2$ and $d',v',\delta_1',\delta_2'$ be the associated maps. From this data we define an $\cS$-complex $(\widetilde C^\otimes_\ast, \widetilde d^\otimes,\chi^\otimes)$. Firstly let $\widetilde C^\otimes_\ast=\widetilde C_\ast\otimes \widetilde C'_\ast$, and then define $\widetilde d^\otimes$ and $\chi^\otimes$ as follows:
\begin{align*}
	\widetilde d^\otimes&:=\widetilde d\otimes 1+\epsilon \otimes \widetilde d'\\
	\chi^\otimes&:=\chi\otimes 1+\epsilon \otimes \chi'
\end{align*}
It is clear that $\widetilde d^\otimes$ and $\chi^\otimes$ satisfy property (i) of Definition \ref{cs-somplex-2}. To see (ii) of Definition \ref{cs-somplex-2}, note that ${\rm image}(\chi^\otimes)$ is generated by the elements of the following form:
\begin{equation}\label{im-chi}
  \underline \alpha\otimes  \alpha'+\epsilon(\alpha)\otimes \underline\alpha',\hspace{1cm}\underline \alpha\otimes \underline \alpha',\hspace{1cm}\underline \alpha\otimes 1,\hspace{1cm}1\otimes \underline \alpha',
\end{equation}
and the kernel of $\chi^\otimes$ is generated by the above elements and $1\otimes 1$.

The subgroup of $\widetilde C_*^\otimes$ generated by the elements in \eqref{im-chi}, after shifting down the degree by $1$, is denoted by $C_*^\otimes$. We also write $d^\otimes$ for the differential on $C_*^\otimes$ defined as $-\widetilde d^\otimes\vert_{C^\otimes}$. The four types of elements in \eqref{im-chi} determine a natural decomposition of $C_*^\otimes$ as follows:
\[
	C^\otimes_\ast = (C \otimes C')_\ast \oplus (C\otimes C')_{\ast - 1} \oplus C_\ast \oplus C'_\ast.
\]
The differential $d^\otimes$ with respect to this decomposition is given by
\[
	d^\otimes =\left[
					\begin{array}{cccc}
						d\otimes 1 + \varepsilon \otimes d' &0&0&0\\
						-\varepsilon v \otimes 1 + \varepsilon \otimes v' & d\otimes 1 - \varepsilon \otimes d' & \varepsilon \otimes \delta_2' & -\delta_2 \otimes 1\\
						\varepsilon \otimes \delta_1' &0&d& 0\\
						\delta_1\otimes 1& 0 & 0 & d'\\
					\end{array}
					\right]
\]

As in \eqref{exact-seq}, we have a short exact sequence of the following form:
\[
  0\longrightarrow C_{*-1}^\otimes\oplus \Z\xhookrightarrow{\hspace{.4cm}}\widetilde C_*^\otimes\xrightarrow{\;\chi^\otimes\;}C_{*-1}^\otimes\longrightarrow 0
\]
We may split this sequence using the following right inverse of the map $\chi^\otimes$:
\[
  \underline \alpha\otimes  \alpha'+\epsilon(\alpha)\otimes \underline\alpha'\mapsto \alpha\otimes  \alpha', \hspace{.7cm}\underline \alpha\otimes \underline \alpha'\mapsto \alpha\otimes \underline \alpha',\hspace{.7cm}\underline \alpha\otimes 1\mapsto \alpha\otimes 1,
  \hspace{.7cm}1\otimes \underline \alpha' \mapsto 1\otimes \alpha'.
\]
This gives rise to a splitting of $\widetilde C_*^\otimes$ as $C_*^\otimes\oplus C_{*-1}^\otimes\oplus \Z$ and the maps $v^\otimes:C_*\to C_{*-2}$, $\delta_1^\otimes:C_1\to \Z$, $\delta_2^\otimes:\Z \to C_{-2}$. The endomorphism $v^\otimes:C_\ast^\otimes\to C_{\ast-2}^\otimes$ is given as follows:
\begin{equation}
	v^\otimes =\left[
					\begin{array}{cccc}
						v\otimes 1 &0& 0& \delta_2\otimes 1\\
						0 & v\otimes 1 & 0 &0\\
						0 & 0& v & 0\\
						0 & \delta_1\otimes 1& 0  & v'\\
					\end{array}
					\right] \label{eq:vtensor}
\end{equation}
We also have $\delta_1^\otimes=[0, 0, \delta_1,\delta_1']$ and $\delta_2^\otimes = [0,0,\delta_2,\delta_2']^\intercal$.

\begin{remark}
	The non-symmetrical form of $v^\otimes$ is due to our non-symmetrical choice of the right inverse for the map $\chi^\otimes$. In the case that we replace the integers with a ring in which $2$ is invertible, we can modify the above right inverse for $\chi^\otimes$
	by mapping $\underline \alpha\otimes \underline \alpha'$ to $\frac{1}{2}(\alpha\otimes \underline \alpha'+\epsilon(\underline \alpha)\otimes  \alpha')$. Then the new map $v^\otimes$ with respect to the induced splitting of $\widetilde C_*^{\otimes}$ has the following form:
	\[
	  \frac{1}{2} \left[
			\begin{array}{cccc}
				v\otimes 1 +  1 \otimes v' &0& 1 \otimes \delta_2'& \delta_2\otimes 1\\
				0 & v\otimes 1 + 1 \otimes v' & 0 &0\\
				0 & 1 \otimes \delta_1' & 2v & 0\\
				0 & \delta_1\otimes 1& 0  & 2v'\\
			\end{array}
			\right]
	\]
	which is symmetrical with respect to switching $\widetilde C_*$ and $\widetilde C_*'$. $\diamd$
\end{remark}

\begin{remark}\label{sign}
	Suppose $(\widetilde C_*^0,\widetilde d^0,\chi^0)$ is the trivial $\cS$-complex with $\widetilde C_*^0=\Z$,
	$\widetilde d^0=0$ and $\chi^0=0$.
	Our sign convention for tensor products and duals of $\cS$-complexes implies that the following natural pairing 
	on $\widetilde C_*^\dagger$ and $ \widetilde C_*$ defines a 
	morphism $Q:\widetilde C_*^\dagger\otimes \widetilde C_*\to \widetilde C_*^0$:
	\[
	  (f,\alpha)\longmapsto  f(\alpha). \quad \diamd
	\]
\end{remark}

In Subsection \ref{equiv-model-1}, to any $\cS$-complex we associated a series of equivariant theories which fit into an exact triangle. For example, $(\widetilde C_*,\widetilde d,\chi)$ gives rise to the complex $(\hrC_*,\widehat d)$ over $\Z[x]$. We denote the corresponding objects associated to the $\cS$-complexes $(\widetilde C'_*,\widetilde d')$ and $(\widetilde C^\otimes_*,\widetilde d^\otimes)$ by $(\hrC_*',\widehat d')$ and $(\hrC_*^\otimes,\widehat d^\otimes)$. We follow a similar terminology for the other objects introduced in Subsection \ref{equiv-model-1}. Our goal is to study how $h(\widetilde C^\otimes_\ast)$ is related to $h(\widetilde C_\ast)$ and $h(\widetilde C_\ast')$.

\begin{lemma}\label{ideal-tensor}
	There are $\Z[x]$-module isomorphisms:
	\begin{equation*}
		\widehat T:\hrC_*\otimes_{\Z[x]}\hrC_*'\to \hrC_*^\otimes \hspace{1cm}
		\widebar T: \brC_*\otimes_{\Z[\![x^{-1},x]}\brC_*' \to \brC_*^\otimes
	\end{equation*}	
	which are chain maps, and the following diagram commutes:
	\begin{equation}\label{T-comm}
		\xymatrix{
		\hrC_*\otimes_{\Z[x]}\hrC_*'\ar[d]^{i\otimes i'}\ar[rr]^{\widehat T}&&\hrC_*^\otimes\ar[d]^{i^\otimes}\\
		\brC_*\otimes_{\Z[\![x^{-1},x]}\brC_*'\ar[rr]^{\widebar T} &&\brC_*^\otimes\\}
	\end{equation}	
	Moreover, we have $\frak I\otimes \frak I'\subseteq \frak I^\otimes$.
\end{lemma}

Note that although we have previously considered $\brC_*$ and $\brC_*' $ as $\Z[x]$-modules, they are also modules over the larger ring $\Z[\![x^{-1},x]$ in the obvious way, and this is used in the above statement. 

\begin{proof}
	We define $\widehat T$ and $\widebar T$ to be the isomorphism of chain complexes which are induced by the following natural isomorphisms:
	\[
	  \Z[x]\otimes_{\Z[x]}\Z[x]=\Z[x],\hspace{1cm}\Z[\![x^{-1},x]\otimes_{\Z[\![x^{-1},x]}\Z[\![x^{-1},x]=\Z[\![x^{-1},x].
	\]
	The definition of the $\cS$-complex structure on $\widetilde C_*^\otimes$ immediately implies that $\widehat T$ and $\widebar T$  are chain maps.
	It is also straightforward to check that Diagram \eqref{T-comm} commutes.
	
	Suppose $\sum_{i=-\infty}^{N}a_ix^i\in \fI$ and $\sum_{i=-\infty}^{N'}a_i'x^i\in \fI'$ such that $a_N$ and $a_{N'}'$ are non-zero.
	Then Lemma \ref{equiv-2-models} implies that we have
	\[(\sum_{i=-\infty}^{N}\sum_{j=0}^{\infty}v^j\delta_2(a_i)x^{i-j-1},0,\sum_{i=-\infty}^{N}a_ix^i)\in {\rm image}(i_*),\]
	  and there is a corresponding element in $\text{image}(i_\ast')$, replacing $a_i$ and $N$ by  $a_i'$ and $N'$.
	 The tensor product of these two elements, using the isomorphism $\overline T$, gives rise to an element in the image of the map
	 $i^\otimes_*:H(\hrC_*^\otimes) \to H(\brC_*^\otimes)$, which is of the following form:
	 \[(A,0,\sum_{i=-\infty}^{N}\sum_{j=-\infty}^{N'}a_ia_j'x^{i+j})\]
	 where $A\in C_*^\otimes \otimes\Z[\![x^{-1},x]$. By Lemma \ref{equiv-2-models}, the map $\overline \Phi$ sends this element to $\fI^\otimes $ 
	 which has the form $\sum_{i=-\infty}^{N}\sum_{j=-\infty}^{N'}a_ia_j'x^{i+j}$.
	 Consequently, $\frak I\otimes \frak I'\subseteq \frak I^\otimes$.
\end{proof}

\begin{cor}\label{h-sub-add} $h(\widetilde C^\otimes_\ast)= h(\widetilde C_\ast)+h(\widetilde C_\ast')$. 	
\end{cor}

\begin{proof}
	Lemma \ref{ideal-tensor} implies that $h(\widetilde C^\otimes_\ast)\geq h(\widetilde C_\ast)+h(\widetilde C_\ast')$.
	The morphism $Q$ in Remark \ref{sign} induces a morphism from 
	$\widetilde C_\ast^\dagger \otimes \widetilde C^\otimes_\ast$ to $\widetilde C_\ast'$. Then we have:
	\begin{align*}
		h(\widetilde C_\ast')&\geq h(\widetilde C_\ast^\dagger \otimes \widetilde C^\otimes_\ast) \geq -h(\widetilde C_\ast)+h(\widetilde C^\otimes_\ast).
	\end{align*}
	The first inequality frollows from Corollary \ref{s-mor-h}, while the second inequality follows from Proposition \ref{h-dual} and another application of Lemma \ref{ideal-tensor}. 	
\end{proof}

\subsection{Local equivalence groups}\label{sec:localequiv}

We next introduce an equivalence relation on the set of $\cS$-complexes, essentially that of ``(chain) local equivalence'' as defined in \cite{stoffregen}; see also \cite{hmz, hendrickshomlidman}. In those references, the focus is on equivariant homological algebra over $\text{Pin}(2)$ and $\Z/4$. In this sense our setup, which is in principle that of $S^1$-equivariant homological algebra, is more basic. However, we allow for more general coefficient rings, and will later consider a filtered analogue of these constructions, providing a broader context.

A {\emph{preorder}} on a set $S$ is a binary relation $\dot{\leqslant}$ which is reflexive and transitive. The quotient set $S/\sim$ obtained by identifying elements $s,t\in S$ with $s\, \dot{\leqslant}\, t$ and $t\, \dot{\leqslant}\, s$ is a partially ordered set with the induced binary relation, denoted $\leqslant$.

To any category, we may define a preorder on objects as follows: two objects $s$ and $t$ have $s\, \dot{\leqslant} \, t$ if there is a morphism from $s$ to $t$. We apply this to the category of $\Z$-graded $\cS$-complexes over a commutative ring $R$, and call the resulting set $\Theta_{R}^\cS$:
\[
	\Theta_{R}^\cS := \left\{ \cS\text{-complexes over } R \right\} /\sim
\]
\[
	 (\widetilde C, \widetilde d, \chi) \sim ( \widetilde C', \widetilde d', \chi') \;\;\; \Longleftrightarrow \;\;\; \exists\;  f:\widetilde C_\ast \to \widetilde C_\ast', \; f':\widetilde C'_\ast \to \widetilde C_\ast
\]
Here $f$ and $f'$ are morphisms of $\cS$-complexes. We refer to the equivalence relation $\sim$ as {\emph{local equivalence}}. An equivalence class in $\Theta_{R}^\cS$ will be denoted $[(\widetilde C_\ast, \widetilde d,\chi)]$. Then
\[
	[(\widetilde C_\ast, \widetilde d,\chi)] \leqslant [(\widetilde C'_\ast, \widetilde d',\chi')] \;\;\; \Longleftrightarrow \;\;\; \exists\;  f:\widetilde C_\ast \to \widetilde C_\ast'
\]
Thus $\Theta_R^\cS$ is a partially ordered set. Our above work implies the following:

\begin{prop}
	The set $\Theta_R^\cS$ has the natural structure of an abelian group, compatible with the partial order $\leqslant$, making it a partially ordered abelian group.
\end{prop}

\begin{proof}
	Addition is given by the tensor product of $\cS$-complexes. Associativity follows because the natural tensor product associativity isomorphism, written
	\[
		(\widetilde C_\ast \otimes \widetilde C_\ast' ) \otimes \widetilde C_\ast'' \cong \widetilde C_\ast \otimes ( \widetilde C_\ast' \otimes \widetilde C_\ast''  ),
	\]
	 is an isomorphism of $\cS$-complexes. Similarly, commutativity follows because the map $\widetilde C_\ast \otimes  \widetilde C'_\ast\to \widetilde C'_\ast \otimes  \widetilde C_\ast$ defined by $\alpha\otimes \alpha'\mapsto (-1)^{ij}\alpha'\otimes \alpha$, where $\alpha$ and $\alpha'$ have respective gradings $i$ and $j$, is an isomorphism of $\cS$-complexes.
	 
	  Next, the addition operation is well-defined on $\Theta_R^\cS$, as it is straightforward to verify that the tensor product of two $\cS$-chain homotopy equivalences is again an $\cS$-chain homotopy equivalence. The identity element is represented by the trivial $\cS$-complex $\widetilde C^0_\ast = R$ with $\widetilde d^0=0$, and the inverse of an $\cS$-complex $\widetilde C_\ast$ is given by its dual $\widetilde C_\ast^\dagger$. Indeed, the natural pairing of Remark \ref{sign} and the morphism $R=\widetilde C^0_\ast \to \widetilde C_\ast^\dagger \otimes \widetilde C_\ast$ sending $1$ to the identity show that $\widetilde C_\ast^\dagger \otimes \widetilde C_\ast$ is equivalent to $\widetilde C^0_\ast$.
\end{proof}

When we want to consider $\Z/2N$-graded $\cS$-complexes over $R$, we may form an analogous partially ordered group, and we denote this by $\Theta_{R,\Z/2N}^{\cS}$:
\[
	\Theta_{R,\Z/2N}^\cS := \left\{ \Z/2N\text{-graded } \cS\text{-complexes over } R \right\} /\sim
\]
 When our $\cS$-complexes are graded by $\Z/N$ and $R$ has characteristic $2$, we similarly have $\Theta_{R,\Z/N}^{\cS}$. If in this latter case $N=1$ (i.e. there are no gradings) then we write $\Theta_{R,\emptyset}^{\cS}$. 
 
 The $h$-invariant of Definition \ref{h-def} extends to the case when the coefficient ring is any commutative ring $R$. In the case that $R$ is an integral domain, $h$ is a homomorphism:

\begin{prop}
	Let $R$ be an integral domain. The $h$-invariant induces a homorphism $h: \Theta_{R}^\cS \to \Z$ of partially ordered abelian groups. If $R$ is a field, then $h$ an isomorphism. In general, the $h$-invariant factors as follows, where $\text{{\emph{Frac}}}(R)$ is the field of fractions of $R$:
	\begin{align*}
		\Theta^\cS_{R}  &\longrightarrow \Theta^\cS_{\text{{\emph{Frac}}}(R),\Z/2} \xrightarrow{\,\;h\;\,} \Z \qquad &\text{\emph{char}}(R)\neq 2\\
		\Theta^\cS_{R}  &\longrightarrow \Theta^\cS_{\text{{\emph{Frac}}}(R),\emptyset} \xrightarrow{\,\;\;h\;\;\;} \Z \qquad &\text{\emph{char}}(R)= 2
	\end{align*}
	
\end{prop}

\begin{proof}
	That $h$ is a homomorphism follows from Corollaries \ref{s-mor-h} and \ref{h-sub-add} and Proposition \ref{h-dual}, which directly adapt when $\Z$ is replaced by any integral domain $R$. 
	
	Next, suppose that $R$ is a field. Let $\widetilde C_\ast$ be an $\cS$-complex over $R$ such that $h(\widetilde C)=0$. Proposition \ref{h-g-reinterpret} implies that there is $\alpha\in C_\ast$ such that $\delta_2(1)=d\alpha$. Then a morphism $\widetilde \lambda:\widetilde C_*^0 \to \widetilde C_\ast$, from the trivial $\cS$-complex $\widetilde C_\ast^0$ to $\widetilde C_\ast$, is given in components by $\lambda=\mu=\Delta_1=0$ and $\Delta_2(1)=\alpha$. The same construction applies to the dual, giving a morphism $\widetilde C_*^0 \to \widetilde C^\dagger_\ast$ whose dual is a morphism $\widetilde C_\ast \to \widetilde C_\ast^0$. Thus $\widetilde C_\ast$ is locally equivalent to the trivial complex. This implies that $h$ is injective. To see that $h$ is surjective, take the complex $\widetilde C_\ast = C_\ast\oplus C_{\ast-1}\oplus R$ with $C_\ast$ freely generated by a single element $\alpha$, with $d=v=\delta_2=0$ and $\delta_1(\alpha)=1$.  Proposition \ref{h-g-reinterpret} implies that $h(\widetilde C_\ast)=1$, so $h$ is surjective.
	
	Proposition \ref{h-g-reinterpret} also makes it clear that the $h$-invariant is the same whether we work over $R$ or $\text{Frac}(R)$, and similarly the grading of $\widetilde C_\ast$ plays no essential role in its determination.
\end{proof}

In particular, the $h$ invariant depends on the weakest possible type of grading we allow for $\cS$-complexes over $R$, and moreover only sees the field of fractions of $R$. In particular, the $h$ invariant defined with $\Z$-coefficients is the same as if it is defined with $\Q$-coefficients, and in this case only the $\Z/2$-gradings of the complexes are necessary.

\subsection{Nested sequences of ideals}\label{subsec:ideals}

Here we describe a refinement of the invariant $h(\widetilde C )$ in the form of a sequence of ideals which in general depends on more than the field of fractions of $R$. 

We begin with the concrete example in which $R=\Z$. In this case, we have another natural invariant associated to the local equivalence class of an $\cS$-complex. Let $(\widetilde C_\ast, \widetilde d,\chi)$ be an $\cS$-complex over $\Z$, and let $\fI\subset \Z[\![x^{-1},x]$ be the associated $\Z[x]$-module $\text{im}(\mathfrak{i}_\ast)$. Set
\[
	  \fg(\widetilde C):={\rm gcd}\{m\mid\text{ $m$ is the leading factor of $Q(x)\in \fI$ with ${\rm Deg}(Q(x))=-h(\widetilde C)$}\} \in \Z_{>0}
\]
An application of Corollary \ref{localization} shows that if there is a morphism from $\widetilde C$ to $\widetilde C'$, then 		
	$\fg(\widetilde C)$ is divisible by $\fg(\widetilde C')$. From Subsection \ref{sec:tensor} we gather that $\fg(\widetilde C\otimes \widetilde C')$ divides $\fg(\widetilde C)\fg(\widetilde C')$. There is also an analogue of Proposition \ref{h-g-reinterpret} for $\fg(\widetilde C)$:
	\begin{equation*}\label{def-g-pos}
		 \fg(\widetilde C)=\begin{cases} {\rm gcd}\{\delta_1v^{k-1}(\alpha)\mid\text{ $\alpha$ satisfies \eqref{h-pos}}\}, & h(\widetilde C)>0 \\ 
		 {\rm gcd}\{a_{-k}\mid\text{  $a_0$, $\dots$, $a_{-k}$ and $\alpha$ satisfy \eqref{h-neg}}\}, & h(\widetilde C) \leqslant 0 \end{cases}
	\end{equation*}
	We have an induced map $\fg:\Theta_\Z^\cS\to \Z_{>0}$ of partially ordered sets, where $\Z_{>0}$ is given the partial order of divisibility. More generally, for any commutative ring $R$, we obtain a similar map from $\Theta_R^\cS$ to the set of ideals of $R$.

	We expand on the above construction in the case that $R$ is a general integral domain. Let $(\widetilde C_\ast, \widetilde d,\chi)$ be an $\cS$-complex over $R$, and again let $\fI\subset R[\![x^{-1},x]$ be the associated $R[x]$-submodule $\text{im}(\mathfrak{i}_\ast)$. Then we make the following:
	
\begin{definition}\label{j-def}
	For an $\cS$-complex $\widetilde C$ as above we define its associated ideal sequence
	\begin{equation}
		\cdots \subseteq J_{i+1} \subseteq J_{i} \subseteq J_{i-1} \subseteq \cdots \subseteq  R \label{eq:nestedidealdef}
	\end{equation}
	where we also write $J_{i}= J_i(\widetilde C)$, as follows:
	\begin{equation}
		J_i(\widetilde C) := \left\{ a_0\in R \mid \exists \; a_0x^{-i} + a_{-1}x^{-i-1} + \cdots \in \fI \right\}.\;\; \diamd \label{eq:idealjidef}
	\end{equation}
\end{definition}

The ideals $J_i$ defined by \eqref{eq:idealjidef} are nested as in \eqref{eq:nestedidealdef} by virtue of the fact that $\fI$ is an $R[x]$-submodule. The maximum $i\in \Z$ such that $J_i\neq 0$ is by definition the invariant $h:=h(\widetilde C)$, and so we may write
\begin{equation}
	J_h \subseteq J_{h-1} \subseteq J_{h-2} \subseteq \cdots \subseteq  R \label{eq:sequenceofideals}
\end{equation}
Note $J_h(\widetilde C)=\fg(\widetilde C)$. 
Corollary \ref{localization} implies that the nested sequence of ideals \eqref{eq:sequenceofideals} is an invariant of the local equivalence class of the $\cS$-complex $(\widetilde C, \widetilde d, \chi)$. More generally, for a morphism $\widetilde C \to \widetilde C'$, with associated ideals $J_i$ and $J'_i$, respectively, we have $J_i\subseteq J_i'$. Lemma \ref{ideal-tensor} implies that these ideals behave with respect to tensor products as follows:
\[
	J_i(\widetilde C) \cdot J_{j}(\widetilde C') \subset J_{i+j}(\widetilde C\otimes \widetilde C').
\]
With respect to taking duals of $\cS$-complexes, we have
\[
	J_i(\widetilde C^\dagger ) = 0 \qquad \Longleftrightarrow \qquad J_{-i-1}(\widetilde C) \neq 0,
\]
which follows from the characterization of the $h$-invariant in terms of $J_i$ mentioned above.

\newpage

%!TEX root = main.tex

\section{Equivariant invariants from singular instanton theory}\label{sec:eq-I}

In this section we apply the machinary of the previous section to the framed instanton $\cS$-complex defined in Section \ref{sec:tilde}, associated to a based knot in an integer homology 3-sphere $(Y,K)$. The output is a triangle of equivariant Floer homology groups,  $\hrI_*(Y,K)$, $\crI_*(Y,K)$ and $\brI_*(Y,K)$, the Fr\o yshov-type invariant $\hinv_\Z(Y,K)$, and a nested sequence of ideals.

\subsection{Equivariant Floer homology groups}\label{sec:eq-I-i}

For an oriented based knot $K$ in an integer homology sphere $Y$, we constructed a $\Z/4$-graded $\cS$-complex $(\widetilde C_\ast(Y,K), \widetilde d)$ whose $\cS$-chain homotopy type is a natural invariant of the pair $(Y,K)$. Associated to this $\cS$-complex, we have equivariant chain complexes $(\hrC_\ast (Y, K),\widehat  d)$, $(\crC_\ast (Y, K),\widecheck d)$ and $(\brC_\ast (Y, K),\widebar d)$ as defined in Subsection \ref{equiv-model-1}. We write
\begin{equation}
	\hrI_*(Y,K), \qquad \crI_*(Y,K), \qquad \brI_*(Y,K) \label{eq:equivsinggroups}
\end{equation}
for the homology groups of these chain complexes and call them the {\it equivariant singular instanton homology} groups of $(Y,K)$. These homology groups are $\Z[x]$-modules. Our notation is motivated by the notation in the monopole Floer homology of \cite{km:monopole}, and the three groups \eqref{eq:equivsinggroups} are respectively called ``{\rm I-from}'', ``{\rm I-to}'' and ``{\rm I-bar}''.

We may alternatively use the small model for equivariant Floer homology groups in Subsection \ref{small-model} to define equivariant singular instanton homology groups. In particular,
\begin{equation*}
	\brI_*(Y,K)\cong \Z[\![x^{-1},x] 
\end{equation*}
 Invariance of the $\cS$-chain homotopy type of $(\widetilde C_\ast (Y, K),\tilde d)$ implies that these $\Z[x]$-modules are invariants of $(Y,K)$. Moreover, we have exact triangles:
\begin{equation}\label{top-equiv-triangle}
	\xymatrix{
	\crI_*(Y,K) \ar[rr]^{j_*}& &
	\hrI_*(Y,K)\ar[dl]^{i_*}\\
	& \brI_\ast(Y,K) \ar[ul]^{p_*} &}
\end{equation}

\begin{equation}\label{top-equiv-triangle-2}
	\xymatrix{
	\hrI_*(Y,K) \ar[rr]^{x_*}& &
	\hrI_*(Y,K)\ar[dl]^{y_*}\\
	& \widetilde I_\ast(Y,K) \ar[ul]^{z_*} &}
\end{equation}
induced by \eqref{equiv-triangle} and \eqref{equiv-triangle-3}. The equivariant singular instanton homology groups are $\Z/4$-graded over the graded ring $\Z[x]$, where $x$ has grading $-2$. With respect to these gradings, the maps $p_\ast$ and $i_\ast$ have degree zero, while $j_\ast$ has degree $-1$. Moreover, the maps $x_*$, $y_*$ and $z_*$ have respective degrees $-2$, $0$ and $1$.

We may similarly define the equivariant singular instanton {\emph{cohomology}} groups $\brI^\ast(Y,K)$, $\hrI^\ast(Y,K)$ and $\brI^\ast(Y,K)$. These satisfy similar properties. In the triangles \eqref{top-equiv-triangle} and \eqref{top-equiv-triangle-2}, the arrows are reversed and the maps are replaced by $i^\ast$, $j^\ast$, $p^\ast$, $x^*$, $y^*$ and $z^*$.

Equivariant singular instanton homology groups are functorial with respect to negative definite pairs. A negative definite pair $(W,S):(Y,K)\to (Y',K')$ induces a morphism $\widetilde \lambda_{(W,S)}:\widetilde C_\ast (Y, K)\to 
\widetilde C_\ast (Y', K')$. The discussion of Subsection \ref{equiv-model-1} shows that this morphism induces morphisms of $\Z[x]$-modules $\hrI_*(W,S): \hrI_*(Y,K)\to \hrI_*(Y',K')$, $\crI_*(W,S): \crI_*(Y,K)\to \crI_*(Y',K')$ and $\brI_*(W,S): \brI_*(Y,K)\to \brI_*(Y',K')$.

\begin{theorem}
The equivariant singular instanton homology groups define functors
\begin{align*}
	\crI_\ast : \mathcal{H} \longrightarrow \text{{\emph{Mod}}}^{\Z/4}_{\Z[x]}\\
	\hrI_\ast : \mathcal{H} \longrightarrow \text{{\emph{Mod}}}^{\Z/4}_{\Z[x]}\\
	\brI_\ast :  \mathcal{H} \longrightarrow \text{{\emph{Mod}}}^{\Z/4}_{\Z[x]}
\end{align*}
from the category $\mathcal{H}$ of based knots in homology 3-spheres to the category of $\Z/4$-graded modules over the graded ring $\Z[x]$. The maps $i_\ast$, $j_\ast$, $p_\ast$ determine natural transformations. Similarly, we have cohomology functors, satisfying the same properties:
\begin{align*}
	\crI^\ast : \mathcal{H} \longrightarrow \text{{\emph{Mod}}}^{\Z/4}_{\Z[x]}\\
	\hrI^\ast : \mathcal{H} \longrightarrow \text{{\emph{Mod}}}^{\Z/4}_{\Z[x]}\\
	\brI^\ast :  \mathcal{H} \longrightarrow \text{{\emph{Mod}}}^{\Z/4}_{\Z[x]}
\end{align*}
\end{theorem} 

\begin{remark}
	In this article, we have restricted our attention to negative definite pairs, in the sense of Definition \ref{def:negdef}. However, we hope that the functoriality of the equivariant singular instanton homology groups can be extended to other cobordisms. $\diamd$
\end{remark}

\begin{remark}
Technically, our constructions assign to a pair $(Y,K)$ a transitive system of equivariant singular instanton homology modules, indexed by the choices of auxiliary data, and to a cobordism $(W,S)$ a morphism of such transitive systems. We may then assign to each transitive system a module, as discussed after Theorem \ref{thm:framedcat}. $\diamd$
\end{remark}

Let $(Y,K)$ be a based knot in an integer homology 3-sphere and let $(-Y,-K)$ be its orientation reversal. From \eqref{eq:orrevgr} and the discussion in Subsection \ref{sec:dual} we conclude that the $\Z/4$-graded $\cS$-complex associated to $(-Y,-K)$ is naturally identified with the dual of the $\Z/4$-graded $\cS$-complex of $(Y,K)$. This implies the following:

\begin{prop}
	Let $r:\mathcal{H}\to \mathcal{H}$ denote the functor which reverses orientations, i.e. $r(Y,K)=(-Y,-K)$, and $r(W,S)=(-W,-S)$. Then we have the following equalities:
	\begin{align*}
	\crI^\ast \circ r =  \hrI_\ast, \qquad \hrI^\ast \circ r = \crI_\ast, \qquad \brI^\ast \circ r = \brI_\ast,\\
	\crI_\ast \circ r =  \hrI^\ast, \qquad \hrI_\ast \circ r = \crI^\ast, \qquad \brI_\ast \circ r = \brI^\ast.
\end{align*}
\end{prop}

\begin{remark}
	Observe that we are working over the coefficient ring $\Z$, and yet we do not need homology orientations, as is necessary, for example, in \cite{km:monopole}. This is because we have restricted our attention to knots in homology 3-spheres and cobordisms which are negative definite pairs, which have canonical homology orientations. $\diamd$
\end{remark}

\begin{remark}
	All of the above works if $\Z$ is replaced by any commutative ring $R$. $\diamd$
\end{remark}

\subsection{Local equivalence and the concordance invariant $\hinv$}

Recall the category $\mathcal{H}$, whose objects are based knots in integer homology 3-spheres, and whose morphisms are negative definite pairs. Consider the set $\Theta_\Z^{3,1}$ obtained from $\mathcal{H}$ by the general procedure described in Subsection \ref{sec:localequiv}. We obtain the following description:
\[
	\Theta_\Z^{3,1} := \left\{ (Y,K): Y \text{ an integer homology 3-sphere}, K\subset Y \text{ a knot}  \right\} /\sim
\]
\[
	 (Y,K) \sim ( Y', K') \;\;\; \Longleftrightarrow \;\;\; \begin{array}{c} \exists\; \text{ negative definite pairs}\\ (Y,K)\to (Y',K'),\;\; (Y',K')\to (Y,K)\end{array}
\]
The partially ordered set $\Theta_\Z^{3,1}$ has a group operation: the identity is represented by the unknot in the 3-sphere, the group operation is connected sum of knots, and inverses are obtained by reversing orientation. This abelian group is also a partially ordered group, with $[(Y,K)]\leqslant [(Y',K')]$ if and only if there is a negative definite pair from $(Y,K)$ to $(Y',K')$. Furthermore, there is a natural homomorphism to $\Theta_\Z^{3,1}$ from the homology concordance group defined in the introduction. We will prove the following:

\begin{theorem}\label{thm:slocequivhomom}
	Let $R$ be a commutative ring. The assignment $(Y,K)\mapsto (\widetilde C_\ast(Y,K;R),\widetilde d,\chi)$ induces a homorphism $\Xi:\Theta_\Z^{3,1}\to \Theta_{R,\Z/4}^\cS$ of partially ordered abelian groups.
\end{theorem}

\noindent That the assignment induces a well-defined map $\Theta_\Z^{3,1}\to \Theta_{R,\Z/4}^\cS$ of partially ordered sets follows from the discussion in Subsection \ref{sec:framed}. That the map is a homomorphism will follow from our connected sum theorem, to be proved in Section \ref{sec:consum}.

\begin{definition}
	For a based knot in an integer homology 3-sphere $(Y,K)$, we define $h(Y,K)=\hinv_\Z(Y,K)$ to be the Fr\o yshov invariant of the $\cS$-complex $(\widetilde C(Y,K), \widetilde d)$. That is, $\hinv$ is the invariant of the equivalence class $[(Y,K)]\in \Theta_\Z^{3,1}$ obtained from the composition
	\[
		\hinv:\Theta_\Z^{3,1} \xrightarrow{\;\;\Xi\;\;}\Theta_{\Z,\Z/4}^\cS \xrightarrow{\;\;h\;\;} \Z
	\]
	When $Y$ is the 3-sphere, we simply write $\hinv(K)$. More generally, we write $\hinv_R(Y,K)$ for the Fr\o yshov invariant obtained using a coefficient ring $R$ which is an integral domain. $\diamd$
\end{definition}

\begin{remark}
	In Section \ref{sec:loccoeffs}, we will generalize this invariant to the collection of invariants $\hinv_\sS(Y,K)$ for $\sS$-algebras over $\sR=\Z[U^{\pm 1}, T^{\pm 1}]$ using local coefficient systems. These more general versions are the ones discussed in the introduction. $\diamd$
\end{remark}

Let $\mathcal{C}_\Z$ be the homology concordance group. There is a natural homomorphism 
\[
	\mathcal{C}_\Z\longrightarrow \Theta_\Z^{3,1}
\]
In this way, any integer valued function defined on $ \Theta_{R,\Z/4}^\cS$ gives rise to a homology concordance invariant for knots. We use the same notation $\hinv_R:\cC_\Z\to \Z$ for the homomorphism induced by $\hinv_R$ in this way.

\begin{theorem}\label{thm:hcheckprops}
	Let $R$ be an integral domain. Then $\hinv_R$ induces a homology concordance invariant which is a homomorphism of partially ordered groups:
	\[
	\hinv_R:\cC_\Z\to \Z.
	\]
	If $(W,S):(Y,K) \to (Y',K')$ is a negative definite pair, then $\hinv_R(Y,K)\leqslant \hinv_R(Y',K')$.
\end{theorem}

We may refine the invariants $\hinv_R$ using Definition \ref{j-def} to obtain a sequence of ideals
\[
	J^R_{h}(Y,K) \subseteq J^R_{h-1}(Y,K) \subseteq \cdots \subseteq R
\]
where $h=\hinv_R(Y,K)$, the construction of which again factors through $\Theta_\Z^{3,1}$. The properties of these ideals carry over from the discussion in Subsection \ref{subsec:ideals}. We will also generalize this construction in Section \ref{sec:loccoeffs}.

\newpage

%!TEX root = main.tex
\section{The connected sum theorem}\label{sec:consum}

Let $(Y,K)$ and $(Y',K')$ be pairs of integer homology 3-spheres with embedded oriented based knots. Then the connected sum $(Y\# Y', K\# K')$, performed at the distinguished basepoints of $K$ and $K'$, is also an oriented  based knot. The main result of this section is:

\begin{theorem}\label{thm:connectedsum}  {\emph{\bf{(Connected Sum Theorem for Knots)}}} In the situation described above, there is a chain homotopy equivalence of $\Z/4$-graded $\cS$-complexes:
\begin{equation}
	\widetilde C(Y\# Y',K\# K') \simeq  \widetilde C(Y,K)\otimes \widetilde C(Y',K') \label{eq:connsumforknots}
\end{equation}
The statement holds over any coefficient ring. This equivalence is natural, up to $\cS$-chain homotopy, with respect to split cobordisms.
\end{theorem}

\noindent Each framed instanton chain complex that appears in the statement has some fixed choices of metric and perturbation, which are as usual suppressed from the notation. The connected sum theorem, together with Lemma \ref{ideal-tensor}, implies the following result for $\hrI_\ast$.

\begin{cor}
	Let $(Y,K)$ and $(Y',K')$ be based knots in integer homology 3-spheres. There is a chain homotopy equivalence of $\Z/4$-graded complexes over $\Z[x]$:
	\begin{align*}
		\hrC_*(Y,K)\otimes_{\Z[x]}  \hrC_*(Y',K') \simeq \hrC_*(Y\# Y',K\# K')
	\end{align*}
	natural up to homotopy with respect to split cobordisms. In particular, if $R$ is a field, then there is a K\"unneth formula 
	relating $\hrI_*(Y,K;R)$, $\hrI_*(Y',K';R)$ and $\hrI_*(Y\# Y',K\# K';R)$.
\end{cor}

\noindent A similar statement holds for the $\brI_\ast$ theory, and the two are intertwined by the map $i_\ast$. We remark also that Theorem \ref{thm:connectedsum} completes the proof of Theorem \ref{thm:slocequivhomom}.

In this section we prove the equivalence \eqref{eq:connsumforknots} and its naturality (explained in Subsection \ref{subsec:natcon}) over $\Z$, as the case for arbitrary coefficients follows from this. 

Theorem \ref{thm:connectedsum} is a singular instanton homology analogue of Fukaya's connected sum theorem for the instanton Floer homology of integer homology 3-spheres \cite{fukaya}. In fact, our result goes further than Fukaya's theorem, which does not determine the full $\cS$-complex for the connected sum. Our proof is an adaptation of the one described by Donaldson \cite[Section 7.4]{donaldson-book} in the non-singular setting. Apart from our having repackaged the algebra, the main difference between our proof of Theorem \ref{thm:connectedsum} and Donaldson's proof in the non-singular case occurs in the proof of Proposition \ref{prop:morphchqi}, where a singular analogue of \cite[Theorem 7.16]{donaldson-book} is used; see Remark \ref{rmk:cfdonaldsonglue} for more details.

\subsection{Topology of the connected sum theorem}

There is a standard cobordism of pairs $(Y,K)\sqcup(Y',K')\to (Y\#Y',K\#K')$ which we denote by $(W,S)$. The cobordism $W:Y\sqcup Y'\to Y\#Y'$ is obtained by attaching a 4-dimensional 1-handle $H$ to $[0,1] \times ( Y \sqcup Y')$ along 3-ball neighborhoods of the basepoints $p\times \{ 1 \}$ and $p'\times \{ 1 \}$. The surface cobordism $S:K\sqcup K'\to K\#K'$ is similarly obtained by attaching a 2-dimensional 1-handle, embedded inside $H$. Note that if $K$ and $K'$ are unknots, then $S$ is a pair of pants. For this reason we depict $(W,S)$ by a directed pair of pants, i.e. {\smash{\raisebox{-.2\height}{\includegraphics[scale=0.25]{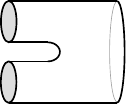}}}}. We have a similar cobordism of pairs $(W',S'):(Y\#Y',K\#K')\to (Y,K)\sqcup(Y',K')$, which is obtained from $(W,S)$ by swapping the roles of incoming and outgoing ends, and reversing orientation. This is depicted by {\smash{\raisebox{-.2\height}{\includegraphics[scale=0.25]{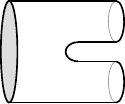}}}}. Both surfaces $S$ and $S'$ admit framings which are compatible with the Seifert framings of the knots $K$, $K'$ and $K\#K'$.

\begin{figure}[h]
\labellist
\pinlabel $\gamma$ at 38 94
\pinlabel $\gamma'$ at 133 110
\pinlabel $\gamma^\#$ at 240 88
\pinlabel $\gamma\cup\gamma'\cup\gamma^\#$ at 323 69
\pinlabel $\sigma$ at 20 18
\pinlabel $\sigma'$ at 120 33
\pinlabel $\sigma^\#$ at 215 40
\pinlabel $\sigma\cup\sigma'\cup\sigma^\#$ at 323 -7
\endlabellist
\centering
\vspace{.3cm}
\includegraphics[scale=.8]{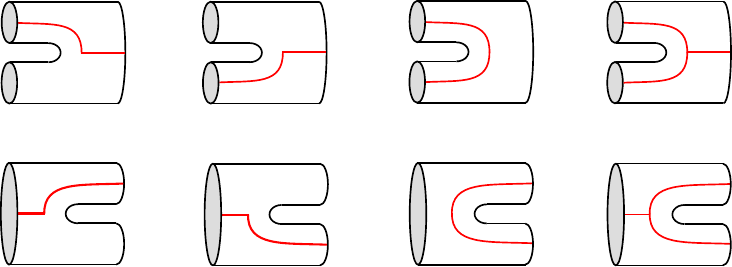}
\caption{}
\label{fig:connsumpaths}
\end{figure}

Fix three oriented, piecewise-differentiable paths $\gamma$, $\gamma'$ and $\gamma^\#$ on the surface $S\subset W$ in the following way. The path $\gamma$ (resp. $\gamma'$) begins at the base point $p$ of $Y$ (resp. $p'$ of $Y'$) and ends at the base point $p^\#$ of $Y\#Y'$. The path $\gamma^\#$ begins at $p$ and ends at $p'$. These three paths together form a graph in the shape of the letter $Y$ as it is shown in Figure \ref{fig:connsumpaths}. The holonomy of any connection along $\gamma$ is equal to the product of its holonomies along the paths $\gamma'$ and $\gamma^\#$. Similarly, we denote by $\sigma$, $\sigma'$ and $\sigma^\#$ the paths on the surface $S'$ which are the mirrors of the paths $\gamma$, $\gamma'$ and $\gamma^\#$, as depicted in Figure \ref{fig:connsumpaths}.

The composite cobordism $(W,S)\circ (W',S')$ has an embedded loop on $S\circ S'$, depicted in Figure \ref{fig:connsumcomposites1}, formed by joining together $\gamma^\#$ and $\sigma^\#$. A regular neighborhood $N$ of this loop is diffeomorphic to the pair $(S^1\times D^3, S^1\times D^1)$, the boundary of which is the pair $(S^1\times S^2, S^1\times \text{2 pts})$. Excising $N$ and gluing back in a copy of $(D^2\times  S^2, D^2\times \text{2 pts} )$ produces a cobordism isomorphic to $[0,1]\times (Y\# Y', K\# K')$, the identity cobordism.\\

\begin{figure}[h]
\labellist
\pinlabel $(W,S)\circ (W',S')$ at 55 58
\pinlabel $[0,1]\times (Y\#Y',K\#K')$ at 231 58
\endlabellist
\centering
\vspace{.45cm}
\includegraphics[scale=1]{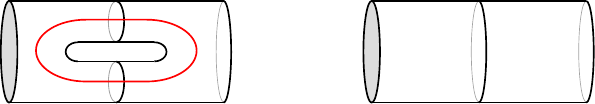}
\caption{}
\label{fig:connsumcomposites1}
\end{figure}

\noindent Now we consider the other composite, $(W',S')\circ (W,S)$. Within this cobordism there is an embedded pair $(S^3,S^1)$; in Figure \ref{fig:connsumcomposites2} below, this $S^1$ is the horizontal circle. Cutting along this 3-sphere and circle, and gluing in two pairs of the form $(B^4,D^2)$, yields a cobodism isomorphic to the identity cobordism $[0,1]\times \left((Y,K)\sqcup (Y',K')\right)$.\\

\begin{figure}[h]
\labellist
\pinlabel $(W',S')\circ (W,S)$ at 55 58
\pinlabel $[0,1]\times \left((Y,K)\sqcup (Y',K')\right)$ at 231 58
\endlabellist
\centering
\vspace{.45cm}
\includegraphics[scale=1]{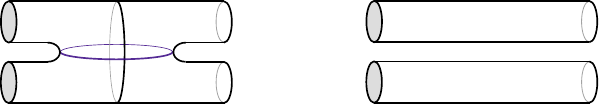}
\caption{}
\label{fig:connsumcomposites2}
\end{figure}

\subsection{Moduli spaces on the cobordisms $(W,S)$ and $(W',S')$} \label{conn-moduli-spaces}

Throughout this section, we write $\alpha$, $\alpha'$ and $\alpha^\#$ for gauge equivalence classes of critical points for the perturbed Chern-Simons functionals on $(Y,K)$, $(Y',K')$ and $(Y\# Y',K\#K')$. Similarly, we write $\theta$, $\theta'$ and $\theta^\#$ for the corresponding reducible classes. We use the abbreviated notation $M(\alpha,\alpha';\alpha^\#)_d$ for the instanton moduli space $M(W,S;\alpha,\alpha',\alpha^\#)_d$. Similarly, $M(\alpha^\#;\alpha,\alpha')_d$ denotes the moduli space $M(W',S';\alpha^\#,\alpha,\alpha')_d$.

We can use the paths $\gamma$ and $\gamma' $ defined above to define maps as in Subsection \ref{cut-down}:
\begin{equation*}
	H^\gamma:\sB(W,S;\alpha,\alpha',\alpha^\#) \to S^1\hspace{1cm}
	H^{\gamma'}:\sB(W,S;\alpha,\alpha',\alpha^\#) \to S^1
\end{equation*}
Note that in contrast to our convention from Section \ref{loop-mu}, for the sake of brevity, we omit the critical limits from the notation of these holonomy maps. To define $H^\gamma$ (resp. $H^{\gamma'}$) we need $\alpha$ (resp. $\alpha'$) to be irreducible, and both maps require $\alpha^\#$ irreducible. By picking generic points $h, h'\in S^1$, we define the following cut-down moduli spaces:
\begin{equation}\label{Mg}
	M_{\gamma}(\alpha,\alpha';\alpha^\#)_{d}:=\{[A]\in M(\alpha,\alpha';\alpha^\#)_{d+1}\mid 
	H^\gamma([A])=h\},
\end{equation}
\begin{equation}\label{Mgp}
	M_{\gamma'}(\alpha,\alpha';\alpha^\#)_{d}:=\{[A]\in M(\alpha,\alpha';\alpha^\#)_{d+1}\mid 
	H^{\gamma'}([A])=h'\},
\end{equation}
\begin{equation}\label{Mggp}
	M_{\gamma\gamma'}(\alpha,\alpha';\alpha^\#)_{d}:=\{[A]\in M(\alpha,\alpha';\alpha^\#)_{d+2}\mid 
	H^{\gamma}([A])=h,\, H^{\gamma'}([A])=h'\}.
\end{equation}
The moduli spaces \eqref{Mg}, \eqref{Mgp} and \eqref{Mggp} are defined only in the case that $\alpha^\#$ is irreducible. Moreover, we need irreducibility of $\alpha$ (resp. $\alpha'$) to define the moduli spaces in \eqref{Mg} (resp. \eqref{Mgp}) and \eqref{Mggp}. There is another obvious way in which we can define a cut-down moduli space in the case that $\alpha$ and $\alpha'$ are both irreducible:
\begin{equation}\label{Mgs}
	M_{\gamma^\#}(\alpha,\alpha';\alpha^\#)_{d}:=\{[A]\in M(\alpha,\alpha';\alpha^\#)_{d+1}\mid 
	H^{\gamma^\#}([A])= h'\cdot h^{-1}\}
\end{equation}
We will mainly be concerned with the moduli spaces in \eqref{Mg}, \eqref{Mgp}, \eqref{Mggp} and \eqref{Mgs} in the case that $d=0$ or $d=1$. By choosing $h$ and $h'$ generically, we may assume that all such moduli spaces are smooth manifolds. 

\begin{remark}
	We follow similar orientation conventions as before to orient the moduli spaces. For example, to orient the moduli space 
	$M(\alpha,\alpha';\alpha^\#)$, we use the canonical homology orientation $o_W$ for the pair $(W,S)$, which is an element of 
	$\Lambda[W,S;\theta_+,\theta_+';\theta^\#_-]$, defined as in Subsection \ref{subsec:ors}. Given elements 
	$o_\alpha\in \Lambda[\alpha]$, $o_{\alpha'}\in \Lambda[\alpha']$ and $o_{\alpha^\#}\in \Lambda[\alpha^\#]$, 
	we can fix $o_{\alpha,\alpha';\alpha^\#}\in \Lambda[W,S;\theta_+,\theta_+';\theta^\#_-]$, and hence an orientation of 
	$M(\alpha,\alpha';\alpha^\#)$, by demanding:
	\[
	  \Phi(o_{\alpha}\otimes o_{\alpha'}\otimes o_W)=\Phi(o_{\alpha,\alpha';\alpha^\#}\otimes o_{\alpha^\#}). \;\diamd
	\]
\end{remark}

\subsection{Proof of the connected sum theorem}

Let us write $(\widetilde C_\ast, \widetilde d,\chi)$, $(\widetilde C'_\ast, \widetilde d',\chi')$ and $(\widetilde C_\ast^\#,\widetilde d^\#,\chi^\#)$ for the framed instanton $\cS$-complexes for the based knots $(Y,K)$, $(Y',K')$ and $(Y\# Y', K\# K')$, respectively. Write $(\widetilde C^\otimes_\ast,\widetilde d^\otimes,\chi^\otimes)$ for the tensor product $\cS$-complex defined using $(\widetilde C_\ast, \widetilde d,\chi)$ and $(\widetilde C'_\ast, \widetilde d',\chi')$ as in Subsection \ref{sec:tensor}, so that its underlying chain complex is simply $(\widetilde C_\ast, \widetilde d)\otimes (\widetilde C'_\ast, \widetilde d')$.

The moduli spaces discussed in Subsection \ref{conn-moduli-spaces} will be used to define morphisms
\begin{align*}
  \widetilde \lambda_{(W,S)}:& (\widetilde C_*^\otimes,\widetilde d^\otimes,\chi^\otimes) \to(\widetilde C_*^\#,\widetilde d^\#,\chi^\#),\\
  \widetilde \lambda_{(W',S')}: & (\widetilde C_*^\#,\widetilde d^\#,\chi^\#) \to (\widetilde C_*^\otimes,\widetilde d^\otimes,\chi^\otimes).
\end{align*}
In Subsection \ref{CWS}, we define ${\smash{\widetilde \lambda_{(W,S)}}}$ and show that it is a morphism of $\cS$-complexes. The definition of ${\smash{\widetilde \lambda_{(W',S')}}}$ is similar and is given in Subsection \ref{CWS-p}. Finally, in Subsection \ref{comp}, we show that these maps are $\cS$-chain homotopy equivalences.

\subsubsection{Definition of the map $\smash{\widetilde \lambda_{(W,S)}}$}\label{CWS}

Using the $\cS$-compex decomposition of $\widetilde C_*^\otimes$ from Subsection \ref{sec:tensor} and the notation of Definition \ref{def:morphism}, giving a morphism \[\widetilde \lambda_{(W,S)}:(\widetilde C_*^\otimes,\widetilde d^\otimes,\chi^\otimes) \to(\widetilde C_*^\#,\widetilde d^\#,\chi^\#)\] amounts to defining four maps, $\lambda:C^\otimes_\ast \to C^\#_\ast$, $\mu:C^\otimes_\ast \to C^\#_{\ast-1}$, $\Delta_1: C^\otimes_0 \to \Z$, and $\Delta_2:\Z\to C_{-1}^\#$. Upon further decomposing $C^\otimes_\ast$, these may be written as maps
\begin{align*}
   & \lambda:(C\otimes C')_\ast\oplus (C\otimes C')_{\ast-1}\oplus C_*\oplus C_*' \to C_*^\# \\
  &\mu:(C\otimes C')_\ast\oplus (C\otimes C')_{\ast-1}\oplus C_*\oplus C_*' \to C_{*-1}^\# \\
  &\Delta_1:(C\otimes C')_0 \oplus (C\otimes C')_{-1}\oplus C_0\oplus C_0'  \to \Z \\
  &\Delta_2:\Z \to C_{-1}^\#.
\end{align*}
The maps $\lambda$, $\mu$ and $\Delta_1$ can be further decomposed using the decomposition of their domains into four components; we write
\[
  \lambda=[\lambda_1,\lambda_2,\lambda_3,\lambda_4],\hspace{.75cm}\mu=[\mu_1,\mu_2,\mu_3,\mu_4],\hspace{.75cm}
  \Delta_1=[\Delta_{1,1},\Delta_{1,2},\Delta_{1,3},\Delta_{1,4}].
\]
We proceed to define these maps. Suppose $\alpha$, $\alpha'$ and $\alpha^\#$ are all irreducible. Then define:
\[
\begin{array}{ll}
    \langle \lambda_1(\alpha\otimes \alpha'), \alpha^\#\rangle = \# M_{\gamma^\#}(\alpha,\alpha';\alpha^\#)_0 \hspace{.5cm}&  \Delta_{1,1}(\alpha\otimes \alpha') = \# M_{\gamma^\#}(\alpha,\alpha';\theta^\#)_0  \\
   \langle \lambda_2(\alpha\otimes \alpha'), \alpha^\#\rangle = \# M(\alpha,\alpha';\alpha^\#)_0 & \Delta_{1,2}(\alpha\otimes \alpha') = \# M(\alpha,\alpha';\theta^\#)_0 \\
   \langle \lambda_3(\alpha), \alpha^\#\rangle = \# M(\alpha,\theta';\alpha^\#)_0  & \Delta_{1,3}(\alpha) = \# M(\alpha,\theta';\theta^\#)_0 \\
   \langle \lambda_4( \alpha'), \alpha^\#\rangle = \# M(\theta,\alpha';\alpha^\#)_0 &   \Delta_{1,4}(\alpha') = \# M(\theta,\alpha';\theta^\#)_0  \\
\end{array}
\]
We also define $\langle \Delta_{2}(1),\alpha^\#\rangle=\# M(\theta,\theta';\alpha^\#)_0$. Finally, we define $\mu$ as follows:
\[
\begin{array}{l}
    \langle \mu_1(\alpha\otimes \alpha'), \alpha^\#\rangle = \# M_{\gamma\gamma'}(\alpha,\alpha';\alpha^\#)_0 \hspace{.5cm}\\
      \langle \mu_{2}(\alpha\otimes \alpha'),\alpha^\#\rangle = \# M_{\gamma}(\alpha,\alpha';\alpha^\#)_0\\
   \langle \mu_3(\alpha), \alpha^\#\rangle = \# M_\gamma(\alpha,\theta';\alpha^\#)_0 \\
   \langle \mu_4( \alpha'), \alpha^\#\rangle = \# M_{\gamma'}(\theta,\alpha';\alpha^\#)_0 \\
\end{array}
\]
Using the pictorial calculus introduced in Section \ref{sec:tilde}, we may write these maps as follows:
\[
\lambda = \left[\;\;\raisebox{-.3\height}{\includegraphics[scale=.35]{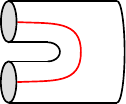}},\;\;
		\raisebox{-.3\height}{\includegraphics[scale=.35]{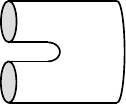}},\;\;
		\raisebox{-.3\height}{\includegraphics[scale=.35]{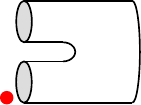}},\;\;
		\raisebox{-.3\height}{\includegraphics[scale=.35]{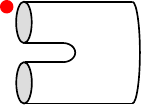}}\;\; \right]
		\hspace{1cm}
\Delta_1 = \left[\;\;\raisebox{-.3\height}{\includegraphics[scale=.35]{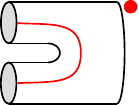}},\;\;
		\raisebox{-.3\height}{\includegraphics[scale=.35]{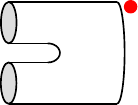}},\;\;
		\raisebox{-.3\height}{\includegraphics[scale=.35]{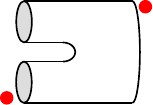}},\;\;
		\raisebox{-.3\height}{\includegraphics[scale=.35]{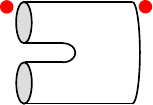}}\;\; \right]
\]
\[
\Delta_2 = \raisebox{-.3\height}{\includegraphics[scale=.35]{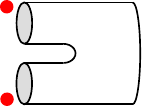}}
\hspace{1cm}
\mu = \left[\;\;\raisebox{-.3\height}{\includegraphics[scale=.35]{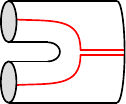}},\;\;
		\raisebox{-.3\height}{\includegraphics[scale=.35]{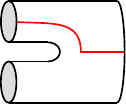}},\;\;
		\raisebox{-.3\height}{\includegraphics[scale=.35]{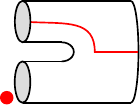}},\;\;
		\raisebox{-.3\height}{\includegraphics[scale=.35]{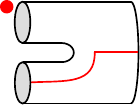}}\;\; \right]
\]

\begin{remark}
	There is a unique reducible ASD connection $A_0$ on $(W,S)$ up to gauge equivalence, which is unobstructed and has index $-1$, as can be verified using arguments similar to those in Subsection \ref{subsec:red}. From this one can deduce that $\widetilde \lambda$ has degree $0$ (mod $4$), as is already implicit in the above notation. $\diamd$
\end{remark}

\begin{remark}\label{rmk:vsquared}
	We have indicated the convention, as in the case of $\mu_1$, that a picture involving more than one path represents a map defined by cutting down moduli spaces by holonomy constraints $H^\gamma([A]) = h_\gamma$ for each path $\gamma$, where $h_\gamma\in S^1$. In the generic case, as is always assumed, the parameters $h_\gamma$ are distinct from one another. Thus we have the relation
	\[
		\raisebox{-.3\height}{\includegraphics[scale=.75]{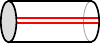}} = 0,
	\]
	because this map counts instantons whose holonomy along one path is equal to two distinct quantities. We note here, in passing, that whenever two paths overlap, our convention is to draw them slightly separated from one another. $\diamd$
\end{remark}

\begin{remark}\label{rmk:holrel}
	Because $\gamma^\#$ is homotopic to $\gamma$ concatenated with the reverse of $\gamma'$, we have the relation $H^\gamma \cdot ( H^{\gamma'})^{-1} = H^{\gamma^\#}$. From this we obtain the relation
	\[
		\raisebox{-.3\height}{\includegraphics[scale=.5]{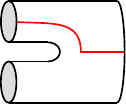}} - \raisebox{-.3\height}{\includegraphics[scale=.5]{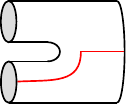}}= \raisebox{-.3\height}{\includegraphics[scale=.5]{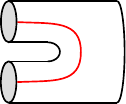}}
	\]
	This follows from Remark \ref{rmk:deg} and the elementary fact that for two maps $f$ and $f'$ from a closed oriented 1-manifold $M$ to $S^1$, we have $\text{deg}(f\cdot f') = \text{deg}(f)+\text{deg}(f')$. $\diamd$
\end{remark}

\begin{figure}
\centering
\includegraphics[scale=.65]{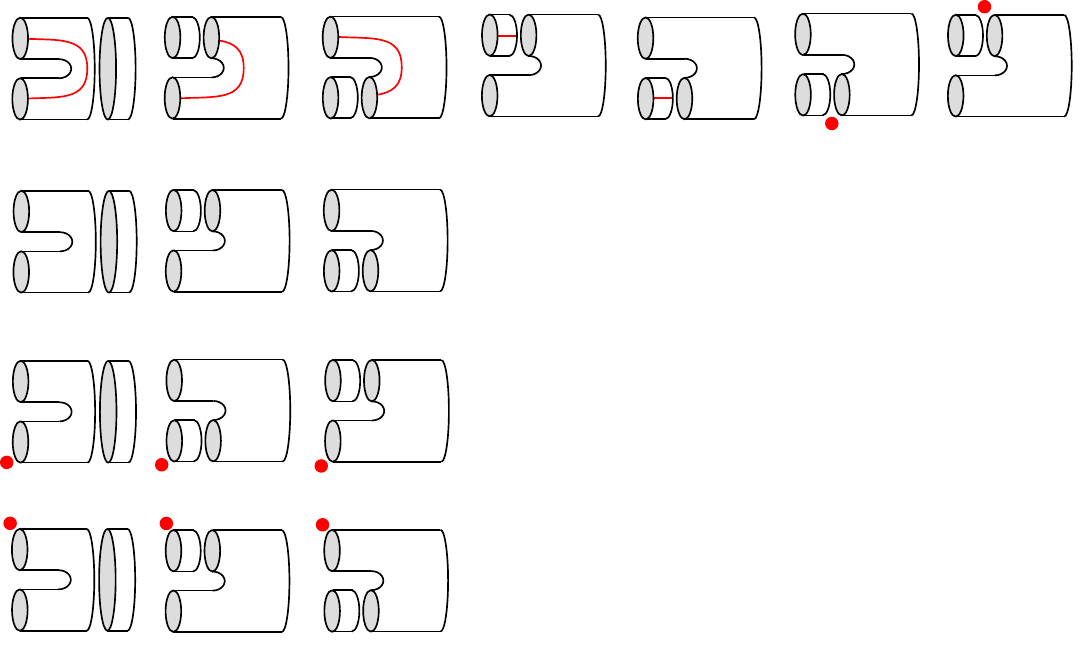}
\caption{}\label{fig:lambdarel}
\end{figure}

\begin{figure}
\centering
\includegraphics[scale=.65]{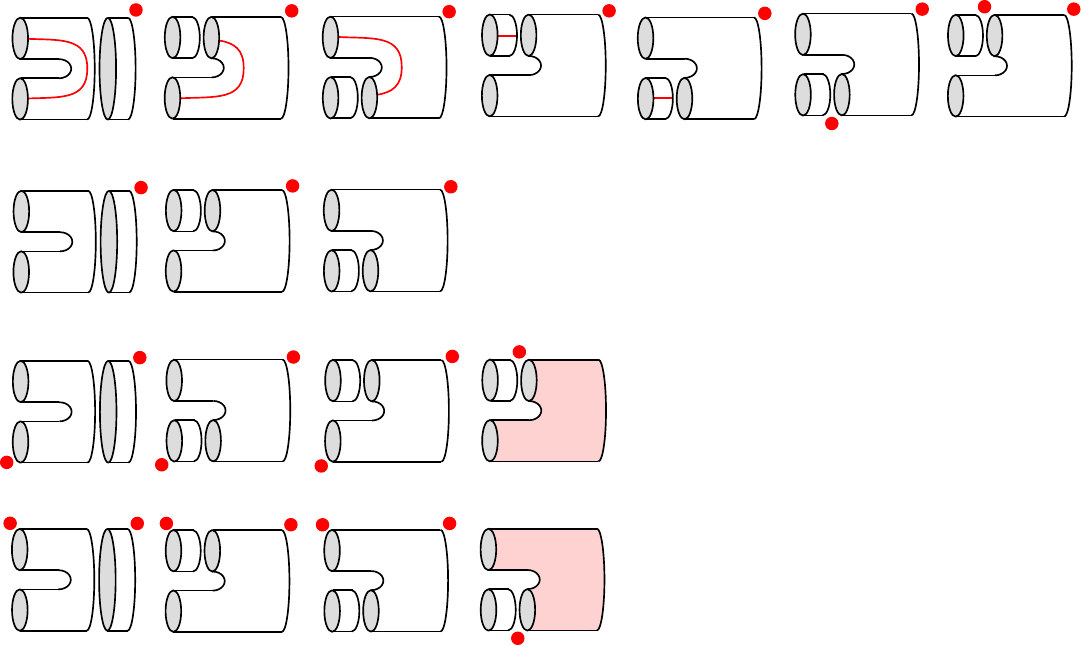}
\caption{}\label{fig:delta1rel}
\end{figure}

\begin{prop}\label{CWS-C}
	The maps $\lambda$, $\mu$, $\Delta_1$ and $\Delta_2$ define a morphism of $\cS$-complexes.
	That is to say, they satisfy the following identities:
	\begin{align*}
		d^\# \circ \lambda&=\lambda \circ d^{\otimes},\\
		\delta_1^\#\circ \lambda&=\Delta_1\circ d^{\otimes}+\delta_1^\otimes,\\
		\lambda\circ \delta_2^\otimes &=\delta_2^\#  -d^{\#} \circ \Delta_2,\\
		 d^{\#}\circ \mu+\mu \circ d^{\otimes}
		&=v^\# \circ \lambda-\lambda \circ v^\otimes+\delta_2^\#\circ \Delta_1-\Delta_2\circ \delta_1^\otimes.
	\end{align*}
\end{prop}

\begin{proof}
The relations listed will follow by analyzing the ends of certain 1-dimensional moduli spaces using the gluing theory which is outlined in Subsection \ref{subsec:red}.

First, the relation $d^\# \circ \lambda = \lambda\circ d^\otimes$ splits into the equations:
\begin{align*}
d^\# \lambda_1 &= \lambda_1 (d\otimes 1) + \lambda_1 (\varepsilon\otimes d') - \lambda_2(\varepsilon v \otimes 1) + \lambda_2 (\varepsilon\otimes v') + \lambda_3  (\varepsilon\otimes \delta_1') + \lambda_4 (\delta_1 \otimes 1)\\
d^\#  \lambda_2 &= \lambda_2  (d\otimes 1) - \lambda_2  (\varepsilon \otimes d')\\
d^\#  \lambda_3 &= \lambda_2  (\varepsilon \otimes \delta_2') + \lambda_3  d \\
d^\#  \lambda_4 &= -\lambda_2 (\delta_2\otimes 1) + \lambda_4  d'
\end{align*}
These relations follow by counting the boundary points of the 1-dimensional manifolds $\smash{M^+_{\gamma^\#}(\alpha,\alpha';\alpha^\#)_1}$, $M^+(\alpha,\alpha';\alpha^\#)_1$, $M^+(\alpha,\theta';\alpha^\#)_1$ and $M^+(\theta,\alpha';\alpha^\#)_1$, respectively. The boundary points in the four cases correspond to certain factorizations of instantons which are depicted by the four rows in Figure \ref{fig:lambdarel}. The details of this analysis are completely analogous to the proofs of Propositions \ref{prop:cobdelta} and \ref{prop:vmaprels2}.

Similarly, the relation $\delta_1^\#\circ \lambda=\Delta_1\circ d^{\otimes}+\delta_1^\otimes$ splits into four equations, which correspond to the four rows in Figure \ref{fig:delta1rel}, obtained by counting boundary points of the moduli spaces $\smash{M^+_{\gamma^\#}(\alpha,\alpha';\theta^\#)_1}$, $M^+(\alpha,\alpha';\theta^\#)_1$, $M^+(\theta,\alpha';\theta^\#)_1$ and $M^+(\theta,\alpha';\theta^\#)_1$. 

The relation $\lambda\circ \delta_2^\otimes =\delta_2^\#  -d^{\#} \circ \Delta_2$ is equivalent to
\[
	\lambda_3\circ \delta_2 + \lambda_4 \circ \delta_2'  = \delta_2^\# - d^\# \circ \Delta_2,
\]
the terms corresponding to the boundary points of moduli spaces $M^+(\theta,\theta';\alpha^\#)_1$:
\[
\includegraphics[scale=.5]{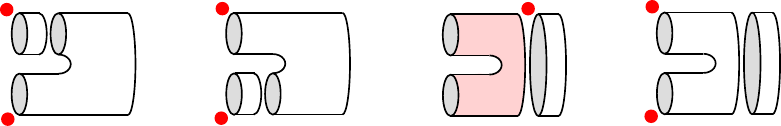}
\]

\begin{figure}[t]
\centering
\includegraphics[scale=.65]{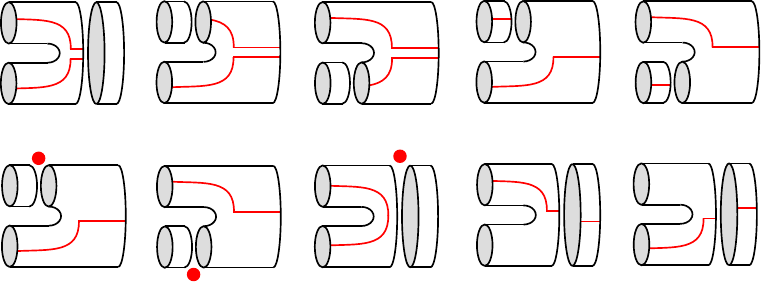}\hspace{.5cm}
\caption{}\label{fig:mu1rel}
\end{figure}

Finally, consider the fourth relation 
\begin{equation}
	d^{\#}\circ \mu+\mu \circ d^{\otimes} = v^\# \circ \lambda-\lambda \circ v^\otimes+\delta_2^\#\circ \Delta_1-\Delta_2\circ \delta_1^\otimes\label{eq:murelfromprop}.
\end{equation}
This splits into four equations. The first equation follows by counting the boundary points of 1-dimensional moduli spaces $M^+_{\gamma\gamma'}(\alpha,\alpha';\alpha^\#)_1$. See Figure \ref{fig:mu1rel}. Note that gluing theory gives an additional contribution to Figure \ref{fig:mu1rel} of the form 
\[
	\includegraphics[scale=.65]{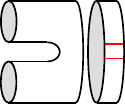}
\]
However, this term vanishes by Remark \ref{rmk:vsquared}. The term in Figure \ref{fig:mu1rel} depicted as
\begin{equation}
	\includegraphics[scale=.65]{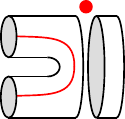} \label{pic:mu}
\end{equation}
comes from broken trajectories in $M^+_{\gamma\gamma'}(\alpha,\alpha';\alpha^\#)_1$ which break along a reducible on $Y\# Y'$. Although the two curves $\gamma$, $\gamma'$ travel through $[1,\infty ) \times Y \# Y'$, it is clear that
\[
	M^+_{\gamma\gamma'}(\alpha,\alpha';\alpha^\#)_1 = M^+_{\gamma^\#\gamma'}(\alpha,\alpha';\alpha^\#)_1.
\]
Thus from the viewpoint of the latter moduli space, we only need understand how such trajectories interact with the holonomy map of $\gamma'$, from which the contribution \eqref{pic:mu} follows just as in Proposition \ref{prop:vmap}. In verifying the relation at hand from 
Figure \ref{fig:mu1rel}, we use Remark \ref{rmk:holrel} several times, for example:
	\[
		\raisebox{-.3\height}{\includegraphics[scale=.6]{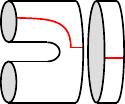}}\;\; -\;\; \raisebox{-.3\height}{\includegraphics[scale=.6]{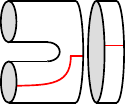}}\;\;=\;\; \raisebox{-.3\height}{\includegraphics[scale=.6]{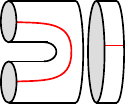}}
		\vspace{.25cm}
	\]
	\[
		\raisebox{-.3\height}{\includegraphics[scale=.6]{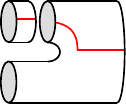}}\;\; -\;\; \raisebox{-.3\height}{\includegraphics[scale=.6]{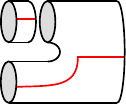}}\;\;=\;\; \raisebox{-.3\height}{\includegraphics[scale=.6]{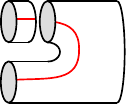}}
	\vspace{.35cm}
	\]

The second equation stemming from \eqref{eq:murelfromprop} is obtained from the boundary points of the moduli space $M^+_{\gamma}(\alpha,\alpha';\alpha^\#)_1$, which are represented in Figure \ref{fig:mu2rel}. The third equation follows from considering the boundary of $M^+_\gamma(\alpha,\theta';\alpha^\#)_1$; the fourth and final equation is similar, and uses $M^+_{\gamma'}(\theta,\alpha';\alpha^\#)_1$ and the relation of Remark \ref{rmk:holrel}.
\begin{figure}[t]
\centering
\includegraphics[scale=.65]{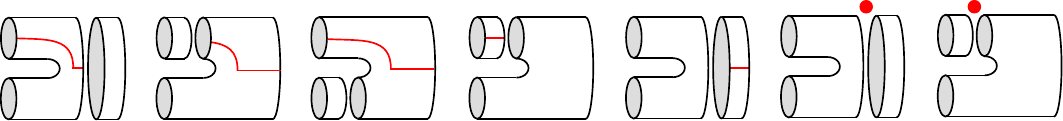}\hspace{.5cm}
\caption{}\label{fig:mu2rel}
\end{figure}
\end{proof}

\subsubsection{Definition of the map $\smash{\widetilde \lambda_{(W',S')}}$}\label{CWS-p}
We have a similar decomposition for the map $ \widetilde \lambda_{(W',S')}$:
\begin{align*}
   & \lambda': C_*^\#\to (C\otimes C')_\ast\oplus (C\otimes C')_{\ast-1}\oplus C_*\oplus C_*'  \\
  &\mu':  C_{*}^\#\to (C\otimes C')_{\ast-1}\oplus (C\otimes C')_{\ast-2}\oplus C_{*-1}\oplus C_{*-1}' \\
  &\Delta'_1: C_1^\# \to \Z \\
  &\Delta'_2:\Z \to (C\otimes C')_{-1} \oplus (C\otimes C')_{-2}\oplus C_{-1}\oplus C_{-1}' .
\end{align*}
The maps $\lambda'$, $\mu'$ and $\Delta'_2$ can be further decomposed using the decomposition of their codomains into four components, as seen above; we write
\[
  \lambda'=[\lambda'_1,\lambda'_2,\lambda'_3,\lambda'_4]^\intercal,\hspace{.75cm}\mu=[\mu'_1,\mu'_2,\mu'_3,\mu'_4]^\intercal,\hspace{.75cm}
  \Delta'_2=[\Delta'_{2,1},\Delta'_{2,2},\Delta'_{2,3},\Delta'_{2,4}]^\intercal.
\]
We proceed to define these maps. Suppose $\alpha$, $\alpha'$ and $\alpha^\#$ are all irreducible. Then define:
\[
\begin{array}{ll}
    \langle \lambda'_1(\alpha^\#), \alpha\otimes \alpha'\rangle = \# M(\alpha^\#;\alpha,\alpha')_0 \hspace{.5cm}&  \langle \Delta'_{2,1}(1),\alpha \otimes \alpha'\rangle = \# M(\theta^\#;\alpha,\alpha')_0  \\
   \langle \lambda'_2(\alpha^\#), \alpha\otimes \alpha'\rangle = \# M_{\sigma^\#}(\alpha^\#;\alpha,\alpha')_0 & \langle \Delta'_{2,2}(1),\alpha \otimes \alpha'\rangle = \# M_{\sigma^\#}(\theta^\#;\alpha,\alpha')_0 \\
   \langle \lambda'_3(\alpha^\#), \alpha\rangle = \# M(\alpha^\#;\alpha,\theta')_0  &\langle \Delta'_{2,3}(1),\alpha \rangle = \# M(\theta^\#;\alpha,\theta')_0 \\
   \langle \lambda'_4( \alpha^\#), \alpha'\rangle = \# M(\alpha^\#;\theta,\alpha')_0 &  \langle \Delta'_{2,4}(1), \alpha'\rangle = \# M(\theta^\#;\theta,\alpha')_0  \\
\end{array}
\]
We also define $ \Delta'_{1}(\alpha^\#)=\# M(\alpha^\#;\theta,\theta')_0$. Finally, we define $\mu'$ as follows:
\[
\begin{array}{l}
    \langle \mu'_1(\alpha^\#), \alpha\otimes \alpha'\rangle =\# M_{\sigma}(\alpha^\#;\alpha,\alpha')_0  \\
    \langle \mu'_2(\alpha^\#), \alpha\otimes \alpha'\rangle =\# M_{\sigma\sigma'}(\alpha^\#;\alpha,\alpha')_0   \\
   \langle \mu'_3(\alpha^\#), \alpha\rangle = \# M_{\sigma}(\alpha^\#;\alpha,\theta')_0 \\
   \langle \mu'_4( \alpha^\#), \alpha' \rangle = \# M_{\sigma'}(\alpha^\#;\theta,\alpha')_0 \\
\end{array}
\]
The proof of the following proposition is similar to the proof of Proposition \ref{CWS-C}. All the relations are obtained from Subsection \ref{CWS} by reversing the pictures from right to left.
\begin{prop}\label{CWS-p-C}
	The maps $\lambda'$, $\mu'$, $\Delta'_1$ and $\Delta'_2$ define a morphism of $\cS$-complexes.
	That is to say, they satisfy the following identities:
	\begin{align*}
		d^{\otimes} \circ \lambda'&=\lambda' \circ d^\#,\\
		\delta_1^\otimes\circ \lambda'&=\Delta_1'\circ d^{\#}+\delta_1^\#,\\
		\lambda'\circ \delta_2^\# &=\delta_2^\otimes  -d^{\otimes} \circ \Delta_2',\\
		d^{\otimes}\circ \mu'+\mu' \circ d^{\#}&=
		v^\otimes \circ \lambda'-\lambda' \circ v^\#+\delta_2^\otimes\circ \Delta_1'-\Delta_2'\circ \delta_1^\#.
	\end{align*}
\end{prop}

\subsubsection{Chain homotopies of compositions}\label{comp}

We firstly identify the composition $\widetilde \lambda_{(W,S)}\circ \widetilde \lambda_{(W',S')}$, which is a morphism of $\mathcal S$-complexes, as a morphism associated to $(W^\circ,S^\circ):=(W\circ W',S\circ S')$. Let $\rho^\#:=\gamma^\#\circ \sigma^\#$ be the closed loop embedded in the surface $S^\circ$, and similarly set $\rho = \gamma\circ \sigma$, $\rho'=\gamma'\circ \sigma'$. Let 
\[
	\widetilde \lambda_{(W^\circ, S^\circ,\rho^\#)}:\widetilde{C}_\ast^\#\to \widetilde{C}_\ast^\#
\]
be the $\cS$-morphism defined by components $\lambda^\circ$, $\mu^\circ$, $\Delta_1^\circ$ and $\Delta_2^\circ$, where
\[
\begin{array}{l}
    \langle \lambda^\circ(\alpha^\#), \beta^\# \rangle =\# M_{\rho^\#}(W^\circ,S^\circ;\alpha^\#,\beta^\#)'_0 \\
    \langle \mu^\circ(\alpha^\#), \beta^\#\rangle = \# M_{\rho^\#\rho}(W^\circ,S^\circ;\alpha^\#,\beta^\#)'_0 \\
    \langle\Delta_1^\circ(\alpha^\#), 1\rangle = \# M_{\rho^\#}(W^\circ,S^\circ;\alpha^\#,\theta^\#)'_0 \\
   \langle \Delta_2^\circ(1), \beta^\# \rangle = \# M_{\rho^\#}(W^\circ,S^\circ;\theta^\#,\beta^\#)'_0 \\
\end{array}
\]
In all of these moduli spaces, we use a slightly larger gauge group than usual; this is indicated by the primed superscripts. To say more, viewing our moduli spaces as consisting of $SO(3)$ (orbifold) adjoint connections, we mod out by not only determinant-1 gauge transformations, but all $SO(3)$ (orbifold) gauge transformations. In the case at hand, the determinant-1 gauge group is of index 2 in this larger group, as $H^1(W^\circ;\Z/2)\cong \Z/2$. The residual $\Z/2$ action is free and orientation-preserving on the determinant-1 moduli spaces  \cite[Subsection 5.1]{KM:unknot}, so for example we have
\[
	\# M_{\rho^\#}(W^\circ,S^\circ;\alpha^\#,\beta^\#)_0 = 2 \#M_{\rho^\#}(W^\circ,S^\circ;\alpha^\#,\beta^\#)'_0
\]
This modification of gauge groups is to avoid factors of $2$ in our chain relations below. 

Otherwise, our notation is just as before; for example, we have
\[
	M_{\rho^\#\rho}(W^\circ,S^\circ;\alpha^\#,\beta^\#)'_0:=\{[A]\in M(W^\circ,S^\circ;\alpha^\#,\beta^\#)'_2
\mid H^{\rho^\#}([A])=s, H^{\rho}([A])=t\}
\]
for generic fixed $s,t\in S^1$, and $H^{\rho^\#}$ and $H^\rho$ are modified holonomy maps. In pictures:
\[
\lambda^\circ = \raisebox{-.3\height}{\includegraphics[scale=.45]{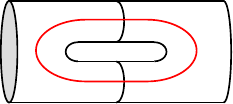}}
		\hspace{1cm}
\mu^\circ =\; \;\raisebox{-.3\height}{\includegraphics[scale=.45]{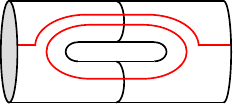}}
\]
\[
\Delta_1^\circ = \raisebox{-.3\height}{\includegraphics[scale=.45]{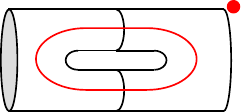}}
		\hspace{1cm}
\Delta_2^\circ = \raisebox{-.3\height}{\includegraphics[scale=.45]{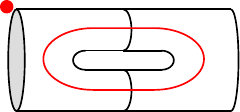}}
\]
That the map $\widetilde \lambda_{(W^\circ, S^\circ,\rho^\#)}$ just defined is a morphism of $\cS$-complexes follows, for the most part, from the usual arguments. The one essential difference is that $(W^\circ, S^\circ)$ is not a negative definite pair in the sense of Definition \ref{def:negdef}, because
$H_1(W^\circ\setminus S^\circ;\Z)$ is free abelian of rank 2. In considering its reducible traceless representations in $\mathscr{X}(W^\circ,S^\circ)$, one of these generators, upon conjugating, must go to $i\in SU(2)$, and the other is then of the form $e^{i\theta}\in SU(2)$. Thus there is not one reducible, but a circle's worth. Nonetheless, all the moduli spaces used in the definition are cut down by holonomy around the loop $\rho^\#$, and this has the effect of picking out a single reducible which is unobstructed. Similar matters are discussed in the proof of Proposition \ref{prop:morphchqi} below.

\begin{figure}[t]
\centering
\includegraphics[scale=.5]{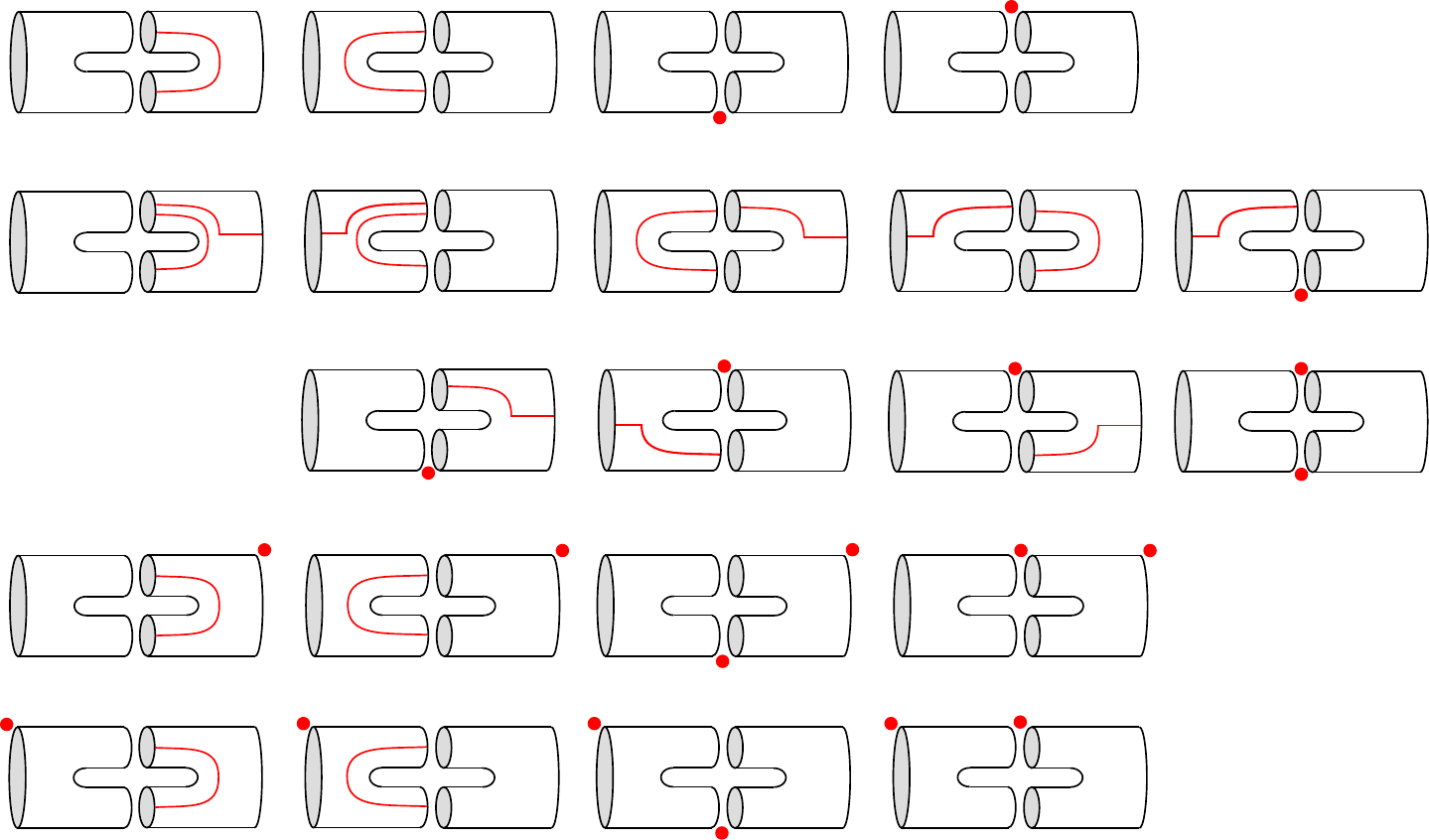}\hspace{.5cm}
\caption{}\label{fig:circrels}
\end{figure}

\begin{prop}\label{prop:comphomotopy1}
	There is an $\mathcal S$-chain homotopy equivalence between the $\mathcal S$-morphism
	$\widetilde \lambda_{(W,S)}\circ \widetilde \lambda_{(W',S')}$ and the map
	$\widetilde \lambda_{(W\circ W',S\circ S',\rho^\#)}$.
\end{prop}
\begin{proof}
	Choose a path of metrics $G$ on $(W^\circ,S^\circ)$, starting at $g_0$ and ending at a broken metric $g_\infty$, the latter of which is broken along the gluing region of the composition $(W,S)\circ (W',S')$. Define $K^\circ$, $L^\circ$, $M_1^\circ$ and $M_2^\circ$ as follows:
	\begin{align*}
		\langle K^\circ (\alpha^\#), \beta^\#\rangle &= \#\{[A]\in \bigcup_{g\in G} M^g(W^\circ,S^\circ;\alpha^\#,\beta^\#)'_0
\mid H^{\rho^\#}([A])=s\} \\
		\langle L^\circ (\alpha^\#), \beta^\#\rangle & = \#\{[A]\in \bigcup_{g\in G} M^g(W^\circ,S^\circ;\alpha^\#,\beta^\#)'_1
\mid H^{\rho^\#}([A])=s, H^{\rho}([A])=t\}\\
\langle M_1^\circ (\alpha^\#), 1\rangle &= \#\{[A]\in \bigcup_{g\in G} M^g(W^\circ,S^\circ;\alpha^\#,\theta^\#)'_0
\mid H^{\rho^\#}([A])=s\} \\
\langle M_2^\circ (1),\beta^\#\rangle &= \# \{[A]\in \bigcup_{g\in G} M^g(W^\circ,S^\circ;\theta^\#,\beta^\#)'_0
\mid H^{\rho^\#}([A])=s\} 
\end{align*}
These maps define a chain homotopy as in Definition \ref{def:shomotopy} between $\widetilde \lambda_{(W^\circ,S^\circ)}$ defined above and the map $\widetilde \lambda^\infty_{(W^\circ,S^\circ)}$ defined similarly to $\widetilde \lambda_{(W^\circ,S^\circ)}$ but using the broken metric $g_\infty$ in place of $g_0$. That is, if we write the components of $\widetilde \lambda^\infty_{(W^\circ,S^\circ)}$ as $\lambda^\infty$, $\mu^\infty$, $\Delta_1^\infty$ and $\Delta_2^\infty$, then
\begin{align*}
	d^\# K + K d^\# &= \lambda^\infty - \lambda^\circ\\
	v^\#K - d^\# L + \delta_2^\#M_1 + Ld^\# - Kv^\# + M_2 \delta_1^\# & = \mu^\infty - \mu^\circ\\
	\delta_1^\# K + M_1d^\# &= \Delta_1^\infty - \Delta_1^\circ\\
	-d^\#M_2- K\delta_2^\# & = \Delta_2^\infty - \Delta_2^\circ
\end{align*}
These relations are proved in the usual way; for the first, consider the 1-dimensional moduli space $\bigcup_{g\in G} M^g_{\rho^\#}(W^\circ,S^\circ;\alpha^\#,\beta^\#)'_0$. Counting the ends of this moduli space that contain sequences of pairs $([A_i],g_i)$ where the metrics $g_i$ converge to the interior of $G$ yield the left hand side of the equation $d^\#K + Kd^\# = \lambda^\infty - \lambda^\circ$, and ends containing sequences with $g_i\to g_0$ (resp. $g_\infty$) contribute to $\lambda^\circ$ (resp. $\lambda^\infty$) on the right side. Finally, we claim that 
\[
	 \widetilde \lambda^\infty_{(W^\circ,S^\circ)} = \widetilde \lambda_{(W,S)}\circ \widetilde \lambda_{(W',S')}
\]
This amounts to straightforward verifications of the following identities:
	\begin{align*}
	\lambda^{\infty}&=\lambda_1\lambda_1'+\lambda_2\lambda_2'+\lambda_3\lambda_3'+\lambda_4\lambda_4'\\
	\mu^{\infty}&=\mu_1\lambda_1'+\mu_2\lambda_2'+\mu_3\lambda_3'+\mu_4\lambda_4'+
				\lambda_1\mu_1'+\lambda_2\mu_2'+\lambda_3\mu_3'+\lambda_4\mu_4'+\Delta_2\Delta_1'\\
	\Delta_1^{\infty}&=\Delta_{1,1}\lambda_1'+\Delta_{1,2}\lambda_2'+
				\Delta_{1,3}\lambda_3'+\Delta_{1,4}\lambda_4'\\
	\Delta_2^{\infty} &=\lambda_1\Delta_{2,1}'+\lambda_2\Delta_{2,2}'+
	\lambda_3\Delta_{2,3}'+\lambda_4\Delta_{2,4}'
\end{align*}
The right hand sides are represented by the rows of Figure \ref{fig:circrels}. (If we had not modified our gauge groups, factors of $2$ would appear on the right hand sides, from a discrete gluing parameter, multiplying by $-1$ on one of $Y$ or $Y'$, as in \cite[Section 3.2]{bd}.) We also remark that in verifying the relation for $\mu^\infty$ we use the following identity, and its symmetries:
\[
		\raisebox{-.3\height}{\includegraphics[scale=.5]{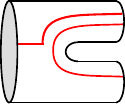}} =\raisebox{-.3\height}{\includegraphics[scale=.5]{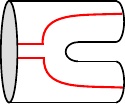}}
\]
\end{proof}
\begin{prop}\label{prop:morphchqi}
	The morphism $\widetilde \lambda_{(W^\circ,S^\circ,\rho^\#)}$
	is $\mathcal S$-chain homotopic to $ \text{\emph{id}}$.
\end{prop}

\begin{proof}
	Recall from Figure \ref{fig:connsumcomposites1} that embedded in the composite $(W^\circ,S^\circ) = (W\circ W',S\circ S')$ is a copy of $(S^1\times D^3,S^1\times D^1)$, surgery on which yields a product cobordism. Write $(W^c,S^c)$ for the closure of the complement of $(S^1\times D^3,S^1\times D^1)$, so that
	\[
		(W^\circ, S^\circ) = (W^c,S^c) \cup (S^1\times D^3,S^1\times D^1),
	\]
	where the two pieces are glued along $(S^1\times S^2,S^1\times \{2\text{ pts}\})$. The basic idea of the proof is to relate the morphism associated to $(W^\circ,S^\circ)$ to a morphism associated to the pair in which $ (S^1\times D^3,S^1\times D^1)$ is replaced by $(D^2\times S^2,D^2\times 2\text{ pts})$.
	
	Stretching the metric along a collar neighborhood of the gluing region provides a 1-parameter family of metrics, starting from our initial choice of metric, and limiting to a metric broken along $S^1\times S^2$. Along this 1-parameter family of metrics, we homotop the loops $\rho$, $\rho'$ and $\rho^\#$ and vary the constant $s\in S^1$ continuously such that when the metric is broken along $S^1\times S^2$, the loop $\rho^\#$ is contained in $S^1\times D^1$, away from the region of stretching, but that $\rho$ and $\rho'$ are in the interior of $W^c$ and $s=1$. We also arrange that the perturbation data both near the gluing region and on the component $(S^1\times D^3,S^1\times D^1)$ are zero. (These assumptions will be justified in the course of the proof.) Write $\widetilde \lambda^{+}$ for the map defined just as $\widetilde \lambda_{(W^\circ,S^\circ,\rho^\#)}$ was defined, but using the limiting broken metric; write its components as $\lambda^+, \mu^+, \Delta_1^+, \Delta_2^+$. The family of metrics determines an $\mathcal S$-chain homotopy from $\widetilde \lambda_{(W^\circ,S^\circ,\rho^\#)}$ to $\widetilde \lambda^{+}$ in the usual way.
	
We now describe $\widetilde \lambda^{+}$. First, the critical set $\fC$ of the unperturbed Chern-Simons functional on $(S^1\times S^2, S^1\times 2\text{ pts})$ may be identified via holonomy with the traceless character variety, similarly to \eqref{eq:charvar}. This latter set is identified with $S^1$ as follows. Let $\mu$ be a meridian generator for the fundamentral group of $S^2\setminus 2\text{ pts}$, and $\nu$ a generator corresponding to the $S^1$ factor. The condition $\tr \rho(\mu)=0$ implies there is some $g_\rho\in SU(2)$ such that $g_\rho\rho(\mu)g_\rho^{-1}=i$. As $\mu$ and $\nu$ commute, we have $g_\rho\rho(\nu )g_\rho^{-1} =e^{i\theta}\in S^1$. Sending $\rho$ to $g_\rho \rho(\nu) g_\rho^{-1}$ gives the bijection. In summary, we have an induced bijection $\fC \to S^1$. 

Note that the stabilizer of each point in $\fC$ is isomophic to $U(1)$. We next claim that $\fC$ is Morse--Bott non-degenerate. As $\fC$ has been identified with the smooth 1-manifold $S^1$, this amounts to showing, for each class in $\fC$, that $H^1$ of the associated deformation complex has dimension 1. This in turn is equivalent to 
\begin{equation}
	\dim H^1 ( \Sigma; \pi^\ast \check{B}^\text{ad})^{\widetilde{\tau}^\ast}=1\label{eq:dbh1}
\end{equation}
where $\pi:\Sigma\to S^1\times S^2$ is the double branched cover, $\check{B}^\text{ad}$ is the orbifold adjoint connection associated to $[B]\in \fC$, and $\widetilde{\tau}^\ast$ is the action induced by a lift $\widetilde{\tau}$ of the covering involution $\tau$ on $\Sigma$ to the adjoint bundle. Each class is reducible, and so $ \pi^\ast \check{B}^\text{ad}$ is the sum of some $U(1)$-connection $B'$ and a trivial connection. Thus $H^1 ( \Sigma; \pi^\ast \check{B}^\text{ad})$ is given by
\begin{equation}
	H^1(S^1\times S^2; B') \oplus H^1(S^1\times S^2;\R).\label{eq:dbh1s}
\end{equation}
The action of $\widetilde{\tau}^\ast$ is by $-1$ on the left factor of \eqref{eq:dbh1s} and the identity on the right factor. This implies the relation (\ref{eq:dbh1}).

Next, let $M(S^1\times D^3,S^1\times D^1)^{\text{red}}$ denote the set of reducible instantons on the pair $(S^1\times D^3,S^1\times D^1)$, with cylindrical end attached. Then we have the map 
\begin{equation}
	M(S^1\times D^3,S^1\times D^1)^{\text{red}} \to \fC \label{eq:redlimitmap}
\end{equation}
which associates to an instanton its flat limit. This is a bijection, as every flat connection on $(S^1\times S^2,S^1\times 2 \text{ pts})$ extends uniquely over $(S^1\times D^3,S^1\times D^1)$.

Note that each flat instanton in $M(S^1\times D^3,S^1\times D^1)^{\text{red}}$ is unobstructed, because the branched cover $S^1\times D^3$ is negative definite. Furthermore, each such instanton has $\dim H^1=0$ in the deformation complex, as follows from a similar computation to that of \eqref{eq:dbh1}. In particular, the index of the ASD operator for any $[A]\in M(S^1\times D^3,S^1\times D^1)^{\text{red}}$ is equal to $-1$. Moreover, any instanton on $(S^1\times D^3,S^1\times D^1)$ with index $-1$ is neccesarily in $M(S^1\times D^3,S^1\times D^1)^{\text{red}}$, because all such instantons must have the same energy, and must therefore all be flat.

Write $M(W^c,S^c;\alpha^\#,\beta^\#)_d$ for the union of $M(W^c,S^c;\alpha^\#,\gamma,\beta^\#)_{d-1}$ over all $\gamma\in \fC$, and similarly for  $M(S^1\times D^3,S^1\times D^1)^\text{irr}_d$. Then each of these is a smooth manifold whose dimension is recorded in the subscript. From the above discussion, 
\begin{equation}
	M(S^1\times D^3,S^1\times D^1)^\text{irr}_{d} = \emptyset,  \qquad d\leqslant 0 \text{   or   } d \not\equiv 0 \mod 4\label{eq:dimempty}
\end{equation} 

We now employ gluing theory in the Morse--Bott case, see e.g. \cite[Section 4.5.2]{donaldson-book}. Write $M_{\rho^\#}(S^1\times D^3,S^1\times D^1)^\text{irr}_{d}$ for the subspace of $[A]\in M(S^1\times D^3,S^1\times D^1)^\text{irr}_{d+1}$ satisfying a holonomy condition $H^{\rho^\#}([A])=1$. Consider the diagram
\begin{equation}\label{fiberprod}
	\xymatrix{
	M(W^c,S^c;\alpha^\#,\beta^\#)_d \ar[dr]_{r}&  & M_{\rho^\#}(S^1\times D^3,S^1\times D^1)^\text{irr}_{d'}
	\ar[dl]^{r'}\\
	&  \fC &}
\end{equation}
where each of $r$ and $r'$ record the limit $\gamma\in\fC$ along the cylindrical ends. Note that all instantons in $M(W^c,S^c;\alpha^\#,\beta^\#)_d$ are irreducible. Similarly, we have a diagram 
\begin{equation}\label{fiberprod22}
	\xymatrix{
	M(W^c,S^c;\alpha^\#,\beta^\#)_1 \ar[dr]_{r}&  & M_{\rho^\#}(S^1\times D^3,S^1\times D^1)^\text{red}
	\ar[dl]^{r'}\\
	&  \fC &}
\end{equation}
where $r'$ is a restriction of the map \eqref{eq:redlimitmap}. The Morse--Bott gluing theory tells us that for our limiting broken metric, the moduli space $M_{\rho^\#}(W^\circ,S^\circ;\alpha^\#,\beta^\#)_0$ may be identified with union of the fiber products of \eqref{fiberprod} for $d+d'=1$, along with the fiber product \eqref{fiberprod22}. However, \eqref{eq:dimempty} implies that \eqref{fiberprod} is empty for any pair $(d,d')$ with $d+d'=1$, so we may restrict our attention to \eqref{fiberprod22}.

The holonomy constraint $H^{\rho^\#}([A])=1$ picks out exactly two points in the domain of the bijection \eqref{eq:redlimitmap}. This is because $H^{\rho^\#}$ is defined by first taking the adjoint connection, which has the effect of squaring the holonomy in $S^1$. We conclude that the map
\[
	r':M_{\rho^\#}(S^1\times D^3,S^1\times D^1)^\text{red}\to \fC
\]
is an embedding of two points into $\fC$ with image being two elements $\theta_\pm$ of $\fC$ that have holonomies $\pm 1$ along the $S^1$-factor. By picking appropriate metric and perturbation on the interior of $(W^c,S^c)$, we may assume that $r$ is transverse to $\theta_\pm \in\fC$.
Thus 
\begin{align*}
	\# M^g_{\rho^\#}(W^\circ,S^\circ;\alpha^\#,\beta^\#)_0 &= \# M(W^c,S^c;\alpha^\#,\theta_+,\beta^\#)_0+\# M(W^c,S^c;\alpha^\#,\theta_-,\beta^\#)_0\\
	&=2\# M(W^c,S^c;\alpha^\#,\theta_+,\beta^\#)_0
\end{align*}
where $g$ is the metric broken along $S^1\times S^2$. The second equality holds because there is again an element of the $SO(3)$ gauge group which maps $M(W^c,S^c;\alpha^\#,\theta_+,\beta^\#)_0$ into $M(W^c,S^c;\alpha^\#,\theta_-,\beta^\#)_0$ in an orientation preserving way. As the map $\lambda^+$ is defined using the moduli space $M^g_{\rho^\#}(W^\circ,S^\circ;\alpha^\#,\beta^\#)'_0$, a quotient of $M^g_{\rho^\#}(W^\circ,S^\circ;\alpha^\#,\beta^\#)_0$ by a free involution, the factor of $2$ is absorbed and we have the identity
\[
	\langle \lambda^+(\alpha^\#),\beta^\#\rangle = \# M(W^c,S^c;\alpha^\#,\theta_+,\beta^\#)_0.
\]
The other components may be described similarly:
\begin{align*}
	\langle \mu^+(\alpha^\#), \beta^\#\rangle &= \# M_\rho(W^c,S^c;\alpha^\#,\theta_+,\beta^\#)_0 \\
	\Delta_1^+(\alpha^\#) &=  \# M(W^c,S^c;\alpha^\#,\theta_+,\theta^\#)_0\\
	  \langle \Delta_2^+(1),\beta^\#\rangle &=  \# M(W^c,S^c;\theta^\#,\theta_+,\beta^\#)_0
\end{align*}
Recall from Figure \ref{fig:connsumcomposites1} that replacing $(S^1\times D^3,S^1\times D^1)$ with $(D^2\times S^2, D^2\times 2\text{ pts})$ results in the product cobordism $[0,1]\times (Y\# Y', K\# K')$. By attaching $(D^2\times S^2, D^2\times 2\text{ pts})$ metrically, with a cylindrical end, this product cobordism inherits a broken metric. Then the associated moduli space $M([0,1]\times (Y\# Y', K\# K');\alpha^\#,\beta^\#)_0$ may be identified with union of the following fiber product:
\begin{equation*}\label{fiberprod2}
	\xymatrix{
	M(W^c,S^c;\alpha^\#,\beta^\#)_1 \ar[dr]_{r}&  & M(D^2\times S^2,D^2\times 2\text{ pts})^\text{red}
	\ar[dl]^{r'}\\
	&  \fC &}
\end{equation*}
We have ruled out the possibility of fiber products involving $M(D^2\times S^2,D^2\times 2\text{ pts})^\text{irr}_d$ as in the previous case for $S^1\times D^3$. The moduli space $M(D^2\times S^2,D^2\times 2\text{ pts})^\text{red}$ consists of one point, the unique flat connection which extends $\theta_+\in \fC$. Note that this connection is unobstructed because the branched cover, also identified with $D^2\times S^2$, is negative definite, and it has $\dim H^1=0$ because $b_1(D^2\times S^2)=0$. Thus 
\[
	\langle \lambda^+(\alpha^\#),\beta^\#\rangle = \# M([0,1]\times (Y\# Y', K\# K');\alpha^\#,\beta^\#)_0
\]
and similarly for $\mu^+$, $\Delta_1^+$ and $\Delta_2^+$. In summary, we may write
\begin{equation}
	\widetilde \lambda^+ = \widetilde \lambda^\infty_{[0,1]\times (Y\# Y' ,K\# K')}\label{eq:plusmap}
\end{equation}
where the map on the right side is the usual $\cS$-morphism associated to a cobordism, with the understanding that the auxiliary data involves a metric broken along $S^1\times S^2$. Finally, using an $\cS$-chain homotopy induced by the family of metrics which starts at this broken metric and ``unstretches'' to the product metric, we obtain an $\cS$-chain homotopy from the right side of \eqref{eq:plusmap} to the $\cS$-morphism defined using the product metric, which is the identity.
\end{proof}

\begin{remark}\label{rmk:cfdonaldsonglue}
	Consider a negative definite pair $(W,S):(Y,K)\to (Y',K')$ as in Subsection \ref{subsec:closedloops}, and an embedded 2-sphere $F$ with $F\cdot F=0$, which intersects $S$ transversally in 2 points, such that $F\cdot S=0$. A neighborhood of $F$ is diffeomorphic to the pair $(D^2\times S^2,D^2\times 2\text{ pts})$, and we can cut this out and reglue a copy of $(S^1\times D^3,S^1\times 2\text{ pts})$ to obtain a pair $(W',S')$. Let $\gamma$ denote the closed loop which is the core of $S^1\times D^3 \subset W'$. From Subsection \ref{subsec:closedloops} we have a holonomy induced map $\mu_{(W',S',\gamma)}:C(Y,K)\to C(Y',K')$. To the original cobordism $(W,S)$ we consider the usual cobordism-induced map $\lambda_{(W,S)}:C(Y,K)\to C(Y',K')$, which counts isolated instantons. Then the argument in the proof of Proposition \ref{prop:morphchqi} shows:
\end{remark}
\begin{prop}\label{prop:chainhomotopyexcision}
		The maps $\mu_{(W',S',\gamma)}$ and $2\lambda_{(W,S)}$ are chain homotopic.
\end{prop}
\noindent This is a singular and relative analogue of \cite[Theorem 7.16]{donaldson-book}. To remove the factor of $2$, we can work with a slightly larger gauge group when defining $\mu_{(W',S',\gamma)}$, as done in the proof of Proposition \ref{prop:morphchqi}. $\diamd$\\

Next, we analyze the reverse composition.

\begin{prop}\label{prop:lasthomotopy}
	$\widetilde \lambda_{(W',S')} \circ \widetilde \lambda_{(W,S)}$
	is $\cS$-chain homotopic to an isomorphism.
\end{prop}

\begin{proof}
Set $(W^\text{I},S^\text{I}):=(W',S')\circ (W,S)$. We consider an $\cS$-chain homotopy $H^\text{I}$ such that
\begin{equation}\label{chain-htpy}
	\widetilde{d}^\otimes H^\text{I} + H^\text{I} \widetilde{d}^\otimes = \widetilde \lambda_{(W',S')} \circ \widetilde \lambda_{(W,S)} - \widetilde \lambda_{(W^\text{I},S^\text{I})}^\infty
\end{equation}
The moduli spaces for the map $H^\text{I}$ are defined similarly to the chain homotopy used in the proof of Proposition \ref{prop:comphomotopy1}, but this time using a path of metrics $G$ that starts at the broken metric $g_0$ for the composite $(W^\text{I},S^\text{I}):=(W',S')\circ (W,S)$, broken along $(Y\# Y', K\# K')$, and ending at the metric $g_\infty$ broken along the $(S^3,S^1)$ from Figure \ref{fig:connsumcomposites2}. The components of $H^\text{I}$ are defined in a straightforward manner, by looking at the shape of the corresponding component in $\widetilde \lambda_{(W',S')} \circ \widetilde \lambda_{(W,S)}$, and defining the component of $H^\text{I}$ using the same kind of moduli space but incorporating the metric family $G$. For example, if we write
\[
	H^\text{I} = \left[\begin{array}{ccc}
				K^\text{I} & 0 & 0 \\
				L^\text{I} & - K^\text{I} & M^\text{I}_2 \\
				M^\text{I}_1 & 0 &0
			\end{array}\right]
\]
then $K^\text{I}$ is a $4\times 4$ matrix, with entries $K^\text{I}_{ij}$. Now, the $(1,1)$-entry of $\lambda\lambda'$ is equal to $\lambda_1\lambda_1'$, which is defined by counting instantons on $(W^\text{I},S^\text{I})$ with metric $g_0$ and constrained holonomy along $\rho^\#$. To define $K_{11}^\text{I}$ we use $G$ instead of the single metric $g_0$:
\[
	\langle K^\text{I}_{11}(\alpha\otimes \alpha'),\beta\otimes\beta'\rangle = \#\{ [A]\in \bigcup_{g\in G} M^g(W^\text{I},S^\text{I}; \alpha,\alpha'; \beta,\beta')_0 \mid H^{\gamma^\#}([A])=s \}
\]
The other components of $H^\text{I}$ are defined similarly. Although tedious, checking that $H^\text{I}$ is indeed an $\cS$-chain homotopy as in \eqref{chain-htpy} is straightforward and analogous to previous computations.

\begin{figure}[t]
\centering
\includegraphics[scale=.32]{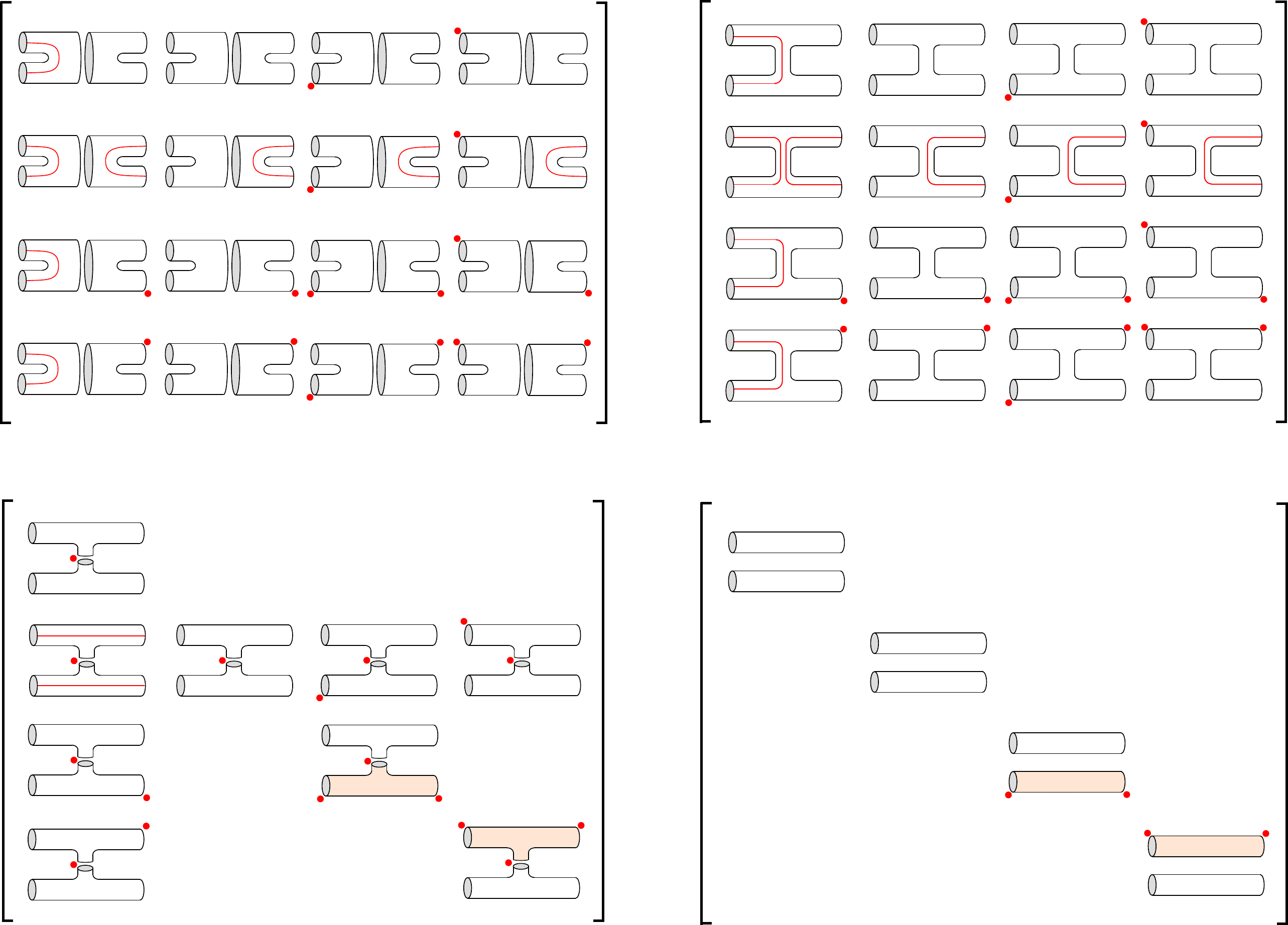}
\caption{A homotopy from $\lambda\lambda'$ to the identity is depicted in stages, from left to right, top to bottom. The top left matrix represents $\lambda\lambda'$, the top right matrix a map halfway through the chain homotopy of $K^\text{I}$, the bottom left is the map $\lambda^\text{I}$, and the bottom right represents the identity map. (Our depiction of a cylinder here, for a cobordism map, should not be confused with our prior use of a cylinder representing the boundary map $d$.)}\label{fig:lambdalambdaprime}
\end{figure}

We claim that the map $\widetilde \lambda_{(W^\text{I},S^\text{I})}^\infty$ is $\cS$-chain homotopic to an isomorphism. Write its $\cS$-morphism components as $\lambda^\text{I}$, $\mu^\text{I}$, $\Delta_1^\text{I}$ and $\Delta_2^\text{I}$. Then $\lambda^\text{I}$ is depicted in the bottom left matrix of Figure \ref{fig:lambdalambdaprime}. To see this, it is convenient to also have in mind the map obtained halfway through the homotopy from $\lambda\lambda'$ to $\lambda^{\text{I}}$, depicted in the top right matrix of Figure \ref{fig:lambdalambdaprime}. The vanishing entries of $\lambda^{\text{I}}$ in Figure \ref{fig:lambdalambdaprime} are instances of standard vanishing theorems. For example, consider the component $\lambda^{\text{I}}_{13}$. By gluing theory, the instantons under consideration correspond to pairs of instantons $[A],[A']$ and a gluing parameter in $S^1$. Here $[A]$ and $[A']$ are connections on the punctured cylinders $\R\times Y$ and $\R\times Y'$, respectively, where $A$ has limits $\alpha$ and $\beta$ at the ends of the cylinder and the reducible $\theta_0$ at the puncture, while $A'$ has limits $\alpha'$ and $\beta'$ along its cylinder, and $\theta_0$ at the puncture. As $\lambda^{\text{I}}_{13}$ counts index 0 instantons, the relevant moduli spaces containing $[A]$ and $[A']$ are empty. The vanishing of the other entries is argued similarly. 

An argument similar to that in the proof of Proposition \ref{prop:morphchqi} shows that $\lambda^\text{I}$ is equal to the map defined using moduli spaces on the unpunctured cylinder $\R \times (Y\sqcup Y')$ using a metric $g$ which is broken along chosen 3-spheres (surrounding the prior punctures) in each of the two cylinders. Using a family of metrics $G^\text{I}$ from $g$ to a translation-invariant metric, we obtain a homotopy from $\lambda^\text{I}$ to the identity. Indeed, the off-diagonal terms in $\lambda^{\text{I}}$, such as 
\[
	\includegraphics[scale=.5]{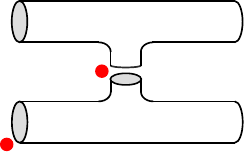}
\]
go under this chain homotopy to zero, because there are no translation-invariant instantons of index $0$ from a reducible to an irreducible; and the diagonal terms in $\lambda^\text{I}$ go to identity maps, because translation-invariant instantons of index $0$ with irreducible limits are (perturbed) flat, and induce the identity maps on the chain level.

In fact, the 1-parameter family of metrics $G^\text{I}$ induces in the usual manner an $\cS$-chain homotopy between $\widetilde \lambda_{(W^\text{I},S^\text{I})}^\infty$ and an $\cS$-morphism which from the previous paragraph has $\lambda$-component equal to the identity. The proof is completed by Lemma \ref{lemma:smoriso} below.
\end{proof}

\begin{lemma}\label{lemma:smoriso}
	If $\widetilde{\lambda}: (\widetilde C_\ast, \widetilde d)\to (\widetilde C'_\ast, \widetilde d')$ is a morphism of $\cS$-complexes with its component $\lambda:C_\ast \to C_\ast'$ an isomorphism, then in fact $\widetilde{\lambda}$ is an isomorphism of $\cS$-complexes.
\end{lemma}

\begin{proof}
	After identifying $C_\ast$ and $C'_\ast$ using $\lambda$, we are reduced to the case in which $\widetilde{\lambda}$ is a morphism from $(\widetilde C_\ast, \widetilde d)$ to itself, and $\widetilde{\lambda}$ has decomposition given in \eqref{eq:tilde-map} where $\lambda=1$ is the identity. The inverse of $\widetilde{\lambda}$ in this case is given by
	\begin{equation*}
\left[
					\begin{array}{ccc}
						1 &0&0\\
						-\mu+\Delta_2\Delta_1 & 1 & -\Delta_2\\
						-\Delta_1&0&1\\
					\end{array}
					\right]
\end{equation*}
which is of course a morphism of $\cS$-complexes. Alternatively, a morphism of $\cS$-complexes $\widetilde C_\ast\to \widetilde C'_\ast$, after reordering the summands, is an upper triangular matrix whose diagonal entries are $\lambda, 1, \lambda$, and the inverse is of the same form.
\end{proof}

We have established that the compositions
\[
	\widetilde \lambda_{(W,S)}\circ \widetilde \lambda_{(W',S')},   \qquad \widetilde \lambda_{(W',S')}\circ \widetilde \lambda_{(W,S)}
\]
are $\cS$-chain homotopic to the identity and an isomorphism, respectively. It follows formally that the second composition is in fact $\cS$-chain homotopic to the identity. This completes the proof of the chain homotopy equivalence \eqref{eq:connsumforknots}.

\subsubsection{Naturality of the equivalences}\label{subsec:natcon}

Finally, we discuss the naturality of \eqref{eq:connsumforknots} with respect to split cobordisms. Let us consider negative definite pairs $(X,F):(Y_1,K_1)\to (Y_2,K_2)$ and $(X',F'):(Y'_1,K_1')\to (Y_2',K_2')$. As usual, our knots are based, and in the cobordisms $(X,F)$ and $(X',F')$ we choose neatly embedded arcs whose endpoints are the distinguished basepoints. Using these arcs, we may form the glued cobordism $(X\# X', F\# F'):(Y_1\# Y_1',K_1\# K_1')\to (Y_2\# Y_2',K_2\# K_2')$ in a standard manner. This is what is meant by a {\emph{split}} cobordism. We consider the square:
	\begin{equation}\label{diagram:nat}
		\xymatrix{
		\widetilde C_*(Y_1,K_1)\otimes \widetilde C_*(Y_1',K_1') \ar[d]_{\widetilde \lambda_{(X,F)}\otimes \widetilde \lambda_{(X',F')}}\ar[rr]^{\widetilde \lambda_{(W_1,S_1)}}&&\widetilde C(Y_1\# Y_1',K_1\# K_1')\ar[d]^{\widetilde \lambda_{(X\# X',F\# F')}}\\
		\widetilde  C_*(Y_2,K_2)\otimes \widetilde C_*(Y_2',K_2') \ar[rr]_{\widetilde \lambda_{(W_2,S_2)}} &&\widetilde C_*(Y_2\# Y_2',K_2\# K_2')}
	\end{equation}	
The vertical maps are the usual morphisms of $\cS$-complexes we associate to given cobordisms, while the horizontal maps are the maps as constructed in Subsection \ref{CWS}. We have implicitly chosen metrics and perturbations for each of the cobordisms. The cobordisms
\[
	(X\# X', F\# F') \circ (W_1,S_1), \qquad  (W_2,S_2)\circ \left((X, F) \sqcup (X', F')\right)
\]
are topologically equivalent, an elementary fact which is left to the reader. Thus the two different compositions obtained in the above square may be viewed as induced by the same cobordism, but with different (broken) metrics and perturbation choices. Then, choosing a 1-parameter family of auxiliary data interpolating these choices gives rise, in the usual manner, to an $\cS$-chain homotopy between the two morphisms
\[
\widetilde \lambda_{(X\# X', F\# F')} \circ \widetilde \lambda_{(W_1,S_1)}, \qquad  \widetilde \lambda_{(W_2,S_2)}\circ \left(\widetilde \lambda_{(X, F)}\otimes \widetilde \lambda_{(X', F')}\right)
\]
This establishes the commutativity of \eqref{diagram:nat} up to $\cS$-chain homotopy. There is a similar square with horizonal arrows reversed, using the morphisms defined in Subsection \ref{CWS-p}, and the same statements hold.

\newpage

%!TEX root = main.tex

\section{Local coefficient systems and filtrations}\label{sec:loccoeffs}

In this section we describe how to generalize our previous constructions to the case of local coefficients. Our general local coefficient system is a hybrid of the one defined by Kronheimer and Mrowka in \cite{KM:YAFT}, which roughly measures the holonomy of connections along the longitudal direction of the knot, and one the defined by the Chern-Simons functional. The latter ingredient has more structure, inherited from the fact that the Chern-Simons functional is (almost) non-increasing for (perturbed) instantons. In this context, we carry over constructions from the non-singular setting given in \cite{AD:CS-Th}.

\subsection{Local coefficients}\label{subsec:local}

Let $(Y,K)$ be a based knot in an integer homology 3-sphere. The most general local coefficient system on $\mathcal B(Y,K)$ that we consider is defined over the ring 
\[
	\mathscr{R}:=\Z[U^{\pm1},T^{\pm 1}].
\]
To an element of $\mathcal B(Y,K)$ represented by the connection $B$, we associate the $\mathscr R$-module
\[
  \Delta_{[B]}:=\Z[U^{\pm1},T^{\pm 1}]\cdot U^{{\rm CS}(B)} T^{{\rm hol}_K(B)}
\]
where ${\rm hol}_K(B)\in \R/\Z$ is defined analogously to \eqref{hol-map-4d}, and roughly gives the holonomy of the $S^1$-connection induced by $B^{\text{ad}}$ along $K$. To define ${\rm hol}_K(B)$, we must choose a framing of our knot to fix our procedure for taking the holonomy. For more details see \cite[Section 3.9]{KM:YAFT}. A knot in an integral homology 3-sphere has a canonical framing induced by a Seifert surface, and we always use this framing.

 The holonomy ${\rm hol}_K(B)$ is related to the monopole number $\nu(A)$ in the same way that the Chern-Simons functional is related to the topological energy $\kappa(A)$. In particular, we have an analogue of relation \eqref{eq:csmodz}: if $A$ is a connection on a cobordism of pairs $(W,S):(Y,K)\to (Y',K')$ then we have the relation
\begin{equation}
	\nu(A) \equiv  {\rm hol}_{K'}(B') - {\rm hol}_K(B) \pmod \Z \label{eq:numodz}
\end{equation}
where $B$ and $B'$ are the limiting connections on $(Y,K)$ and $(Y',K')$, respectively.

Let $\gamma:[-1,1]\to \mathcal B(Y,K)$ be a path from $\alpha_1$ to $\alpha_2$. Let $A$ be a singular connection on $\R\times Y$ representing $\gamma$, and representing the pull-backs of $\alpha_1$ and $\alpha_2$ for $t<-1$ and $t>1$, respectively. Then we define $\Delta_{\gamma}:\Delta_{\alpha_1}\to \Delta_{\alpha_2}$ to be multiplication by $U^{-2\kappa(A)}T^{\nu(A)}$, which is well-defined by relations \eqref{eq:csmodz} and \eqref{eq:numodz}. Monomials in $\Delta_{\alpha}$ can be identified with homotopy classes of paths $\gamma$ from $\alpha$ to the reducible class $\theta$. The action and monopole number of the path $\gamma$ determines a pair of real numbers $(\widetilde {\rm CS}(\gamma), \widetilde {\rm hol}_K(\gamma))$. We use this pair to define two (real valued) gradings on $\Delta_{\alpha}$, which are respectively called {\it instanton} and {\it monopole} gradings. Moreover, the ASD index of the path $\gamma$ can be used to define a $\Z$-grading $\widetilde {\rm gr}$ on $\Delta_\alpha$ which is an integer lift of ${\rm gr}(\alpha)$. This grading satisfies $\widetilde{\rm gr}(1)=0$ for $1\in \Delta_\theta$. Further, multiplication by $U^{\pm 1}$ changes the $\Z$-grading by $\pm 4$ and it is fixed by multiplication by $T^{\pm 1}$.

Suppose $K$ is a based knot in an integer homology sphere $Y$. Let $(C_*(Y,K;\Delta),d)$ be the $\Z$-graded chain complex over $\mathscr{R}$ defined as follows:
\begin{equation}\label{eq:irrdiff-loc}
C_*(Y,K;\Delta) = \bigoplus_{\alpha\in\fC_\pi^\text{irr}} \Delta_\alpha, \hspace{1cm}
	d(\alpha_1) = \sum_{\substack{\alpha_2\in\fC_\pi^\text{irr}\\ [A] \in \breve{M}(\alpha_1,\alpha_2)_0}} 
	\Delta(A)( \alpha_2)
\end{equation}
Here $\Delta(A)=\pm \Delta_\gamma$ such that $\gamma$ is the path in $\mathcal B(Y,K)$ from $\alpha_1$ to $\alpha_2$ determined by $A$ and the sign is given by the orientation of the moduli space $\breve{M}(\alpha_1,\alpha_2)_0$. A discussion similar to the one in Subsection \ref{sing-Floer} shows that $(C_*(Y,K;\Delta),d)$ is indeed a chain complex over $\mathscr R$, where $d$ has degree $-1$.

Suppose $(W,S):(Y,K)\to (Y',K')$ is a negative definite pair in the sense of Definition \ref{def:negdef}. Choose, as usual, a path from the basepoint of $K$ to the basepoint of $K'$. By a slight modification of \eqref{eq:irrdiff-loc}, we may define a cobordism map
\begin{equation}\label{chain-map-sing-flo}
	\lambda_{(W,S;\Delta)}(\alpha) = \sum_{\substack{\alpha'\in \fC_{\pi'}^\text{irr} \\ [A]\in M(W,S;\alpha,\alpha')_0}} \Delta(A)\cdot \alpha'.
\end{equation}
Here $\Delta(A):=\pm U^{-2\kappa(A)} T^{\nu(A)}$ where the sign is determined as usual by the orientation of the moduli space $M(W,S;\alpha,\alpha')_0$.

We may continue in this fashion, and adapt all of the constructions in Section \ref{sec:tilde} to the setting of local coefficients. We obtain an $\cS$-complex
\begin{equation}
	(\widetilde C_\ast(Y,K;\Delta),\widetilde d) \label{eq:defofscomploccoeff}
\end{equation}
over the ring $\mathscr R$. To a negative definite pair $(W,S)$ as above, we may associate a morphism $\widetilde \lambda_{(W,S;\Delta)}$ of $\cS$-complexes, and so forth. The same arguments as before show that the $\cS$-chain homotopy type of the $\cS$-complex $(\widetilde C(Y,K;\Delta),\widetilde d)$ is independent of the choice of the orbifold metric on $Y$ and the perturbation of the Chern-Simons functional. 

We may recover the $\cS$-complex $(\widetilde C_*(Y,K),\widetilde d)$ over $\Z$ from \eqref{eq:defofscomploccoeff} using the change of basis that evaluates $U$ and $T$ at $1$. Given an $\mathscr{R}$-algebra $\mathscr{S}$, we obtain an $\cS$-complex over the ring $\mathscr{S}$ by performing a change of basis on our local coefficient system:
\[
	(\widetilde C_*(Y,K;\Delta_\mathscr{S}),\widetilde d), \qquad \Delta_\mathscr{S} := \Delta \otimes_\mathscr{R} \mathscr{S}
\]
In general, this $\cS$-complex is only $\Z/4$-graded. However, it becomes a $\Z$-graded algebra in the obvious way if $\mathscr{S}$ is a $\Z$-graded $\mathscr{R}$-algebra, which means that multiplication by $U^{\pm 1}$ changes the grading by $\pm1$ and multiplication by $T^{\pm1}$ does not change the grading.

\begin{remark} \label{rem:tring}
If $\mathscr{T}=\Z[T^{\pm 1}]$ is the $\mathscr{R}$-algebra obtained by setting $U=1$, then we obtain an $\cS$-complex over $\mathscr{T}$ denoted $(\widetilde C(Y,K;\Delta_\mathscr{T}), \widetilde d)$. The system $\Delta_\mathscr{T}$ is essentially the local coefficient system considered in \cite[Section 3.9]{KM:YAFT}. $\diamd$
\end{remark}

\begin{remark}\label{cS-non-sing-loc-coef}
	If $\mathscr{U}=\Z[U^{\pm 1} ]$ is the $\mathscr{R}$-algebra obtained by setting $T=1$, then we obtain an $\cS$-complex over $\mathscr{U}$ 
	denoted $(\widetilde C(Y,K;\Delta_\mathscr{U}), \widetilde d)$. 
	There is an analogous construction in the non-singular case, where one can define an $\cS$-complex 
	$(\widetilde C_*(Y;\Delta),\widetilde d)$, such that the local coefficient system $\Delta$ is over the ring $\Q[U^{\pm 1}]$. The definition of $\Delta$ follows a similar pattern to $\Delta_\mathscr{U}$, where we use the action $\kappa(A)$ (for a non-singular connection) to define a homomorphism associated to a path of connections. $\diamd$
\end{remark}

\begin{remark}
	Strictly speaking the $\Z$-graded complex $(\widetilde C_*(Y,K;\Delta),\widetilde d)$ does not fit into the definition of  $\cS$-complexes that we have been using so far 
	because the coefficient ring 
	$\mathscr R$ is graded. However, one can
	modify the definition to include such $\cS$-complexes. On its face, it might seem that $\widetilde C_*(Y,K;\Delta)$
	does not give any extra information in comparison to $\widetilde C(Y,K;\Delta_\mathscr{T})$ of 
	Remark \ref{rem:tring} because $\widetilde C_*(Y,K;\Delta)$ is the periodic $\Z$-graded complex obtained by unrolling  the $\Z/4\Z$-graded complex $\widetilde C(Y,K;\Delta_\mathscr{T})$.
	However, in the remaining part of this section we shall show that $\widetilde C_*(Y,K;\Delta)$ can be equipped with an ``almost-filtration'' which gives rise to additional information. 
	$\diamd$
\end{remark}

The machinery of Section \ref{sec:equivtheories} applies to this setting: for any $\sR$-algebra $\sS$ we obtain three equivariant singular instanton homology groups, denoted
\begin{equation*}
	\hrI_*(Y,K;\Delta_\sS), \qquad \crI_*(Y,K;\Delta_\sS), \qquad \brI_*(Y,K;\Delta_\sS).
\end{equation*}
These are $\Z$-graded $\sS[x]$-modules, and the discussion of Subsection \ref{sec:eq-I-i} carries over to this setting in a straightforward manner. For each $\sR$-algebra $\sS$ we have a Fr\o yshov invariant
\begin{equation}
	h_\sS(Y,K)\in \Z \label{eq:froyshovinvloccoeff}
\end{equation}
by taking the algebraic Fr\o yshov invariant of the $\cS$-complex $(\widetilde C_\ast(Y,K;\Delta_\sS),\widetilde d)$ as given by Definition \ref{h-def} or equivalently Proposition \ref{h-g-reinterpret}. Further, we also have ideals
\begin{equation}
	J_h^\sS(Y,K) \subset J_{h-1}^\sS(Y,K) \subset \cdots \subset \sS \label{eq:locjideals}
\end{equation}
by applying Definition \ref{j-def} to the $\cS$-complex $(\widetilde C_\ast(Y,K;\Delta_\sS),\widetilde d)$. Here $h=h_\sS(Y,K)$. Theorems \ref{thm:hinvproperties} and \ref{thm:jideals} of the introduction about the properties of these invariants follow from our discussions in Section \ref{sec:equivtheories}.

The connected sum theorem also generalizes to the setting of local coefficients.

\begin{theorem}\label{thm:connectedsumlocceoff} Let $(Y,K)$ and $(Y',K')$ be based knots in integer homology 3-spheres. There is a chain homotopy equivalence of $\Z$-graded $\cS$-complexes over $\mathscr{R}=\Z[U^{\pm 1}, T^{\pm 1}]$:
\begin{equation*}
	\widetilde C(Y\# Y',K\# K';\Delta) \simeq  \widetilde C(Y,K;\Delta)\otimes_\mathscr{R} \widetilde C(Y',K';\Delta) 
\end{equation*}
This equivalence is natural, up to $\cS$-chain homotopy, with respect to split cobordisms.
\end{theorem}

The proof is for the most part the same as that of Theorem \ref{thm:connectedsum}. In particular, all maps defined in the proof are modified to follow the same pattern as in \eqref{chain-map-sing-flo}, where we now keep track of the terms $\kappa(A)$ and $\nu(A)$ for each instanton in the exponents of our formal variables. The only part of the proof that requires additional commentary is Proposition \ref{prop:morphchqi}. The key observation is that the instantons in the proof that appear in the moduli spaces
\begin{equation}
	M_{\rho^\#}(S^1\times D^3, S^1\times D^1)^{\rm red} , \qquad M(D^2 \times S^2, D^2 \times 2 \text{ pts})^{\rm red} \label{eq:connsuminstantons}
\end{equation}
have $\kappa(A)=\nu(A)=0$, which follows because all of these instantons are flat. (In fact, the vanishing of $\kappa(A)$ and $\nu(A)$ here are not essential for the proof of Theorem \ref{thm:connectedsumlocceoff}, but $\kappa(A)=0$ will play an important role below.) This shows, in particular, that Proposition \ref{prop:chainhomotopyexcision} holds with local coefficients, and that the proof of Section \ref{sec:consum} carries through to prove Theorem \ref{thm:connectedsumlocceoff}. As a consquence, we have:

\begin{theorem}
	Let $\mathscr{S}$ be an $\mathscr{R}$-algebra. The assignment $(Y,K)\mapsto (\widetilde C_\ast(Y,K;\Delta_\mathscr{S}),\widetilde d,\chi)$ induces a homomorphism $\Theta_\Z^{3,1}\to \Theta_{\mathscr{S},\Z/4}^\cS$ 
	of partially ordered abelian groups. 
\end{theorem}

In particular, the Fr\o yshov invariant \eqref{eq:froyshovinvloccoeff} descends to a homomorphism $\Theta_\Z^{3,1} \to \Z$ of partially ordered abelian groups, and satisfies the analogue of Theorem \ref{thm:hcheckprops}.

\subsection{The Chern-Simons filtration}

The topological energy $\kappa(A)$ of the elements of our moduli spaces satisfies a positivity property which gives rise to more structure on $(\widetilde C(Y,K;\Delta),\widetilde d)$. This idea, applied to the complex $(\widetilde C_*(Y;\Delta),\widetilde d)$ of Remark \ref{cS-non-sing-loc-coef}, was used in \cite{AD:CS-Th} to produce invariants of the homology cobordism group of integral homology 3-spheres. (More precisely, the complex $(\widetilde C_*(Y;\Delta),\widetilde d)$ there is obtained by the change of basis associated to the inclusion of $\Q[U^{\pm 1}]$ into a Novikov ring.) These constructions can be adapted to the present set up to produce concordance invariants.

If $A$ is a non-flat ASD connection on a 4-manifold, which determines a path in $\cB(Y,K)$ from $\alpha$ to $\alpha'$, then the Chern-Weil integral defining the topological energy is positive:
\begin{equation*}
	\kappa(A)=\frac{1}{8\pi^2}\int_{W^+\setminus S^+}\tr(F_A\wedge F_A) = \frac{1}{8\pi^2}\int_{W^+\setminus S^+} |F_A|^2 > 0.
\end{equation*}
In particular, if the perturbation of the Chern-Simons functional is trivial in the definition of the complex $\widetilde C_*(Y,K;\Delta)$, the differential $\widetilde d$ strictly decreases the instanton grading. This structure may be formalized as follows. 

\begin{definition}
	An {\it I-graded $\cS$-complex (of level $\delta$)} over $R[U^{\pm 1}]$ is an $\cS$-complex $(\widetilde C, \widetilde d,\chi)$ over $R[U^{\pm 1}]$ with a $\Z\times \R$-bigrading as an $R$-module, which satisfies the following properties. Writing $\widetilde C_{i,j}$ for the $(i,j)\in \Z\times \R$-graded summand, we have:
	\begin{itemize}
		\item[(i)] $U\; \widetilde C_{i,j} \subset \widetilde C_{i+4,j+1}$
		\item[(ii)] $\widetilde d\; \widetilde C_{i,j} \subset \bigcup_{k< j+\delta} \widetilde C_{i-1,k}$
		\item[(iii)] $\chi\; \widetilde C_{i,j} \subset  \widetilde C_{i+1,j}$
	\end{itemize}
	Further, $\widetilde C$ is freely, finitely generated as an $R[U^{\pm 1}]$-module by homogeneously bigraded elements. The distinguished summand $R[U^{\pm 1}]\subset \widetilde C$ has $1\in R[U^{\pm 1}]$ in bigrading $(0,0)$. 
	We denote the integer and the 
	real gradings on $\widetilde C_{i,j}$ by $\widetilde {\rm gr}$ and $\deg_I$. $\diamd$ 
\end{definition}

Note that an I-graded $\cS$-complex of level $\delta$ is also one of level $\delta'$ for all $\delta'>\delta$. Concretely, an I-graded $\cS$-complex over $R[U^{\pm 1}]$ has underlying chain complex of the form:
\begin{equation}
	\widetilde C = \left(\bigoplus_{k=1}^n R[U^{\pm 1}]\gamma_k\right)\oplus \left(\bigoplus_{k=1}^n R[U^{\pm 1}]\underline\gamma_k\right)\oplus R[U^{\pm}]\gamma_0 \label{eq:igrscx}
\end{equation}
This gives the $\cS$-complex decomposition $\widetilde  C_\ast = C_\ast \oplus C_{\ast-1} \oplus R[U^{\pm 1}]$, where the indicated subscript gradings are the $\Z$-gradings. In particular, $\gamma_0$ generates the ``trivial'' summand, and the bigrading of $\gamma_0$ is $(0,0)$. The bigrading $(i_k,j_k)\in \Z\times \R$ of $\gamma_k$ is arbitrary, although it gives the bigrading of $\underline \gamma_k$ as $(i_k+1,j_k)$. The real I-grading $\deg_I$ can be extended to any element of $\widetilde C$ as follows:
\[
  \deg_I\left(\sum s_k \zeta_k\right)=\max\{\deg_I(\zeta_k)\mid s_k\neq 0\}
\]
where the $\zeta_k$ belong to distinct summands $ \widetilde C_{i,j}$ of $ \widetilde C$.

Thus $(\widetilde C_*(Y,K;\Delta),\widetilde d)$, if defined with a trivial perturbation, has the structure of an I-graded $\cS$-complex (of level $0$) over $R[U^{\pm 1}]$, where $R=\Z[T^{\pm }]$. In the form \eqref{eq:igrscx}, the generators $\gamma_k$ ($k\geqslant 1$) are choices of homotopy classes of paths from irreducible critical points to $\theta$, while $\gamma_0$ is the constant path at $\theta$. In general, in the presence of a perturbation $\pi$ the differential can possibly increase the instanton grading, but only less than some $\delta_\pi\geqslant 0$ determined by the perturbation. Thus in the general case,  $(\widetilde C_*(Y,K;\Delta),\widetilde d)$ is an I-graded $\cS$-complex of level $\delta_\pi$ over $R[U^{\pm 1}]$. For morphisms, we have:

\begin{definition}
	A morphism $\widetilde \lambda:\widetilde C\to \widetilde C'$ of level $\delta> 0 $ of I-graded $\cS$-complexes (of any levels) is an $R[U^{\pm 1}]$-module homomorphism and morphism of $\cS$-complexes such that
	\[
		\widetilde \lambda \;\widetilde C_{i,j} \subset \bigcup_{k \leqslant  j + \delta} \widetilde C'_{i,k}. \;\; 
	\]
	A level $\delta$ $\cS$-chain homotopy $\widetilde K$ between morphisms $\widetilde \lambda$ and 
	$\widetilde \lambda'$ of I-graded $\cS$-complexes is an $R[U^{\pm 1}]$-module homomorphism and an 
	$\cS$-chain homotopy between $\widetilde \lambda$ and $\widetilde \lambda'$ such that
	\[
		\widetilde K \;\widetilde C_{i,j} \subset \bigcup_{k \leqslant  j + \delta} \widetilde C'_{i+1,k}. \;\; \diamond
	\]
\end{definition}

If the perturbation of the ASD equation in the definition of a cobordism map $\widetilde \lambda_{(W,S;\Delta)}$ is trivial, then this morphism does not increase the instanton grading. Hence it is a morphism of I-graded complexes of level $0$. In the general case, the map induced by $(W,S)$ is a morphism of I-graded complexes of some level determined by the perturbation.

\subsection{Enriched $\cS$-complexes}\label{enriched-cS}

Ideally, we would like to associate an I-graded $\cS$-complex of level 0 to $(Y,K)$. As the zero perturbation is not always admissible, we settle for the following limiting structure.

\begin{definition}
	An {\it enriched $\cS$-complex} $\widetilde \fE$ is a sequence $\{(\widetilde C^i, \widetilde d^i,\chi^i)\}_{i\geqslant 1}$ of I-graded $\cS$-complexes over $R[U^{\pm 1}]$ of levels $\delta_i$, and morphisms $\phi^j_i:\widetilde C^i\to \widetilde C^j$ of levels $\delta_{i,j}$ satisfying:
	\begin{itemize}
		\item[(i)] $\phi^i_i = \text{id}$ and $\phi^j_k\circ \phi_i^k$ is $\cS$-chain homotopy equivalent to $\phi_i^j$ 
		via an $\cS$-chain homotopy of some level $\delta_{i,k,j}$.
		\item[(ii)] For each $\delta >0$ there exists an $N$ such that $i>N$ implies $\delta_i\leqslant \delta$, and $i,j>N$ (resp. $i,j,k>N$) implies $\delta_{i,j}\leqslant \delta$ (resp. $\delta_{i,k,j} \leqslant \delta$). $\diamd$
	\end{itemize}
\end{definition}

For a based knot $(Y,K)$ in an integer homology 3-sphere, we may take a sequence
\[
	\widetilde \fE(Y,K;\Delta):=\{(\widetilde C_*^i(Y,K;\Delta),\widetilde d^i,\chi^i)\}
\]
of I-graded $\cS$-complexes associated to a sequence of perturbations of the Chern-Simons functional that go to zero. Thus $(\widetilde C_*^i(Y,K;\Delta),\widetilde d^i,\chi^i)$ is of some level $\delta_i$ determined by the chosen perturbation. For any pair $i$, $j$ there is a morphism $\phi_i^j:\widetilde C^i(Y,K;\Delta) \to \widetilde C^j(Y,K;\Delta)$ of I-graded $\cS$-complexes of level $\delta_{i,j}$, determined by a path of auxiliary data. Moreover, $\widetilde \fE(Y,K;\Delta)$ satisfies the properties of an enriched $\cS$-complex over $R[U^{\pm 1}]$, where here $R=\Z[T^{\pm 1}]$. The proofs of these claims are identical to the proofs of analogous results in the non-singular setting given in \cite{AD:CS-Th}. Next, we define morphisms in this setting:

\begin{definition}
	A {\emph{morphism}} $\mathfrak{L}:\widetilde \fE(1) \to \widetilde \fE(2)$ of enriched complexes, where
	\[
		\widetilde \fE(r)=(\{(\widetilde C_*^i(r),\widetilde d^i(r),\chi^i(r))\},\phi_i^j(r)) \qquad r\in \{1,2\},
	\]
	is a collection of morphisms $\widetilde \lambda_i^j:\widetilde C^i(1)\to \widetilde C^j(2)$ of I-graded $\cS$-complexes 
	of level $\delta_{i,j}$ such that the following hold:
	\begin{itemize}
		\item[(i)] $\widetilde \lambda_k^j\circ \phi_i^k(1)$ and $\phi_k^j(2)\circ \widetilde \lambda_i^k$ are $\cS$-chain 
		homotopy equivalent to $\widetilde \lambda_i^j$ via an $\cS$-chain homotopy of some level $\delta_{i,k,j}$.
		\item[(ii)]  For each $\delta>0$, there exists an $N$ such that $i,j>N$ (resp. $i,j,k>N$) implies that $\delta_{i,j}<\delta$ 
		(resp. $\delta_{i,k,j}<\delta$). 
	\end{itemize}
The morphism is a {\emph{chain homotopy equivalence}} of enriched $\cS$-complexes if each $\widetilde \lambda_i^j$ is an $\cS$-chain homotopy equivalence where the involved $\cS$-chain homotopy equivalences have levels which converge to $0$. $\diamd$
\end{definition}

We have thus constructed the category of enriched $\cS$-complexes over $R[U^{\pm 1}]$. There is a forgetful functor to the category of $\cS$-complexes, that to any enriched $\cS$-complex associates the first $\cS$-complex in its sequence, and to any enriched $\cS$-morphism as above, we also associate the $\cS$-morphism $\lambda_1^1$. This forgetful map does not remember the positivity property of enriched complexes with respect to the instanton gradings.

To a negative definite pair $(W,S):(Y,K)\to (Y',K')$ we can associate a morphism $\mathfrak{L}_{(W,S)}$ of enriched $\cS$-complexes $\widetilde\fE (Y,K;\Delta)\to\widetilde\fE (Y',K';\Delta)$ by taking a sequence of $\cS$-morphisms, defined as usual, after choosing an appropriate sequence of auxiliary data.

We obtain the analogue of Theorem \ref{thm:framedcat}.

\begin{theorem}
	The assignments $(Y,K)\mapsto \widetilde\fE (Y,K;\Delta)$ and $(W,S)\mapsto \mathfrak{L}_{(W,S)}$ induce a functor from $\mathcal{H}$ to the homotopy category of enriched $\cS$-complexes over $\Z[T^{\pm 1}][U^{\pm 1}]$.
\end{theorem}

The enriched $\cS$-complex $\widetilde \fE(Y,K;\Delta)$ is the ``universal'' invariant defined in this paper, in the sense that all of our $\cS$-complexes associated to based knots in integer homology 3-spheres may be derived from this invariant by a change of basis (coefficient ring), and by possibly applying a forgetful functor.

\subsection{Local equivalence for enriched complexes}

We may apply the general procedure described in Subsection \ref{sec:localequiv} to the category of enriched $\cS$-complexes over $R[U^{\pm 1}]$. That is, we declare that two enriched $\cS$-complexes $\widetilde \fE$ and $\widetilde \fE'$ are equivalent $\widetilde\fE \sim \widetilde\fE'$ if there are morphisms $\widetilde\fE \to \widetilde\fE'$ and $\widetilde\fE' \to \widetilde\fE$. The resulting set
\[
	\Theta^\fE_{R[U^{\pm 1}]} = \left\{ \text{ enriched $\cS$-complexes over $R[U^{\pm 1}]$ } \right\} / \sim
\]
is a partially ordered abelian group, where $[\widetilde\fE]\leqslant [\widetilde\fE']$ if there is a morphism $\widetilde\fE\to \widetilde\fE'$. The group structure is inherited from that of the local equivalence group for $\cS$-complexes, by performing operations component-wise for each sequence. Furthermore, the forgetful functor from the category of enriched $\cS$-complexes to $\cS$-complexes induces a surjective homomorphism of partially ordered abelian groups:
\[
	\Theta^\fE_{R[U^{\pm 1}]} \longrightarrow \Theta^\cS_{R[U^{\pm 1}],\Z}
\] 
The target is the local equivalence group of $\Z$-graded $\cS$-complexes over $R[U^{\pm 1}]$. The grading is inherited from the $\Z$-grading of the I-graded $\cS$-complex.

\begin{theorem}
	The assignment $(Y,K)\mapsto \widetilde \fE(Y,K;\Delta)$ induces a homomorphism of partially ordered abelian groups $\Omega:\Theta^{3,1}_\Z\to \Theta^\fE_{R[U^{\pm 1}]}$ where $R=\Z[T^{\pm 1}]$.
\end{theorem}

We also have an analogue of the connected sum theorem in the setting of enriched $\cS$-complexes. For this, we note that the tensor product of two I-graded $\cS$-complexes is naturally an I-graded $\cS$-complex, and similarly for enriched $\cS$-complexes.

\begin{theorem}\label{thm:connectedsumenriched} Let $(Y,K)$ and $(Y',K')$ be based knots in integer homology 3-spheres. There is a chain homotopy equivalence of enriched $\cS$-complexes over $\mathscr{R}=\Z[U^{\pm 1}, T^{\pm 1}]$:
\begin{equation*}
	\widetilde \fE(Y\# Y',K\# K';\Delta) \simeq  \widetilde \fE(Y,K;\Delta)\otimes_\mathscr{R} \widetilde \fE(Y',K';\Delta) 
\end{equation*}
This equivalence is natural, up to enriched chain homotopy, with respect to split cobordisms.
\end{theorem}

The proof follows the remarks after Theorem \ref{thm:connectedsumlocceoff}, and relies on the fact that $\kappa(A)=0$ for the instantons in \eqref{eq:connsuminstantons}. We also use that in the proof we can choose the perturbations on the two cobordisms as small as we like.

\subsection{The concordance invariant $\Gamma^{R}_{(Y,K)}$}

The local equivalence class of the enriched $\cS$-complex associated to a based knot $(Y,K)$ is expected to be a strong invariant, as it has the behavior and values of the Chern-Simons functional built into its structure. Here we describe one way to extract numerical information from this local equivalence class which leads to a concordance invariant $\Gamma^R_{(Y,K)}$, an analogue of the invariant $\Gamma_Y$ for homology 3-spheres from \cite{AD:CS-Th}.

The definition of the invariant $\Gamma^R_{(Y,K)}$ factors through an algebraic map defined on the local equivalence group of enriched $\cS$-complexes:
\begin{equation}
	\Gamma:\Theta^{\fE}_{R[U^{\pm 1}]} \longrightarrow \text{Map}_{\geqslant}(\Z, \overline \R_{\geqslant 0})\label{eq:gammamap}
\end{equation}
In fact, a similar algebraic map can be used to define $\Gamma_Y$. Here and throughout this section, $R$ is an integral domain and an algebra over $\Z[T^{\pm 1}]$. The codomain in \eqref{eq:gammamap} is the set of non-decreasing functions from $\Z$ to the extended positive real line $\overline \R_{\geqslant 0} = \R_{\geqslant 0} \cup \infty$. The map $\Gamma$ is defined as follows. Take an enriched $\cS$-complex $\widetilde \fE$ defined by a sequence $\{(\widetilde C^j,\widetilde d^j , \chi^j)\}$ of I-graded $\cS$-complexes over $R[U^{\pm 1}]$. 
To each $\cS$-complex $(\widetilde C^j,\widetilde d^j , \chi^j)$ we have the associated chain group $C^j$ and the $R[U^{\pm 1}]$-module homomorphisms:
\begin{gather*}
	d^j:C^j\to C^j,\hspace{1cm}v^j:C^j\to C^j,\\
	\delta_1^j:C^j\to R[U^{\pm1}],\hspace{1cm}\delta_2^j:R[U^{\pm1}]\to C^j.
\end{gather*}
Then for each $k\in \Z^{>0}$ we define:
\[
	\Gamma(\widetilde \fE)(k) := \lim_{j\to \infty} \inf_{\alpha}\left(\deg_I(\alpha)\right) \in \overline \R_{\geqslant 0}
\]
where the infimum is over all $\alpha\in C^j$ with bigrading $(2k-1,\deg_I(\alpha))$ such that:
\begin{equation*}\label{Gamma-pos-inf}
  d^j(\alpha)=0,\hspace{.9cm}k-1=\min\{i\in \Z^{\geq 0}\mid \delta_1^j(v^j)^{i}(\alpha)\neq 0 \}.
\end{equation*}
For each $k\in \Z^{\leq 0}$ we define:
\[
	\Gamma(\widetilde \fE)(k) := \max \Bigg( \lim_{j\to \infty} \inf_{\alpha}\left(\deg_I(\alpha) \right),0  \Bigg) \in \overline \R_{\geqslant 0}
\]
where the infimum is over all $\alpha\in C^j$ with bigrading $(2k-1,\deg_I(\alpha))$ such that there are $\{a_0, a_1, \dots, a_{-k}\}\subset R[U^{\pm 1}]$ satisfying:
\begin{equation}\label{Gamma-neg-inf}
  d^j(\alpha)=\sum_{i=0}^{-k}(v^j)^{i}\delta_2^j(a_i).
\end{equation}
Note that in \eqref{Gamma-neg-inf}, a straightforward degree consideration implies that we can limit ourselves to the following case, where $s_i\in R$:
\[
  a_i=\left\{
  \begin{array}{ll}
 	s_iU^{\frac{k+i}{2}}&i\equiv k \mod 2\\
	0&i\nequiv k \mod 2\\
  \end{array}
  \right.
\]

If there is a morphism $\widetilde\fE \to \widetilde\fE'$ between enriched $\cS$-complexes then $\Gamma(\widetilde \fE)(k)\geq \Gamma(\widetilde \fE')(k)$ for any integer $k$. In particular, locally equivalent enriched complexes have the same $\Gamma$ functions. An equivalent definition of $\Gamma(\widetilde \fE)$ can also be given in terms of the small equivariant complexes associated to the $\cS$-complexes.  We refer to \cite{AD:CS-Th} for more details.

\begin{definition}
	Let  $R$ be an integral domain which is an algebra over $\Z[T^{\pm 1}]$. For a based knot in an integer homology 3-sphere $(Y,K)$, we define the function $\Gamma^R_{(Y,K)}$ as follows:
	\[
		\Gamma^{R}_{(Y,K)}:=\Gamma\left( \widetilde \fE(Y,K;\Delta_{R[U^{\pm 1}]}) \right)
	\]
	If $Y=S^3$ we write $K$ in place of $(Y,K)$. $\diamd$
\end{definition}

	That is, $\Gamma^R_{(Y,K)}$ is the invariant of the equivalence class $[(Y,K)]\in \Theta^{3,1}_\Z$ obtained as:
	\[
		\Gamma^R_{(Y,K)}:\Theta^{3,1}_\Z\xrightarrow{\;\;\Omega\;\;} \Theta^\fE_{R[U^{\pm 1}]} \xrightarrow{\;\;\Gamma\;\;} \text{Map}_{\geqslant 0} (\Z,\overline\R_{\geqslant 0})
	\]
	The following summarizes the basic properties of this function. The proofs are entirely analogous to those of Theorems 1--4 and Proposition 1 of \cite{AD:CS-Th}. 
	
\begin{theorem}\label{Gamma-Y-K}
	Let $(Y,K)$ be a based knot in an integer homology 3-sphere.
	\begin{itemize}
		\item[\emph{(i)}] The function $\Gamma_{(Y,K)}^R$ is an invariant of $[(Y,K)]\in \Theta^{3,1}_\Z$.
		
		\item[\emph{(ii)}] $\Gamma_{(Y,K)}^R$ is a non-decreasing function $\Z\to \overline\R_{\geqslant 0}$ which is positive for $i\in\Z_{>0}$.
		
		\item[\emph{(iii)}] If $(W,S):(Y,K)\to (Y',K')$ is a negative definite cobordism of pairs, then
		\[
			\Gamma_{(Y',K')}^R(i) \leqslant \begin{cases} \Gamma_{(Y,K)}^R(i)-\eta(W,S) & i>0 \\ \max(\Gamma_{(Y,K)}^R(i)-\eta(W,S),0) & i\leqslant 0 \end{cases}
		\]
		where $\eta(W,S)\in \R_{\geqslant 0}$ is an invariant of $(W,S)$. Furthermore, $\eta(W,S)>0$ unless there is a traceless $SU(2)$ representation of $\pi_1(W\setminus S)$ which extends irreducible traceless $SU(2)$ representations of $\pi_1(Y\setminus K)$ and $\pi_1(Y'\setminus K')$.
		
		\item[\emph{(iv)}] For each $i\in\Z$, we have $\Gamma_{(Y,K)}^R(i)<\infty$ if and only if $i\leqslant h_R(Y,K)$.
		
		\item[\emph{(v)}] For each $i\in \Z$, if $\Gamma_{(Y,K)}^R(i)\not\in\{ 0, \infty\}$ then it is congruent to $\text{{\emph{CS}}}(\alpha)$ {\emph{(mod $\Z$)}} for some irreducible singular flat $SU(2)$ connection $\alpha$ on $(Y,K)$.
	\end{itemize}
\end{theorem}
	
	The invariant $\eta(W,S)$ is defined to be the infimum of $2\kappa(A)$, as $A$ ranges over all finite energy singular ASD connections which limit to irreducible flat connections on the ends.\\

	In this subsection, we did not attempt to systematically exploit the Chern-Simons filtration to study concordances, and we content ourselves with the
	definition of one homology concordance invariant. For example, we believe that by a slight modification of our axiomatization of enriched complexes one can define analogues of the invariants $r_s$ introduced in \cite{NST:filtered}. Another possible direction is to apply the construction of Subsection \ref{subsec:ideals} in the context of enriched $\cS$-complexes.

\newpage

%!TEX root = main.tex

\section{Connections to Kronheimer and Mrowka's constructions}\label{sec:kmgroups}

In \cite{KM:YAFT,KM:unknot}, Kronheimer and Mrowka defined several instanton Floer homology groups associated to a given link in a 3-manifold, for cases in which no reducibles are present. In \cite{km-tait, km-def, km-barnatan}, over rings of characteristic $2$ they extended their constructions to webs, which are embedded trivalent graphs. Here we recall some of these constructions, and discuss their relationship to the invariants introduced earlier.

\subsection{Instanton homology for admissible links}\label{subsec:kmgroups}

We first recall the instanton homology groups $I^\omega(Y,L)$ construced in \cite{KM:YAFT,KM:unknot} for certain links in 3-manifolds. Special cases of this construction are the singular instanton groups $I^\natural(Y,L)$ and $I^\#(Y,L)$, where in this latter case $(Y,L)$ is a pair of any 3-manifold $Y$ with an embedded link $L$. We begin with:

\begin{definition}
	An {\emph{admissible link}} is a triple $(Y,L,\omega)$ where $L$ is an unoriented link embedded in a closed, oriented, connected 3-manifold $Y$, and  $\omega\subset Y$ is an unoriented 1-manifold embedded in $Y$ with $\partial \omega = \omega\cap L$, transversely, satisfying the following condition: there exists a closed oriented surface $\Sigma\subset Y$ such that either
	\begin{itemize}
\item $\Sigma$ is disjoint from $L$ and intersects $\omega$ transversely an odd number of times, or\label{eq:nonint1}
\item $\Sigma$ is transverse to $L$ and intersects it an odd number of times.
\end{itemize}
This condition for $\omega$ is called the {\emph{non-integrality}} condition.  $\diamd$
\end{definition}

We remark that with this terminology, a knot in an integer homology 3-sphere is not an admissible link for any choice of $\omega$.

Kronheimer and Mrowka associate to an admissible link $(Y,L,\omega)$ a relatively $\Z/4$-graded abelian group $I^\omega(Y,L)$, defined as follows. From the 1-manifold $\omega$ one may construct an $SO(3)$-bundle $P\to Y\setminus L$ whose second Stiefel-Whitney class is Poincar\'{e} dual to $[\omega]\in H_1(Y,L;\Z/2)$. Then define a chain complex $\left(C_\ast^\omega(Y,L),d\right)$ by setting
\[
	C^\omega_\ast = C_\ast^\omega(Y,L) =  \bigoplus_{\alpha\in\fC_\pi} \Z\cdot \alpha
\]
where $\fC_\pi$ are the critical points modulo the determinant-1 gauge group of a suitably perturbed Chern-Simons functional for $P$. The non-integrality condition on $\omega$ ensures that for small $\pi$, all such critical points on $P$ are {\emph{irreducible}}, i.e. $\fC_\pi=\fC_\pi^{\text{irr}}$. The differential $d$ is defined just as in \eqref{eq:irrdiff}, and $I^\omega(Y,L)$ is defined to be the homology of this chain complex:
\[
	I^\omega(Y,L) := H_\ast\left(C^\omega(Y,L),d\right).
\]
As before, unlike the group $I^\omega(Y,L)$, the chain complex $\left(C^\omega(Y,L),d\right)$ depends on a choice of metric and perturbation, which are suppressed from the notation.

Given {\emph{any}} link $L\subset Y$ with a basepoint $p\in L$, Kronheimer and Mrowka define
\[
	I^\natural(Y,L) := I^\omega(Y,L\# H)
\]
where $H\subset S^3$ is the Hopf link and $\omega$ is a unknotted arc connecting the two components of $H$; see Figure \ref{fig:hopf}. The connect sum is taken at the basepoint $p\in L$. The connected sum $L\# H$ may be identified with $L\cup \mu$ where $\mu$ is a small meridional component around $L$ near $p$. Note that the resulting $\omega$ always satisfies the non-integrality condition: take $\Sigma$ to be the 2-torus boundary of a small regular neighborhood of $\mu$. The homology $I^\natural(Y,L)$ is a relatively $\Z/4$-graded abelian group, and is an invariant of the based link $(Y,L,p)$.  Note that just as before we omit the basepoint $p$ from the notation. The group $I^\natural(Y,L)$ is in fact absolutely $\Z/4$-graded. In the sequel, we will use the notation
\[
	(C^\natural_\ast,d^\natural) = (C^\natural_\ast(Y,L) , d^\natural)
\]
for the chain complex $(C^\omega_\ast(Y,L\# H),d)$, so that $I^\natural_\ast(Y,L)$ is the homology of $(C^\natural_\ast,d^\natural)$. 

In a similar vein, for any link $L\subset Y$, Kronheimer and Mrowka define the group 
\[
 	I^\#(Y,L):=I^\omega(Y,L\sqcup H).
\]
Here we take the disjoint union of $L$ with the Hopf link $H$, along with its arc $\omega$. In order to perform this construction, we choose a small ball in $Y\setminus L$ in which to embed $H$. We write 
\[
	(C^\#_\ast,d^\#) = (C^\#_\ast(Y,L) , d^\#)
\]
for the chain complex $(C^\omega_\ast(Y,L\sqcup H),d)$, so that $I^\#_\ast(Y,L)$ is the homology of $(C^\#_\ast,d^\#)$. This group is also absolutely $\Z/4$-graded.

\begin{figure}[t]
\centering
\includegraphics[scale=1.65]{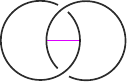}
\caption{The Hopf link $H$ with the arc $\omega$.}
\label{fig:hopf}
\end{figure}

Let $(Y,L,\omega)$ be an admissible link, and let $p\in L$ be a basepoint. The construction of Subsection \ref{sec:vmap} carries through in this setting to define a map $v:C^\omega_\ast \to C^\omega_{\ast-2}$ associated to $p$, using $S^1$ holonomy along the cylinder. Here we have:

\begin{prop}\label{prop:vmap2}
	$d \circ v - v \circ d  = 0$. 
\end{prop}

\noindent The proof is similar to that of Proposition \ref{prop:vmap}; the absence of any reducible critical points, in this case, precludes the appearance of the term $\delta_2\circ \delta_1$.

\begin{remark}\label{v-maps}
	There are at least two other ways that one can define a degree $2$ 
	operator on the complex $C^\omega_\ast(Y,L)$ using the basepoint $p$. Connected sum of 
	$[-1,1]\times L$ at the point $(0,p)$ with a standard torus determines a cobordism of pairs 
	$([-1,1]\times Y,S):(Y,L)\to (Y,L)$ and the induced cobordism map is a degree 
	$2$ chain map $\sigma$ acting on $C^\omega_\ast(Y,L)$ \cite[Subsection 8.3]{KM:unknot}. 
	Alternatively, the standard construction of $\mu$-maps \cite[Chapter 5]{DK} assigns to the point $p$ 
	a cohomology class of degree $2$ in the space of 
	singular connections on the bundle $P$ associated to $(Y,L,\omega)$ \cite{Kr:obs}. 
	This cohomology class in \cite{Kr:obs} is rational; 
	to obtain an integral class we consider $-2$ times the 2-dimensional point class in \cite{Kr:obs}.
	Cup product with this cohomology class 
	defines another chain map $\sigma'$ of degree $2$ on $C^\omega_\ast(Y,L)$. 
	The argument of \cite[Proposition 5.1]{Kr:obs} shows that the operators $\sigma$ and $\sigma'$ 
	are chain homotopy equivalent. 	
	Moreover, the operator $\sigma'$ is also chain homotopy equivalent to $v$. 
	The proof is analogous to the corresponding result in the 
	non-singular setting in \cite[Subsection 7.3.2]{donaldson-book}. $\diamd$
\end{remark}

The analogue of the $\cS$-complex $\widetilde C(Y,K)$ from Subsection \ref{sec:framed} in this setting is simply a mapping cone complex of $v$; we define $(\widetilde C^\omega(Y,L),\widetilde d)$ by
\begin{equation}\label{eq:mappingconecx}
	\widetilde C^\omega(Y,L) := C^\omega_\ast \oplus C^\omega_{\ast-2}, \hspace{1cm} \widetilde d = \left[\begin{array}{cc} d & 0 \\ v & -d \end{array}\right]
\end{equation}
We leave it to the interested reader to formulate the analogue of Theorem \ref{thm:framedcat} in this setting, describing a functor from a category whose objects are base-pointed admissible links $(Y,L,\omega)$ to a suitable category of mapping cone complexes.

We now describe variations of Theorem \ref{thm:connectedsum} obtained by replacing one or both of $(Y,K)$ and $(Y',K')$ by a based admissible link. Let $(Y,K)$ be an integer homology 3-sphere with an embedded based knot. Let $(Y',L',\omega')$ be a based admissible link. Then $\widetilde C(Y,K)$ is an $\cS$-complex, while \eqref{eq:mappingconecx} defines the chain complex $\widetilde{C}^{\omega'}(Y',L')$ as a mapping cone complex.

\begin{theorem}  {\emph{\bf{(Connected Sum Theorem for a knot and an admissible link)}}} \label{thm:oneadmis} There is a chain homotopy equivalence of relatively $\Z/4$-graded chain complexes:
\[
	\widetilde C^{\omega'}(Y\# Y',K\# L') \simeq  \widetilde C(Y,K)\otimes \widetilde C^{\omega'}(Y',L')
\]
Furthermore, the tensor product is naturally isomorphic to a mapping cone complex, making the chain homotopy equivalence one of mapping cone complexes. The equivalence is natural, up to mapping cone chain homotopies, with respect to split cobordisms.
\end{theorem}

\begin{remark}
	One can generalize the definition of an $\cS$-complex to include mapping cone complexes, the latter being viewed as $\cS$-complexes with the distinguished summand $\Z$ replaced by $0$, and all maps modified accordingly. The notion of morphisms can be similarly generalized, and the theorems in this section can then be stated as homotopy equivalences between $\cS$-complexes in this larger category. $\diamd$
\end{remark}

\begin{remark}
	The naturality in Theorem \ref{thm:connectedsum}, as explained in Subsection \ref{subsec:natcon}, assumes that the cobordisms involved are negative definite pairs. However, for the naturality in the above Theorem \ref{thm:oneadmis}, we allow the cobordism on the side of the admissible links to be of the general sort considered in \cite{KM:YAFT, KM:unknot}. A similar remark holds for the other variations of the connected sum theorem stated below. $\diamd$
\end{remark}

The proof of this result is very similar to that of Theorem \ref{thm:connectedsum}. If $(Y',L',\omega')$ is admissible, then one modifies the proof above in which $(Y',K')$ is a knot by omitting the $\Z$-summand in the associated complex $(\widetilde C_\ast', \widetilde d')$ and all maps that have anything to do with it; in short, the reducible $\theta'$ is eliminated. That the tensor product of an $\cS$-complex and a mapping cone complex is naturally a mapping cone complex follows from the discussion in Subsection \ref{sec:tensor} by simply deleting the $\Z$-summand of one $\cS$-complex.

There is another variation where $(Y,K)$ is also replaced by a based admissible link $(Y,L,\omega)$. The statement in this situation is as follows:

\begin{theorem}  {\emph{\bf{(Connected Sum Theorem for admissible links)}}} \label{thm:connsumadlinks} Let $(Y,L,\omega)$, $(Y',L',\omega')$ be based admissible links.
There is a relatively $\Z/4$-graded chain homotopy equivalence
\[
	\widetilde C^{\omega \cup \omega'}(Y\# Y',L\# L') \simeq  \widetilde C^\omega(Y,L)\otimes \widetilde C^{\omega'}(Y',L')
\]
This equivalence is one of mapping cone complexes, and is natural, up to mapping cone chain homotopies, with respect to split cobordisms.
\end{theorem}

These variations have counterparts in non-singular instanton Floer homology, involving connected sums between homology 3-spheres and 3-manifolds with non-trivial admissible bundles, see e.g. \cite{scadutothesis}.

\subsection{Computing $I^\natural(Y,K)$ and $I^\#(Y,K)$ from the framed complex}\label{sec:naturalsharp}

We may now relate the framed instanton homology $\widetilde I(Y,K)$, or more precisely its underlying chain complex, to Kronheimer and Mrowka's instanton homology groups $I^\natural (Y,K)$ and $I^\#(Y,K)$.  We first consider $I^\natural(Y,K)$. Recall that this group has an absolute $\Z/4$-grading defined in \cite[Section 4.5]{KM:unknot}. 

\begin{theorem}\label{thm:tildeconnsum} 
	Let $(Y,K)$ be a based knot in an integer homology 3-sphere. There is a chain homotopy equivalence $C^\natural (Y,K ) \simeq \widetilde C(Y,K)$, natural up to chain homotopy, and homogeneous with respect to $\Z/4$-gradings. In particular, there is a natural isomorphism
	\[
		I^\natural(Y,K) \cong \widetilde I(Y,K)
	\]
	In the case that $Y$ is the 3-sphere, the isomorphism has degree $\sigma(K) \pmod 4$.
\end{theorem}

The chain homotopy $C^\natural (Y,K ) \simeq \widetilde C(Y,K)$ in the statement of the above theorem is a chain homotopy of $\Z/4$-graded chain complexes. That is to say, we forget the $\cS$-complex structure of $\widetilde C(Y,K)$ given by the endomorphism $\chi$.

\begin{proof}
	Recall from Subsection \ref{subsec:kmgroups} that $C^\natural(Y,K)$ is defined to be $C^\omega(S^3\# Y,H\# K)$ where $H\subset S^3$ is the Hopf link and $\omega$ is a small arc as in Figure \ref{fig:hopf}. We apply Theorem \ref{thm:oneadmis} in this situation to obtain a chain homotopy equivalence
	\begin{equation}
		\widetilde C^\omega(S^3\# Y,H\# K) \simeq \widetilde C^\omega (S^3,H)  \otimes \widetilde C(Y,K). \label{eq:hopfconn}
	\end{equation}
	The complex $C^\omega (S^3,H)$ is free abelian on one generator, with zero differential. The $v$-map, of degree $2$ (mod 4), is necessarily zero. Thus the mapping cone complex $\widetilde C^\omega(S^3,H)$ is free abelian of rank two with zero differential. We may then identity the right side of \eqref{eq:hopfconn} with two copies of $\widetilde C(Y,K)$; it is the mapping cone for the zero map on $\widetilde C(Y,K)$. As the chain homotopy \eqref{eq:hopfconn} is one of {\emph{mapping cone}} complexes, we conclude that 
	\[
		C^\natural (Y,K) = C^\omega(S^3\# Y,H\# K) \simeq \widetilde C(Y,K).
	\]
	Now suppose $Y$ is the 3-sphere. As the established equivalence is homogeneous with respect to gradings, to compute its degree, it suffices to compare the grading of the reducible generator on each side.  By definition, the reducible generator in $\widetilde C(Y,K)$ has grading zero (mod 4). On the other hand, the grading of the reducible in $C^\natural(K)$ is computed by Poudel and Saveliev in \cite[Theorem 1]{PS} to be $\sigma(K)$ (mod 4). (Note that the signature of a knot is always even, so we do not have to determine a sign.)
\end{proof}

Note that under the equivalence of Theorem \ref{thm:tildeconnsum}, the $v$-map on the complex $C^\natural(Y,K)$ corresponds to the map on $\widetilde C(Y,K)= C_\ast \oplus C_{\ast-1} \oplus \Z$ by formula \eqref{eq:vtensor} to the zero map. This recovers a special case of \cite[Proposition 4.6]{xie}.

Next, we observe that Theorem \ref{thm:tildeconnsum} combined with Theorem \ref{thm:connectedsum} recovers the following connected sum theorem for $I^\natural(Y,K)$:
\begin{cor}
	Let $(Y,K)$ and $(Y',K')$ be knots in integer homology 3-spheres. Then over a field there is a natural isomorphism of vector spaces
	\begin{equation}
		I^\natural (Y\# Y', K\# K') \cong I^\natural(Y,K)\otimes I^\natural (Y',L').\label{eq:inaturalconnsumthm}
	\end{equation}
	which preserves the $\Z/4$-gradings.
\end{cor}

Note that, from our viewpoint, the preservation of the $\Z/4$-gradings in \eqref{eq:inaturalconnsumthm} follows from the additivity of the knot signature under connected sums.

We now turn to $I^\#(Y,K)$. Recall that on the $\cS$-complex $\widetilde C_\ast(Y,K) = C_\ast \oplus C_{\ast-1} \oplus \Z$ the map $\chi:\widetilde C_\ast(Y,K)\to \widetilde C_\ast(Y,K)$ defined with respect to this decomposition by
\[
	\chi = \left[ \begin{array}{ccc} 0 & 0 & 0 \\ 1 & 0 & 0 \\ 0 & 0 & 0 \end{array} \right]
\]
is an anti-chain map. Note that $\chi$ sends $C_\ast$ to $C_{\ast-1}$ identically and is otherwise zero. We may form $\text{Cone}(2\chi)$, the mapping cone of $2\chi$ acting on $\widetilde C_\ast(Y,K)$.

\begin{theorem}\label{thm:sharpconnsum} 
	Let $(Y,K)$ be a based knot in an integer homology 3-sphere. There is a chain homotopy equivalence $C^\# (Y,K )\simeq \text{\emph{Cone}}(2\chi)$. This equivalence is natural up to chain homotopy, and homogeneous with respect to $\Z/4$-gradings.
\end{theorem}

\begin{proof}
	Recall from Subsection \ref{subsec:kmgroups} that $C^\#(Y,K)$ is defined to be $C^\omega(S^3\# Y,H\sqcup K)$ where $H\subset S^3$ is the Hopf link and $\omega$ is a small arc as in Figure \ref{fig:hopf}. We may view $(S^3\# Y,H\sqcup K)$ as the connected sum of $(S^3,H\sqcup U_1)$, a Hopf link with a disjoint unknot (the latter of which contains the basepoint) with the based knot $(Y,K)$. Apply Theorem \ref{thm:oneadmis} to obtain
	\begin{equation}
		\widetilde C^\omega(S^3\# Y,H\sqcup K) \simeq \widetilde C^\omega (S^3,H\sqcup U_1)  \otimes \widetilde C(Y,K).\label{eq:sharpapp}
	\end{equation}
	The complex $C^\#(U_1)=C^\omega (S^3,H\sqcup U_1)$ contains two generators, $\mathbf{v}_+$ and $\mathbf{v}_-$, which differ in degree by $2$ (mod $4$). Indeed, the traceless character variety for $(S^3,H\sqcup U_1)$ is a 2-sphere, and we may perturb the Chern-Simons functional using a standard Morse function for $S^2$ leaving us with two critical points. The differential on $C^\#(U_1)$ is zero for grading reasons, and we have a natural identification between $C^\# ( U_1)$ and its homology $I^\#(U_1)$.
	
	We may then align our notation of generators $\mathbf{v}_+$ and $\mathbf{v}_-$ with \cite{KM:unknot}, 
	where the $v$-map, denoted there by $\sigma$, is computed on $I^\#(U_1)$ as follows 
	(see Remark \ref{v-maps}):
	\[
			v(\mathbf{v}_+)= 2\mathbf{v}_-, \qquad v(\mathbf{v}_-)=0.
	\]
	Having determined $\widetilde C^\omega (S^3,H\sqcup U_1)$ to be the mapping cone of $v$ as above on 
	$\Z\mathbf{v}_+ \oplus \Z\mathbf{v}_-$, using Subsection \ref{sec:tensor} we compute 
	$C^\omega(S^3\# Y,H\sqcup K)\subset \widetilde C^\omega(S^3\# Y,H\sqcup K)$ to be
	\[
		 \Z \mathbf{v}_+ \otimes \widetilde C(Y,K)  \oplus \Z \mathbf{v}_- \otimes \widetilde C(Y,K)
	\]
	with differential $1\otimes \widetilde d \oplus 1\otimes \widetilde d - 2F$ where $F$ sends $\Z\mathbf{v}_+\otimes C_\ast$ to $\Z\mathbf{v}_-\otimes C_{\ast-1}$ identically. This chain complex is clearly the same as $\text{Cone}(-2\chi)$, which is isomorphic to $\text{Cone}(2\chi)$.
\end{proof}

Note that the results above fit together to form an exact triangle:
\begin{equation}\label{eq:naturalsharptriangle}
\xymatrixcolsep{.5cm}
	\xymatrix{
	\cdots  I^\natural(Y,K)  \ar[d] \ar[r] & I^\#(Y,K) \ar[d] \ar[r] &  I^\natural(Y,K)  \ar[d] \ar[r] & I^\natural(Y,K)  \ar[d] \cdots\\
	\cdots  H_\ast(\widetilde C(Y,K)) \ar[r] &  H_\ast(\text{Cone}(2\chi))  \ar[r] &   H_\ast(\widetilde C(Y,K))  \ar[r]^{\;\;\;[2 \epsilon \chi]\;\;\;\;\quad } &  H_{\ast -1}(\widetilde C(Y,K))  \cdots
	}
\end{equation}
Here, the vertical maps are induced by the equivalences of Theorems \ref{thm:tildeconnsum}  and \ref{thm:sharpconnsum}, the bottom horizontal arrows are induced by the short exact sequence for a mapping cone complex, and the top horizontal arrows are defined to commute. There is similar long exact sequence involving $I^\natural(Y,K)$ and $I^\#(Y,K)$ obtained from Kronheimer and Mrowka's unoriented skein exact triangle applied to the situation of Figure \ref{fig:hopftriangle}, see \cite[Section 8.7]{KM:unknot}.

\begin{figure}[t]
\centering
\includegraphics[scale=1]{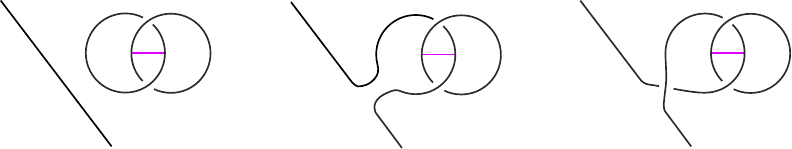}
\caption{The skein triple involving $(S^3\# Y, (H \sqcup U_1)\# K)$ (left) and two instances of $(S^3 \# Y, H\# K)$ (middle and right). The basepoint on $H\sqcup U_1$ lives on the unknot $U_1$.}
\label{fig:hopftriangle}
\end{figure}

The skein triple in Figure \ref{fig:hopftriangle} may be viewed as obtained from the skein triple for $(S^3,H\sqcup U_1)$, $(S^3,H)$ and $(S^3,H)$ (setting $(Y,K)=(S^3,U_1)$ in Figure \ref{fig:hopftriangle}), and then connect summing with $(Y,K)$ at a point on $U_1\subset H\sqcup U_1$ away from the crossing resolutions. The naturality of our equivalences with respect to split cobordisms, together with the computations in \cite[Section 8.7]{KM:unknot}, implies the following:

\begin{prop}
	The exact triangle \eqref{eq:naturalsharptriangle} is isomorphic to the exact triangle obtained from the unoriented skein exact triangle as described above.
\end{prop}

If we work over the coefficient ring $\bF=\Z/2$, then $2\chi=0$, and \eqref{eq:naturalsharptriangle} splits.

\begin{cor}
	Over the field $\bF=\Z/2$ there is a natural $\Z/4$-graded isomorphism 
	\[
		I^\#(Y,K;\bF)_\ast \cong I^\natural(Y,K;\bF)_\ast\oplus I^\natural(Y,K;\bF)_{\ast+2}
	\]
\end{cor}

This last corollary is essentially \cite[Lemma 7.7]{km-tait}.

\subsection{Local coefficients and the concordance invariant $s^\#(K)$}\label{subsec:rasmussen}

Let $(Y,K)$ be a based knot in an integer homology 3-sphere. Consider the ring 
\[
	\mathscr{T}_\Q= \sT\otimes \Q = \Q[T^{\pm 1}]
\]
of Remark \ref{rem:tring} tensored by $\Q$, so that $\widetilde C_\ast(Y,K;\Delta_{\mathscr{T}_\Q})$ is a $\Z/4$-graded $\cS$-complex over $\mathscr{T}_\Q$. Recall that to each critical point $[B]$ we assign the module $T^{\text{hol}_K(B)}\mathscr{T}_\Q$. A variation of our connected sum theorem with local coefficients and one admissible link implies the following chain homotopy equivalence of $\Z/4$-graded mapping cone complexes over $\mathscr{T}_\Q$:
\begin{equation*}
	\widetilde C_\ast^\#(Y,K;\Delta_{\sT_\Q}) \simeq \widetilde C^\omega_\ast(S^3,H\sqcup U_1;\Delta_{\sT_\Q}) \otimes_{\mathscr{T}_\Q}    \widetilde C_\ast(Y,K;\Delta_{\mathscr{T}_\Q})  \label{eq:loccoeffrasconn}
\end{equation*}
As in Subsection \ref{sec:naturalsharp}, the link $(S^3,H\cup U_1)$ has its base point on $U_1$. The local coefficient system on $(S^3,H\cup U_1)$ is defined just as for $(Y,K)$ but using only $U_1$.

Thus we are in the situation of \eqref{eq:sharpapp}, but with local coefficients. Still we have that $C^\omega_\ast(S^3,H\sqcup U_1;\Delta_{\sT_\Q})$ is isomorphic to a rank 2 module $\sT_\Q\mathbf{v}_+ \oplus \sT_\Q\mathbf{v}_-$ with trivial differential. However, the $v$-map is different here: it is determined by
	\[
			v(\mathbf{v}_+)= 2\mathbf{v}_-, \qquad v(\mathbf{v}_-)=(2T^2+2T^{-2}-4)\mathbf{v}_+.
	\]
	This follows from the computation of $p(u)$ and $q(u)$ in the proof of Proposition 4.1 of \cite{km-rasmussen}. For us, $u=T$, and $v(\mathbf{v}_+)=q(u)\mathbf{v}_-$ and $v(\mathbf{v}_-)=p(u)\mathbf{v}_+$. This leads to the following description of $C_\ast^\#(Y,K;\Delta_{\sT_\Q})$: it is chain homotopy equivalent to the complex
	\begin{equation}\label{eq:sharplocalconnsumchain}
		\left(\widetilde C(Y,K;\Delta_{\sT_\Q})_\ast \oplus \widetilde C(Y,K;\Delta_{\sT_\Q})_{\ast+2},\;  \left[\begin{array}{cc} \widetilde d  & (2T^2+2T^{-2}-4)\chi \\ 2\chi & \widetilde d  \end{array} \right] \right)
	\end{equation}
	where $\chi$ is as defined in the previous subsection. Note that if we set $T=1$ we obtain the mapping cone of $2\chi$, as expected.
	
	 The following was proved in \cite{km-rasmussen} for knots in the 3-sphere.

	\begin{prop}\label{prop:unreducedrank}
		For $(S^3,K)$ any based knot in the $3$-sphere, $I_\ast^\#(S^3,K;\Delta_{\sT})$ has rank 2 as a module over $\sT$, with generators in gradings which differ by $2$ (mod $4$).
	\end{prop}

	We remark that our grading convention in this case is different from the one of \cite{km-rasmussen} by a shift. The two generators for  $I^\#(S^3,K;\Delta_{\sT_\Q})$ are typically denoted $\mathbf{z}_+$ and $\mathbf{z_-}$, and have gradings $1$ and $-1$ (mod $4$) in the convention of \cite{km-rasmussen}. This structure is exploited in \cite{km-rasmussen} to define a concordance invariant $s^\#(K)$ for knots in $S^3$.
	
	The construction of $s^\#(K)$ is as follows. Choose a surface cobordism $S:U_1\to K$ from the unknot $U_1$ to $K$ with a path between basepoints. This induces a map
	\[
		I^\#(S;\Delta_{\sT_\Q})':	I^\#(U_1;\Delta_{\sT_\Q})' \longrightarrow I^\#(K;\Delta_{\sT_\Q})'
	\]
	where the superscript in $I^\#(K;\Delta_{\sT_\Q})'$ indicates that we mod out by torsion. Then there are elements $\sigma_\pm(S)\in \sT_\Q$ such that we have $I^\#(S;\Delta_{\sT_\Q})'(\mathbf{v}_\pm)=\sigma_\pm(S)\mathbf{z}_\pm$ if $g(S)$ is even, and otherwise $I^\#(S;\Delta_{\sT_\Q})'(\mathbf{v}_\pm)=\sigma_\pm(S)\mathbf{z}_\mp$. Pass to the local ring of $\sT_\Q$ at $T=1$, and let $\lambda=T-T^{-1}$. Then there are unique natural numbers $m^\#_\pm(S)$ such that in this local ring $\sigma_\pm(S)$ is up to a unit equal to $\smash{\lambda^{m^\#_\pm(S)}}$. Finally,
	\[
		s^\#(K) := 2g(S) - \frac{1}{2}(m^\#_+(S) + m^\#_-(S)).
	\]
	In particular, $s^\#(K)$ is determined by the cobordism map $I^\#(S;\Delta_{\sT_\Q})$. Now suppose $([0,1]\times S^3,S)$ is a negative definite pair. From the naturality of our connected sum theorem with respect to split cobordisms, the map $I^\#(S;\Delta_{\sT_\Q})$ is induced by the element $(0,\Delta_2(1),1)\oplus (0,\Delta_2(1),1)$ in the chain complex \eqref{eq:sharplocalconnsumchain}. We summarize:

\begin{prop}
	Let $(Y,K)$ be a based knot in an integer homology 3-sphere. Then there is a natural chain homotopy equivalence from $C_\ast^\#(Y,K;\Delta_{\sT_\Q})$ to the complex \eqref{eq:sharplocalconnsumchain}. If $Y=S^3$ and there is a surface cobordism $S:U_1\to K$ in $[0,1]\times S^3$ with negative definite branched cover, then the concordance invariant $s^\#(K)$ is determined by the $\cS$-complex $\widetilde C_\ast(K;\Delta_{\sT_\Q})$ and the element $\Delta_2(1)$ therein induced by $S$.
\end{prop}

\begin{remark}\label{rmk:cobremove}
	Upon developing our theory for more general cobordism maps, we expect that the negative definite cobordism condition on $S$ can be removed, and one can obtain an interpretation of $s^\#(K)$ in terms of the $\cS$-chain homotopy type of $\widetilde C_\ast(K;\Delta_{\mathscr{T}_\Q})$ for an arbitrary knot $K$. $\diamd$
\end{remark}

\subsection{Instanton homology for strongly marked webs}
\newcommand{\web}{L}

In \cite{km-tait}, Kronheimer and Mrowka defined a variation of singular instanton homology for webs in 3-manifolds. A {\emph{web}} in a closed, oriented and connected 3-manifold $Y$ is an embedded trivalent graph $\web\subset Y$. Let $\omega\subset Y$ be an embedded unoriented 1-manifold which may intersect the edges of $\web$ transversely, but misses the vertices of $L$. We say a web is {\emph{admissible}} if it satisfies the non-integrality condition stated above for admissible links, in which the link $L$ is replaced by a web.

Given a web $(Y,K)$, there is an associated {\emph{bifold}}, denoted $\check{Y}$, which is an orbifold whose underlying space is $Y$. The orbifold $\check{Y}$ has points with isotropy $\Z/2$ given by the edges of $K$, and points with isotropy the Klein-four group $V_4$ given by the vertices of $K$; all other points in the orbifold have trivial isotropy. In \cite{km-tait}, the authors pass freely between the web $(Y,\web)$ and its associated bifold $\check{Y}$.

There is a notion of {\emph{marking data}} $\mu$ for a web $(Y,\web)$, which consists of a pair $(U_\mu,E_\mu)$ where $U_\mu\subset Y$ is any subset and $E_\mu \to U_\mu \setminus \web$ is any $SO(3)$ bundle. When $\mu$ is {\emph{strong}}, Kronheimer and Mrowka define the instanton homology
\[
	J(\check{Y};\mu) =J(Y,\web;\mu)
\]
which is a vector space over $\bF=\Z/2$. This is constructed using $SO(3)$ singular instanton gauge theory. The marking data specifies a region for which we only allow determinant-1 gauge transformations. When $(Y,\web,\omega)$ is an admissible web, the marking data is strong if it is all of $Y$. In particular, when $\web$ also has no vertices, in this case
\begin{equation}
	J(\check{Y};\mu)  = I^\omega(Y,\web;\bF)\label{eq:jgroup}
\end{equation}
That is to say, in this case $(\web,\omega)$ is an admissible link, and we recover the instanton homology for admissible links with $\bF$-coefficients. We will write $(C(\check{Y};\mu),d)$ for the chain complex that computes the $\bF$-vector space $J(\check{Y};\mu)$. In general, this is not graded, but see \cite[Section 8.4]{km-tait}. Note in \eqref{eq:jgroup} that $\omega \subset Y$ is determined by the marking data $\mu$. More generally, $(C(\check{Y};\mu),d)$ is defined with coefficients in any ring with characteristic two.

To carry the construction of the map $v$ from Subsection \ref{sec:vmap} over to this setting, we must address the issue of bubbling. For context, we briefly recall why $J(\check{Y};\mu)$ is not defined with general coefficient rings, as explained in \cite[Section 3.3]{km-tait}. In describing $d^2$, we consider the ends of 1-dimensional moduli spaces $\breve{M}(\alpha_1,\alpha_2)_1$. The ends of this moduli space are as before, unless $\alpha_1=\alpha_2$. In this case, in addition to ends of the form $[0,\infty)\times \breve{M}(\alpha_1,\beta)_0 \times  \breve{M}(\beta,\alpha_1)_0$, there is an end of the form
\begin{equation}
	[0,\infty) \times V\times  V_4 \label{eq:webbubbles}
\end{equation}
where $V\subset L$ is the subset of vertices of our web. This end represents bubbling at the vertices of the web, a phenomenon which is absent in the case for links. However, because the Klein-four group $V_4$ has 4 elements, the relation $d^2=0$ holds if we work over any ring of characteristic two, for example.

Now choose a basepoint $p\in L$ away from the vertices. We have a holonomy map $h_{\alpha_1\alpha_2}:\breve{M}(\alpha_1,\alpha_2)_1\to S^1$ defined as before. We must address the possibility of bubbling, represented by the end \eqref{eq:webbubbles}. A connection class on this end with $t\in [0,\infty)$ large is obtained by gluing, along a vertex in $L$, an instanton on $\R^4/V_4$ to the flat connection on $\R\times \check{Y}$ which is the pull-back of $\alpha_1$. The key point is that because our basepoint $p\in L$ is away from the vertices, the holonomy of any such glued instanton is close to the holonomy of the flat connection on $\R\times \check{Y}$ determined by $\alpha_1$, the latter of which is trivial.

In conclusion, we may modify our holonomy map to a map $H_{\alpha_1,\alpha_2}$ just as in Subsection \ref{sec:vmap}, without making any modifications along any end where bubbling occurs. Then our observation from the previous paragraph implies that the cut-down moduli space $\{[A]\in\breve{M}(\alpha_1,\alpha_2)_1: H_{\alpha_1,\alpha_2}([A])=h\}$ for a generic $h\in S^1\setminus \{1\}$ is, as before, a finite set of points. We may then define the endomorphism $v$ on $C(\check{Y};\mu)$. Proposition \ref{prop:vmap2} continues to hold in this setting, and we may form the mapping cone complex of $v$:
\begin{equation}\label{eq:mappingconecxweb}
	\widetilde C(Y,L;\mu) := C(\check{Y};\mu) \oplus C(\check{Y};\mu), \hspace{1cm} \widetilde d = \left[\begin{array}{cc} d & 0 \\ v & d \end{array}\right]
\end{equation}
We have the following variation of Theorem \ref{thm:connectedsum} when one of the based knots is replaced by a strongly marked web; it generalizes Theorem \ref{thm:oneadmis} over $\bF$.

\begin{theorem}  {\emph{\bf{(Connected Sum Theorem for a knot and a strongly marked web)}}} \label{thm:oneweb} Let $(Y,K)$ be a based knot in an integer homology 3-sphere and $(Y',L')$ a based web with strong marking data $\mu$ containing the basepoint of $L'$. Let $\mu^\#$ be marking data on the connected sum formed by connect summing the marking data which is all of $Y$ with $\mu$. There is a chain homotopy equivalence of chain complexes over $\bF$:
\[
	\widetilde C(Y\# Y',K\# L';\mu^\#) \simeq  \widetilde C(Y,K;\bF)\otimes \widetilde C(Y',L';\mu)
\]
 Furthermore, the tensor product is naturally isomorphic to a mapping cone complex, making the chain homotopy equivalence one of mapping cone complexes. 
The equivalence is natural, up to mapping cone chain homotopies, with respect to split cobordisms.
\end{theorem}

We may also consider the case of a connected sum between two strongly marked webs. The following generalizes Theorem \ref{thm:connsumadlinks} over $\bF$.

\begin{theorem}  {\emph{\bf{(Connected Sum Theorem for strongly marked webs)}}}  Let $(Y,L)$, $(Y',L')$ be based webs with strong marking data $\mu$, $\mu'$ containing the basepoints of $L$, $L'$, respectively. There is a chain homotopy equivalence of mapping cone chain complexes over $\bF$:
\[
	\widetilde C(Y\# Y',L\# L';\mu\cup \mu') \simeq  \widetilde C(Y,L;\mu)\otimes \widetilde C(Y',L';\mu').
\]
Here $\mu\cup \mu'$ is marking data obtained from gluing $\mu$ and $\mu'$ on the connected sum.
\end{theorem}

\begin{figure}[t]
\centering
\includegraphics[scale=.65]{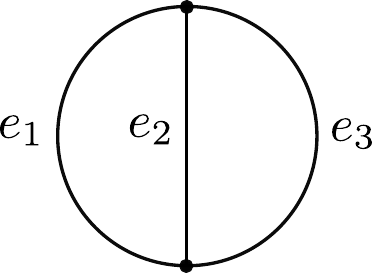}
\caption{}
\label{fig:theta}
\end{figure}

\subsection{Connect summing with a theta web}\label{sec:thetawebz2}

Let $(Y,K)$ be a based knot in an integer homology 3-sphere. Let $(S^3,\Theta)$ be the theta web, as given in Figure \ref{fig:theta}, with edges $e_1$, $e_2$ and $e_3$. Consider the connected sum $(Y\# S^3,K\# \Theta)$. The marking data $\mu=(U_\mu,E_\mu)$, where $U_\mu$ is all of $Y\# S^3$, and $E_\mu$ is trivial, is strong. We use the same notation $\mu$ for the marking data restricted to $(S^3,\Theta)$. The argument in Section 2.2 of \cite{km-def} provides an isomorphism
\begin{equation}
	J(Y\# S^3,K\# \Theta;\mu) \cong I^\natural(K;\bF). \label{eq:thetaisnatural}
\end{equation}
We may recover this, and obtain a chain-level refinement, by applying Theorem \ref{thm:oneweb}. The complex $C_\ast(\Theta;\mu)$ is in fact $\Z/2$-graded, and has one generator in degree 0. By the computation in Section 4.4 of \cite{km-def}, the $v$-map on this complex vanishes, and the mapping cone complex $\widetilde C_\ast(\Theta;\mu)$ is thus isomorphic to the $\Z/2$-graded vector space $\bF_{(0)} \oplus \bF_{(1)}$ with trivial differential. In other words, $\widetilde C_\ast(\Theta;\mu)$ is identical to the Hopf link complex $\widetilde C^\omega_\ast(H;\bF)$. Applying Theorem \ref{thm:oneweb} gives a chain homotopy equivalence 
\[
	\widetilde C_\ast(Y\# S^3,K\# \Theta;\mu) \simeq \widetilde C_\ast(Y,K;\bF) \otimes \widetilde C_\ast(\Theta;\mu),
\]
of $\Z/2$-graded mapping cone complexes, similar to the proof of Theorem \ref{thm:tildeconnsum}. In particular, we have a chain homotopy equivalence of $\Z/2$-graded complexes over $\bF$,
\[
	C_\ast(Y\# S^3,K\# \Theta;\mu) \simeq C_\ast^\natural(Y, K;\bF)
\]
which is natural with respect to split cobordism maps up to chain homotopy. Taking homology, we recover \eqref{eq:thetaisnatural}. Similar reasoning allows us to recover $I^\#(K;\bF)$ by taking the connected sum of $K$ with an unknot union a theta web.

\subsection{Local coefficients from theta webs}

More recently, in \cite{km-concordance} Kronheimer and Mrowka have used their instanton homology of webs to extract new concordance invariants. Here we explain how a variation of our connected sum theorem implies some structural results about these invariants.

\begin{figure}[t]
\centering
\includegraphics[scale=.57]{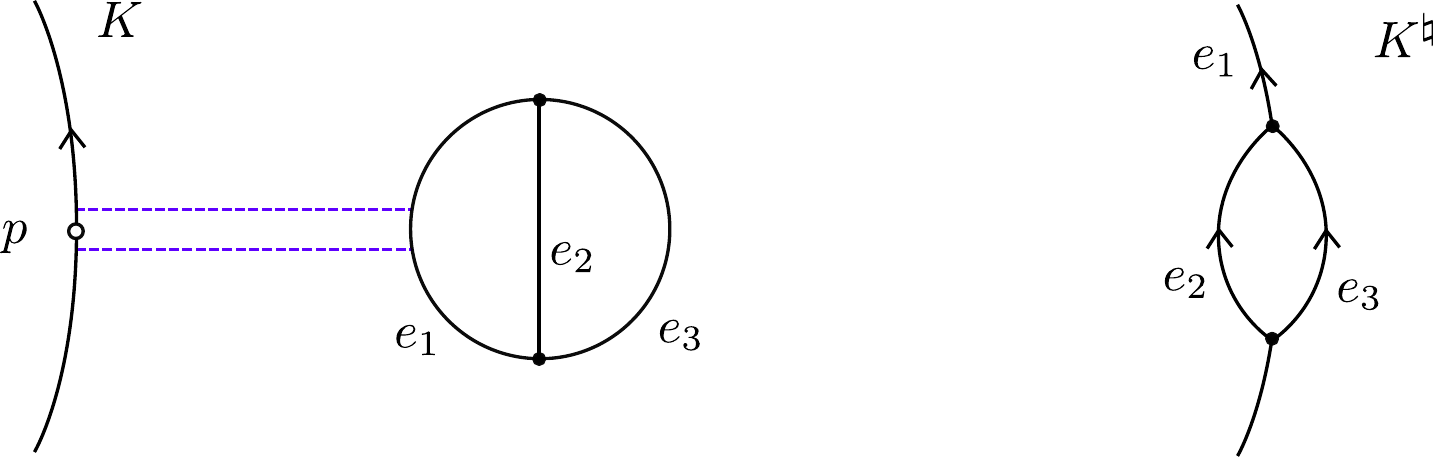}
\caption{On the left, we take our oriented knot $K$ and perform a connected sum with a standard theta web near the basepoint $p\in K$. The result is $K^\natural$ on the right, and we orient the three edges of $K^\natural$ as indicated.}
\label{fig:thetasum}
\end{figure}

Let $(Y,K)$ be a based knot in an integer homology 3-sphere. Consider the ring 
 \[
	\mathscr{S}_{BN}=\bF[T_1^{\pm 1},T_2^{\pm 1}, T_3^{\pm 1}]
\]
Now orient $K$, and in a small ball surrounding the basepoint $p\in K$, take the connected sum of $K$ with a theta web. We call the result $K^\natural$. This is a web with three edges $e_1$, $e_2$, $e_3$, each oriented in a fashion determined by the orientation of $K$, as in Figure \ref{fig:thetasum}. Let $\mu$ be the marking data of the previous subsection. The instanton complex $C(Y,K^\natural;\mu)$ can be upgraded to a complex with a local coefficient system, 
\begin{equation}
	(C_\ast^\natural(Y,K;\Delta_{BN}),d) \label{eq:naturalbncoeff}
\end{equation}
which is a $\Z/2$-graded chain complex over $\mathscr{S}_{BN}$. The local coefficient system $\Delta_{BN}$ of $\sS_{BN}$-modules is defined using holonomies along the three edges of the web $K^\natural$: the variable $T_1$ is associated to the arc containing the edge formed by the merging of $K$ and $e_1$, while $T_2$ and $T_3$ are associated to the edges $e_2$ and $e_3$, respectively. We refer to \cite{km-barnatan} for more details. The homology of the complex \eqref{eq:naturalbncoeff} is denoted
\[
	I_\ast^\natural (Y,K;\Delta_{BN}) = H_\ast\left(C^\natural(Y,K;\Delta_{BN}),d^\natural\right)
\]
and is a $\Z/2$-graded $\mathscr{S}_{BN}$-module. The recycling of the notation $I^\natural$ here is justified, as is in \cite{km-barnatan}, by the observations from Subsection \ref{sec:thetawebz2}. The $v$-map is defined in this setting, and there are no reducibles, so we may as usual form the mapping cone complex of $v$:
\[
	(\widetilde C_\ast^\natural(Y,K;\Delta_{BN}), \widetilde d^\natural)
\]
Let $\sT_\bF=\sT\otimes \bF = \bF[T^{\pm 1}]$. A variation of our connected sum theorem implies the following chain homotopy equivalence of $\Z/2$-graded mapping cone complexes over $\mathscr{S}_{BN}$:
\begin{equation}
	\widetilde C_\ast^\natural(Y,K;\Delta_{BN}) \simeq \widetilde C_\ast(Y,K;\Delta_{{\mathscr{T}_\bF}}) \otimes_{\mathscr{T}_\bF} \widetilde C_\ast^\natural(S^3,U_1;\Delta_{BN}) \label{eq:connsumloctheta}
\end{equation}
where $U_1$ is an unknot. Here and in what follows, $\mathscr{S}_{BN}$ is viewed as a $\mathscr{T}_\bF$-algebra by identifying $T=T_1$. The mapping cone complex $\widetilde C_\ast^\natural(S^3,U_1;\Delta_{BN})$ is computed in Section 4.4 of \cite{km-def}: there is one irreducible critical point for the theta web, so we have $C_\ast^\natural(S^3,U_1;\Delta_{BN})=\sT_\bF$, while the $v$-map $\mathscr{T}_\bF\to \mathscr{T}_\bF$ is multiplication by $P\in \mathscr{T}_\bF$, where
\begin{equation}
	P := T_1T_2T_3 + T_1^{-1}T_2^{-1} T_3 + T_1^{-1}T_2T_3^{-1} + T_1T_2^{-1}T_3^{-1} \label{eq:defnofp}
\end{equation}
Unravelling the right-hand side of the equivalence \eqref{eq:connsumloctheta}, we obtain 
\[
	\left(C^\natural_\ast (Y, K;\Delta_{BN}),d^\natural\right) \simeq  \left( \widetilde C_\ast (Y, K;\Delta_{\mathscr{T}_\bF})  \otimes_{\mathscr{T}_\bF} \mathscr{S}_{BN} , \;\;  \widetilde d \otimes 1_{\mathscr{S}_{BN}} + P\cdot\chi\right)
\]
where $P\cdot\chi$ is multiplication by $P$ between two summands of $\widetilde C_\ast (Y, K;\Delta_{\mathscr{T}_\bF})  \otimes_{\mathscr{T}_\bF} \mathscr{S}_{BN}$. That is, according to its decomposition as an $\cS$-complex over $\mathscr{S}_{BN}$, we have
\begin{equation}
P\cdot\chi = \left( \begin{array}{ccc} 0 & 0 & 0 \\ P & 0 & 0 \\ 0 & 0 & 0\end{array}\right)\label{eq:pmap}
\end{equation}
In conclusion, the complex $C^\natural_\ast (Y, K;\Delta_{BN})$ and its homology $I^\natural_\ast (Y, K;\Delta_{BN})$ featured in \cite{km-barnatan, km-concordance}, while defined as modules over the ring $\mathscr{S}_{BN}=\bF[T_1^{\pm 1}, T_2^{\pm 1}, T_3^{\pm 1}]$ in three variables, is entirely determined up to chain homotopy equivalence by the $\cS$-complex $\widetilde C_\ast (Y, K;\Delta_{\mathscr{T}_\bF})$ over the ring $\mathscr{T}_\bF$ in one variable. In fact, in the next subsection we will see a more precise statement at the level of homology. A similar argument to that of Proposition \ref{prop:unreducedrank} yields the following, which is proved for $Y=S^3$ by a different method in \cite{km-barnatan}.

\begin{prop}
	For $(Y,K)$ any based knot in an integer homology 3-sphere, the homology group $I_\ast^\natural(Y,K;\Delta_{BN})$ has rank 1 as a module over $\sS_{BN}$.
\end{prop}

We now assume that $Y$ is the 3-sphere, and omit it from the notation. The importance of the modules $I^\natural(K;\Delta_{BN})$, from the viewpoint of \cite{km-barnatan, km-concordance}, is two-fold. First is a connection to combinatorial link homology: Corollary 6.6 of \cite{km-barnatan} gives a spectral sequence whose $E_2$-page is a version of Bar-Natan's Khovanov homology over $\mathscr{S}_{BN}$ and which converges to $I^\natural(K;\Delta_{BN})$. Second, there is defined in \cite{km-concordance} a concordance invariant $z_{BN}^\natural(K)$ in the setting of $I^\natural(K;\Delta_{BN})$, which we now briefly review.

Let $\text{Frac}(\mathscr{S}_{BN})$ be the quotient field of $\mathscr{S}_{BN}$. A {\emph{fractional ideal}} is an $\mathscr{S}_{BN}$-module $M\subset \text{Frac}(\mathscr{S}_{BN})$ such that there is some $s\in \mathscr{S}_{BN}$ with $s M\subset \mathscr{S}_{BN}$. Let $S:U_1\to K$ be a cobordism of knots from the unknot to $K$. We have an induced map of rank-1 $\mathscr{S}_{BN}$-modules
\[
	I^\natural(S;\Delta_{BN})':I^\natural(U_1;\Delta_{BN})' \longrightarrow I^\natural(K;\Delta_{BN})'
\]
where $I^\natural(K;\Delta_{BN})'$ denotes the $\mathscr{S}_{BN}$-module $I^\natural(K;\Delta_{BN})$ modulo torsion. Then
\[
	z^\natural_{BN}(K) := P^{g}\cdot \left[I^\natural(K;\Delta_{BN})':\text{im}I^\natural(S;\Delta_{BN})'\right] \subset \text{Frac}(\mathscr{S}_{BN})
\]
where $g$ is the genus of $S$, and $[N:M]$ is the generalized module quotient,
\[
	[N:M] := \left\{ a/b \in \text{Frac}(\mathscr{S}_{BN}): \; aM \subset bN \right\}.
\]
The ideal $z^\natural_{BN}(K)$ is a fractional ideal, and is proven in \cite[Section 5]{km-concordance} to be a concordance invariant of the knot $K$. 

Now suppose that $( [0,1]\times S^3, S)$ is a negative definite pair. The existence of $S$ is equivalent to the assumption in \eqref{slice-genus-cond} on the slice genus of $K$. Then we have a morphism of $\cS$-complexes $\widetilde C(S;\Delta_{\mathscr{T}_\bF}):\widetilde C(U_1;\Delta_{\mathscr{T}_\bF})\to \widetilde C(K;\Delta_{\mathscr{T}_\bF})$. Upon tensoring with $\widetilde C(\Theta;\Delta_{BN})$ and using the naturality of our connected sum theorem, we obtain that $C^\natural(S;\Delta_{BN})$ is chain homotopy equivalent to the map
\[
	 \left(0, \Delta_2\otimes 1, 1 \right):\mathscr{T}_\bF\otimes_{\mathscr{T}_\bF} \mathscr{S}_{BN} \longrightarrow \widetilde C_\ast (Y, K;\Delta_{\mathscr{T}_\bF})  \otimes_{\mathscr{T}_\bF} \mathscr{S}_{BN}
\]
We summarize some of our observations in the following.

\begin{prop}\label{prop:inaturalthetastructure}
	Let $(Y,K)$ be a based knot in an integer homology 3-sphere. The module $I^\natural(Y,K;\Delta_{BN})$ over $\mathscr{S}_{BN}=\bF[T_1^{\pm 1}, T_2^{\pm 1}, T_3^{\pm 1}]$ is determined by the $\cS$-complex $\widetilde C_\ast(Y,K;\Delta_{\mathscr{T}_\bF})$ over $\mathscr{T}_\bF=\bF[T_1^{\pm 1}]$. Specifically, we have a $\Z/2$-graded isomorphism
	\[
		I_\ast^\natural(Y,K;\Delta_{BN}) \cong H_\ast \left( \widetilde C_\ast (Y, K;\Delta_{\mathscr{T}_\bF})  \otimes_{\mathscr{T}_\bF} \mathscr{S}_{BN} ,  \;\;  \widetilde d \otimes 1_{\mathscr{S}_{BN}} + P\cdot\chi\right)
	\]
	where $P\cdot\chi$ is given by \eqref{eq:pmap}. Suppose $Y=S^3$ and that there is a surface cobordism $S:U_1\to K$ in $[0,1]\times S^3$ with negative definite double branched cover. Then the image of the induced map $I^\natural(S;\Delta_{BN})$ in $I^\natural(K;\Delta_{BN})$ corresponds to the inclusion of the element
	\[
		\widetilde \lambda_S(1) = \left(0,\Delta_2(1),1\right) \in \widetilde C_\ast (Y, K;\Delta_{\mathscr{T}_\bF}).
	\]
	Thus in this case, the concordance invariant $z^\natural_{BN}(K)$ is determined by the $\cS$-complex $\widetilde C_\ast(Y,K;\Delta_{\mathscr{T}_\bF})$ over $\mathscr{T}_\bF=\bF[T_1^{\pm 1}]$ and the element $\Delta_2(1)$ therein.
\end{prop}

All of the above may be carried out in the framework of the unreduced theory $I^\#(Y,K)$ with local coefficients induced by the theta web. In this case, we consider the ring
\[
	\sR^\# = \bF[T_0^{\pm 1}, T_1^{\pm 1}, T_2^{\pm 1}, T_3^{\pm 1}]
\]
The group $I^\#(Y,K;\Delta_{\sR^\#})$ is a module over $\sR^\#$ and its local coefficient system is defined by holonomies around $K$ and the three arcs of the theta web; the variable $T_0$ is associated with $K$, and $T_1,T_2,T_3$ with the theta web as before. It is the homology of
\begin{equation}
	(C_\ast^\#(Y,K;\Delta_{\sR^\#}), d^\#)\label{eq:sharprsharpchain}
\end{equation}
which is a chain complex associated to $(Y,K)\#(S^3, U_1\sqcup \Theta)$ where as before the connected sum involves $U_1$, not $\Theta$, and which has marking data $\mu$ containing all of $Y\# S^3$, and local coefficient system $\Delta_{\sR^\#}$. We form the mapping cone complex with respect to the $v$-map, and apply a variation of our connected sum theorem, to obtain the following equivalence:
\begin{equation*}
	\widetilde C_\ast^\#(Y,K;\Delta_{\sR^\#}) \simeq \widetilde C_\ast(Y,K;\Delta_{{\mathscr{T}_\bF}}) \otimes_{\mathscr{T}_\bF} \widetilde C_\ast^\#(S^3,U_1;\Delta_{\sR^\#}) 
\end{equation*}
Here $\sR^\#$ is a $\sT_\bF$-module by identifying $T=T_0$. Now just as in Subsection \ref{sec:naturalsharp}, the complex $C_\ast^\#(S^3,U_1;\Delta_{\sR^\#})$ may be taken to have two generators in even grading, with trivial differential. The $v$-map in this case is multiplication by $P$. Indeed, this follows from \cite[Proposition 5.9]{km-barnatan}, after identifying the $v$-map with the map induced by a genus 1 cobordism $U_1\to U_1$ (see Remark \ref{v-maps}). Then \eqref{eq:sharprsharpchain} is chain homotopy equivalent to
\begin{equation}
	\left( \widetilde C_\ast (Y, K;\Delta_{\mathscr{T}_\bF})  \otimes_{\mathscr{T}_\bF} \mathscr{R}^\# ,  \;\;  \widetilde d \otimes 1_{\mathscr{R}^\#} + P\cdot\chi\right)^{\oplus 2}\label{eq:sharpthetabigcomplex}
\end{equation}
In particular, the theory naturally splits into two chain complexes. Further, when we set $T_0=T_1$ we recover two copies of the chain complex computed above for $I^\natural(Y,K;\Delta_{BN})$. 

Similar to Proposition \ref{prop:inaturalthetastructure}, we obtain that the module $I^\#(Y,K;\Delta_{\sR^\#})$ over $\mathscr{R}^\#$ is determined by the $\cS$-complex $\widetilde C_\ast(Y,K;\Delta_{\mathscr{T}_\bF})$ defined over $\mathscr{T}_\bF=\bF[T_0^{\pm 1}]$. Specifically, we have a $\Z/2$-graded isomorphism between $I^\#(Y,K;\Delta_{\sR^\#})$ and the homology of \eqref{eq:sharpthetabigcomplex}. Furthermore, if $Y=S^3$ and there is a cobordism of pairs $([0,1]\times S^3, S)$ from $U_1$ to $K$ which is negative definite, then the concordance invariant $z^\#(K)$ of \cite{km-barnatan} is determined by the cobordism map on $\cS$-complexes.

\subsection{Relations to equivariant homology groups}\label{subsec:relntoequivgrps}

Several of the constructions of the above chain complexes bare similarities to the equivariant homology groups discussed in Section \ref{sec:equivtheories}. To begin making these connections more precise, recall that given an $\cS$-complex $(\widetilde C_\ast, \widetilde d, \chi)$ over a ring $R$, the equivariant complex $(\hrC_*,\widehat d)$ over $R[x]$ is defined as follows, where $\widehat d$ is defined on $\widetilde C\subset \widehat C$ and extended $R[x]$-linearly:
\[
	\hrC_*=\widetilde C_\ast \otimes_R R[x], \qquad \widehat d = -\widetilde d + x\cdot \chi
\]
We slightly extend this construction to incorporate base changes of $R[x]$, as follows. Suppose $S$ is another ring, and $\varphi:R[x]\to S$ is a ring homomorphism. Define $(\hrC_*^\varphi,\widehat d^\varphi)$ by:
\begin{align*}
   \hrC_*^\varphi =S\otimes_R \widetilde C_*, \qquad  \widehat d^\varphi(s \cdot \zeta)  =-s \cdot \widetilde d\zeta+\phi(x)s\cdot \chi(\zeta)
\end{align*}
Here $s\in S$, $\zeta\in \widetilde C$, and $S$ is considered an $R$-module by restriction. In other words, we have the identification $\hrC_*^\varphi=\hrC_* \otimes_{R[x]}S$ as chain complexes over $S$. Note that when $\phi:R[x]\to R$ sends $x$ to zero, the complex $\widehat C_\ast^\phi$ is naturally identified with $\widetilde C_\ast$.

As a simple example, recall that Theorem \ref{thm:sharpconnsum} expresses the chain homotopy type of $C_\ast^\#(Y,K)$ as the mapping cone of $2\chi$ acting on $\widetilde C_\ast(Y,K)$. We may write this as
\[
	C_\ast^\#(Y,K) \simeq \widehat C^\phi_\ast(Y,K)
\]
where $\phi:\Z[x]\to \Z[x]/(x^2)$ is the ring homomorphism which is determined by sending $x$ to the equivalence class $2x$ (mod $x^2$).

For a more interesting example, we consider $I^\#(Y,K;\Delta_{\sR^\#})$ as discussed in the previous subsection. We saw there that the chain complex for this group is chain homotopy equivalent to the chain complex given by \eqref{eq:sharpthetabigcomplex}. From this we obtain
\begin{equation}
	C^\#_\ast(Y,K;\Delta_{\sR^\#}) \simeq \widehat C^\phi_\ast(Y,K;\Delta_{\sT_\bF})^{\oplus 2} = \left(\widehat C_\ast(Y,K;\Delta_{\sT_\bF})\otimes_{\sT_{\bF}[x]} \sR^\#\right)^{\oplus 2} \label{eq:bnsharpequiv}
\end{equation}
where $\phi:\sT_\bF[x]\to \sR^\#=\bF[T_{0}^{\pm 1},T_1^{\pm 1},T_2^{\pm 1}, T_3^{\pm 1}]$ is the $\sT_\bF=\bF[T_0^{\pm 1}]$-linear homomorphism determined by sending $x$ to $P$, where $P$ is given as before by \eqref{eq:defnofp}.

Note that $\bF[T_1^{\pm 1},T_2^{\pm 1}, T_3^{\pm 1}]$ a free module over $\bF[P]$. (This is an immediate consequence of Proposition \ref{SBN-free} below.) Tensoring with the ring $\bF[T_0^{\pm 1}]$ then shows that $\sR^\#$ is a free module over $\sT_\bF[x]$ with the above module structure. As a consequence, the relationship \eqref{eq:bnsharpequiv} holds at the level of homology.

\begin{cor}\label{equiv-sharp}
	Let $(Y,K)$ be a based knot in an integer homology 3-sphere. The module 
	$I^\#(Y,K;\Delta_{\sR^\#})$ over $\sR^\#$ is isomorphic to 
	$(\widehat I(Y,K;\Delta_{\mathscr{T}_\bF})\otimes_{\mathscr{T}_\bF[x]} \sR^\#)^{\oplus 2}$ where the 
	ring structure of $\sR^\#$ over $\mathscr{T}_\bF[x]$ is given by sending $x\mapsto P$.
\end{cor}

\begin{remark}
	In \cite{km-barnatan}, an operator $\Lambda$ is defined on $I^\#(Y,K;\Delta_{\sR^\#})$, associated to the identity cobordism with a dot placed on the singular surface $[0,1]\times K$. The $\sR^\#$-module structure of $I^\#(Y,K;\Delta_{\sR^\#})$ then lifts to an $\sF$-module structure, where
	\[
		\sF = \sR^\#[\Lambda]/(\Lambda^2 + P \Lambda + Q).
	\]
	We can lift the isomorphism of Corollary \ref{equiv-sharp} to one of $\sF$-modules. The result is
	\[
		I^\#(Y,K;\Delta_{\sR^\#}) \cong \widehat I(Y,K;\Delta_{\mathscr{T}_\bF})\otimes_{\mathscr{T}_\bF[x]} I^\#(U_1;\Delta_{\sR^\#})
	\]
	where the $\sF$-module structure is induced by that of $I^\#(U_1;\Delta_{\sR^\#})$. This follows from viewing $\Lambda$ as induced by a cobordism (with a dot) and the naturality of the connected sum theorem. As $I^\#(U_1;\Delta_{\sR^\#})$ is isomorphic as an $\sF$-module to $\sF$, $I^\#(Y,K;\Delta_{\sR^\#})$ is an $\sF$-module obtained from  $\widehat I(Y,K;\Delta_{\mathscr{T}_\bF})$ by the base change $\sT_{\bF}[x]\to \sF$ which is the base change of Corollary \ref{equiv-sharp} followed by inclusion $\sR^\#\to \sF$. $\diamd$
\end{remark}

Similar results hold for the reduced theory $I^\natural(Y,K;\Delta_{BN})$, starting from Proposition \ref{prop:inaturalthetastructure}. The description of the chain homotopy-type of $C^\natural_\ast(Y,K;\Delta_{\sT_\bF})$ given there may be represented as a chain homotopy equivalence
\begin{equation}
	C^\natural_\ast(Y,K;\Delta_{BN}) \simeq \widehat C^\phi_\ast(Y,K;\Delta_{\sT_\bF}) = \widehat C_\ast(Y,K;\Delta_{\sT_\bF})\otimes_{\sT_{\bF}[x]} \sS_{BN}\label{eq:bnnaturalequiv}
\end{equation}
where $\phi:\sT_\bF[x]\to \sS_{BN}=\bF[T_1^{\pm 1},T_2^{\pm 1}, T_3^{\pm 1}]$ is the $\sT_\bF=\bF[T_1^{\pm 1}]$-linear homomorphism determined by sending $x$ to $P$. This module structure is elucidated by the following.

\begin{prop}\label{SBN-free}
	The ring $\mathscr{S}_{BN}=\bF[T_1^{\pm 1}, T_2^{\pm 1}, T_3^{\pm 1}]$ is free as a module over
	$\bF[T_1^{\pm 1},P]$.
\end{prop}

As a consequence, the relationship \eqref{eq:bnnaturalequiv} also holds at the level of homology.

\begin{cor}\label{equiv-natural}
	Let $(Y,K)$ be a based knot in an integer homology 3-sphere. The module 
	$I^\natural(Y,K;\Delta_{BN})$ over $\mathscr{S}_{BN}$ is isomorphic to 
	$\widehat I(Y,K;\Delta_{\mathscr{T}_\bF})\otimes_{\mathscr{T}_\bF[x]} \mathscr{S}_{BN}$ where the 
	ring structure of $\mathscr{S}_{BN}$ over $\mathscr{T}_\bF[x]$ is given by sending $x\mapsto P$.
\end{cor}

The proof of Proposition \ref{SBN-free} needs some preparation. Firstly, we fix the dictionary order on $\Z_{\geq 0}\times \Z_{\geq 0}$ by declaring $(i,j)\geq (i',j')$ if either $i>i'$ or $i=i'$, $j>j'$. For a non-zero element $Q=\sum R_{i,j}T_2^iT_3^j$ of the ring $\mathscr{S}_{BN}$ where $R_{i,j}\in \bF[T_1^{\pm 1}]$, define
\[
  {\rm Deg}(Q):=\max_{i,j}\{(|i|,|j|)\mid R_{i,j}\neq 0\}\in \Z_{\geq 0}\times \Z_{\geq 0}
\] 
where the maximum is defined with respect to the dictionary order. We also define $L(Q)$, the leading terms of such a non-zero $Q$, to be the following expression:
\[
  L(Q):=\sum_{(|i|,|j|)={\rm Deg}(Q)} R_{i,j}T_2^iT_3^j \in \sS_{BN}
\]
Proposition \ref{SBN-free} is a consequence of the following two lemmas.

\begin{lemma}
	Suppose a subset $G\subset \mathscr{S}_{BN}$ is given such that the set
	\[
	  \{L(gP^i)\mid g\in G,\,i\geq 0\}
	\] 
	forms a basis of $\mathscr{S}_{BN}$ over  $\bF[T_1^{\pm 1}]$. 
	Then $\mathscr{S}_{BN}$ is free over $\bF[T_1^{\pm 1},P]$ with basis $G$.
\end{lemma}
\begin{proof}
	Let $Q\in \mathscr{S}_{BN}$ be non-zero. By assumption, there are $R_0,\ldots,R_n\in \bF[T_1^{\pm 1}]$ and $g_0,\ldots,g_n\in G$
	such that the leading terms of $Q$ are given as the sum
	\[
	  L(Q)=\sum_{i}R_iL(g_iP^i)
	\]
	In particular, either $Q=\sum_{i}R_ig_iP^i$ or we have the following:
	\[{\rm Deg}(Q-\sum_{i}R_ig_iP^i)<{\rm Deg}(Q).\]
	Thus by induction we can write $Q$ as a linear combination of the elements in $G$ over the ring 
	$\bF[T_1^{\pm 1},P]$. Next, suppose we have a relation of the form
	\[
	  \sum_{m=1}^{N}R_mg_mP^{k_m}=0
	\]
	with $R_m\in \bF[T_1^{\pm 1}]$ not all zero, $k_m\in \Z_{\geq 0}$, and $g_m\in G$. Then we define $(i_0,j_0)\in \Z_{\geq 0}\times \Z_{\geq 0}$ to be the maximum of 
	${\rm Deg}(g_mP^{k_m})$ among all $m$ that $R_m\neq 0$. Therefore, 
	\[
	  \sum_{{\rm Deg}(g_mP^{k_m})=(i_0,j_0)}R_mL(g_m P^{k_m})=0.
	\]
	This contradicts our assumption and implies that $R_m=0$ for all $m$. In particular, $G$ gives an
	$\bF[T_1^{\pm 1},P]$-basis for $\mathscr{S}_{BN}$.
\end{proof}

\begin{lemma}
	Let $G$ be the following subset of $\mathscr{S}_{BN}$:
	\[
	  G:=\{T_2^m,\,T_3^m,\,T_2^nT_3^{-1},\,T_2^{-n}T_3,\,T_2^{-1}T_3^{n},\,T_2T_3^{-n},\,T_2T_3\mid m\in \Z,
	  n\in \Z_{\geq 1}\}.
	\]
	Then $\{L(gP^k)\mid g\in G,\,k\geq 0\}$ is an $\bF[T_1^{\pm 1}]$-basis for $\mathscr{S}_{BN}$.
\end{lemma}
\begin{proof}
	For a given $(i,j)\in \Z_{\geq 0}\times \Z_{\geq 0}$, we characterize all the elements of the form 
	$gP^k$ with $g\in G$ having ${\rm Deg}(gP^k)=(i,j)$. Because of the symmetrical role of $T_2$ and $T_3$, we assume 
	that $i\geq j$. In the following table we list all such elements $gP^k$ with degree $(i,j)$.
	
	\begin{center}
\def\arraystretch{.35}% 
\begin{tabular}{|c|l|}
 \hline
 &\\
  $i=j=0$&$1\cdot P^0$\\ 
 &\\
 \hline
  &\\
  $i>0$, $j=0$&$T_2^i\cdot P^0$,\hspace{.10cm} $T_2^{-i}\cdot P^0$\\ 
   &\\
 \hline
  &\\
   $i=j>0$&$P^i$,\hspace{.10cm} $T_2T_3P^{i-1}$,\hspace{.10cm} $T_2T_3^{-1}P^{i-1}$,\hspace{.10cm}$T_2^{-1}T_3P^{i-1}$\\ 
  &\\
 \hline
   &\\
   $i>j>0$&$T_2^{i-j}P^j$,\hspace{.10cm} $T_2^{j-i}P^j$,\hspace{.10cm} $T_2^{i-j+1}T_3^{-1}P^{j-1}$,\hspace{.10cm}$T_2^{j-i-1}T_3P^{j-1}$\\ 
  &\\
 \hline
\end{tabular}
\end{center}	
It is clear that in each case the leading terms of the listed elements give an $\bF[T_1^{\pm 1}]$-basis of
	\[
	  \bF[T_1^{\pm 1}]\{T_2^iT_3^j, T_2^iT_3^{-j},T_2^{-i}T_3^j,T_2^{-i}T_3^{-j}\},
	\]  
	the direct sum of which give a decomposition of $\sS_{BN}$ into $\bF[T_1^{\pm 1}]$-modules.
\end{proof}

\subsection{Remarks on some concordance invariants}\label{subsection:remarks-conc}

For a knot $K$ whose slice genus satisfies the assumption in \eqref{slice-genus-cond}, we have given rather specific recipes for how to describe the invariants of \cite{km-rasmussen, km-concordance} in terms of our theory of $\cS$-complexes with local coefficients. It is natural to ask whether the concordance invariants $s^\#(K)$, $z_{BN}^\natural(K)$, $z^\#(K)$ are related to the various concordance invariants constructed in this paper. A closely related question is the following:
\begin{question}
	Are the invariants $s^\#(K)$, $z_{BN}^\natural(K)$, $z^\#(K)$ determined by the local equivalence class of the $\cS$-complex 
	$\widetilde C(S^3,K;\Delta_\sT)$ with local coefficients?
\end{question}

We now suggest a possible recipe for how the above concordance invariants may be derived more directly from our equivariant homology theories. We first mimic the definition of $z^\natural_{BN}(K)$, as far as we can, in the setting of $\widehat{I}(K;\Delta_{\sT_\bF})$, motivated by Corollary \ref{equiv-natural}. To ensure that we have a cobordism map at our disposal, we {\emph{assume}} that there is a connected, oriented surface cobordism $S:U_1\to K$ in $[0,1]\times S^3$ with negative definite branched double cover. Then we have an induced map
\[
	\widehat{I}(S;\Delta_{\sT_\bF})':\widehat{I}(U_1;\Delta_{\sT_\bF})' \longrightarrow \widehat{I}(K;\Delta_{\sT_\bF})'
\]
where the prime superscripts indicate that we mod out by the torsion elements over $\sT_\bF[x]$. Let $g$ be the genus of $S$. We define
\begin{equation}\label{eq:zhatdef}
	\widehat{z}(K) := x^g\cdot [ \widehat{I}(K;\Delta_{\sT_\bF})' : \text{im}\, \widehat{I}(S;\Delta_{\sT_\bF})' ] \subset \text{Frac}(\sT_\bF[x])
\end{equation}
Because the equivariant homology $\widehat{I}(K;\Delta_{\sT_\bF})$ is a rank 1 module over $\sT_\bF[x]$, it is isomorphic as a module to an ideal $I\subset \sT_\bF[x]$. Then $\widehat z(K)$ may be described as
\[
	\widehat z(K) = x^g\zeta^{-1}\cdot I \subset \text{Frac}(\sT_\bF[x])
\]
where $\zeta= \widehat{I}(S;\Delta_{\sT_\bF})(1)$. Thus $\widehat{z}(K)$ is a fractional ideal for the ring $\sT_\bF[x]=\bF[T^{\pm 1}, x]$. Now the ring homomorphism $\sT_\bF[x]\to \sS_{BN}$ induced by sending $T\mapsto T_1$ and $x\mapsto P$ induces a map from fractional ideals of $\sT_\bF[x]$ to fractional ideals of $\sS_{BN}$. By Proposition \ref{prop:inaturalthetastructure} and the naturality of the isomorphism in Corollary \ref{equiv-natural}, we conclude that $\widehat{z}(K)$ is sent to $z^\natural(K)$ under this correspondence.

Similar remarks hold for the relationship between $\widehat{z}(K)$ and $z^\#(K)$. Consider the homomorphism $\sT_\bF[x]\to \sR^\#$ induced by sending $T\mapsto T_0$ and $x\mapsto P$, and the homomorphism $\sR^\#\to \sS_{BN}$ which sets $T_0=T_1$. We obtain the commutative diagram on the left:

\noindent \begin{minipage}{.5\textwidth}
\begin{equation*}
	\begin{tikzcd}[column sep=1ex, fill=none, row sep=5ex, fill=none, /tikz/baseline=-10pt]
 & \sT_\bF[x] \arrow{ld} \arrow{rd} &\\
\sR^\# \arrow{rr} & & \sS_{BN} 
\end{tikzcd}
\end{equation*}
\end{minipage}
\begin{minipage}{.5\textwidth}
\begin{equation*}
	\begin{tikzcd}[column sep=1ex, fill=none, row sep=5ex, fill=none, /tikz/baseline=-10pt]
 & \widehat{z}(K)   &\\
z^\#(K)  \arrow[mapsfrom]{ru} & & z^\natural_{BN}(K) \arrow[mapsfrom]{lu} \arrow[mapsfrom]{ll}
\end{tikzcd}
\end{equation*}
\end{minipage}
These homomorphisms induce maps between the fractional ideals of the three rings, and via these correspondences the invariant $\widehat{z}(K)$ is sent to $z^\#(K)$ and $z^\natural_{BN}(K)$, as indicated in the above diagram on the right.

There is a relation between $\widehat{z}(K)$ and the nested sequence of ideals $J_i^{\sT_\bF}(K)$ defined in \eqref{eq:locjideals} which can be described as follows. Given any $\sT_\bF[x]$-module $Z$ in $\text{Frac}(\sT_\bF[x])$ and any integer $i$, we define 
\[
  \fJ_i(Z):=\left \{\frac{a}{b}\mid \exists \: \frac{a x^{j-i}+a_{j-i-1}x^{j-i-1}+\dots+a_0}{b x^{j}+b_{j-1}x^{j-1}+
  \dots+b_0} \in Z\right\}\subset \text{Frac}(\sT_\bF).
\]
This is clearly a $\sT_\bF$-submodule of $\text{Frac}(\sT_\bF)$.

\begin{prop}
	For any $K$ as above, $\fJ_{i+\sigma(K)/2}(\widehat z(K))\subset J_i^{\sT_\bF}(K)$. In particular, $\fJ_i(\widehat z(K))$ is an ideal of $\sT_\bF$.
\end{prop}
\begin{proof}
	Suppose a cobordism $S:U_1\to K$ is chosen as above. Recall that 
	\[
		\fI(K)=\text{im}(i_\ast:\widehat{I}(K;\Delta_{\sT_\bF})\to\sT_\bF[\![x^{-1},x] ),
	\]
	\[
	  J_i^{\sT_\bF}(K)=\left \{a\in \sT_\bF\mid \exists \; ax^{-i} + a_{-i-1}x^{-i-1} + \cdots\in \fI \right\}.
	\]
	First, with notation as in \eqref{eq:zhatdef}, we claim that
	\begin{equation}\label{z-hat-re}
	 [\widehat{I}(K;\Delta_{\sT_\bF})' : \text{im}\, \widehat{I}(S;\Delta_{\sT_\bF})']=
	 [\fI (K): \widebar{I}(S;\Delta_{\sT_\bF})(1)]
	\end{equation}
	where $1\in \sT_\bF[\![x^{-1},x]$ in \eqref{z-hat-re} is a generator of $\fI (U_1)=\sT_\bF[x]$ as a module over $\sT_\bF[x]$. 
	This claim is a straightforward consequence of the exact triangle in \eqref{top-equiv-triangle} and the fact that 
	all elements in $\widecheck{I}(K;\Delta_{\sT_\bF})$ are $\sT_\bF[x]$-torsion. Corollary \ref{localization} implies that
	\[
	  \widebar{I}(S;\Delta_{\sT_\bF})(1)=1+\sum_{i=-\infty}^{-1}b_ix^i.
	\]
	Let $a/b\in \fJ_{i+\sigma(K)/2}(\widehat z(K))$. Then \eqref{z-hat-re} implies there are $P(x), Q(x)\in \sT_\bF[x]$ such that
	\[
	  \frac{P(x)(1+\sum_{i=-\infty}^{-1}b_ix^i)}{Q(x)}\in \fI (K),  
	\]
	and the ratio of the leading term of $P(x)/Q(x)$ is equal to $\frac{a}{b}x^{-i}$. Thus $a/b\in  J_i^{\sT_\bF}(K)$.
\end{proof}

The above proposition and the preceding discussion provides a partial answer to Question \ref{question:kmconc} for the family of knots with slice genus $-\sigma(K)/2$.

\begin{question}
	Can the definition of $\widehat{z}(K)$ be extended to all knots in the 3-sphere such that the relationship to $z^\#(K)$ and  $z^\natural_{BN}(K)$ described above still holds?
\end{question}

Note that if one constructs general cobordism maps for the equivariant theory, then there is an obvious way of extending the definition of $\widehat{z}(K)$ to all knots using \eqref{eq:zhatdef}. 

A similar discussion holds for $s^\#(K)$. In particular, if there is a surface cobordism $S:U_1\to K$ as above, we may define $\widehat{s}(K)$ by replacing $\sT_\bF$ in \eqref{eq:zhatdef} with $\sT_\bQ$, and from our discussion in Subsection \ref{subsec:rasmussen}, $s^\#(K)$ may be recovered from $\widehat{s}(K)$.

\newpage

%!TEX root = main.tex

\section{Computations}\label{sec:computations}

In this section we study our invariants for two-bridge knots and torus knots. In particular, we prove Theorems \ref{thm:hinvcomps} and \ref{thm:sasahiracomparison} from the introduction. 

After discussing the utility of passing to the double branched cover in Subsection \ref{sec:branchedcover}, we turn to two-bridge knots, where much of the structure of our invariants can be described combinatorially. We rely on previous results about instantons on $\R\times L(p,q)$, which depend upon equivariant ADHM constructions. Along the way we describe all of our invariants for the right-handed trefoil, from which the computation $\hinv_\sT=1$ follows. We also make contact with Sasahira's instanton homology for lens spaces, from which Theorem \ref{thm:sasahiracomparison} follows. We then discuss the $(3,5)$ and $(3,4)$ torus knots, making use of Austin's work \cite{austin}, which also relies on equivariant ADHM constructions. Subsection \ref{subsection:irr-torus-knots} discusses the irreducible Floer homology of torus knots. Finally, in Subsection \ref{sec:morevanishing}, we discuss some more examples for which $\hinv(K)$ vanishes.

\subsection{Passing to the branched cover}\label{sec:branchedcover}

The geometrical input involved in the singular instanton Floer complexes $\widetilde C_\ast(K)$ for a knot $K\subset S^3$ as defined in Section \ref{sec:tilde} can be related to the corresponding data on the double branched cover $\pi : \Sigma\to S^3$ over $K$. On the level of critical sets, this is established in \cite{PS}, and the description extends to the cylinder in a straightforward manner.

Recall that the critical set $\fC(K)= \fC(S^3,K)$ of the unperturbed singular Chern-Simons functional for $K$ may be identified with the $SU(2)$ traceless character variety of $K$ from \eqref{eq:charvar}, denoted $\sX(K)=\sX(S^3,K)$. For a closed oriented 3-manifold $Y$, we define
\[
	\sX(Y) := \{\rho:\pi_1(Y)\to SU(2)\} /SU(2),
\]
and via holonomy $\sX(Y)$ may be identified with the critical set $\fC(Y)$ of flat connections modulo gauge transformations for the $SU(2)$ Chern-Simons functional on $Y$. If $Y$ is a $\Z/2$-homology 3-sphere, as is the case for $\Sigma$, then $\sX(Y)$ is naturally identified with the corresponding $SO(3)$ character variety, by taking adjoints. Note in this case there are unique $SU(2)$ and $SO(3)$ bundles over $Y$ up to isomorphism.

Let $\tau:\Sigma\to \Sigma$ be the covering involution of $\Sigma$. Fix an $SU(2)$ bundle over $\Sigma$ and let $\fg_\Sigma$ be its adjoint bundle. A lift $\widetilde \tau$ of $\tau$ to $\fg_\Sigma$ is specified by choosing a bundle isomorphism $f:\tau^\ast \fg_\Sigma \to \fg_\Sigma$ via the relation $\widetilde \tau = f\circ p^{-1}$ where $p:\tau^\ast \fg_\Sigma \to \fg_\Sigma$ is the pullback map. Such lifts fall into two types, depending on whether $\widetilde{\tau}|_K$ is the identity or of order two. We restrict our attention here to the latter case; these are locally conjugate to the model given in \eqref{eq:locmodelbranched}. For such a lift $\widetilde{\tau}$, the quotient $\check{\fg}_K=\fg_\Sigma/\widetilde \tau $ is an $SO(3)$ orbifold bundle over $(S^3,K)$, isomorphic to the adjoint orbifold bundle $\check{\fg}_E$ considered in Section \ref{sec:connections}. Let $\fC^\tau(\Sigma)$ denote the subset of $\fC(\Sigma)$ of classes that are represented by a connection whose adjoint is fixed under the induced action for some such choice of a lift $\widetilde \tau$. Now, define a map
\[
	\Pi:\fC(K) \longrightarrow \fC(\Sigma)^\tau
\]
as follows: given an $SU(2)$ singular connection representing a class in $\fC(K)$, take its $SO(3)$ adjoint, pull back the induced orbifold connection to obtain an $SO(3)$ connection on $\Sigma$, and then take the gauge equivalence class of its unique $SU(2)$ lift. In terms of representations, $\fC(\Sigma)^\tau$ corresponds to the subset $\sX(\Sigma)^\tau\subset \sX(\Sigma)$ consisting of $\rho:\pi_1(\Sigma)\to SU(2)$ such that $\tau^\ast \rho = u\rho u^{-1}$ for some order four element $u\in SU(2)$, up to conjugacy.

The fibers of the map $\Pi$ are either one or two points. More precisely, we may divide the classes in $\fC(\Sigma)^\tau$ into three types, determined by their gauge stabilizers:
\begin{enumerate}
	\item[(i)] ({\emph{trivial}}) the trivial connection class $\theta_\Sigma$ with stabilizer $SU(2)$;
	\item[(ii)] ({\emph{abelian}}) non-trivial classes with stabilizer isomorphic to $U(1)$;
	\item[(iii)] ({\emph{irreducible}}) classes with stabilizer $\{\pm1 \}$.
\end{enumerate}
The map $\Pi$ is onto, and $\Pi^{-1}(\theta_\Sigma)=\{\theta\}$ where $\theta$ is the reducible singular flat connection for $K$; the fiber over a class of type (ii) in $\fC(\Sigma)^\tau$ consists of a unique irreducible class in $\fC(K)$; and over a class of type (iii) are {\emph{two}} irreducible class in $\fC(K)$. In particular, a non-irreducible class in $\fC(\Sigma)^\tau$ may come from an irreducible class in $\fC(K)$. In terms of representations, as is described in \cite[Section 4]{PS}, (i) corresponds to the trivial homomorphism, (ii) to non-trivial representations with abelian image, and (iii) to non-abelian representations. An abelian representation in $\sX(\Sigma)^\tau$ lifts to a unique binary dihedral representation in $\sX(K)$, and an irreducible representation has two irreducible lifts to $\sX(K)$. The map $\Pi$ factors as
\[
	\fC(K) \longrightarrow \fC(K)/\iota \longrightarrow \fC(\Sigma)^\tau
\]
where the first map is the quotient map associated to the flip symmetry $\iota$ of Subsection \ref{subsec:flip}, and the second map is a bijection. The fibers of $\Pi$ consisting of one point are given by fixed points of $\iota$, while the fibers with two points are the free orbits of $\iota$. In particular, the involution $\iota$ restricted to $\fC(K)$ acts freely on flat connections whose pullbacks to $\Sigma$ are irreducible, and fixes all other connection classes.

For $\alpha,\beta\in \fC(K)$ we have an $\R$-invariant flip symmetry $\iota:M(\alpha,\beta)\to  M(\iota\alpha,\iota\beta)$ between moduli spaces on $\R\times (S^3,K)$, and $\iota^2=\text{id}$. Thus $\iota$ acts on $M(\alpha,\beta)\cup  \iota M(\alpha,\beta)$ as an involution. If either $\alpha$ or $\beta$ is fixed by $\iota$ we have an identification
\[
	\left(\breve{M}(\alpha,\beta)\cup  \iota \breve{M}(\alpha,\beta))\right)/\iota= \breve{M}(\alpha_\Sigma,\beta_\Sigma)^{\tau}
\]
where $\breve{M}(\alpha_\Sigma,\beta_\Sigma)^\tau$ is the moduli space of instantons on $\R \times \Sigma$ which are fixed by some lift $\widetilde{\tau}$ of the branched covering $\tau$ extended to the cylinder. Here $\alpha_\Sigma = \Pi(\alpha)$ and $\beta_\Sigma=\Pi(\beta)$. We also assume that our metric on $\Sigma$ is invariant with respect to the branched covering involution. (We assume for simplicity that we do no need any holonomy perturbations to achieve regularity.) Then $\breve{M}(\alpha_\Sigma,\beta_\Sigma)^{\tau}$ is the fixed point set of a $\Z/2$-action on $\breve{M}(\alpha_\Sigma,\beta_\Sigma)$ induced by $\tau$. If $\iota$ acts freely on both $\alpha$ and $\beta$ we have instead an identification
\[
	\left(\breve{M}(\alpha,\beta)\cup  \iota \breve{M}(\alpha,\beta))\right)/\iota \cup \left(\breve{M}(\alpha,\iota\beta)\cup \breve{M}(\iota\alpha,\beta))\right)/\iota = \breve{M}(\alpha_\Sigma,\beta_\Sigma)^{\tau}
\]
where the two sets on the left hand side are disjoint. 

\begin{lemma}\label{lemma:flipsymmetry}
	The involution $\iota$ acts freely on $\breve{M}(\alpha,\beta)\cup  \iota\breve{M}(\alpha,\beta)$, and $\nu(\iota[A])=-\nu([A])$.
\end{lemma}

\begin{proof}

	Let $[A]\in \breve{M}(\alpha,\beta)$ be a (non-constant) instanton fixed by $\iota$. Then the pull-back of $[A]$ to $\R\times \Sigma$ is a non-constant reducible. 
	Since there is no such instanton we conclude that the action of $\iota$ is free.
	The behavior of the monopole number with respect to $\iota$ is given in \eqref{iota-effect}.
\end{proof}

\subsection{Two-bridge knots}\label{subsec:twobridge}

Let $p, q$ be relative prime, with $p$ odd. Write $K_{p,q}$ for the two-bridge knot whose two-fold branched cover is the lens space $L(p,q)$. The critical set $\fC(L(p,q))$ is easy to describe. Write $\smash{\xi_{L(p,q)}^i}$ for the flat $SU(2)$ connection class on $L(p,q)$ corresponding to the conjugation class of the representation $\pi_1(L(p,q))=\Z/p\to SU(2)$ defined by $\zeta\mapsto \zeta^i\oplus \zeta^{-i}$. Then
\[
	\fC(L(p,q)) = \{ \xi^1_{L(p,q)},\ldots, \xi^{(p-1)/2}_{L(p,q)}\}
\]
Our convention is to identify $L(p,q)$ with the quotient of $S^3\subset \C^2$ by the action of $\Z/p$, where $\Z/p$, viewed as the $p^\text{th}$ roots of unity, acts as $\zeta\cdot (z_1,z_2) = (\zeta z_1, \zeta^q z_2)$. Furthermore, the two-bridge knot $K_{p,q}$ is the fixed point set of $L(p,q)$ under the involution induced by the conjugation action $(z_1,z_2)\mapsto (\overline z_1 , \overline z_2)$. Thus the orbifold $(S^3,K_{p,q})$ is the quotient of $S^3$ by the action of the dihedral group of order $2p$.

As observed in \cite{PS}, all of the classes in $\fC(L(p,q))$ are fixed by the action of $\tau$, and so $\fC(L(p,q))=\fC(L(p,q))^\tau$. Furthermore, each $\smash{\xi^i_{L(p,q)}}$ is reducible, and thus uniquely lifts to a class $\xi^i\in \fC(K)$ which is fixed by $\iota$. Each of these is non-degenerate. Note that $\xi^0=\theta$ is the flat reducible, while $\xi^i$ for $1\leqslant i \leqslant (p-1)/2$ is irreducible. Thus our irreducible Floer chain complex has underlying (ungraded) group given by
\[
	C(K_{p,q}) = \bigoplus_{i=1}^{(p-1)/2} \Z\cdot \xi^i
\]
The gradings may also be computed, as recalled below. Consequently, the framed $\cS$-complex $\widetilde C(K_{p,q}) = C(K_{p,q}) \oplus C(K_{p,q})\oplus \Z$ has rank $p$. On the other hand, by Theorem \ref{thm:tildeconnsum} the homology of $\widetilde C(K_{p,q})$ is isomorphic to Kronheimer and Mrowka's $I^\natural (K_{p,q})$, which is also of rank $p$ by \cite[Corollary 1.6]{KM:unknot}. Thus $\widetilde d=0$, and in particular all the chain-level maps $d$, $v$, $\delta_1$ and $\delta_2$ must vanish. Proposition \ref{h-g-reinterpret} implies:

\begin{prop}
	For a two-bridge knot $K_{p,q}$ we have $\hinv(K_{p,q})=0$.
\end{prop}

This result is no longer true if we use local coefficients. The key observation is that some of the zero-dimensional moduli spaces $\breve{M}(\xi^i,\xi^{j})_0$ are non-empty, and in fact consist of exactly two points $[A]$ and $[A']$, which descend to a unique instanton on $\R\times L(p,q)$. The flip symmetry $\iota$ interchanges these instantons, reversing orientations, and so their contributions to the differential considered above cancel. However, it will happen that sometimes the monopole number $\nu([A])= -\nu([A'])$ is nonzero, in which case the contributions will not cancel in the setting of the local coefficient system $\Delta_\sT$.

Moduli spaces of instantons on $\R\times L(p,q)$ were studied in \cite{austin, furuta-invariant, furuta-hashimoto}. See also \cite[Section 4.1]{sasahira-lens} for a nice summary. An argument using the Weitzenb\"{o}ck formula shows that all such moduli are unobstructed and smooth. The $0$-dimensional moduli spaces can be described explicitly, and are determined as follows. First, consider the congruence
\begin{equation}
	 a+qb \equiv 0 \;(\text{mod } p) \label{eq:congruence}
\end{equation}
We count solutions $(a,b)$ in a rectangle determined by $k_1,k_2\in \Z_{>0}$ as follows:
\begin{align}
	N_1(k_1,k_2;p,q) &:= \#\left\{ (a,b)\in \Z^2 \text{ solving } \eqref{eq:congruence}, \;\; \; |a|<k_1, |b|<k_2\right\} \label{eq:n1defn}\\
	N_2(k_1,k_2;p,q) &:= \#\left\{  (a,b)\in \Z^2 \text{ solving } \eqref{eq:congruence}, \;  \begin{array}{c}|a|<k_1, |b|=k_2 \text{ or }\\ |a|=k_1, |b|<k_2 \phantom{ \text{ or }}\end{array}\right\}\label{eq:n2defn}
\end{align}

\begin{theorem}[\cite{austin, furuta-invariant}]\label{thm:lensmodulipoint}
	Let $0\leqslant i, j\leqslant (p-1)/2$, where $i\neq j$. Suppose there exists $k_1, k_2\in \Z_{>0}$ and $\varepsilon_1,\varepsilon_2\in \{+1,-1\}$ such that the following hold:
	\begin{equation}
		k_1 \equiv \varepsilon_1 i +  \varepsilon_2 j\;\; (\text{{\emph{mod }}}p), \quad qk_2 \equiv -\varepsilon_1 i  + \varepsilon_2 j \;\;(\text{{\emph{mod }}}p)\label{eq:k1k2}
	\end{equation}
	\begin{equation}
		N_1(k_1,k_2;p,q) = 1, \quad N_2(k_1,k_2;p,q)=0\label{eq:n1n2}
	\end{equation}
	Then $\breve{M}(\xi_{L(p,q)}^{i}, \xi_{L(p,q)}^{j})_0$ defined for $\R\times L(p,q)$ is a point. Otherwise it is empty.
\end{theorem}

Furthermore, the positive integers $k_1$, $k_2$ above are related to topological energy as follows: when there exists an instanton $[A']$ on $\R\times L(p,q)$ as in the theorem, we have
\begin{equation}\label{eq:chernsimonsk1k2}
	\kappa(A') = \frac{1}{8\pi^2}\int_{\R\times L(p,q)} \text{tr}(F_{A'}\wedge F_{A'}) = \frac{1}{p} \cdot c_2(\widetilde E) = \frac{k_1 k_2}{p}
\end{equation}
where $\widetilde E\to S^4$ is an $SU(2)$ bundle that supports an extension of the pullback of the instanton $A'$ to the 4-sphere compactification of $\R \times S^3$. 

We now return to the orbifold $(S^3,K_{p,q})$. Consider a moduli space $\breve{M}(\xi_{L(p,q)}^{i}, \xi_{L(p,q)}^{j})_d$ of instantons on $\R\times L(p,q)$. It follows from the computations of \cite[Section 7.1]{PS} that the $\tau$-invariant moduli space has dimension $d/2$. Thus pullback induces an embedding
\[
	\left(\breve{M}(\xi^{i}, \xi^{j})_{\frac{d}{2}}\right)/\iota \hookrightarrow \breve{M}(\xi_{L(p,q)}^{i}, \xi_{L(p,q)}^{j})_d
\]
of smooth manifolds, whose image is the fixed point set of an involution on the codomain. In particular, setting $d=0$, we find that $\breve{M}(\xi^{i}, \xi^{j})_0$ is entirely determined by the above theorem, combined with the behavior of the symmetry $\iota$ as described in Lemma \ref{lemma:flipsymmetry}.

\begin{cor}
	Let $0\leqslant i, j\leqslant (p-1)/2$, where $i\neq j$, and consider the corresponding moduli space $\breve{M}(\xi^{i}, \xi^{j})_0$ of instantons on $\R\times (S^3,K_{p,q})$. 
	\begin{itemize}
	\item[\emph{(i)}] $\breve{M}(\xi^{i}, \xi^{j})_0=\emptyset$  if and only if  $\smash{\breve{M}(\xi_{L(p,q)}^{i}, \xi_{L(p,q)}^{j})}_0=\emptyset$.
	\item[\emph{(ii)}] If $\breve{M}(\xi^{i}, \xi^{j})_0\neq \emptyset$, then it consists of two oppositely oriented points. 
	\end{itemize}
\end{cor}

In the case that $\breve{M}(\xi^{i}, \xi^{j})_0\neq \emptyset$, Theorem \ref{thm:lensmodulipoint} implies that $\breve{M}(\xi^{i}, \xi^{j})_0$ consists of two points. These two points are oppositely oriented because of the vanishing of the maps $d$, $\delta_1$ and $\delta_2$. To compute the maps $d$, $\delta_1$ and $\delta_2$ with local coefficients, we have:

\begin{prop}\label{prop:2bridgemonopoleno}
	Suppose $\breve{M}(\xi^{i}, \xi^{j})_0\neq \emptyset$, so that it contains two instantons $[A]$ and $\iota [A]$. Let $k=k_1k_2$ where $k_1,k_2\in\Z_{>0}$ are solutions to \eqref{eq:k1k2}, \eqref{eq:n1n2}. Then
	\[
		\left\{\nu([A]), \nu(\iota[A])\right\}  = \begin{cases} \{0\} & \text{if }k\equiv 0 \mod 2\\
															\{2,-2\} & \text{if } k\equiv 1 \mod 2 \end{cases}
	\]
\end{prop}

To prove this we utilize the twisted spin Dirac operator. To set this up, let $[A]\in \breve{M}(\xi^i,\xi^j)$, so that $A$ is a singular instanton on $\R\times (S^3, K_{p,q})$. Pull back $A$ to an instanton on $\R\times S^3$. By Uhlenbeck's removable singularity theorem, this pull back connection extends, after possibly gauge transforming, to an instanton $\widetilde A$  on the compactified $S^4$. Write $\widetilde E\to S^4$ for the bundle on which $\widetilde A$ is supported, and set $k:=c_2(\widetilde E)$.

Recall that $(S^3,K_{p,q})$ is the quotient of $S^3$ by the dihedral group $D_{2p}$, generated by $\zeta\in \Z/p\subset U(1)$ and $\tau\in\Z/2$, where $\zeta\cdot (z_1,z_2)= (\zeta z_1, \zeta^q z_2)$ and $\tau\cdot (z_1,z_2) = (\overline z_1 , \overline z_2)$. Here we view $S^3\subset \C^2$ as the unit sphere. This action extends to $\C^2\cup \infty=S^4$. We may lift the action of $D_{2p}$ to an action of the binary dihedral group $\widetilde D_{4p}$ of order $4p$ on the bundle $\widetilde E$. This lift may be chosen such that $\widetilde{A}$ is invariant under the action of $\widetilde  D_{4p}$.

The action of $D_{2p}$ on the 4-sphere also lifts to an action of $\widetilde D_{4p}$ on the spinor bundles $S^\pm \to S^4$. Thus we may consider the spin Dirac operator coupled to $\widetilde A$:
\[
	\slashed{D}_{\widetilde A}:\Gamma(\widetilde E\otimes S^- ) \longrightarrow \Gamma(\widetilde E\otimes S^+)
\]
The group $\widetilde D_{4p}$ induces actions on the domain and codomain, and $\slashed{D}_{\widetilde A}$ is equivariant with respect to these actions. Write $\widetilde \tau$ for the lift of the action of $\tau$ to $\widetilde D_{4p}$.

\begin{prop}\label{prop:lefschetzmonopole}
	$\text{\emph{Lef}}(\widetilde \tau,\slashed{D}_{\widetilde A})=\nu(A)/2$.
\end{prop}

\begin{proof} We use the Atiyah--Segal--Singer equivariant index theorem, as stated in \cite[Theorem 6.16]{bgv}, applied to the operator $\slashed{D}_{\widetilde A}$ and the action of $\widetilde\tau \in \widetilde D_{4p}$:
	\[
		\text{Lef}(\widetilde \tau,\slashed{D}_{\widetilde A}) = -\frac{1}{2\pi} \int_{S^2} \frac{\widehat{A}(S^2)\text{ch}(\widetilde \tau,\widetilde E)}{\det(1-\tau_1 \cdot \text{exp}(-F_N))^{1/2}}
	\]
	Here $S^2=\R^2\cup \infty \subset S^4$ is the fixed point set of $\tau$; the term $\text{ch}(\widetilde \tau , \widetilde E)$ is defined to be $\text{Tr}(\widetilde \tau\cdot \text{exp}(-F_{\widetilde A}))$; $F_N$ is the component of Riemannian curvature form of $S^4$ normal to $S^2$; and $\tau_1$ is the action of $\tau$ on the normal bundle over $S^2$. We have $\widehat{A}(S^2)=1$. The connection $\widetilde A$ descends to the singular connection $A$, and so by assumption there is a preferred reduction $\widetilde L \oplus \widetilde L^{-1}\to S^2$ over the fixed point set, at which $\widetilde A$ splits as $\widetilde A_0\oplus \widetilde A_0^\ast$. Furthermore, the action of $\widetilde \tau$ on $\widetilde L \oplus \widetilde L^{-1}$ is of the form $\text{diag}(i,-i)$. Thus
	\begin{align*}
		\text{ch}(\widetilde \tau , \widetilde E) 
									 =  \text{Tr}\left(\left[\begin{array}{cc} i & 0 \\ 0 & -i \end{array}\right]\cdot\text{exp}\left[\begin{array}{cc}- F_{\widetilde A_0} & 0 \\ 0 & F_{\widetilde A_0} \end{array}\right]\right)
									 = -2i F_{\widetilde A_0}
	\end{align*}
	The action of $\tau$ on the normal bundle of $S^2$ is by negation, so for the denominator we have
	\[
		\det(1-\tau_1 \cdot \text{exp}(-F_N))^{-1/2} = \det(1+ \text{exp}(-F_N))^{-1/2} = \frac{1}{2}\left( 1 + \frac{1}{4} F_N \right)
	\]
	We find that only the constant term of the denominator contributes to the final integral:
	\[
		\text{Lef}(\widetilde \tau, \slashed{D}_{\widetilde A}) = -\frac{1}{2\pi}\int_{S^2} -2iF_{\widetilde A_0} \cdot \frac{1}{2} = \frac{i}{2\pi}\int_{S^2} F_{\widetilde A_0} = \frac{1}{2}\nu(A)
	\]
	The last equality follows because the fixed point set $S^2$, with $0$ and $\infty$ removed, is mapped with the connection $\widetilde A$ isomorphically to $\R\times K$ with the connection $A$.
\end{proof}

The character of $\slashed{D}_{\widetilde A}$ restricted to the subgroup $\Z/2p \subset \widetilde D_{4p}$, which we write as $\text{ind}_{\Z/2p}(\slashed{D}_{\widetilde A})$, is computed in \cite[Section 4.2]{sasahira-lens}. For $0\leqslant i \leqslant 2p-1$, we write $\chi_j$ for the character of $\Z/2p$ defined by sending the generator to $e^{j \pi i  /p}$. Then
\begin{equation}
	\text{ind}_{\Z/2p}(\slashed{D}_{\widetilde A}) = \sum_{j=0}^{2p-1} M(j,k_1,k_2;p,q) \cdot \chi_j \label{eq:diraccyclicchar}
\end{equation}
where the coefficient $M(j,k_1,k_2;p,q)$ is defined to be the number of solutions $(c,d)\in\Z^2$ with $0\leqslant c \leqslant k_1-1$ and $0\leqslant d\leqslant k_2-1$ to the congruence
\begin{equation}
	-k_1 + 2c + 1 + q ( -k_2 + 2d + 1 ) \equiv j \;\; (\text{mod }2p) \label{eq:congruencedirac}
\end{equation}
Here $k_1$ and $k_2$ are solutions to \eqref{eq:k1k2}; in the situation we consider, they will be uniquely determined as also satisfying \eqref{eq:n1n2}. The binary dihedral group $\widetilde D_{4p}$ has $p-1$ two-dimensional representations, whose characters we denote by $\widetilde \chi_j$ for $1\leqslant j \leqslant p-1$; and also 4 one-dimensional representations, dented $\widetilde \chi^\pm_0$, $\smash{\widetilde \chi_p^\pm}$. Here $\widetilde \chi_0^+$ is the trivial representation and $\widetilde \chi_0^-$ has $\widetilde \tau$ acting by $-1$. The restriction from $\widetilde D_{4p}$ to the subgroup $\Z/2p$ sends:
\begin{equation}
	\widetilde \chi_j \mapsto \chi_j + \chi_{2p-j} \quad  (1\leqslant j \leqslant p-1), \qquad 
	\widetilde \chi_0^\pm \mapsto \chi_0, \qquad \widetilde \chi_{p}^\pm\mapsto \chi_p \label{eq:binihedraltocyclic}
\end{equation}
In particular, if we consider the $\widetilde{D}_{4p}$ character of $\slashed{D}_{\widetilde A}$, we may write
\[
	\text{ind}_{\widetilde D_{4p}}(\slashed{D}_{\widetilde A}) =n_0^+ \widetilde \chi_0^+ + n_0^- \widetilde \chi_0^- +n_p^+ \widetilde \chi_p^+ + n_p^- \widetilde \chi_p^- +  \sum_{j=1}^{p-1} n_j \widetilde \chi_j  
\]
The operator $\slashed{D}_{\widetilde A}$ is surjective, so each of these coefficients is non-negative. We then have, combining \eqref{eq:binihedraltocyclic} and \eqref{eq:diraccyclicchar}, the following:
\begin{align*}
	 M(j,k_1,k_2;p,q) + M(2p-j,k_1,k_2;p,q) & = n_j \qquad(1\leqslant j \leqslant p-1) \\ M(0,k_1,k_2;p,q) &= n_0^+ + n_0^-\\ M(p,k_1,k_2;p,q) &= n_p^+ + n_p^-
\end{align*}
Observe that $\widetilde \chi_j(\widetilde \tau)=0$, while $\widetilde \chi^\pm_0(\widetilde \tau)=\pm 1$ and $\widetilde \chi^\pm_p(\widetilde \tau)=\pm i$. Consequently, we have $\text{Lef}(\widetilde \tau, \slashed{D}_{\widetilde A})=n_0^+-n_0^-$. Note that this Lefschetz number is real by Proposition \ref{prop:lefschetzmonopole}, so necessarily $n_p^+=n_p^-$.

\begin{prop}\label{prop:spindiracindex}
	Let $0\leqslant i, j\leqslant (p-1)/2$, where $i\neq j$, satisfying \eqref{eq:k1k2} and \eqref{eq:n1n2} for some $k_1,k_2\in \Z_{>0}$, so that the moduli space $\smash{\breve{M}(\xi_{L(p,q)}^{i}, \xi_{L(p,q)}^{j})_0} = \{ [ A' ] \}$ is a point. Then the associated coupled spin Dirac operator $\slashed{D}_{A'}$ defined over $\R\times L(p,q)$ is surjective and
	\begin{equation*}
		\dim \text{\text{\emph{ker}}} \slashed{D}_{A'} = 
		\begin{cases}
			0 & \text{\emph{ if }}\;\; k_1k_1\equiv 0 \mod 2\\
			1 & \text{\emph{ if }}\;\; k_1k_2 \equiv 1 \mod 2
		\end{cases}
	\end{equation*}
\end{prop}

\begin{proof}
	Suppose $M(0,k_1,k_2;p,q)>0$, so that there exists a solution $(c,d)$ to \eqref{eq:congruencedirac} with $j=0$, where $0 \leqslant c\leqslant k_1-1$ and $0\leqslant d\leqslant k_2-1$. Then $(a,b)$ with $a:= -k_1+2c+1$ and $b:=-k_2+2d+1$ is a solution to \eqref{eq:congruence} with $|a|<k_1$ and $|b|<k_2$. However, the assumption $N_1(k_1,k_2;p,q)=1$ implies the only such solution is $(a,b)=(0,0)$. Thus 
\[
	c=(k_1-1)/2, \qquad d=(k_2-1)/2
\]
are uniquely determined, and we must have $M(0,k_1,k_2;p,q)=1$ and $k_1k_2\equiv 1$ (mod 2). The same reasoning shows that if $k_1k_2\equiv 0$ (mod 2) then there is no solution to \eqref{eq:congruencedirac}. It remains to observe that $M(0,k_1,k_2;p,q)$ is precisely $\text{ind}\slashed{D}_{A'}=\dim \text{ker} \slashed{D}_{A'}$.
\end{proof}

\begin{proof}[Proof of Proposition \ref{prop:2bridgemonopoleno}]
	By Proposition \ref{prop:spindiracindex}, $n_0^++n_0^- =0$ if $k=k_1k_2\equiv 0$ (mod $2$), and $n_0^++n_0^- =1$ if $k=k_1k_2\equiv 1$ (mod $2$). Note $n_0^\pm \in \Z_{\geqslant 0}$. By Proposition \ref{prop:lefschetzmonopole}, $n_0^+-n_0^-=\text{Lef}(\widetilde \tau,\slashed{D}_{\widetilde A})=\nu(A)/2$, from which the result follows.
\end{proof}

The assumption that $[A]\in \breve{M}(\xi^i,\xi^j)$ lies in the $0$-dimensional component of the moduli space was only used in the proof of Proposition \ref{prop:spindiracindex}, and there we only relied on the relation $N_1(k_1,k_2;p,q)=1$. In general, $[A]\in \breve{M}(\xi^i,\xi^j)_d$ where
\[
	d = N_1(k_1,k_2;p,q) + \frac{1}{2}N_2(k_1,k_2;p,q)-1
\]
for some $k_1$, $k_2$ satisfying \eqref{eq:k1k2}. Note that the involution $(a,b)\mapsto (-a,-b)$ on the sets appearing in \eqref{eq:n1defn}, \eqref{eq:n2defn} shows that $N_1(k_1,k_2;p,q)$ is odd and $N_2(k_1,k_2;p,q)$ is even. In particular, if $d=1$, then we must have $N_1(k_1,k_2;p,q)=1$ and $N_2(k_1,k_2;p,q)=2$. Thus the above work carries through in this situation as well.

\begin{cor}\label{cor:monopolevmap}
	If $[A]\in \breve{M}(\xi^i,\xi^j)_1$ then the conclusion of Proposition \ref{prop:2bridgemonopoleno} still holds.
\end{cor}

\subsubsection{The irreducible chain complex with local coefficients}\label{subsec:2bridgecomplexes}

We now have an algorithm to compute the irreducible chain complex $(C_\ast(K_{p,q};\Delta \otimes \bF),d)$ which is a module over $\sR\otimes \bF=\bF[U^{\pm 1},T^{\pm 1}]$, along with the maps $\delta_1$ and $\delta_2$. We have
\[
	C_\ast(K_{p,q};\Delta\otimes{\bF}) = \bigoplus_{i=1}^{(p-1)/2} U^{c_i}\cdot \bF[U^{\pm 1},T^{\pm 1}]\cdot \xi^i
\]
Fixing $1\leqslant i\leqslant (p-1)/2$, let $k_1,k_2\in \Z_{>0}$ be any pair of solutions to \eqref{eq:k1k2} after setting $j=0$. Using \eqref{eq:chernsimonsk1k2} the Chern--Simons invariant $c_i$ associated to $\xi^i$ is given by 
\[
	c_i := 2\cdot k_1k_2/2p = k_1k_2/p
\]
The factor of $2$ appears in the denominator because the orbifold $(S^3,K_{p,q})$ is the quotient of $L(p,q)$ by the branched cover involution, while the factor of $2$ appears in the numerator because the Chern--Simons functional is related to the scaled action $2\kappa$. Furthermore, the $\Z/4$-grading of the generator $\xi^i$ is given by:
\[
	\text{gr}(\xi^i)\equiv N_1(k_1,k_2;p,q) + \frac{1}{2}N_2(k_1,k_2;p,q) \;\; (\text{mod } 4)
\]
Next, fix $0\leqslant i,j\leqslant (p-1)/2$ and $i\neq j$. Define $a_{ij}\in \bF[U,T^{\pm 1}]$ to be
\[
	a_{ij} :=  \begin{cases}U^{-k_1k_2/p}(T^{ 2}-T^{ -2}), &  \text{ if } \exists k_1,k_2\in \Z_{>0} \text{ solving }\eqref{eq:k1k2}, \eqref{eq:n1n2}, k_1k_2 \text{ odd} \\ 0, & \text{  otherwise  } \end{cases}
\]
Then the maps $d$, $\delta_1$, and $\delta_2$ are determined as follows:
\[
	 \langle d \xi^i,\xi^j \rangle = a_{ij}, \qquad \delta_1(\xi^i) = a_{i0}, \qquad \langle \delta_2(1),\xi^i\rangle = a_{0i} 
\]
All that remains, in order to describe the entire $\cS$-complex $\widetilde C_\ast(K_{p,q};\Delta\otimes{\bF})$ (and in fact all of its structure as an enriched $\cS$-complex), is to compute the $v$-maps. Properties of the 2-dimensional moduli spaces $\smash{\breve{M}(\xi^i_{L(p,q)},\xi^j_{L(p,q)})_2}$ are described in \cite{austin}, and it seems probable that the $v$-maps can be computed directly, starting from the equivariant ADHM constructions described therein.

\begin{remark}
Although we have not computed the $v$-maps here, note that Corollary \ref{cor:monopolevmap} determines the monopole numbers of instantons which contribute to the $v$-map. $\diamd$
\end{remark}

\begin{remark}\label{rem:2bridgetoz}
If we work over the general local coefficient system $\Delta$, without tensoring by $\bF$, we have only determined each map $d$, $\delta_1$, $\delta_2$ up to a sign, i.e. each $a_{ij}$ should be written instead as $\pm a_{ij}$. However, as long as there are no ``cycles'' in the differential, it is straightforward to verify that these signs do not matter, and our description determines the equivalence class of the complex over $\Delta$.  $\diamd$
\end{remark}

\subsubsection{Sasahira's instanton homology for lens spaces}\label{subsec:sasahira}

The complex computed above is closely related to a version of instanton homology for lens spaces defined by Sasahira \cite{sasahira-lens}, the idea for which goes back to Furuta \cite{furuta-invariant}. We first define the underlying chain group, which is a vector space over $\bF=\Z/2$, to be
\[
	C_\ast(L(p,q)) = \bigoplus_{i=1}^{(p-1)/2} \bF \cdot \xi_{L(p,-q)}^i
\]
The minus sign in $-q$ is included to properly align our orientation conventions with \cite{sasahira-lens}. There is of course a natural identification of this group with $C_\ast(K_{p,-q};\bF)$, via $\smash{\xi_{L(p,-q)}^i\mapsto \xi^i}$, and we define a $\Z/4$-grading on $C_\ast(L(p,q))$ using this identification. Next, consider
\[
	\bF_4 := \bF[x]/(x^2+x+1),
\]
the field with four elements. We have a ring homomorphism $f:\bF[U^{\pm 1}, T^{\pm 1}]\to \bF_4$ which is determined by $f(U)=1$ and $f(T)=x$. Note $f(T^{2}-T^{-2})=1$. Then 
\[
	\langle d \xi^i_{L(p,-q)} , \xi^j_{L(p,-q)} \rangle := f(a_{ij}) \in \bF \subset \bF_4
\]
The homology of the complex $(C_\ast(L(p,q)), d)$ is denoted $I_\ast(L(p,q))$, and is a $\Z/4$-graded $\bF$-vector space. That this is the same as the definition in \cite{sasahira-lens} follows from Proposition \ref{prop:spindiracindex}. Another notation for $I_\ast(L(p,q))$ in \cite{sasahira-lens} is given by $\smash{I^{(0)}_\ast(L(p,q))}$.

\begin{theorem}
	There is an isomorphism from $(C_\ast(K_{p,q};\Delta_{\bF_4}),d)$, where the local system $\Delta_{\bF_4}$ is obtained from $\Delta\otimes{\bF}$ via the base change $f:\bF[U^{\pm 1},T^{\pm 1}]\to \bF_4$, to Sasahira's chain complex $(C_\ast(L(p,-q))\otimes \bF_4, d)$ tensored over $\bF_4$. As a result,
	\[
		I_\ast(K_{p,q};\Delta_{\bF_4}) \cong I_\ast(L(p,-q))\otimes \bF_4
	\]
	as $\Z/4$-graded vector spaces over the field $\bF_4$.
\end{theorem}

In particular, the euler characteristics are equal, and given by $\sigma(K_{p,q})/2$. Thus we obtain a way of computing the signature of $K_{p,q}$ from the arithmetic functions $N_1(k_1,k_2;p,-q)$ and $N_2(k_1,k_2;p,-q)$, although the authors suspect that this is probably not new.

\begin{remark} Proposition \ref{prop:spindiracindex} simplifies the construction of $(C_\ast(L(p,q)),d)$ given by Sasahira. Indeed, in \cite{sasahira-lens}, the coefficient $d:=\smash{\langle d \xi_{L(p,-q)}^i,\xi_{L(p,-q)}^j\rangle}$ is computed in two steps: (1) first, check if the relevant 0-dimensional moduli space is empty or a point $[A']$ using Theorem \ref{thm:lensmodulipoint}; then, (2) compute $\text{ind} \slashed{D}_{A'}$, and its parity will give the answer for $d\in \bF$. However, Proposition \ref{prop:spindiracindex} says that $\text{ind} \slashed{D}_{A'}\equiv k_1k_2$ (mod $2$), where $k_1$ and $k_2$ have already been determined in step (1), via Theorem \ref{thm:lensmodulipoint}. $\diamd$
\end{remark}

\begin{cor}
	Let $k\in \Z$. Then $\hinv_\sS(K_{8k+1,-2})=0$ for any coefficient system $\Delta_\sS$.
\end{cor}

\begin{proof}
	Sasahira computes in \cite[Proposition 4.9]{sasahira-lens} that $I_\ast(L(8k+1,2))=0$. Essentially the same computation shows $I_\ast(K_{8k+1,2};\Delta_\sS)=0$, which implies the result.
\end{proof}

\subsubsection{Invariants for the right-handed trefoil}

We describe the full structure of the our invariant for the right-handed trefoil, which with our conventions is the $(3,-1)$ two-bridge knot. Recall that we write $\sR=\Z[U^{\pm 1}, T^{\pm 1}]$. From Subsection \ref{subsec:2bridgecomplexes} and Remark \ref{rem:2bridgetoz}, we have the following:
\[
	C_\ast(K_{3,-1};\Delta) = U^{c_1} \sR \cdot \xi^1
\]
To compute $c_1$, we note that the solution $(k_1,k_2)$ to \eqref{eq:k1k2}, \eqref{eq:n1n2} for $i=1$ and $j=0$ is $(k_1,k_2)=(1,1)$. Thus $c_1=k_1k_2/p = 1/3$. This also gives the usual grading of $U^{1/3}\xi^1$ as $1$ (mod 4). Note that there is no room for a $v$-map, because there is only 1 generator. The map $\delta_2$ is zero for grading reasons, while $\delta_1$ is determined by
\[
	\delta_1(\xi^1) = \pm U^{-1/3}(T^{ 2}-T^{ -2}),
\]
using the data of $k_1$, $k_2$ discussed above. The generator $U^{1/3}\xi^1$ will be identified with the homotopy class of the path from $\xi^1$ to $\theta$ determined by the unique instanton that contributes to the computation of $\delta_1$. The instanton grading of this generator is given by the Chern--Simons invariant of this path, which is $1/3$.

Next, the $\cS$-complex $\widetilde C_\ast := \widetilde C_\ast(K_{3,-1};\Delta)$ is
\[
	\widetilde C_\ast = \left(U^{1/3} \sR\cdot \xi^1\right) \oplus \left(U^{1/3}\sR \cdot \chi\xi^1	\right) \oplus  \sR \cdot \theta
\]
where the generator $U^{1/3}\chi\xi^1$ has grading $2$ (mod 4), and the differential is simply
\[
	\widetilde d = \left[ \begin{array}{ccc} 0 & 0 & 0 \\ 0 & 0 & 0 \\ \pm U^{-1/3}(T^{ 2}-T^{ -2}) & 0  & 0  \end{array} \right]
\]
This determines the I-graded $\cS$-complex of the right-handed trefoil, and, as no perturbations are necessary, also the isomorphism-type of its enriched $\cS$-complex.

We may also compute the equivariant homology groups. Let us first continue to work over our universal coefficient ring $\sR=\Z[T^{\pm 1}, U^{\pm 1}]$. The exact triangle \eqref{top-equiv-triangle-intro} splits:
\begin{equation*}\label{eq:trefoilexact1}
	0 \longrightarrow \hrI( K_{3,-1} , \Delta_{\sR}) \xrightarrow{i_\ast} \brI( K_{3,-1} , \Delta_{\sR}) \xrightarrow{p_\ast}  \crI( K_{3,-1} , \Delta_{\sR}) \longrightarrow 0
\end{equation*}
More precisely, this is a short exact sequence of $\sR[x]$-modules, and is computed from our description of the $\cS$-complex for $K_{3,-1}$ given above and the definitions in Subsection \ref{sec:equivtheories}:
\begin{equation}\label{eq:exactseqfortrefoil}
	0 \longrightarrow U^{1/3}\sR\cdot \xi^1 \oplus \sR[x] \xrightarrow{i_\ast} \sR[\![ x^{-1} , x] \xrightarrow{p_\ast}  \frac{\sR[\![ x^{-1} , x]}{(T^2-T^{-2})x^{-1} + \sR[x]} \longrightarrow 0
\end{equation}
Here $p_\ast$ is the obvious projection, $i_\ast$ embeds $\sR[x]$ into $\sR[\![x^{-1},x]$, and 
\[
	i_\ast(U^{1/3}\xi^1)= (T^2-T^{-2})x^{-1}.
\]
The $\sR[x]$-module structures on the second and the third groups in the exact sequence \eqref{eq:exactseqfortrefoil} are the obvious ones, while on the first group we have
\begin{equation}
	x\cdot U^{1/3}\xi^1 = T^2-T^{-2} \in \sR[x], \label{eq:actiononxi1}
\end{equation}
with the usual module structure on the $\sR[x]$-factor. From this description it is clear that the ideal $\fI=\text{im}(i_\ast)$ introduced in Subsection \ref{small-model} is:
\[
	\fI = (T^2-T^{-2})x^{-1} + \sR[x] \subset \sR[\![x^{-1},x].
\]
Our nested sequence of ideals of $\sR$, from \eqref{eq:locjideals}, is given by
\[
	J^\sR_1 = (T^2-T^{-2}), \qquad J^\sR_i = 0\;\; (i\geqslant 2), \qquad J^\sR_i = \sR \;\; (i\leqslant 0).
\]
We summarize some more consequences of these computations.

\begin{prop}
	For the right-handed trefoil we have $\hinv_\sR =1$. The same holds for $\hinv_\sS$ for any integral domain base change $\sS$ for which $U$ and $(T^2-T^{-2})$ remain nonzero. Furthermore, the function $\Gamma^R_{K_{3,-1}}:\Z\to \R_{\geqslant 0}\cup \infty$ for $R={\Z[T^{\pm 1}]}$ satisfies the following:
	\begin{equation*}
		\Gamma^R_{K_{3,-1}}(k)=\begin{cases}   0, & k\leqslant 0\\
					1/3, &  k=1\\
					\infty, & k\geqslant 2
		\end{cases}
	\end{equation*}
\end{prop}

We turn to the base change $\sT_\bF = \bF[T^{\pm 1}]$. Using Corollary \ref{equiv-natural} we may recover the computation of $I^\natural(K_{3,-1};\Delta_{BN})$ from \cite{km-concordance}. Here we note that
\[
	\widehat{C}(K_{3,-1},\Delta_{\sT_\bF}) = \widehat{I}(K_{3,-1},\Delta_{\sT_\bF}) = \sT_\bF \cdot \xi^1 \oplus \sT_\bF[x]
\]
where similar to before the $\sT_\bF[x]$-module structure is the standard one on the summand $\sT_\bF[x]$, but just as in \eqref{eq:actiononxi1} acts on the generator $\xi^1$ as:
\begin{equation}
	x\cdot \xi^1 = T^2- T^{-2} \in \sT_\bF[x]\label{eq:actionxi1t}
\end{equation}
Now Corollary \ref{equiv-natural} gives us an isomorphism of $\sS_{BN}$-modules
\begin{equation}
	I^\natural(K_{3,-1};\Delta_{BN})\cong \widehat I(K_{3,-1};\Delta_{\mathscr{T}_\bF})\otimes_{\mathscr{T}_\bF[x]} \mathscr{S}_{BN}\label{eq:isofromcorequivnat}
\end{equation}
We have a surjective homomorphism to the right hand side of \eqref{eq:isofromcorequivnat}:
\begin{equation}
	\sS_{BN} \oplus \sS_{BN} \longrightarrow \left(\sT_\bF \cdot \xi^1\otimes_{\sT_{\bF}[x]} \sS_{BN} \right)\oplus  \sS_{BN} \label{eq:surjectivehomtrefoil}
\end{equation}
which sends $(a,b)\mapsto (\xi^1\otimes a,  b)$. Using relation \eqref{eq:actionxi1t} and that the module structure of $\sS_{BN}=\bF[T_1^{\pm 1}, T_2^{\pm 1}, T_3^{\pm 1}]$ over $\sT_{\bF}[x]$ sends $x\mapsto P$, we find that the kernel of \eqref{eq:surjectivehomtrefoil} consists of elements $(Pa, (T^2-T^{-2})b)$. Thus $I^\natural(K_{3,-1};\Delta_{BN})$ has the presentation
\[
	\sS_{BN} \longrightarrow \sS_{BN} \oplus \sS_{BN}, \qquad 1\mapsto (P,T^2-T^{-2}).
\]
We have recovered part of \cite[Proposition 8.1]{km-concordance}, proven by an entirely different method:

\begin{cor} \label{trefoil-computation}
	Let $K$ be the right-handed trefoil. Then $I^\natural(K;\Delta_{BN})$ is isomorphic as an $\sS_{BN}$-module to the ideal $(P,T^2-T^{-2})\subset \sS_{BN}$.
\end{cor}

Since the slice genus condition in \eqref{slice-genus-cond} is satisfied by the right-handed trefoil, we can define the ideal $\widehat{z}(K_{3,-1})$ and use \cite[Proposition 8.1]{km-concordance} to see that it is equal to $(x,T^2-T^{-2})$.

\subsection{The $(3,5)$ torus knot}\label{sec:35}

We now turn back to the coefficient ring $\Z$, and study some knots which have non-trivial $h=h_\Z$ invariants. We begin with the torus knot $K=T_{3,5}$, whose double branched cover $\Sigma$ is the Poincar\'{e} homology sphere $\Sigma(2,3,5)$. Recall that for an integer homology 3-sphere such as $\Sigma$, Floer's chain complex $(C_\ast(\Sigma),d)$ for the $\Z/8$-graded instanton homology $I_\ast(\Sigma)$ is generated by the irreducible critical points in $\fC(\Sigma)$, perhaps after a suitable holonomy perturbation. For the Poincar\'{e} sphere, no perturbation is needed, and there are two non-degenerate irreducibles $\{\alpha_\Sigma, \beta_\Sigma\} \subset \fC(\Sigma)$, so we can write
\[
	C_\ast(\Sigma) = \Z_{(1)} \oplus \Z_{(5)} = I_\ast(\Sigma).
\]
The Floer gradings are $\text{gr}(\alpha_\Sigma)\equiv 1$ (mod 8) and $\text{gr}(\beta_\Sigma)\equiv 5$ (mod 8). Our convention is that $\text{gr}(\alpha_\Sigma)$ is the expected dimension modulo 8 of $M(\alpha_\Sigma,\theta_\Sigma)$, the moduli space of finite energy instantons on $\R\times \Sigma$ with limits $\alpha_\Sigma$ and $\theta_\Sigma$ at $-\infty $ and $+\infty$, respectively, where $\theta_\Sigma$ is a trivial connection.

From our discussion in Subsection \ref{sec:branchedcover}, and the fact that $\fC(\Sigma)^\tau=\fC(\Sigma)$ (see \cite[Proposition 8]{saveliev:brieskorn}), there are four irreducibles in $\fC(K)$, which are non-degenrate: two lifts, $\alpha$ and $\iota \alpha$, of $\alpha_\Sigma$, and two lifts, $\beta$ and $\iota\beta$, of $\beta_\Sigma$. We may use the zero perturbation and the index computations of \cite[Theorem 7.2]{PS} to write
\begin{equation}
	C_\ast(K) = \Z_{(1)}^2 \oplus \Z_{(3)}^2 = I_\ast(K) \label{eq:complex35}
\end{equation}
where each $\Z_{(1)}$ is generated by one of $\alpha$ and $\iota \alpha$, and each $\Z_{(3)}$ by one of $\beta$ and $\iota\beta$. For grading reasons, as already indicated in \eqref{eq:complex35}, we have $d=0$. 

For grading reasons, $\delta_2=0$. We now turn to $\delta_1$. Recall that $\Sigma(2,3,5)$ is a quotient of the 3-sphere by an action of the binary icosohedral group in $SU(2)\cong S^3$ of order 120. In particular, it has a metric of positive scalar curvature, covered by the round metric on the 3-sphere. Using an equivariant ADHM construction, Austin showed that for this metric, all finite energy instanton moduli on $\R\times \Sigma$ are unobstructed, smooth manifolds, and
\[
	\breve{M}(\alpha_\Sigma,\theta_\Sigma)_0 = \{[A']\}
\]
consists of one unparametrized instanton up to gauge, see \cite[Proposition 4.1]{austin}. The metric may be chosen so that the covering involution $\tau$ that yields the orbifold $(S^3,K)$ is an isometry; see \cite{dunbar} for the classification of spherical orbifolds. Again, our discussion from Subsection \ref{sec:branchedcover} implies that each of $\breve{M}(\alpha,\theta)_0$ and $\breve{M}(\iota\alpha,\theta)_0$ consist of a unique instanton, $[A]$ and $\iota[A]$, respectively, which lift $[A']$. Consequently, $\delta_1:\Z_{(1)}^2\to \Z$ is nonzero, from which $\hinv (T_{3,5})\geqslant 1$ follows.

By Theorem \ref{thm:tildeconnsum}, a chain complex computing $I^\natural(K)$ for the $(3,5)$ torus knot is
\[
	\widetilde C_\ast(K) = \left(\Z_{(1)}^2 \oplus \Z_{(3)}^2\right) \oplus \left(\Z_{(0)}^2 \oplus \Z_{(2)}^2\right) \oplus \Z_{(0)}
\]
\[
	\widetilde d = \left[\begin{array}{ccc}0 & 0 & 0 \\ v & 0 & 0 \\ \delta_1 & 0 & 0 \end{array} \right]
\]
On the other hand, it is known that $I^\natural(K)$ has rank $7$. Indeed, for a general knot $K$, Kronheimer and Mrowka's spectral sequence from \cite{KM:unknot} provides the rank inequality
\[
	\text{rank } Kh^\text{red}(K) \geqslant \text{rank } I^\natural (K),
\]
where $Kh^\text{red}(K)$ is the reduced Khovanov homology of $K$. For the $(3,5)$ torus knot, the left hand side is equal to 7. Furthermore, $\text{rank}I^\natural (K)$ is bounded below by $|\Delta_K|$, the sum of the absolute values of the alexander polynomial, see \eqref{eq:alexanderbound}. For the $(3,5)$ torus knot $|\Delta_K|=7$, so $\text{rank}I^\natural(K)=7$ as claimed. It is straightforward to see that $\delta_1v=0$, for otherwise the homology of $\widetilde C_\ast(K)$ would have rank less than $7$. Thus $\hinv(T_{3,5})\leqslant 1$, and we obtain:

\begin{prop}
	For the $(3,5)$ torus knot $K$ we have $\hinv(K)=1$.
\end{prop}

\subsection{The $(3,4)$ torus knot}\label{sec:34}

The case of the $(3,4)$ torus knot $K$ is similar to that of the $(3,5)$ torus knot. This is because the double branched cover is again a finite quotient of $S^3$, this time by the binary tetrahedral group in $SU(2)\cong S^3$ of order 24. Again $\fC(\Sigma)^\tau=\fC(\Sigma)$ and there are two classes $\alpha_\Sigma$, $\beta_\Sigma\in\fC(\Sigma)$ apart from $\theta_\Sigma$, but in this case one is abelian, say $\beta_\Sigma$. Thus there are two irreducibles $\alpha,\iota\alpha\in \fC(K)$ that pull back to $\alpha_\Sigma$, and only one $\beta\in \fC(K)$ that pulls back to $\beta_\Sigma$. Using \cite{austin} and \cite[Lemma 7.5]{PS} we compute
\begin{equation}
	C_\ast(K) = \Z^2_{(1)} \oplus \Z_{(3)} = I_\ast(K),\label{eq:34}
\end{equation}
where each copy of $\Z_{(1)}$ is generated by one of $\alpha$ and $\iota\alpha$, while $\Z_{(3)}$ is generated by $\beta$. We sketch the computation for the reader.

We begin with the algorithm of Austin \cite{austin}. We start with the graph $G_{T^\ast}$ of the extended Dynkin diagram $\widetilde E_6$ for the binary tetrahedral group $T^\ast$. By the McKay Correspondence, this graph encodes the unitary representation theory of the binary tetrahedral group. We recall that $T^\ast$ is the subgroup of $SU(2)$, viewed as the unit quaternions, given by
\[
	\{ \pm 1, \pm i, \pm j, \pm k, \frac{1}{2}(\pm 1 \pm i \pm j \pm k)\}.
\]
Let $\{R_i\}$ be the set of unitary representations of $T^\ast$ up to isomorphism. We let $R_1$ be the trivial representation of $T^\ast$, $R_2$ the tautological representation including $T^\ast$ into $SU(2)$, and $R_3=R_7^\ast$ the two non-trivial $U(1)$ representations; note that $T^\ast$ has abelianization $\Z/3$. Then the vertices of $G_{T^\ast}$ are $\{R_i\}$, while $R_i$ and $R_j$ are connected by an edge if upon writing $R_i \otimes R_2 = \sum n_{ij} R_j$ we have $n_{ij}=1$; in general, $n_{ij}\in \{0,1\}$.

\begin{figure}[t]
\centering
\begin{tikzpicture}

	\draw (0,0) -- (4,0);
	\draw (2,0) -- (2,2);
	
	\draw[fill=red] (0,0) circle(.1);
	\draw[fill=white] (1,0) circle(.1);
	\draw[fill=white] (2,0) circle(.1);
	\draw[fill=green] (2,1) circle(.1);
	\draw[fill=white] (3,0) circle(.1);
	\draw[fill=red] (4,0) circle(.1);
	\draw[fill=black] (2,2) circle(.1);
	
	\node at (.4,1.5) {$G_{T^\ast}=\widetilde{E}_6$};
	\node at (2.5,2) {$R_1$};
	\node at (2.5,1) {$R_2$};
	\node at (0,-.5) {$R_3$};
	\node at (1,-.5) {$R_4$};
	\node at (2,-.5) {$R_5$};
	\node at (3,-.5) {$R_6$};
	\node at (4,-.5) {$R_7$};
	
	\draw (2,-1);
	
\end{tikzpicture}\quad\quad
\begin{tikzpicture}

    \draw (1,0) -- (3,0);
	\draw (2,0) -- (2,2);
	\draw (1,0) -- (2,-1);
	\draw (2,-1) -- (3,0);

	\draw[fill=white] (1,0) circle(.1);
	\draw[fill=white] (2,0) circle(.1);
	\draw[fill=green] (2,1) circle(.1);
	\draw[fill=white] (3,0) circle(.1);
	\draw[fill=black] (2,2) circle(.1);
	\draw[fill=red] (2,-1) circle(.1);
	
	\node at (1,1.5) {$G'_{T^\ast}$};
	
\end{tikzpicture}\quad\quad
\begin{tikzpicture}

	\draw (2,0) -- (2,2);

	\draw[fill=red] (2,0) circle(.1);
	\draw[fill=green] (2,1) circle(.1);
	\draw[fill=black] (2,2) circle(.1);
	\draw (2,-1);
	
	\node at (1,1.5) {$\mathscr{S}_{T^\ast}$};
	
	\node at (3.25,2) {$\theta = R_1\oplus R_1$};
	\node at (2.85,1) {$\alpha = R_2$};
	\node at (3.25,0) {$\beta = R_3\oplus R_7$};
	
\end{tikzpicture}
\caption{}\label{fig:e6}
\end{figure}

Next, let $G'_{T^\ast}$ be the graph obtained from $G_{T^\ast}$ by identifying conjugate representations. In our case, we only identify $R_3$ and $R_7$. Then form a new graph $\sS_{T^\ast}$ by extracting the vertices of $G'_{T^\ast}$ that are $SU(2)$ representations, and connecting two vertices by an edge if there is a path in $G'_{T^\ast}$ connecting the representations not passing through any other $SU(2)$ representations. Austin shows that for a vertex $v$ in $\sS_{T^\ast}$ corresponding to $\alpha_v\in \fC(\Sigma)$,
\[
	\dim M(\alpha_v,\theta_\Sigma)  \equiv  4 I(v) -3 \mod 8,
\]
where $I(v)$ is the length of a shortest path of edges from $v$ to $R_1\oplus R_1=\theta_\Sigma$ within $\mathscr{S}_{T^\ast}$. From Figure \ref{fig:e6} we compute in our case that $\dim M(\alpha_\Sigma,\theta_\Sigma) \equiv 1$ (mod 8) and $\dim M(\beta_\Sigma,\theta_\Sigma)\equiv 5$ (mod 8).

Finally, from \cite[Lemma 7.5]{PS} we have for any non-trivial $\gamma\in\fC(K)$ the relation
\[
	\dim M(\gamma,\theta) = \frac{1}{2} \left(\dim M(\Pi(\gamma),\theta_\Sigma) + 1 \right) \mod 4
\]
from which we conclude that $\text{gr}(\alpha)\equiv \text{gr}(\iota\alpha)\equiv  1$ (mod 4) and $\text{gr}(\beta)\equiv 3$ (mod 4). This verifies the computation of \eqref{eq:34}. Furthermore, from \cite[Proposition 4.1]{austin} and the fact that $\alpha_\Sigma$ is adjacent to $\theta_\Sigma$ in the graph $\sS_{T^\ast}$, the moduli space $\breve{M}(\alpha_\Sigma,\theta_\Sigma)_0$ is a point. Just as for the $(3,5)$ torus knot, we conclude that $\delta_1\neq 0$ on $C_1(K)$. Now
\[
	\widetilde C_\ast(K) = \left(\Z_{(1)}^2 \oplus \Z_{(3)}\right) \oplus \left(\Z_{(0)}^2 \oplus \Z_{(2)}\right) \oplus \Z_{(0)}
\]
has rank 7, while $I^\natural(K)$ has rank 5, because $\text{rank } Kh^\text{red}(K)=5=|\Delta_K|$. We again conclude that $\hinv(K)=1$, just as for the $(3,5)$ torus knot.

\begin{remark}
	An alternative approach to computing the gradings replaces Austin's algorithm with Fintushel and Stern's method \cite{fs}. However, the most important input above is in showing that $\delta_1\neq 0$, which follows from Austin's equivariant ADHM construction. $\diamd$
\end{remark}

\subsection{The irreducible homology of torus knots}\label{subsection:irr-torus-knots}

Fintushel and Stern \cite{fs} computed $I_\ast(\Sigma(p,q,r))$, Floer's $\Z/8$-graded instanton homology for the Brieskorn integer homology sphere $\Sigma=\Sigma(p,q,r)$. Here $p,q,r$ are relatively prime positive integers. The unperturbed critical set $\fC(\Sigma)$ is non-degenerate, and each generator has odd grading, so that $d=0$ and $C_\ast(\Sigma)=I_\ast(\Sigma)$.

Let $p,q$ be odd. Then $\Sigma(p,q,2)$ is the double branched cover of the torus knot $K=T_{p,q}$. Just as in Subsections \ref{sec:35} and \ref{sec:34}, $\fC(\Sigma)^\tau=\fC(\Sigma)$ by \cite[Proposition 8]{saveliev:brieskorn}. Furthermore, according to \cite[Lemma 4.1]{collin-saveliev}, the gradings of the generators in $C_\ast(\Sigma)$ are congruent mod 2 to the gradings of the corresponding generators in $C_\ast(K)$. Thus here also $d=0$ and $C_\ast(K)=I_\ast(K)$. By Theorem \ref{thm:herald} we obtain
\begin{equation}
	I(T_{p,q}) \cong \Z^{-\sigma(T_{p,q})/2}\label{eq:irrtorus}
\end{equation}
supported in odd gradings (mod 4). The conjecture in \cite[Section 7.4]{PS} is equivalent to the rank of $I_\ast(T_{p,q})$ being evenly distributed between gradings $1$ and $3$ (mod 4). 

The isomorphism \eqref{eq:irrtorus} also holds when one of $p$ or $q$ is even. In fact, in this case and when $p$ and $q$ are odd, one can directly count that the number of irreducible traceless $SU(2)$ representations in $\sX(T_{p,q})$ is equal to $-\sigma(T_{p,q})/2$, using \cite[Theorem 1]{klassen} and formula \eqref{signature} below. These representations correspond to nondegenerate critical points, and thus $C_\ast(T_{p,q})$ has rank $-\sigma(T_{p,q})/2$, providing an upper bound on the rank of $I_\ast(T_{p,q})$. Theorem \ref{thm:irrhomthy} provides the same lower bound, implying the isomorphism \eqref{eq:irrtorus}.

\subsection{More vanishing results for $\hinv(K)$}\label{sec:morevanishing}

We describe some simple conditions under which $\hinv(K)=0$ for a knot $K\subset S^3$. Write $\Delta_K=\sum a_j t^j$ for the symmetrized Alexander polynomial with $\Delta_K(1)=1$, and define 
\[
	|\Delta_K|:=\sum_j |a_j|
\]
to be the sum of the absolute values of its coefficients. The instanton homology $I^\natural (K)$ is isomorphic over $\C$ to the sutured instanton knot homology of $K$ \cite[Proposition 1.4]{KM:unknot}, the latter of which has an additional $\Z$-grading whose graded euler characteristic is the Alexander polynomial \cite{KM:alexander,lim}. This implies
\begin{equation}
	\text{rank } I^\natural (K) \geqslant |\Delta_K|. \label{eq:alexanderbound}
\end{equation}
On the other hand, by Theorem \ref{thm:tildeconnsum}, $I^\natural(K)$ is isomorphic to the homology of the $\cS$-complex $\widetilde C_\ast(K)=C_\ast(K)\oplus C_{\ast-1}(K)\oplus \Z$. In particular, if the irreducible complex $C_\ast(K)$ may be chosen to have $\frac{1}{2}(|\Delta_K|-1)$ generators, then the differential of $\widetilde C_\ast(K)$ is necessarily zero by \eqref{eq:alexanderbound}, and in particular, $\delta_1=\delta_2=0$, implying $\hinv(K)=0$. More precisely, given a perturbation for which $\fC_\pi(K)=\fC_\pi^\text{irr}\cup \{\theta\}$ is non-degenerate, if $|\fC_\pi(K)| = \frac{1}{2}(|\Delta_K|+1)$, then $\hinv(K)=0$. For the examples we will consider, the perturbation $\pi$ will always be zero, so that the implication may be written in terms of the traceless character variety $\sX(K)$. That is, if $\sX(K)$ is a finite set of nondegenerate points then:
\begin{equation}
	|\sX(K)| = \frac{1}{2}(|\Delta_K|+1) \quad \Longrightarrow \quad \hinv(K)=0.\label{eq:vanimp}
\end{equation}
Applying this to the case of torus knots, we obtain the following:

\begin{prop}
	For $K$ a $(p,q)$ torus knot, if $1+|\sigma(K)|=|\Delta_K|$, then $\hinv(K)=0$. This condition is satisfied by the following two families:
        \begin{align*}
 	       	(p,2pk+2), & \qquad k\geqslant 1,\;\; p\equiv \phantom{\pm}1 \pmod 2\\
 	       	(p,2pk\pm (2-p)), & \qquad k \geqslant 1,\;\; p \equiv \pm 1 \pmod 4
        \end{align*}	
\end{prop}
\begin{proof}
	For a $(p,q)$ torus knot we saw in the previous subsection that $\sX(K)$ is nondegenerate and has cardinality equal to $|\sigma(K)|/2+1$. 
	So the first part of the claim follows from \eqref{eq:vanimp}.
	
	The signature of torus knots is given by the following expression \cite{Lith:sign-torus,Kauffman:knots}
	\begin{equation}\label{signature}
		\sigma(T_{p,q})=(p-1)(q-1)-4\cdot B(p,q)
	\end{equation}
	where, assuming that $p$ is odd, we have
	\begin{equation*}\label{B(p,q)}
	  B(p,q)=\#\left \{(m,n) \mid 1\leq m \leq \frac{p-1}{2},\; 1\leq n \leq q-1,\; \frac{1}{2}\leq \frac{m}{p}+\frac{n}{q}\right \}.
	\end{equation*}
	A straightforward computation shows that:
	\begin{equation*}
		B(p,q)=\left\{
		\begin{array}{ll}
			\(\frac{p-1}{2}\)\(\frac{3kp+k-1\pm3}{2}\) & \;\;\; q=2kp\pm2\\
			\(\frac{p-1}{2}\)\((\frac{2k+1}{2})(\frac{3p+1}{2})-\frac{1}{2}\pm2\)\mp\lfloor \frac{p}{4}\rfloor& \;\;\; q=(2k+1)p\pm2\\
		\end{array}
		\right.
	\end{equation*}
	Thus we can use \eqref{signature} to write
	\begin{equation*}\label{signature-example}
		\sigma(T_{p,q})=\left\{
		\begin{array}{ll}
			-\(p-1\)\(kp+k\pm1\) &\;\;\; q=2kp\pm2\\
			-(p-1)\((2k+1)(\frac{p+1}{2})\pm2\)\pm 4\lfloor \frac{p}{4}\rfloor&\;\;\; q=(2k+1)p\pm2\\
		\end{array}
		\right.
	\end{equation*}

	The Alexander polynomial of a torus knot is given by the following formula:
	\[
	  \Delta_{T_{p,q}}=\frac{(t^{pq}-1)(t-1)}{(t^p-1)(t^q-1)}.
	\]
	In the case that $q=lp\pm 2$ with $l\geq 1$, this formula simplifies as follows:
	\begin{equation*}
		\Delta_{T_{p,q}}=\left\{
		\begin{array}{ll}
			1+(t-1)\(\sum_{i=1}^{p-1}\sum_{j=0}^{li-1}t^{pj+2i}+\sum_{i=1}^{\frac{p-1}{2}}t^{2i-1}\)&\;\; q=lp+2\\
			1+(t-1)\(\sum_{i=1}^{p-1}\sum_{j=1}^{li-1}t^{pj-2i}-\sum_{i=1}^{\frac{p-1}{2}}t^{-2i+1}\) &\;\; q=lp-2\\
		\end{array}
		\right.
	\end{equation*}
	These identities imply that:
	\begin{equation*}
		|\Delta_{T_{p,lp\pm 2}}|=\(\frac{p-1}{2}\)\((p+1)l\pm 2\)\pm1
	\end{equation*}
	Form the above identities it is easy to check that the identity $1+|\sigma(K)|=|\Delta_K|$ holds for the mentioned families of torus knots.
\end{proof}

The same method may be applied, for example, to find Montesinos knots $K$ with $\hinv(K)=0$. In this case, the double branched cover $\Sigma$ of $K$ is again a Brieskorn homology sphere $\Sigma(p,q,r)$, and by \cite{fs}, the unperturbed critical set $\fC(\Sigma)=\sX(\Sigma)$ is non-degenerate and of cardinality $2|\lambda(\Sigma)|+1$, where $\lambda$ is the Casson invariant. For example, the $(-2,3,7)$ pretzel knot satisfies \eqref{eq:vanimp}, and consequently has $\hinv=0$.

\newpage

%!TEX root = main.tex

\appendix

\section{Modified holonomy maps}\label{app:hol}

In this appendix we construct the modified holonomy maps that are used to define the $v$-map, and related maps, first used in Subsection \ref{sec:vmap} of the main text. The construction is inspired by one that is described in \cite[\S 7.3.2]{donaldson-book}.

Fix a pair $(Y,K)$ of an integer homology sphere and a knot. Let $\pi$ be a perturbation of the Chern-Simons functional for the pair $(Y,K)$, and for any $\alpha\in\fC^\text{irr}_\pi(Y,K)$ fix a neighborhood $\mathcal U_\alpha$ of $\alpha$ in $\sB(Y,K)$ together with a subset $\mathcal W_\alpha \subset\sC(Y,K)$ that maps diffeomorphically to $\mathcal U_\alpha$ via the projection. For instance, we may use a Coulomb chart around $\alpha$ to pick $\mathcal W_\alpha$. We call elements of $\mathcal W_\alpha$ lifts of the elements of $\mathcal U_\alpha$. Then for each pair of irreducible critical points $\alpha_i\in\fC^\text{irr}_\pi(Y,K)$, we obtain the holonomy map 
\[
	h_{\alpha_1\alpha_2}:\sB(Y,K;\alpha_1,\alpha_2) \longrightarrow S^1
\]
that is invariant with respect to translations, as outlined in Section \ref{sec:tilde}. Recall that the holonomy is of the adjoint connection along $\R\times \{p\}$ where $p\in K$ is a basepoint. The preferred reduction of the bundle along the singular locus guarantees that this holonomy lies in $S^1\cong SO(2)\subset SO(3)$. Below we implicitly assume that all holonomies are of the adjoint connection, without further mention. We wish to modify the map $h_{\alpha_1\alpha_2}$ so that the restrictions of this map to the moduli spaces of instantons have the properties listed in Subsection \ref{sec:vmap}. Our key tool in the construction of modified holonomy maps are {\it almost homomorphisms} of $S^1$.

\begin{definition}\label{almost-homo}
	An {\emph{almost homomorphism}} of $S^1$ is the data of sequence of continuous maps $h_k:[0,\infty)^k\times (S^1)^{k+1}\to S^1$ for any $k\geq 0$ satisfying the following properties.
	\begin{itemize}
		\item[(i)] For $1\leq i\leq k$, if $s_i\geq 1$, then
		\begin{align*}
			h_k(s_1,\dots,s_{i-1},s_i,s_{i+1},\dots&,s_k,g_0,g_1,\dots,g_{k})\\
			=h_{k-i}(s_{i+1},&\dots,s_k,g_i,\dots,g_{k})\cdot h_{i-1}(s_1,\dots,s_{i-1},g_0,\dots,g_{i-1}).
		\end{align*}
		\item[(ii)] For $1\leq i\leq k$, if $s_i\leq  \frac{1}{2}$, then
		\begin{align*}
			h_k(s_1,\dots,s_{i-1},s_i,s_{i+1},\dots&,s_k,g_0,g_1,\dots,g_{k+1})\\
			=h_{k-1}(s_1,\dots,s_{i-1}&,s_{i+1},\dots,s_k,g_0,g_1,\dots,g_{i-2},g_ig_{i-1},g_{i+1},\dots,g_{k+1}). \;\; \diamd
		\end{align*}
	\end{itemize}
\end{definition}

\begin{remark}
	An almost homomorphism of $S^1$ is essentially a version of a homotopy diagram (as defined in \cite{vogt}) for the category $S^1$, where $S^1$ is viewed as a groupoid. $\diamd$
\end{remark}

\begin{example}
	For any $k\geq 0$, define $h_k:[0,\infty)^k\times (S^1)^{k+1}\to S^1$ as
	\begin{equation}\label{can-almost-homo}
		h_k(s_1,s_2,\dots,s_k,g_0,g_1,\dots,g_{k})=g_k\cdot g_{k-1}\cdots g_0.
	\end{equation}
	Clearly this sequence of maps defines an almost homomorphism of $S^1$, which we call the {\it canonical almost homomorphism}. $\diamd$
\end{example}

\begin{definition}
	A {\it homotopy} of almost homomorphisms is a family of continuous maps 
	\[
	  \tilde h_k:[0,1]\times [0,\infty)^k\times (S^1)^{k+1}\to S^1
	\]
	such that for any $t\in [0,1]$, the sequence of maps $\{h_k^t\}_{k\geq 0}$ defined by $h_k^t(\cdot) = \tilde h_k(t,\cdot)$ is an almost homomorphism. We say that $\{\tilde h_k \}_{k\geq 0}$ is a homotopy between the almost homomorphisms $\{ h^0_k \}_{k\geq 0}$ and $\{ h^1_k \}_{k\geq 0}$. $\diamd$ \end{definition}

\begin{lemma}\label{cons-almost-hom}
	Let $h_0:S^1\to S^1$ be a continuous map and $\widetilde h_0:[0,1]\times S^1 \to S^1$ be a homotopy from the identity map to $h_0$. 
	Then $\tilde h_0$ can be extended to a homotopy of almost homomorphisms $\{\tilde h_k\}_{k\geq 0}$ such that $\{h^0_k\}_{k\geq 0}$ is the canonical almost homomorphism.
\end{lemma}

\begin{proof}
	We may construct $\tilde h_k:[0,1]\times [0,\infty)^k\times (S^1)^{k+1}\to S^1$ by induction on $k$. Suppose 
	$\tilde h_k$ is constructed for $k\leq i-1$ such that the maps $h_k^t$ satisfy (i) and (ii) of Definition \ref{almost-homo}, 
	and $h_k^0$ is given by \eqref{can-almost-homo}. We wish to extend this construction to $\tilde h_i$ while the above properties are still satisfied. 
	The properties in Definition \ref{almost-homo} and the condition on $h_i^0$ uniquely determine
	$\tilde h_i(t,s_1,s_2,\dots,s_k,g_0,g_1,\dots,g_{k})$ when $t=0$ or one of $s_i$ belongs to $[0,1/2]\cup [1,\infty)$. These points in the domain of 
	$\tilde h_i$ form a retract of $[0,1]\times [0,\infty)^k\times (S^1)^{k+1}$, and hence we may extend $\tilde h_i$ to the rest of its domain.
\end{proof}

The inductive argument in the proof of Lemma \ref{cons-almost-hom} can be used in the construction of almost homomorphisms that satisfy additional properties. The following lemma gives an instance of such properties.

\begin{lemma} \label{subset-1}
	For any $k\geq 0$, let $S_k$ be a subset of $[0,\infty)^k\times (S^1)^{k+1}$ such that
	$\cup_{k\geq 0}S_k$ is finite. Then there is an almost homomorphism $\{h_k\}_{k\geq 0}$ which is homotopic to the canonical almost homomorphism and such that $h_k$ evaluates to $1$ at the elements of $S_k$.
\end{lemma}

Now for each pair $\alpha_1$, $\alpha_2\in \fC^\text{irr}_\pi(Y,K)$, we define a modified holonomy map 
\[
  H_{\alpha_1\alpha_2}:\sB(Y,K;\alpha_1,\alpha_2) \longrightarrow S^1
\]
that depends on the choice of an almost homomorphism $\{h_k\}_{k\geq 0}$ which is homotopic to the canonical almost homomorphism, and the choices of $\mathcal U_\alpha$ and $\mathcal W_\alpha$ for each $\alpha\in \fC^{\text{irr}}_\pi(Y,K)$. For $[A] \in \sB(Y,K;\alpha_1,\alpha_2)$, let $[A_t]\in \sB(Y,K)$ denote the restriction of $[A]$ to $\{t\}\times Y$. Let $V^0_{[A]}$ be an open subspace of $\Bbb R$ such that $t\in V^0_{[A]}$ if and only if $[A_t]$ belongs to $\mathcal U_\alpha$ for some choice of $\alpha$. We also need an open subspace $V_{[A]}$ of $V^0_{[A]}$ consisting of the points $t$ such that $t$ is in an interval $I$ of length at least $1/3$ and $I\subset V^0_{[A]}$. 

The definition of $V_{[A]}$ implies that it is a union of finitely many open intervals. Thus there is a sequence of real numbers
\[
  b_0<a_1<b_1<\dots <a_k<b_k<a_{k+1}
\]
such that 
\[
  V_{[A]}=(-\infty,b_0) \cup (a_1,b_1) \cup \dots \cup (a_k,b_k) \cup (a_{k+1},\infty).
\]
For $1\leq i \leq k$, define
\[
  s_i:=b_i-a_i,\hspace{2cm} c_i:=\frac{a_i+b_i}{2}.
\]
Since $c_i\in V_{[A]}$, there is a preferred lift of $[A_{c_i}]$ to $\mathcal{W}_\alpha\subset \sC(Y,K)$ for some $\alpha$. In particular, we may obtain a well-defined $g_i\in S^1$ by taking the holonomy of $[A]$ along $\{y\}\times [c_i,c_{i+1}]$ for any $1\leq i\leq k-1$. Here $y$ is the basepoint on the knot $K$ as before. Similarly we may take the holonomy of $[A]$ along $\{y\}\times (-\infty,c_{1}]$ and $\{y\}\times [c_k,\infty)$ to respectively obtain $g_0$ and $g_{k}$. Clearly $h_{\alpha_1\alpha_2}([A])$ is equal to the product $g_k\cdot g_{k-1}\cdots g_0$. We define
\[
  H_{\alpha_1\alpha_2}([A])=h_k(s_1,s_2,\dots,s_k,g_0,g_1,\dots,g_{k}).
\]

\begin{prop}\label{Ha-props}
	The map $H_{\alpha_1\alpha_2}:\sB(Y,K;\alpha_1,\alpha_2) \to S^1$ is continuous and invariant with respect to translation action on 
	$\sB(Y,K;\alpha_1,\alpha_2)$. We may also assume that any given finite subset 
	$S$ of  $\sB(Y,K;\alpha_1,\alpha_2)$ is mapped into $1\in S^1$.
\end{prop}
\begin{proof}
	Property (ii) in Definition \ref{almost-homo} implies that if we remove the intervals in $V_{[A]}$ with length at most
	 $1/2$ to obtain $V_{[A]}'$, and then follow a similar definition as that of $H_{\alpha_1\alpha_2}([A])$ with $V_{[A]}'$
	 we obtain the same value of $H_{\alpha_1\alpha_2}([A])$. For $[A']\in \sB(Y,K;\alpha_1,\alpha_2)$ that is close
	 to $A$, the subspace $V_{[A']}'$ of $\R$ is close to the subspace of $V_{[A]}$ given by intervals of length at least
	 $1/2$. From this and continuity of almost homomorphisms, it is easy to see that $H_{\alpha_1\alpha_2}$ is continuous. It is 
	 clear from the definition that $H_{\alpha_1\alpha_2}$ is invariant with respect to the translation action. 
	 The last part of the proposition is a consequence of Lemma \ref{subset-1}.
\end{proof}

Singular connections on the cylinder associated to $(Y,K)$ can be glued to each other to form new singular connections. For instance for any $\alpha_1,\alpha_3\in\fC_\pi(Y,K)$ and $\alpha_2\in\fC^\text{irr}_\pi(Y,K)$ we have the following gluing map:
\[
 {\rm Gl}:\sB(Y,K;\alpha_1,\alpha_2)\times \R_{>0} \times \sB(Y,K;\alpha_2,\alpha_3)\to \sB(Y,K;\alpha_2,\alpha_3)
\]
To be more specific, let $f:\R\to [0,1]$ be a smooth function such that $f(t)=0$ for $t\leq -1$ and $f(t)=1$ for $t\geq 1$. Suppose $[A]\in \sB(Y,K;\alpha_1,\alpha_2)$ and $[A']\in \sB(Y,K;\alpha_2,\alpha_3)$. We may assume that the representative connection $A$ for $[A]$ is chosen such that it is in temporal gauge and is asymptotic to the fixed lift in $\mathcal{W}_{\alpha_2}$ of $\alpha_2$ as $t\to \infty$. This assumption fixes $A$ uniquely. Similarly, let $A'$ be in temporal gauge and asymptotic to the lift of $\alpha_2$ as $t\to -\infty$. Now ${\rm Gl}([A],T,[A'])$ is the element of $\sB(Y,K;\alpha_1,\alpha_3)$ represented by the connection 
\[
  f\left(\frac{t}{1+T}\right)\tau_T^*(A)+\left(1-f\left(\frac{t}{1+T}\right)\right)\tau_{-T}^*(A)
\]
where $\tau_s:\R\times Y\to \R\times Y$ is the translation $\tau_s(t,y)=(t+s,y)$. The following proposition describes the behavior of modified holonomy maps with respect to the gluing map ${\rm Gl}$.

\begin{prop}\label{Gl-Ha-comp}
	For any $[A]\in \sB(Y,K;\alpha_1,\alpha_2)$ and $[A']\in \sB(Y,K;\alpha_2,\alpha_3)$, we have
	\[
	  \lim_{T\to \infty} H_{\alpha_1\alpha_3}({\rm Gl}([A],T,[A']))=H_{\alpha_2\alpha_3}([A'])\cdot H_{\alpha_1\alpha_2}([A]).
	\]
\end{prop}

\begin{proof}
	For fixed $[A],[A']$, as $T\to \infty$, there is a corresponding index $i$ with parameter $s_i\to \infty$, and the result follows immediately from property (i) in Definition \ref{almost-homo}.
\end{proof}

The restriction of the maps $H_{\alpha_1,\alpha_2}$ to the moduli spaces $\breve{M}(\alpha_1,\alpha_2)_{d}$ can be used to construct the maps  required for the definition of the map $v$ in Subsection \ref{sec:vmap}. By taking $S$ in Proposition \ref{Ha-props} to be the union of $0$-dimensional moduli spaces, we may assume that property (H1) in Subsection \ref{sec:vmap} is satisfied. Properties (H2) and (H3) hold by similar arguments as in Proposition \ref{Gl-Ha-comp}. The only missing point is that our maps at this point are only continuous. (In fact, continuity of these maps would be enough to define the map $v$ and prove Proposition \ref{prop:vmap}.) By induction on $d$ and for $d\leq 2$, we may approximate the constructed maps $\breve{M}(\alpha_1,\alpha_2)_{d} \to S^1$ such that properties (H1)--(H3) still hold and the restriction of $H_{\alpha_1,\alpha_2}$ to each stratum of $\breve{M}(\alpha_1,\alpha_2)_{d}$ is smooth. 

Next, we turn to the definition of modified holonomy maps for a cobordism of pairs $(W,S):(Y,K)\to (Y',K')$ that are compatible with the chosen modified holonomy maps for $(Y,K)$ and $(Y',K')$ in an appropriate sense. To achieve this, first we introduce the corresponding notion for almost homomorphisms.

\begin{definition}\label{almost-hom-cont}
	Suppose $\{ h_k \}_{k\geq 0}$ and $\{ h'_k \}_{k\geq 0}$ are almost homomorphisms of $S^1$. A {\it continuation from $\{ h_k \}_{k\geq 0}$ to 
	$\{ h'_k \}_{k\geq 0}$} is the data of a sequence of continuous maps \[\bh_{k,l}:[0,\infty)^{k+l}\times (S^1)^{k+l+1}\to S^1\]
	for any $k,l\geq 0$ satisfying the following properties.\\
	
	\noindent \text{(i)} For $1\leq i\leq k$, if $s_i\geq 1$, then
		\begin{align*}
			\bh_{k,l}(s_1&,\dots,s_{i-1},s_i,s_{i+1},\dots,s_k,s_1',\dots,s_l',g_1,\dots,g_{k},{\rm g},g_1',\dots,g_{l}')=\\
		\bh_{k-i,l}(&s_{i+1},\dots,s_k,s_1',\dots,s_l',g_{i+1},\dots,g_{k},{\rm g},g_1',\dots,g_{l}')\cdot h_{i-1}(s_1,\dots,s_{i-1},g_1,\dots,g_{i})	,
		\end{align*}
		and for $1\leq i\leq l$, if $s_i'\geq 1$, then 
		\begin{align*}
			\bh_{k,l}(&s_1,\dots,s_k,s_1',\dots,s_{i-1}',s_i',s_{i+1}',\dots,s_l',g_1,\dots,g_{k},{\rm g},g_1',\dots,g_{l}')=\\
			h_{l-i}'(&s_{i+1}',\dots,s_{l}',g_i',\dots,g_{l}')\cdot \bh_{k,i-1}(s_1,\dots,s_k,s_1',\dots,s_{i-1}',g_1,\dots,g_{k},{\rm g},g_1',\dots,g_{i-1}').		
		\end{align*}
		
		\noindent \text{(ii)} For $1\leq i\leq k$, if $s_i\leq  \frac{1}{2}$, then
		\begin{align*}
			\bh_{k,l}(s_1&,\dots,s_{i-1},s_i,s_{i+1},\dots,s_k,s_1',\dots,s_l',g_1,\dots,g_{k},{\rm g},g_1',\dots,g_{l}')=\\
			\bh_{k-1,l}(s_1&,\dots,s_{i-1},s_{i+1},\dots,s_k,s_1',\dots,s_l',g_1,\dots,g_{i-1},g_ig_{i+1},\dots,g_{k},{\rm g},g_1',\dots,g_{l}').
		\end{align*}
		In the above identity, if $i=k$, then $g_ig_{i+1}$ should be replaced with $g_i{\rm g}$.
		Similarly, for $1\leq i\leq l$, if $s_i'\leq  \frac{1}{2}$, then
		\begin{align*}
			\bh_{k,l}(&s_1,\dots,s_k,s_1',\dots,s_{i-1}',s_i',s_{i+1}',\dots,s_l',g_1,\dots,g_{k},{\rm g},g_1',\dots,g_{l}')=\\
			\bh_{k,l-1}(&s_1,\dots,s_k,s_1',\dots,s_{i-1}',s_{i+1}',\dots,s_l',g_1,\dots,g_{k},{\rm g},g_1',\dots,g_{i-1}'g_i',g_{i+1}',\dots,g_{l}'),
		\end{align*}
		and if $i=1$, then $g_{i-1}'g_i'$ should be replaced with ${\rm g}g_i'$. $\diamd$ 
\end{definition}

\begin{example}\label{triv-cont}
	Given any almost homomorphism $\{ h_k \}_{k\geq 0}$, there is a trivial continuation from $\{ h_k \}_{k\geq 0}$ to itself given by $\bh_{k,l}=h_{k+l}$ for 
	any $k,l \geq 0$. $\diamd$
\end{example}

\begin{lemma}\label{cons-cont-almost-hom}
	Let $\{ h_k \}_{k\geq 0}$ and $\{ h'_k \}_{k\geq 0}$ be homotopic almost homomorphisms of $S^1$. Then there is a continuation 
	$\{\bh_{k,l}\}_{k,l\geq 0}$ from $\{ h_k \}_{k\geq 0}$ to $\{ h'_k \}_{k\geq 0}$.
	Moreover, we can assume that the continuation maps evaluate to $1$ at any finite subset $\cup_{k,l\geq 0}S_{k,l}$ contained in the union of the domains of 
	$\{\bh_{k,l}\}_{k,l\geq 0}$.
\end{lemma}
\begin{proof}
	We wish to construct the maps $\bh_{k,l}$ by induction on $k+l$. In each step of induction, the properties in Definition \ref{almost-hom-cont}
	 determine $\bh_{k,l}$ on a subset of the domain. However, there might be an obstruction to extending the map to the whole domain.
	To resolve this issue, we use a similar trick as in Lemma \ref{cons-almost-hom} by constructing a stronger object. Suppose the family 
	of almost homomorphisms $\{\tilde h^t_k\}_{k=0}$ gives a homotopy from $\{ h_k \}_{k\geq 0}$ to $\{ h'_k \}_{k\geq 0}$. Then 
	we inductively construct a map
	\[
	  \widetilde \bh_{k,l}:[0,1]\times [0,\infty)^{k+l}\times (S^1)^{k+l+1}\to S^1
	\]
	such that for any $t\in [0,1]$, the sequence of maps $\{\bh_{k,l}^t\}_{k\geq 0}$ defined by $\bh_{k,l}^t(\cdot) = \widetilde \bh_{k,l}(t,\cdot)$ is 
	a continuation from $\{ h_k \}_{k\geq 0}$ to $\{\tilde h^t_k\}_{k=0}$ and the continuation $\{\bh_{k,l}^0\}_{k\geq 0}$ from 
	$\{ h_k \}_{k\geq 0}$ to $\{ h_k \}_{k\geq 0}$ is provided by Example \ref{triv-cont}. In particular, $\{\bh_{k,l}^1\}_{k\geq 0}$ gives a continuation 
	from $\{ h_k \}_{k\geq 0}$ to $\{ h'_k \}_{k\geq 0}$. Now there is no obstruction in carrying out the induction 
	step. A straightforward examination of the construction shows that the property in the second part of this lemma can 
	be guaranteed in this inductive argument.
 \end{proof}

A similar inductive argument can be used to prove the following lemma.

\begin{lemma}\label{hom-cons-cont-almost-hom}
	There is a homotopy between any two choices of continuations $\{\bh^0_{k,l}\}_{k,l\geq 0}$ and $\{\bh^1_{k,l}\}_{k,l\geq 0}$ provided by Lemma \ref{cons-cont-almost-hom}. To be more precise, there is a continuous map
	\[
	  \widetilde \bh_{k,l}:[0,1]\times [0,\infty)^{k+l}\times (S^1)^{k+l+1}\to S^1
	\]
	for any $k,l\geq 0$ such that $\{\bh^t_{k,l}\}_{\{k,l\geq 0\}}$, given by $\bh^t_{k,l}:=\widetilde \bh_{k,l}(t,\cdot)$, is a continuation that is equal to the given continuations for $t=0$ and $1$.
	Moreover, for any finite subset $S$ of the union over $k$ and $l$ of the spaces $(0,1)\times [0,\infty)^{k+l}\times (S^1)^{k+l+1}$, we may assume that the maps $\widetilde \bh_{k,l}$ evaluate to $1$ at the points of $S$.
\end{lemma}

Now we are ready to define modified holonomy maps for a cobordism of pairs $(W,S):(Y,K)\to (Y',K')$ and a path $\gamma$ connecting the basepoints $y$ and $y'$ of $K$ and $K'$. For the pair $(Y,K)$, fix an almost homomorphism $\{h_k\}_{k\geq 0}$ homotopic to the canonical almost homomorphism, neighborhood $\mathcal U_\alpha$, and a lift $\mathcal W_\alpha$ of $\mathcal U_\alpha$ for each $\alpha\in \fC^{\text{irr}}_\pi(Y,K)$. Fix similar choices for the pair $(Y',K')$. Using Lemma \ref{cons-cont-almost-hom}, we may find a continuation $\{\bh_{k,l}\}_{k,l\geq 0}$ from $\{ h_k \}_{k\geq 0}$ to $\{ h'_k \}_{k\geq 0}$. For $[A]\in \sB(W,S;\alpha,\alpha')$ and $t\in (-\infty,0)$, let $[A_t]\in \sB(Y,K)$ denote the restriction of $[A]$ to $\{t\}\times Y$ on the cylindrical end associated to $(Y,K)$. 

Let $V_{[A]}\subset (-\infty,0)$ consist of $t_0$ such that there is an interval $I$ of size at least $1/3$ containing $t_0$ such that for any $t\in I$, $[A_t]$ belongs to one of $\mathcal U_\alpha$. Similarly, we define a subset $V_{[A]}'\subset (0,\infty)$ using the restriction of $[A]$ to $\{t\}\times Y'$ on the cylindrical end associated to $(Y',K')$. There are real numbers
\[
  b_0<a_1<b_1<\dots <a_k<b_k<0<a_1'<b_1'<\dots <a_l<b_l'<a_{l+1}',
\]
such that the sets $V_{[A]}$ and $V'_{[A]}$ have the following form:
\[
  V_{[A]}=(-\infty,b_0) \cup (a_1,b_1) \cup \dots \cup (a_k,b_k),\hspace{1cm}V_{[A]}'=(a_1',b_1') \cup \dots \cup (a_l',b_l') \cup (a_{l+1}',\infty).
\]
For $1\leq i \leq k$ and $1\leq j\leq l$, define the following real numbers:
\[
  s_i:=b_i-a_i,\hspace{1cm} c_i:=\frac{a_i+b_i}{2},\hspace{2cm}s_j':=b_j'-a_j',\hspace{1cm} c_j':=\frac{a_j'+b_j'}{2}.
\]
For $2\leq i\leq k$ and $1\leq j\leq l-1$, let $g_i\in S^1$ and $g_i'\in S^1$ be respectively the holonomy of $[A]$ along $\{y\}\times [c_{i-1},c_{i}]$  and $\{y'\}\times [c_j',c'_{j+1}]$. We also let $g_1$ be the holonomy of $[A]$ along $\{y\}\times (-\infty,c_{1}]$, $g_l'$ be the holonomy of $[A]$ along $\{y'\}\times [c_l',\infty)$, and ${\rm g}$ be the holonomy of $[A]$ along $\gamma$ from $\{y\}\times\{c_k\}$ to $\{y'\}\times\{c_1'\}$. We define
\[
  H_{\alpha\alpha'}^\gamma([A])=\bh_{k,l}(s_1,\dots,s_k,s_1',\dots,s_l',g_1,\dots,g_{k},{\rm g},g_1',\dots,g_{l}').
\]

Similar to the case of cylinders, $H_{\alpha\alpha'}^\gamma:\sB(W,S;\alpha,\alpha') \to S^1$ is continuous, and we assume it maps a given finite subset $S$ of $\sB(W,S;\alpha,\alpha')$ into $1\in S^1$. There are gluing maps
\[
 {\rm Gl_{in}}:\sB(Y,K;\alpha_1,\alpha_2)\times \R_{>0} \times \sB(W,S;\alpha_2,\alpha')\to\sB(W,S;\alpha_1,\alpha'),
\]
\[
 {\rm Gl_{out}}:\sB(W,S;\alpha,\alpha_1') \times \R_{>0} \times \sB(Y',K';\alpha_1',\alpha_2')\to\sB(W,S;\alpha,\alpha_2')
\]
such that the following relations hold:
\[
  \lim_{T\to \infty} H^\gamma_{\alpha_1\alpha'}({\rm Gl_{in}}([A],T,[A']))=H_{\alpha_2\alpha'}^\gamma([A'])\cdot H_{\alpha_1\alpha_2}([A]),
\]
\[
  \lim_{T\to \infty} H_{\alpha\alpha_2'}^\gamma({\rm Gl_{out}}([A],T,[A']))=H_{\alpha_1'\alpha_2'}([A'])\cdot H^\gamma_{\alpha\alpha_1'}([A]).
\]

We use the restriction of the map $H_{\alpha\alpha'}^\gamma$ to the moduli spaces $M(W,S;\alpha,\alpha')_1$, after a smooth approximation, to define the map $\mu$ in Subsection \ref{cut-down}. Verifying Proposition \ref{prop:vmaprels2} uses compatibility of $H_{\alpha\alpha'}^\gamma$ with respect to the gluing maps spelled out in the last paragraph. Lemma \ref{hom-cons-cont-almost-hom} allows us to show that the map $\mu$ (and the remaining components of the cobordism map associated to $(W,S)$) are invariant of the auxiliary choices (including the choice of the continuation) up to chain homotopy.

\newpage

\bibliography{references}
%\bibliography{main.bbl}
\bibliographystyle{hplain}
\Addresses
\end{document}